\documentclass[numbers=enddot,12pt,final,onecolumn,notitlepage]{scrartcl}%
\usepackage[headsepline,footsepline,manualmark]{scrlayer-scrpage}
\usepackage[all,cmtip]{xy}
\usepackage{amssymb}
\usepackage{amsmath}
\usepackage{amsthm}
\usepackage{framed}
\usepackage{comment}
\usepackage{color}
\usepackage{hyperref}
\usepackage[sc]{mathpazo}
\usepackage[T1]{fontenc}
\usepackage{tikz}
\usepackage{needspace}
\usepackage{tabls}
\usepackage{wasysym}
\usepackage{easytable}
\usepackage{ytableau}
\providecommand{\U}[1]{\protect\rule{.1in}{.1in}}
\usetikzlibrary{arrows.meta}
\usetikzlibrary{chains}
\newcounter{exer}

\numberwithin{exer}{subsection}
\theoremstyle{definition}
\newtheorem{theo}{Theorem}[subsection]
\newenvironment{theorem}[1][]
{\begin{theo}[#1]\begin{leftbar}}
{\end{leftbar}\end{theo}}
\newtheorem{lem}[theo]{Lemma}
\newenvironment{lemma}[1][]
{\begin{lem}[#1]\begin{leftbar}}
{\end{leftbar}\end{lem}}
\newtheorem{prop}[theo]{Proposition}
\newenvironment{proposition}[1][]
{\begin{prop}[#1]\begin{leftbar}}
{\end{leftbar}\end{prop}}
\newtheorem{defi}[theo]{Definition}
\newenvironment{definition}[1][]
{\begin{defi}[#1]\begin{leftbar}}
{\end{leftbar}\end{defi}}
\newtheorem{remk}[theo]{Remark}
\newenvironment{remark}[1][]
{\begin{remk}[#1]\begin{leftbar}}
{\end{leftbar}\end{remk}}
\newtheorem{coro}[theo]{Corollary}
\newenvironment{corollary}[1][]
{\begin{coro}[#1]\begin{leftbar}}
{\end{leftbar}\end{coro}}
\newtheorem{conv}[theo]{Convention}

\newtheorem{quest}[theo]{Question}
\newenvironment{question}[1][]
{\begin{quest}[#1]\begin{leftbar}}
{\end{leftbar}\end{quest}}
\newtheorem{warn}[theo]{Warning}
\newenvironment{warning}[1][]
{\begin{warn}[#1]\begin{leftbar}}
{\end{leftbar}\end{warn}}
\newtheorem{conj}[theo]{Conjecture}

\newtheorem{exam}[theo]{Example}
\newenvironment{example}[1][]
{\begin{exam}[#1]\begin{leftbar}}
{\end{leftbar}\end{exam}}
\newtheorem{exmp}[exer]{Exercise}

\newenvironment{statement}{\begin{quote}}{\end{quote}}
\newenvironment{fineprint}{\medskip \begin{small}}{\end{small} \medskip}
\let\sumnonlimits\sum
\let\prodnonlimits\prod
\let\cupnonlimits\bigcup
\let\capnonlimits\bigcap
\renewcommand{\sum}{\sumnonlimits\limits}
\renewcommand{\prod}{\prodnonlimits\limits}
\renewcommand{\bigcup}{\cupnonlimits\limits}
\renewcommand{\bigcap}{\capnonlimits\limits}
\newenvironment{verlong}{}{}

\newenvironment{noncompile}{}{}

\excludecomment{verlong}
\includecomment{vershort}
\excludecomment{noncompile}
\excludecomment{teachingnote}

\newcommand{\powset}[2][]{\ifthenelse{\equal{#2}{}}{\mathcal{P}\left(#1\right)}{\mathcal{P}_{#1}\left(#2\right)}}

\newcommand{\are}{\ar@{-}}

\newcommand{\MurpF}{\vphantom{\int_g^g}\mathcal{F}}
\ihead{Rook sums in the symmetric group algebra}
\ohead{page \thepage}
\begin{document}

\title{Rook sums in the symmetric group algebra}
\author{Darij Grinberg}
\date{draft, July 29, 2025}
\maketitle

\begin{abstract}
\textbf{Abstract.} Let $\mathcal{A}$ be the group algebra $\mathbf{k}\left[
S_{n}\right]  $ of the $n$-th symmetric group $S_{n}$ over a commutative ring
$\mathbf{k}$. For any two subsets $A$ and $B$ of $\left[  n\right]  $, we
define the elements%
\[
\nabla_{B,A}:=\sum_{\substack{w\in S_{n};\\w\left(  A\right)  =B}%
}w\ \ \ \ \ \ \ \ \ \ \text{and}\ \ \ \ \ \ \ \ \ \ \widetilde{\nabla}%
_{B,A}:=\sum_{\substack{w\in S_{n};\\w\left(  A\right)  \subseteq B}}w
\]
of $\mathcal{A}$. We study these elements, showing in particular that their
minimal polynomials factor into linear factors (with integer coefficients). We
express the product $\nabla_{D,C}\nabla_{B,A}$ as a $\mathbb{Z}$-linear
combination of $\nabla_{U,V}$'s.

More generally, for any two set compositions (i.e., ordered set partitions)
$\mathbf{A}$ and $\mathbf{B}$ of $\left\{  1,2,\ldots,n\right\}  $, we define
$\nabla_{\mathbf{B},\mathbf{A}}\in\mathcal{A}$ to be the sum of all
permutations $w\in S_{n}$ that send each block of $\mathbf{A}$ to the
corresponding block of $\mathbf{B}$. This generalizes $\nabla_{B,A}$. The
factorization property of minimal polynomials does not extend to the
$\nabla_{\mathbf{B},\mathbf{A}}$, but we describe the ideal spanned by the
$\nabla_{\mathbf{B},\mathbf{A}}$ and a further ideal complementary to it.
These two ideals have a \textquotedblleft mutually
annihilative\textquotedblright\ relationship, are free as $\mathbf{k}%
$-modules, and appear as annihilators of tensor product $S_{n}$%
-representations; they are also closely related to Murphy's cellular bases,
Specht modules, pattern-avoiding permutations and even some algebras appearing
in quantum information theory.

\end{abstract}
\tableofcontents

\section*{***}

\emph{Rook theory} is the study of permutations $w$ in a symmetric group
$S_{n}$ that avoid certain input-output pairs (given by \textquotedblleft
boards\textquotedblright, i.e., sets of allowed input-output pairs). Research
done so far -- e.g., \cite{GJW75,BCHR11,LoeRem09} -- has mostly focused on the
enumeration of such permutations (known as \textquotedblleft rook
placements\textquotedblright). In this work, we set out in a new direction:
Instead of counting the rook placements for a given board, we study their sum
in the symmetric group algebra $\mathbf{k}\left[  S_{n}\right]  $ over a
commutative ring $\mathbf{k}$. As with any elements of such a group algebra,
we can ask for their spectral and representation-theoretical properties --
what ideals do they generate? how well do their minimal polynomials factor?

Let us give a quick overview, starting with an example that does not actually
fit into our theory. Let $\mathcal{A}$ be the group algebra $\mathbf{k}\left[
S_{n}\right]  $ of a symmetric group $S_{n}$. For any subset $T$ of $\left[
n\right]  \times\left[  n\right]  $ (where $\left[  n\right]  $ denotes the
set $\left\{  1,2,\ldots,n\right\}  $ as is usual in combinatorics), we set%
\[
\nabla_{T}:=\sum_{\substack{w\in S_{n};\\\left(  i,w\left(  i\right)  \right)
\in T\text{ for each }i\in\left[  n\right]  }}w\in\mathcal{A}.
\]
In rook-theoretical terms (see, e.g., \cite[\S 1.2]{BCHR11}), this is the sum
of the $n$-rook placements on the board $T$.

For example, if $T=\left\{  \left(  i,j\right)  \in\left[  n\right]
\times\left[  n\right]  \ \mid\ i\neq j\right\}  $, then $\nabla_{T}$ is the
sum of all derangements in $S_{n}$, and is well-behaved in many ways: It lies
in the center of $\mathcal{A}=\mathbf{k}\left[  S_{n}\right]  $ (since the set
of all derangements is fixed under conjugation), and its minimal polynomial
(over a field $\mathbf{k}$) factors into linear factors (this is true for any
element in the center of $\mathcal{A}$, because the center of $\mathbb{Q}%
\left[  S_{n}\right]  $ is split semisimple; this is well-known folklore).

Even simpler examples are $T=\left\{  \left(  i,i\right)  \ \mid\ i\in\left[
n\right]  \right\}  $ (yielding $\nabla_{T}=\operatorname*{id}$) and
$T=\left[  n\right]  \times\left[  n\right]  $ (here, $\nabla_{T}$ is the sum
of all permutations in $S_{n}$).

These examples are far from representative. In general, $\nabla_{T}$ will
rarely belong to the center of $\mathbf{k}\left[  S_{n}\right]  $, nor will
its minimal polynomial often factor nicely. For example, if $n=5$ and
$T=\left\{  \left(  i,j\right)  \ \mid\ j\neq i+1\right\}  $, then the minimal
polynomial of $\nabla_{T}$ has irreducible factors of degrees $1$, $4$, $5$
and $6$ (over $\mathbb{Q}$). Thus, the behavior of $\nabla_{T}$ depends
starkly on the structure of $T$.

In this paper, we will study two types of boards $T$. The first one has the
form \textquotedblleft$\left[  n\right]  \times\left[  n\right]  $ with a
rectangle cut out\textquotedblright. Thus, we will define the so-called
\emph{rectangular rook sums}
\[
\nabla_{B,A}:=\sum_{\substack{w\in S_{n};\\w\left(  A\right)  =B}%
}w\ \ \ \ \ \ \ \ \ \ \text{and}\ \ \ \ \ \ \ \ \ \ \widetilde{\nabla}%
_{B,A}:=\sum_{\substack{w\in S_{n};\\w\left(  A\right)  \subseteq
B}}w\ \ \ \ \ \ \ \ \ \ \text{in }\mathcal{A}%
\]
for two subsets $A$ and $B$ of $\left[  n\right]  $. (The $\widetilde{\nabla
}_{B,A}$ here is the $\nabla_{T}$ for $T=\left(  \left[  n\right]
\times\left[  n\right]  \right)  \setminus\left(  A\times\left(  \left[
n\right]  \setminus B\right)  \right)  $. The $\nabla_{B,A}$ are -- strictly
speaking -- redundant, as they equal either $\widetilde{\nabla}_{B,A}$ or $0$
depending on whether $\left\vert A\right\vert =\left\vert B\right\vert $ or
not. Conversely, however, the $\widetilde{\nabla}_{B,A}$ can be expressed as
sums of $\nabla_{B,A}$'s, so the two families of elements are closely
related.) Section \ref{sec.rooksum} is devoted to the study of these elements.
We will show that they span an ideal of $\mathcal{A}$, that they satisfy an
explicit multiplication rule (Theorem \ref{thm.Nabla.prod0}), and that their
minimal polynomials factor into linear factors (Corollary
\ref{cor.Nablatil.triangular}).

In Section \ref{sec.row-to-row}, we will generalize the $\nabla_{B,A}$ to a
wider family. Namely, we define a \emph{set decomposition} of $\left[
n\right]  $ to be a tuple of disjoint subsets of $\left[  n\right]  $ (called
\emph{blocks}) whose union is $\left[  n\right]  $. (Empty blocks are
allowed.) Now, if $\mathbf{A}=\left(  A_{1},A_{2},\ldots,A_{k}\right)  $ and
$\mathbf{B}=\left(  B_{1},B_{2},\ldots,B_{k}\right)  $ are two set
decompositions of $\left[  n\right]  $ having the same length, then we define
the element%
\[
\nabla_{\mathbf{B},\mathbf{A}}:=\sum_{\substack{w\in S_{n};\\w\left(
A_{i}\right)  =B_{i}\text{ for all }i}}w\ \ \ \ \ \ \ \ \ \ \text{of
}\mathcal{A}.
\]
For $k=2$, these recover the elements $\nabla_{B,A}$ studied above. In
general, $\nabla_{\mathbf{B},\mathbf{A}}$ is $\nabla_{T}$ for a certain board
$T$ that is obtained from $\left[  n\right]  \times\left[  n\right]  $ by
cutting out multiple rectangles. While these $\nabla_{\mathbf{B},\mathbf{A}}$
lack the multiplication rule and the factoring minimal polynomials of the
$\nabla_{B,A}$'s (at least we are not aware of a multiplication rule), they
have their own share of interesting properties. In fact, they have already
appeared in the works of Canfield/Williamson \cite{CanWil89} and Murphy
\cite{Murphy92,Murphy95}, where they were used to construct bases of
$\mathcal{A}$ now known as the \emph{Murphy cellular bases}.

Unlike Murphy, we are interested in the $\nabla_{\mathbf{B},\mathbf{A}}$ for
all pairs $\left(  \mathbf{A},\mathbf{B}\right)  $ of equal-length set
decompositions, not just ones that come from standard bitableaux. We consider
-- for any given $k\in\mathbb{N}$ -- the span $\mathcal{I}_{k}$ of all such
elements $\nabla_{\mathbf{B},\mathbf{A}}$ where $\mathbf{A},\mathbf{B}$ are
two set decompositions of $\left[  n\right]  $ having at most $k$ blocks each.

Here is an incomplete survey of our results in Section \ref{sec.row-to-row};
note that several of them have been shown before, but we give new proofs. We
show (in Theorem \ref{thm.row.main}) that $\mathcal{I}_{k}$ is an ideal of
$\mathcal{A}$ and a free $\mathbf{k}$-module whose rank is the number of
$12\cdots\left(  k+1\right)  $-avoiding permutations in $S_{n}$, whereas the
residue classes of the remaining permutations in $S_{n}$ form a basis of the
quotient $\mathcal{A}/\mathcal{I}_{k}$. Moreover, we construct another ideal
$\mathcal{J}_{k}$ of $\mathcal{A}$ that has a mutual annihilation relationship
with $\mathcal{I}_{k}$, meaning that each of the two ideals $\mathcal{I}_{k}$
and $\mathcal{J}_{k}$ is the annihilator of the other ideal from both left and
right as well as its orthogonal complement with respect to the standard dot
product on $\mathcal{A}$. The simplest way to define $\mathcal{J}_{k}$ (at
least for $k<n$) is as the two-sided ideal generated by the single element
\[
\nabla_{X}^{-}:=\sum_{\substack{w\in S_{n};\\w\left(  i\right)  =i\text{ for
all }i\in\left[  n\right]  \setminus U}}\left(  -1\right)  ^{w}w\in\mathcal{A}%
\]
(this is not how we define $\mathcal{J}_{k}$, but is equivalent; see
Proposition \ref{prop.IJ.1} \textbf{(f)}). This ideal $\mathcal{J}_{k}$, too,
is a free $\mathbf{k}$-module, as is the quotient $\mathcal{A}/\mathcal{J}%
_{k}$ (see Theorem \ref{thm.row.main}). Note that the span of the rectangular
rook sums $\nabla_{B,A}$ studied in Section \ref{sec.rooksum} is precisely the
ideal $\mathcal{I}_{2}$, and so we conclude that its rank is the number of
$123$-avoiding permutations in $S_{n}$, that is, the Catalan number $C_{n}$
(Corollary \ref{cor.Nabla.span}).

So far we have not assumed anything about the commutative ring $\mathbf{k}$.
If, however, $n!$ is invertible in $\mathbf{k}$, then $\mathcal{A}%
=\mathcal{I}_{k}\oplus\mathcal{J}_{k}$ as $\mathbf{k}$-modules and
$\mathcal{A}\cong\mathcal{I}_{k}\times\mathcal{J}_{k}$ as $\mathbf{k}$-algebras.

The ideals $\mathcal{I}_{k}$ and $\mathcal{J}_{k}$ are not entirely new; they
can be viewed as spans of subfamilies of the Murphy cellular bases (see Remark
\ref{rmk.IJ=murphy}). However, all our proofs are independent of this fact,
and use only the most elementary algebra and combinatorics.

In Subsection \ref{sec.row.ann}, we describe the ideals $\mathcal{I}_{k}$ and
$\mathcal{J}_{n-k-1}$ as annihilators of certain left $\mathbf{k}\left[
S_{n}\right]  $-modules. Namely:

\begin{enumerate}
\item The ideal $\mathcal{J}_{k}$ is the annihilator of the tensor power
$V_{k}^{\otimes n}$, where $V_{k}=\mathbf{k}^{k}$ is a free $\mathbf{k}%
$-module of rank $k$ and where $S_{n}$ acts on $V_{k}^{\otimes n}$ by
permuting the factors (Theorem \ref{thm.AnnVkn}). This is a classical result
by de Concini and Procesi \cite[Theorem 4.2]{deCPro76}.

\item The ideal $\mathcal{I}_{k}$ is the annihilator of a
sign-twisted\footnote{\textquotedblleft Sign-twisting\textquotedblright\ means
that the action of any permutation $w\in S_{n}$ is additionally scaled by the
sign $\left(  -1\right)  ^{w}$ of $w$.} tensor power $N_{n}^{\otimes\left(
n-k-1\right)  }$ of the natural representation $N_{n}=\mathbf{k}^{n}$ of
$S_{n}$, where $S_{n}$ acts diagonally on the tensor power (Theorem
\ref{thm.AnnNnk}, which is stated in a slightly different form, applying the
sign-twist to the ideal rather than the module). This is a recent result by
Bowman, Doty and Martin \cite[Theorem 7.4 (a)]{BoDoMa18}.
\end{enumerate}

Note once again that no requirements are made on $\mathbf{k}$ here.

When $n!$ is invertible in $\mathbf{k}$, we can characterize the ideals
$\mathcal{I}_{k}$ and $\mathcal{J}_{k}$ in a more mainstream way (Theorem
\ref{thm.IJ.rep}): The ideal $\mathcal{I}_{k}$ consists of those elements
$\mathbf{a}\in\mathcal{A}$ that annihilate all Specht modules $S^{\lambda}$
with $\ell\left(  \lambda\right)  >k$, whereas the ideal $\mathcal{J}_{k}$
consists of those elements $\mathbf{a}\in\mathcal{A}$ that annihilate all
Specht modules $S^{\lambda}$ with $\ell\left(  \lambda\right)  \leq k$. This
fact is responsible for the appearance of $\mathcal{J}_{k}$ in recent work on
quantum information theory \cite[Theorem 2.2]{KlStVo25}. Moreover, it can be
used to give a new proof for the classical fact (Corollary \ref{cor.num-avoid}%
) that the number of permutations $w\in S_{n}$ that avoid $12\cdots\left(
k+1\right)  $ (that is, have no increasing subsequence of length $k+1$) equals
the number of pairs of standard tableaux of shape $\lambda$ for partitions
$\lambda\vdash n$ satisfying $\ell\left(  \lambda\right)  \leq k$.

We end each of the two sections with an open question. In Section
\ref{sec.rooksum}, we propose an \textquotedblleft abstract
lift\textquotedblright\ of the span of the $\nabla_{B,A}$'s. This is a
nonunital $\mathbf{k}$-algebra with basis $\left(  \Delta_{B,A}\right)
_{A,B\subseteq\left[  n\right]  \text{ of equal size}}$ consisting of
\textquotedblleft abstract nablas\textquotedblright\ $\Delta_{B,A}$, which
multiply by the same rules as the $\nabla_{B,A}$'s (see Theorem
\ref{thm.Nabla.prod0}), but do not satisfy any linear dependencies. This
$\mathbf{k}$-algebra has dimension $\dbinom{2n}{n}$, which is $n+1$ times the
dimension of the actual span of the $\nabla_{B,A}$'s; but it appears to share
some of the properties of the latter. We know (Theorem \ref{thm.Nabla-alg.ass}%
) that it is associative, and we conjecture that it is unital when $n!$ is
invertible in $\mathbf{k}$ (certainly not for general $\mathbf{k}$). Section
\ref{sec.row-to-row} leads us to another open question (Question
\ref{quest.num-avoid-gen}), about a basis of the quotient ring $\mathcal{A}%
/\left(  \mathcal{I}_{k}+T_{\operatorname*{sign}}\left(  \mathcal{J}_{\ell
}\right)  \right)  $, closely related to a recent conjecture by Donkin
\cite[Remark 2.4]{Donkin24}.

\subsubsection*{Acknowledgments}

This project owes its inspiration to a conversation with Per Alexandersson and
to prior joint work with Nadia Lafreni\`{e}re. I would also like to thank
Patricia Commins, Zachary Hamaker, Jonathan Novak, Vic Reiner, Brendon
Rhoades, Travis Scrimshaw and Richard P. Stanley for interesting comments, and
Jurij Volcic for informing me of relevant work in quantum information theory.

\section{\label{sec.rooksum}Rook sums in the symmetric group algebra}

\subsection{Definitions}

Let $n$ be a nonnegative integer. Let $\left[  n\right]  :=\left\{
1,2,\ldots,n\right\}  $.

Fix a commutative ring $\mathbf{k}$. (All rings and algebras are understood to
be associative and unital, unless declared otherwise.)

Let $S_{n}$ be the $n$-th symmetric group, defined as the group of all $n!$
permutations of $\left[  n\right]  $. Let $\mathcal{A}:=\mathbf{k}\left[
S_{n}\right]  $ be its group algebra over $\mathbf{k}$.

The \emph{antipode} of the group algebra $\mathcal{A}$ is the $\mathbf{k}%
$-linear map $\mathcal{A}\rightarrow\mathcal{A}$ that sends each permutation
$w\in S_{n}$ to $w^{-1}$. We will denote this map by $S$. It is well-known
(see, e.g., \cite[\S 3.11.4]{sga}) that $S$ is a $\mathbf{k}$-algebra
anti-automorphism and an involution (i.e., satisfies $S\circ
S=\operatorname*{id}$).

For any two subsets $A$ and $B$ of $\left[  n\right]  $, we define the
elements%
\[
\nabla_{B,A}:=\sum_{\substack{w\in S_{n};\\w\left(  A\right)  =B}%
}w\ \ \ \ \ \ \ \ \ \ \text{and}\ \ \ \ \ \ \ \ \ \ \widetilde{\nabla}%
_{B,A}:=\sum_{\substack{w\in S_{n};\\w\left(  A\right)  \subseteq B}}w
\]
of $\mathcal{A}$. We shall refer to these elements as \emph{rectangular rook
sums}.

For instance, for $n=4$, we have%
\begin{align*}
\nabla_{\left\{  2,3\right\}  ,\left\{  1,4\right\}  }  &  =\sum
_{\substack{w\in S_{4};\\w\left(  \left\{  1,4\right\}  \right)  =\left\{
2,3\right\}  }}w\\
&  =\operatorname*{oln}\left(  2143\right)  +\operatorname*{oln}\left(
2413\right)  +\operatorname*{oln}\left(  3142\right)  +\operatorname*{oln}%
\left(  3412\right)  ,
\end{align*}
where $\operatorname*{oln}\left(  i_{1}i_{2}\ldots i_{n}\right)  $ means the
permutation in $S_{n}$ with one-line notation $\left(  i_{1},i_{2}%
,\ldots,i_{n}\right)  $. Moreover, again for $n=4$, we have%
\[
\widetilde{\nabla}_{\left\{  1,2,3\right\}  ,\left\{  1,4\right\}  }%
=\sum_{\substack{w\in S_{4};\\w\left(  \left\{  1,4\right\}  \right)
\subseteq\left\{  1,2,3\right\}  }}w=\sum_{\substack{w\in S_{4};\\w\left(
1\right)  \neq4\text{ and }w\left(  4\right)  \neq4}}w,
\]
which is a sum of altogether $12$ permutations.

The following proposition collects some easy properties of rectangular rook sums:

\begin{proposition}
\label{prop.Nabla.simple}Let $A$ and $B$ be two subsets of $\left[  n\right]
$. Then:

\begin{enumerate}
\item[\textbf{(a)}] We have $\nabla_{B,A}=0$ if $\left\vert A\right\vert
\neq\left\vert B\right\vert $.

\item[\textbf{(b)}] We have $\widetilde{\nabla}_{B,A}=0$ if $\left\vert
A\right\vert >\left\vert B\right\vert $.

\item[\textbf{(c)}] We have $\widetilde{\nabla}_{B,A}=\sum
\limits_{\substack{U\subseteq B;\\\left\vert U\right\vert =\left\vert
A\right\vert }}\nabla_{U,A}$.

\item[\textbf{(d)}] We have $\nabla_{B,A}=\nabla_{\left[  n\right]  \setminus
B,\ \left[  n\right]  \setminus A}$.

\item[\textbf{(e)}] If $\left\vert A\right\vert =\left\vert B\right\vert $,
then $\nabla_{B,A}=\widetilde{\nabla}_{B,A}$.

\item[\textbf{(f)}] The antipode $S$ satisfies $S\left(  \nabla_{B,A}\right)
=\nabla_{A,B}$.

\item[\textbf{(g)}] The antipode $S$ satisfies $S\left(  \widetilde{\nabla
}_{B,A}\right)  =\widetilde{\nabla}_{\left[  n\right]  \setminus A,\ \left[
n\right]  \setminus B}$.

\item[\textbf{(h)}] For any $u\in S_{n}$, we have
\[
u\nabla_{B,A}=\nabla_{u\left(  B\right)  ,A}\ \ \ \ \ \ \ \ \ \ \text{and}%
\ \ \ \ \ \ \ \ \ \ u\widetilde{\nabla}_{B,A}=\widetilde{\nabla}_{u\left(
B\right)  ,A}.
\]

\item[\textbf{(i)}] For any $u\in S_{n}$, we have
\[
\nabla_{B,A}u=\nabla_{B,u^{-1}\left(  A\right)  }%
\ \ \ \ \ \ \ \ \ \ \text{and}\ \ \ \ \ \ \ \ \ \ \widetilde{\nabla}%
_{B,A}u=\widetilde{\nabla}_{B,u^{-1}\left(  A\right)  }.
\]

\end{enumerate}
\end{proposition}

\begin{proof}
[Proof sketch.]\textbf{(a)} A permutation $w\in S_{n}$ cannot satisfy
$w\left(  A\right)  =B$ if $\left\vert A\right\vert \neq\left\vert
B\right\vert $. \medskip

\textbf{(b)} A permutation $w\in S_{n}$ cannot satisfy $w\left(  A\right)
\subseteq B$ if $\left\vert A\right\vert >\left\vert B\right\vert $ (since
$\left\vert w\left(  A\right)  \right\vert =\left\vert A\right\vert $).
\medskip

\textbf{(c)} A permutation $w\in S_{n}$ satisfies $w\left(  A\right)
\subseteq B$ if and only if it satisfies $w\left(  A\right)  =U$ for some
subset $U\subseteq B$ satisfying $\left\vert U\right\vert =\left\vert
A\right\vert $. Moreover, this subset $U$ is uniquely determined (as $w\left(
A\right)  $). \medskip

\textbf{(d)} A permutation $w\in S_{n}$ satisfies $w\left(  A\right)  =B$ if
and only if it satisfies $w\left(  \left[  n\right]  \setminus A\right)
=\left[  n\right]  \setminus B$. \medskip

\textbf{(e)} If $\left\vert A\right\vert =\left\vert B\right\vert $, then a
permutation $w\in S_{n}$ satisfies $w\left(  A\right)  =B$ if and only if it
satisfies $w\left(  A\right)  \subseteq B$ (since $w\left(  A\right)  $ has
the same size as $A$ and therefore as $B$ as well, and thus cannot be a proper
subset of $B$). \medskip

\textbf{(f)} A permutation $w\in S_{n}$ satisfies $w\left(  A\right)  =B$ if
and only if its inverse $w^{-1}$ satisfies $w^{-1}\left(  B\right)  =A$.
\medskip

\textbf{(g)} A permutation $w\in S_{n}$ satisfies $w\left(  A\right)
\subseteq B$ if and only if its inverse $w^{-1}$ satisfies $w^{-1}\left(
\left[  n\right]  \setminus B\right)  \subseteq\left[  n\right]  \setminus A$.
\medskip

\textbf{(h)} Let $u\in S_{n}$. Then, the definition of $\widetilde{\nabla
}_{B,A}$ yields
\begin{align*}
u\widetilde{\nabla}_{B,A}  &  =u\sum_{\substack{w\in S_{n};\\w\left(
A\right)  \subseteq B}}w=\sum_{\substack{w\in S_{n};\\w\left(  A\right)
\subseteq B}}uw\\
&  =\sum_{\substack{w\in S_{n};\\\left(  uw\right)  \left(  A\right)
\subseteq u\left(  B\right)  }}uw\ \ \ \ \ \ \ \ \ \ \left(
\begin{array}
[c]{c}%
\text{since the condition \textquotedblleft}w\left(  A\right)  \subseteq
B\text{\textquotedblright}\\
\text{is equivalent to \textquotedblleft}\left(  uw\right)  \left(  A\right)
\subseteq u\left(  B\right)  \text{\textquotedblright}\\
\text{(because }u\text{ is a bijection)}%
\end{array}
\right) \\
&  =\sum_{\substack{w\in S_{n};\\w\left(  A\right)  \subseteq u\left(
B\right)  }}w\ \ \ \ \ \ \ \ \ \ \left(
\begin{array}
[c]{c}%
\text{here, we have substituted }w\text{ for }uw\\
\text{in the sum}%
\end{array}
\right) \\
&  =\widetilde{\nabla}_{u\left(  B\right)  ,A}\ \ \ \ \ \ \ \ \ \ \left(
\text{by the definition of }\widetilde{\nabla}_{u\left(  B\right)  ,A}\right)
.
\end{align*}
Similarly, $u\nabla_{B,A}=\nabla_{u\left(  B\right)  ,A}$. This proves
Proposition \ref{prop.Nabla.simple} \textbf{(h)}. \medskip

\textbf{(i)} Let $u\in S_{n}$. Then, the definition of $\widetilde{\nabla
}_{B,A}$ yields
\begin{align*}
\widetilde{\nabla}_{B,A}u  &  =\sum_{\substack{w\in S_{n};\\w\left(  A\right)
\subseteq B}}wu=\sum_{\substack{w\in S_{n};\\w\left(  A\right)  \subseteq
B}}wu\\
&  =\sum_{\substack{w\in S_{n};\\\left(  wu\right)  \left(  u^{-1}\left(
A\right)  \right)  \subseteq B}}wu\ \ \ \ \ \ \ \ \ \ \left(  \text{since
}w\left(  A\right)  =\left(  wu\right)  \left(  u^{-1}\left(  A\right)
\right)  \right) \\
&  =\sum_{\substack{w\in S_{n};\\w\left(  u^{-1}\left(  A\right)  \right)
\subseteq B}}w\ \ \ \ \ \ \ \ \ \ \left(
\begin{array}
[c]{c}%
\text{here, we have substituted }w\text{ for }wu\\
\text{in the sum}%
\end{array}
\right) \\
&  =\widetilde{\nabla}_{B,u^{-1}\left(  A\right)  }\ \ \ \ \ \ \ \ \ \ \left(
\text{by the definition of }\widetilde{\nabla}_{B,u^{-1}\left(  A\right)
}\right)  .
\end{align*}
Similarly, $\nabla_{B,A}u=\nabla_{B,u^{-1}\left(  A\right)  }$. This proves
Proposition \ref{prop.Nabla.simple} \textbf{(i)}.
\end{proof}

Proposition \ref{prop.Nabla.simple} \textbf{(c)} shows that the elements
$\nabla_{B,A}$ and $\widetilde{\nabla}_{B,A}$ have the same span as $B$ and
$A$ range over the subsets of $\left[  n\right]  $, or even as $B$ ranges over
all subsets of $\left[  n\right]  $ while $A$ is fixed. Later (in Corollary
\ref{cor.Nabla.span}), we will learn more about this span, and in particular
compute its dimension.

\subsection{The product rule}

Parts \textbf{(h)} and \textbf{(i)} of Proposition \ref{prop.Nabla.simple}
show that the span of the elements $\nabla_{B,A}$ is an ideal of $\mathcal{A}%
$. Hence, this span is a nonunital $\mathbf{k}$-subalgebra of $\mathcal{A}$.
It has an explicit multiplication rule, which we shall state in three
different forms. First, we define an important family of integers:

\begin{definition}
\label{def.omegaBC}For any two subsets $B$ and $C$ of $\left[  n\right]  $, we
define the positive integer%
\[
\omega_{B,C}:=\left\vert B\cap C\right\vert !\cdot\left\vert B\setminus
C\right\vert !\cdot\left\vert C\setminus B\right\vert !\cdot\left\vert \left[
n\right]  \setminus\left(  B\cup C\right)  \right\vert !\in\mathbb{Z}.
\]

\end{definition}

\begin{theorem}
\label{thm.Nabla.prod0}Let $A,B,C,D$ be four subsets of $\left[  n\right]  $
such that $\left\vert A\right\vert =\left\vert B\right\vert $ and $\left\vert
C\right\vert =\left\vert D\right\vert $. Then:

\begin{enumerate}
\item[\textbf{(a)}] We have%
\[
\nabla_{D,C}\nabla_{B,A}=\omega_{B,C}\sum_{\substack{w\in S_{n};\\\left\vert
w\left(  A\right)  \cap D\right\vert =\left\vert B\cap C\right\vert }}w.
\]

\item[\textbf{(b)}] We have%
\[
\nabla_{D,C}\nabla_{B,A}=\omega_{B,C}\sum_{\substack{U\subseteq D,\\V\subseteq
A;\\\left\vert U\right\vert =\left\vert V\right\vert }}\left(  -1\right)
^{\left\vert U\right\vert -\left\vert B\cap C\right\vert }\dbinom{\left\vert
U\right\vert }{\left\vert B\cap C\right\vert }\nabla_{U,V}.
\]

\item[\textbf{(c)}] We have%
\[
\nabla_{D,C}\nabla_{B,A}=\omega_{B,C}\sum_{V\subseteq A}\left(  -1\right)
^{\left\vert V\right\vert -\left\vert B\cap C\right\vert }\dbinom{\left\vert
V\right\vert }{\left\vert B\cap C\right\vert }\widetilde{\nabla}_{D,V}.
\]

\end{enumerate}
\end{theorem}

Before we can prove this theorem, we shall show a few lemmas from enumerative
combinatorics. The first lemma explains the appearance of the numbers
$\omega_{B,C}$:

\begin{lemma}
\label{lem.Nabla.prod0.a}Let $A,B,C,D$ be four subsets of $\left[  n\right]  $
such that $\left\vert A\right\vert =\left\vert B\right\vert $ and $\left\vert
C\right\vert =\left\vert D\right\vert $. Fix a permutation $w\in S_{n}$. Let
$Q_{w}$ denote the set of all pairs $\left(  u,v\right)  \in S_{n}\times
S_{n}$ satisfying $u\left(  C\right)  =D$ and $v\left(  A\right)  =B$ and
$uv=w$. Then:

\begin{enumerate}
\item[\textbf{(a)}] If $\left\vert w\left(  A\right)  \cap D\right\vert
\neq\left\vert B\cap C\right\vert $, then $\left\vert Q_{w}\right\vert =0$.

\item[\textbf{(b)}] If $\left\vert w\left(  A\right)  \cap D\right\vert
=\left\vert B\cap C\right\vert $, then $\left\vert Q_{w}\right\vert
=\omega_{B,C}$.
\end{enumerate}
\end{lemma}

\begin{proof}
\textbf{(a)} Assume that $\left\vert w\left(  A\right)  \cap D\right\vert
\neq\left\vert B\cap C\right\vert $.

Let $\left(  u,v\right)  \in Q_{w}$. Thus, $\left(  u,v\right)  \in
S_{n}\times S_{n}$ and $u\left(  C\right)  =D$ and $v\left(  A\right)  =B$ and
$uv=w$ (by the definition of $Q_{w}$). Now, from $w=uv$, we obtain $w\left(
A\right)  =\left(  uv\right)  \left(  A\right)  =u\left(  v\left(  A\right)
\right)  $, so that
\[
\underbrace{w\left(  A\right)  }_{=u\left(  v\left(  A\right)  \right)  }%
\cap\underbrace{D}_{=u\left(  C\right)  }=u\left(  \underbrace{v\left(
A\right)  }_{=B}\right)  \cap u\left(  C\right)  =u\left(  B\right)  \cap
u\left(  C\right)  =u\left(  B\cap C\right)
\]
(since $u$ is a permutation). Thus, $\left\vert w\left(  A\right)  \cap
D\right\vert =\left\vert u\left(  B\cap C\right)  \right\vert =\left\vert
B\cap C\right\vert $ (again since $u$ is a permutation). This contradicts
$\left\vert w\left(  A\right)  \cap D\right\vert \neq\left\vert B\cap
C\right\vert $.

Thus, we have found a contradiction for each $\left(  u,v\right)  \in Q_{w}$.
Hence, there exists no $\left(  u,v\right)  \in Q_{w}$. Thus, $\left\vert
Q_{w}\right\vert =0$. This proves Lemma \ref{lem.Nabla.prod0.a} \textbf{(a)}.
\medskip

\textbf{(b)} Assume that $\left\vert w\left(  A\right)  \cap D\right\vert
=\left\vert B\cap C\right\vert $.

Set $k:=\left\vert B\cap C\right\vert $. Thus, by assumption, we have
$\left\vert w\left(  A\right)  \cap D\right\vert =\left\vert B\cap
C\right\vert =k$. Furthermore, set $p:=\left\vert B\setminus C\right\vert $
and $q:=\left\vert C\setminus B\right\vert $.

Since $w$ is a permutation, we have $\left\vert w\left(  A\right)  \right\vert
=\left\vert A\right\vert =\left\vert B\right\vert $. Thus,%
\[
\left\vert w\left(  A\right)  \setminus D\right\vert =\underbrace{\left\vert
w\left(  A\right)  \right\vert }_{=\left\vert B\right\vert }%
-\underbrace{\left\vert w\left(  A\right)  \cap D\right\vert }_{=\left\vert
B\cap C\right\vert }=\left\vert B\right\vert -\left\vert B\cap C\right\vert
=\left\vert B\setminus C\right\vert =p
\]
and%
\[
\left\vert D\setminus w\left(  A\right)  \right\vert =\underbrace{\left\vert
D\right\vert }_{=\left\vert C\right\vert }-\underbrace{\left\vert w\left(
A\right)  \cap D\right\vert }_{=\left\vert B\cap C\right\vert }=\left\vert
C\right\vert -\left\vert B\cap C\right\vert =\left\vert C\setminus
B\right\vert =q.
\]
Moreover, the set $\left[  n\right]  \setminus\left(  B\cup C\right)  $
consists of those elements of $\left[  n\right]  $ that belong to none of the
three disjoint subsets $B\cap C$, $B\setminus C$ and $C\setminus B$. Thus,
\[
\left\vert \left[  n\right]  \setminus\left(  B\cup C\right)  \right\vert
=\underbrace{\left\vert \left[  n\right]  \right\vert }_{=n}%
-\underbrace{\left\vert B\cap C\right\vert }_{=k}-\underbrace{\left\vert
B\setminus C\right\vert }_{=p}-\underbrace{\left\vert C\setminus B\right\vert
}_{=q}=n-k-p-q.
\]

Now, we want to compute $\left\vert Q_{w}\right\vert $. In other words, we
want to count the elements of $Q_{w}$. These elements are the pairs $\left(
u,v\right)  \in S_{n}\times S_{n}$ satisfying $u\left(  C\right)  =D$ and
$v\left(  A\right)  =B$ and $uv=w$. Clearly, such a pair $\left(  u,v\right)
$ must satisfy $v=u^{-1}w$ (since $uv=w$), and thus is uniquely determined by
its first entry $u$. Moreover, the requirement $v\left(  A\right)  =B$ on this
pair is equivalent to $u\left(  B\right)  =w\left(  A\right)  $, because of
the following chain of equivalences:%
\begin{align*}
\left(  v\left(  A\right)  =B\right)  \  &  \Longleftrightarrow\ \left(
\left(  u^{-1}w\right)  \left(  A\right)  =B\right)
\ \ \ \ \ \ \ \ \ \ \left(  \text{since }v=u^{-1}w\right) \\
&  \Longleftrightarrow\ \left(  u^{-1}\left(  w\left(  A\right)  \right)
=B\right) \\
&  \Longleftrightarrow\ \left(  w\left(  A\right)  =u\left(  B\right)
\right)  \ \Longleftrightarrow\ \left(  u\left(  B\right)  =w\left(  A\right)
\right)  .
\end{align*}
Hence, the elements $\left(  u,v\right)  $ of $Q_{w}$ are in one-to-one
correspondence with the permutations $u\in S_{n}$ satisfying $u\left(
C\right)  =D$ and $u\left(  B\right)  =w\left(  A\right)  $. Let us call such
permutations \emph{nice}. Thus, a nice permutation $u\in S_{n}$ must send the
subset $C$ to $D$ and send the subset $B$ to $w\left(  A\right)  $. Hence, it
must send the three subsets%
\[
B\cap C,\ \ \ \ \ \ \ \ \ \ B\setminus C,\ \ \ \ \ \ \ \ \ \ C\setminus B
\]
of $\left[  n\right]  $ to the three subsets%
\[
w\left(  A\right)  \cap D,\ \ \ \ \ \ \ \ \ \ w\left(  A\right)  \setminus
D,\ \ \ \ \ \ \ \ \ \ D\setminus w\left(  A\right)  ,
\]
respectively. Moreover, any permutation $u\in S_{n}$ that does so is nice
(since $B=\left(  B\cap C\right)  \cup\left(  B\setminus C\right)  $ and
$C=\left(  B\cap C\right)  \cup\left(  C\setminus B\right)  $). Thus, we can
construct a nice permutation $u\in S_{n}$ as follows:

\begin{enumerate}
\item Choose the values of $u$ on the $k$ elements of $B\cap C$ in such a way
that these values all belong to $w\left(  A\right)  \cap D$ and are distinct.
This can be done in $k!$ ways, since both $\left\vert B\cap C\right\vert $ and
$\left\vert w\left(  A\right)  \cap D\right\vert $ equal $k$.

\item Choose the values of $u$ on the $p$ elements of $B\setminus C$ in such a
way that these values all belong to $w\left(  A\right)  \setminus D$ and are
distinct. This can be done in $p!$ ways, since both $\left\vert B\setminus
C\right\vert $ and $\left\vert w\left(  A\right)  \setminus D\right\vert $
equal $p$.

\item Choose the values of $u$ on the $q$ elements of $C\setminus B$ in such a
way that these values all belong to $D\setminus w\left(  A\right)  $ and are
distinct. This can be done in $q!$ ways, since both $\left\vert C\setminus
B\right\vert $ and $\left\vert D\setminus w\left(  A\right)  \right\vert $
equal $q$.

\item Choose the values of $u$ on the remaining $n-k-p-q$ elements of $\left[
n\right]  $ in such a way that these values are distinct and have not been
chosen yet. This can be done in $\left(  n-k-p-q\right)  !$ ways, since we
have $n-k-p-q$ elements of $\left[  n\right]  $ left that have not been chosen yet.
\end{enumerate}

This process can be done in $k!\cdot p!\cdot q!\cdot\left(  n-k-p-q\right)  !$
ways. Thus, the number of nice permutations $u\in S_{n}$ is $k!\cdot p!\cdot
q!\cdot\left(  n-k-p-q\right)  !$. But we know that $\left\vert Q_{w}%
\right\vert $ is the number of nice permutations $u\in S_{n}$ (since the
elements $\left(  u,v\right)  $ of $Q_{w}$ are in one-to-one correspondence
with the nice permutations $u\in S_{n}$). Hence, we conclude that%
\begin{align*}
\left\vert Q_{w}\right\vert  &  =\underbrace{k}_{=\left\vert B\cap
C\right\vert }!\cdot\underbrace{p}_{=\left\vert B\setminus C\right\vert
}!\cdot\underbrace{q}_{=\left\vert C\setminus B\right\vert }!\cdot
\underbrace{\left(  n-k-p-q\right)  }_{=\left\vert \left[  n\right]
\setminus\left(  B\cup C\right)  \right\vert }!\\
&  =\left\vert B\cap C\right\vert !\cdot\left\vert B\setminus C\right\vert
!\cdot\left\vert C\setminus B\right\vert !\cdot\left\vert \left[  n\right]
\setminus\left(  B\cup C\right)  \right\vert !\\
&  =\omega_{B,C}.
\end{align*}
This proves Lemma \ref{lem.Nabla.prod0.a} \textbf{(b)}.
\end{proof}

The next lemma is a simple and elementary binomial identity:

\begin{lemma}
\label{lem.Nabla.prod0.bn}Let $n,k\in\mathbb{N}$. Then,%
\[
\sum_{r=0}^{n}\left(  -1\right)  ^{r-k}\dbinom{n}{r}\dbinom{r}{k}=%
\begin{cases}
1, & \text{if }n=k;\\
0, & \text{else.}%
\end{cases}
\]

\end{lemma}

\begin{proof}
It is well-known that each $m\in\mathbb{Z}$ satisfies%
\begin{equation}
\sum_{i=0}^{m}\left(  -1\right)  ^{i}\dbinom{m}{i}=%
\begin{cases}
1, & \text{if }m=0;\\
0, & \text{else.}%
\end{cases}
\label{pf.lem.Nabla.prod0.bn.sum-1}%
\end{equation}
(Indeed, this is obvious for $m<0$, since the left hand side is an empty sum
in this case. In the case $m\geq0$, it is proved in \cite[Proposition
1.3.28]{19fco}.)

For each $r\in\mathbb{N}$, we have%
\begin{equation}
\dbinom{n}{r}\dbinom{r}{k}=\dbinom{n}{k}\dbinom{n-k}{r-k}
\label{pf.lem.Nabla.prod0.bn.1}%
\end{equation}
(by the trinomial revision formula \cite[Proposition 1.3.35]{19fco}, applied
to $a=r$ and $b=k$). Hence,%
\begin{align}
&  \sum_{r=0}^{n}\left(  -1\right)  ^{r-k}\underbrace{\dbinom{n}{r}\dbinom
{r}{k}}_{\substack{=\dbinom{n}{k}\dbinom{n-k}{r-k}\\\text{(by
(\ref{pf.lem.Nabla.prod0.bn.1}))}}}\nonumber\\
&  =\sum_{r=0}^{n}\left(  -1\right)  ^{r-k}\dbinom{n}{k}\dbinom{n-k}%
{r-k}=\dbinom{n}{k}\sum_{r=0}^{n}\left(  -1\right)  ^{r-k}\dbinom{n-k}%
{r-k}\nonumber\\
&  =\dbinom{n}{k}\sum_{i=-k}^{n-k}\left(  -1\right)  ^{i}\dbinom{n-k}{i}
\label{pf.lem.Nabla.prod0.bn.2}%
\end{align}
(here, we have substituted $i$ for $r-k$ in the sum). But the two sums%
\[
\sum_{i=-k}^{n-k}\left(  -1\right)  ^{i}\dbinom{n-k}{i}%
\ \ \ \ \ \ \ \ \ \ \text{and}\ \ \ \ \ \ \ \ \ \ \sum_{i=0}^{n-k}\left(
-1\right)  ^{i}\dbinom{n-k}{i}%
\]
differ only in their addends for $i<0$ (if they differ at all), and all these
addends are $0$ (since $i<0$ entails $\dbinom{n-k}{i}=0$ and thus $\left(
-1\right)  ^{i}\underbrace{\dbinom{n-k}{i}}_{=0}=0$) and thus do not affect
the sums. Hence, these two sums are equal. In other words,%
\begin{align*}
\sum_{i=-k}^{n-k}\left(  -1\right)  ^{i}\dbinom{n-k}{i}  &  =\sum_{i=0}%
^{n-k}\left(  -1\right)  ^{i}\dbinom{n-k}{i}\\
&  =%
\begin{cases}
1, & \text{if }n-k=0;\\
0, & \text{else}%
\end{cases}
\ \ \ \ \ \ \ \ \ \ \left(  \text{by (\ref{pf.lem.Nabla.prod0.bn.sum-1}),
applied to }m=n-k\right) \\
&  =%
\begin{cases}
1, & \text{if }n=k;\\
0, & \text{else}%
\end{cases}
\end{align*}
(since the equation $n-k=0$ is equivalent to $n=k$). Substituting this into
(\ref{pf.lem.Nabla.prod0.bn.2}), we obtain%
\begin{align*}
\sum_{r=0}^{n}\left(  -1\right)  ^{r-k}\dbinom{n}{r}\dbinom{r}{k}  &
=\dbinom{n}{k}%
\begin{cases}
1, & \text{if }n=k;\\
0, & \text{else}%
\end{cases}
\ \ =%
\begin{cases}
\dbinom{n}{k}\cdot1, & \text{if }n=k;\\
\dbinom{n}{k}\cdot0, & \text{else}%
\end{cases}
\\
&  =%
\begin{cases}
\dbinom{n}{k}, & \text{if }n=k;\\
0, & \text{else}%
\end{cases}
\ \ =%
\begin{cases}
1, & \text{if }n=k;\\
0, & \text{else}%
\end{cases}
\end{align*}
(since $\dbinom{n}{k}=\dbinom{k}{k}=1$ when $n=k$). This proves Lemma
\ref{lem.Nabla.prod0.bn}.
\end{proof}

\begin{lemma}
\label{lem.Nabla.prod0.b}Let $k\in\mathbb{N}$. Let $Z$ be any finite set.
Then,%
\[
\sum_{U\subseteq Z}\left(  -1\right)  ^{\left\vert U\right\vert -k}%
\dbinom{\left\vert U\right\vert }{k}=%
\begin{cases}
1, & \text{if }\left\vert Z\right\vert =k;\\
0, & \text{else.}%
\end{cases}
\]

\end{lemma}

\begin{proof}
Any subset $U\subseteq Z$ has size $\left\vert U\right\vert \in\left\{
0,1,\ldots,\left\vert Z\right\vert \right\}  $. Thus, we can break up the sum
on the left hand side as follows:%
\begin{align*}
\sum_{U\subseteq Z}\left(  -1\right)  ^{\left\vert U\right\vert -k}%
\dbinom{\left\vert U\right\vert }{k}  &  =\sum_{r=0}^{\left\vert Z\right\vert
}\ \ \sum_{\substack{U\subseteq Z;\\\left\vert U\right\vert =r}%
}\underbrace{\left(  -1\right)  ^{\left\vert U\right\vert -k}\dbinom
{\left\vert U\right\vert }{k}}_{\substack{=\left(  -1\right)  ^{r-k}\dbinom
{r}{k}\\\text{(since }\left\vert U\right\vert =r\text{)}}}\\
&  =\sum_{r=0}^{\left\vert Z\right\vert }\ \ \underbrace{\sum
_{\substack{U\subseteq Z;\\\left\vert U\right\vert =r}}\left(  -1\right)
^{r-k}\dbinom{r}{k}}_{\substack{=\dbinom{\left\vert Z\right\vert }{r}%
\cdot\left(  -1\right)  ^{r-k}\dbinom{r}{k}\\\text{(since there are }%
\dbinom{\left\vert Z\right\vert }{r}\text{ many}\\r\text{-element subsets
}U\text{ of }Z\text{)}}}\\
&  =\sum_{r=0}^{\left\vert Z\right\vert }\dbinom{\left\vert Z\right\vert }%
{r}\cdot\left(  -1\right)  ^{r-k}\dbinom{r}{k}=\sum_{r=0}^{\left\vert
Z\right\vert }\left(  -1\right)  ^{r-k}\dbinom{\left\vert Z\right\vert }%
{r}\dbinom{r}{k}\\
&  =%
\begin{cases}
1, & \text{if }\left\vert Z\right\vert =k;\\
0, & \text{else}%
\end{cases}
\end{align*}
(by Lemma \ref{lem.Nabla.prod0.bn}, applied to $n=\left\vert Z\right\vert $).
This proves Lemma \ref{lem.Nabla.prod0.b}.
\end{proof}

\begin{proof}
[Proof of Theorem \ref{thm.Nabla.prod0}.]\textbf{(a)} The definition of a
rectangular rook sum shows that
\[
\nabla_{D,C}=\sum_{\substack{u\in S_{n};\\u\left(  C\right)  =D}%
}u\ \ \ \ \ \ \ \ \ \ \text{and}\ \ \ \ \ \ \ \ \ \ \nabla_{B,A}%
=\sum_{\substack{v\in S_{n};\\v\left(  A\right)  =B}}v.
\]
Multiplying these two equalities, we find%
\begin{align}
\nabla_{D,C}\nabla_{B,A}  &  =\left(  \sum_{\substack{u\in S_{n};\\u\left(
C\right)  =D}}u\right)  \left(  \sum_{\substack{v\in S_{n};\\v\left(
A\right)  =B}}v\right)  =\sum_{\substack{\left(  u,v\right)  \in S_{n}\times
S_{n};\\u\left(  C\right)  =D\text{ and }v\left(  A\right)  =B}}uv\nonumber\\
&  =\sum_{w\in S_{n}}\left\vert Q_{w}\right\vert w,
\label{pf.thm.Nabla.prod0.1}%
\end{align}
where $Q_{w}$ (for a given permutation $w\in S_{n}$) is defined as the set of
all pairs $\left(  u,v\right)  \in S_{n}\times S_{n}$ satisfying $u\left(
C\right)  =D$ and $v\left(  A\right)  =B$ and $uv=w$. Now,
(\ref{pf.thm.Nabla.prod0.1}) becomes%
\begin{align*}
\nabla_{D,C}\nabla_{B,A}  &  =\sum_{w\in S_{n}}\left\vert Q_{w}\right\vert w\\
&  =\sum_{\substack{w\in S_{n};\\\left\vert w\left(  A\right)  \cap
D\right\vert =\left\vert B\cap C\right\vert }}\underbrace{\left\vert
Q_{w}\right\vert }_{\substack{=\omega_{B,C}\\\text{(by Lemma
\ref{lem.Nabla.prod0.a} \textbf{(b)})}}}w+\sum_{\substack{w\in S_{n}%
;\\\left\vert w\left(  A\right)  \cap D\right\vert \neq\left\vert B\cap
C\right\vert }}\underbrace{\left\vert Q_{w}\right\vert }%
_{\substack{=0\\\text{(by Lemma \ref{lem.Nabla.prod0.a} \textbf{(a)})}}}w\\
&  =\sum_{\substack{w\in S_{n};\\\left\vert w\left(  A\right)  \cap
D\right\vert =\left\vert B\cap C\right\vert }}\omega_{B,C}w=\omega_{B,C}%
\sum_{\substack{w\in S_{n};\\\left\vert w\left(  A\right)  \cap D\right\vert
=\left\vert B\cap C\right\vert }}w.
\end{align*}
This proves Theorem \ref{thm.Nabla.prod0} \textbf{(a)}. \medskip

\begin{noncompile}
\textit{Old proof sketch for Theorem \ref{thm.Nabla.prod0} \textbf{(a)}:} Each
permutation $w$ appearing in the product $\nabla_{D,C}\nabla_{B,A}$ has the
property that $\left\vert w\left(  A\right)  \cap D\right\vert =\left\vert
B\cap C\right\vert $ (because it can be written as $w=uv$ with $u\left(
C\right)  =D$ and $v\left(  A\right)  =B$, and therefore we have
$\underbrace{w\left(  A\right)  }_{=u\left(  v\left(  A\right)  \right)  }%
\cap\underbrace{D}_{=u\left(  C\right)  }=u\left(  \underbrace{v\left(
A\right)  }_{=B}\right)  \cap u\left(  C\right)  =u\left(  B\right)  \cap
u\left(  C\right)  =u\left(  B\cap C\right)  $, so that $\left\vert w\left(
A\right)  \cap D\right\vert =\left\vert B\cap C\right\vert $). It remains to
show that each $w$ with this property appears exactly $\omega_{B,C}$ times in
this product. In other words, given a permutation $w\in S_{n}$ that satisfies
$\left\vert w\left(  A\right)  \cap D\right\vert =\left\vert B\cap
C\right\vert $, we must show that there are exactly $\omega_{B,C}$ ways to
decompose $w$ as $w=uv$ with $u\left(  C\right)  =D$ and $v\left(  A\right)
=B$. But this is an exercise in counting: We want to count the permutations
$v\in S_{n}$ satisfying $v\left(  A\right)  =B$ and $\left(  wv^{-1}\right)
\left(  C\right)  =D$. Such a permutation $v$ must send $A$ to $B$ and send
$w^{-1}\left(  D\right)  $ to $C$. In other words, it must send the four
subsets $A\cap w^{-1}\left(  D\right)  $, $A\setminus w^{-1}\left(  D\right)
$, $w^{-1}\left(  D\right)  \setminus A$ and $\left[  n\right]  \setminus
\left(  A\cup w^{-1}\left(  D\right)  \right)  $ to the respectively
equinumerous subsets $B\cap C$, $B\setminus C$, $C\setminus B$ and $\left[
n\right]  \setminus\left(  B\cup C\right)  $, respectively. The number of ways
to do this is
\[
\left\vert B\cap C\right\vert !\cdot\left\vert B\setminus C\right\vert
!\cdot\left\vert C\setminus B\right\vert !\cdot\left\vert \left[  n\right]
\setminus\left(  B\cup C\right)  \right\vert !,
\]
which is exactly $\omega_{B,C}$. Thus, the proof of Theorem
\ref{thm.Nabla.prod0} \textbf{(a)} is complete.
\end{noncompile}

\textbf{(b)} We have%
\begin{align}
&  \sum_{\substack{U\subseteq D,\\V\subseteq A;\\\left\vert U\right\vert
=\left\vert V\right\vert }}\left(  -1\right)  ^{\left\vert U\right\vert
-\left\vert B\cap C\right\vert }\dbinom{\left\vert U\right\vert }{\left\vert
B\cap C\right\vert }\underbrace{\nabla_{U,V}}_{\substack{=\sum_{\substack{w\in
S_{n};\\w\left(  V\right)  =U}}w\\\text{(by the definition of }\nabla
_{U,V}\text{)}}}\nonumber\\
&  =\sum_{\substack{U\subseteq D,\\V\subseteq A;\\\left\vert U\right\vert
=\left\vert V\right\vert }}\left(  -1\right)  ^{\left\vert U\right\vert
-\left\vert B\cap C\right\vert }\dbinom{\left\vert U\right\vert }{\left\vert
B\cap C\right\vert }\sum_{\substack{w\in S_{n};\\w\left(  V\right)
=U}}w\nonumber\\
&  =\sum_{w\in S_{n}}\left(  \sum_{\substack{U\subseteq D,\\V\subseteq
A;\\\left\vert U\right\vert =\left\vert V\right\vert ;\\w\left(  V\right)
=U}}\left(  -1\right)  ^{\left\vert U\right\vert -\left\vert B\cap
C\right\vert }\dbinom{\left\vert U\right\vert }{\left\vert B\cap C\right\vert
}\right)  w. \label{pf.thm.Nabla.prod0.c.1}%
\end{align}

Now, let $w\in S_{n}$ be arbitrary. We shall simplify the sum%
\[
\sum_{\substack{U\subseteq D,\\V\subseteq A;\\\left\vert U\right\vert
=\left\vert V\right\vert ;\\w\left(  V\right)  =U}}\left(  -1\right)
^{\left\vert U\right\vert -\left\vert B\cap C\right\vert }\dbinom{\left\vert
U\right\vert }{\left\vert B\cap C\right\vert }.
\]
Indeed, we observe that

\begin{itemize}
\item the condition \textquotedblleft$\left\vert U\right\vert =\left\vert
V\right\vert $\textquotedblright\ under the summation sign is redundant (since
it follows from $w\left(  V\right)  =U$ because $w$ is a permutation);

\item the set $V$ in this sum is uniquely determined by $U$ via the condition
\textquotedblleft$w\left(  V\right)  =U$\textquotedblright\ (since $w$ is a
permutation), and thus can be simply replaced by $w^{-1}\left(  U\right)  $
instead of being summed over.
\end{itemize}

Thus, we can rewrite the sum as follows:%
\begin{align}
&  \sum_{\substack{U\subseteq D,\\V\subseteq A;\\\left\vert U\right\vert
=\left\vert V\right\vert ;\\w\left(  V\right)  =U}}\left(  -1\right)
^{\left\vert U\right\vert -\left\vert B\cap C\right\vert }\dbinom{\left\vert
U\right\vert }{\left\vert B\cap C\right\vert }\nonumber\\
&  =\sum_{\substack{U\subseteq D;\\w^{-1}\left(  U\right)  \subseteq
A}}\left(  -1\right)  ^{\left\vert U\right\vert -\left\vert B\cap C\right\vert
}\dbinom{\left\vert U\right\vert }{\left\vert B\cap C\right\vert }\nonumber\\
&  =\sum_{\substack{U\subseteq D;\\U\subseteq w\left(  A\right)  }}\left(
-1\right)  ^{\left\vert U\right\vert -\left\vert B\cap C\right\vert }%
\dbinom{\left\vert U\right\vert }{\left\vert B\cap C\right\vert }\nonumber\\
&  \ \ \ \ \ \ \ \ \ \ \ \ \ \ \ \ \ \ \ \ \left(
\begin{array}
[c]{c}%
\text{here, we rewrote the condition \textquotedblleft}w^{-1}\left(  U\right)
\subseteq A\text{\textquotedblright}\\
\text{under the summation sign as \textquotedblleft}U\subseteq w\left(
A\right)  \text{\textquotedblright}%
\end{array}
\right) \nonumber\\
&  =\sum_{U\subseteq w\left(  A\right)  \cap D}\left(  -1\right)  ^{\left\vert
U\right\vert -\left\vert B\cap C\right\vert }\dbinom{\left\vert U\right\vert
}{\left\vert B\cap C\right\vert }\nonumber\\
&  =%
\begin{cases}
1, & \text{if }\left\vert w\left(  A\right)  \cap D\right\vert =\left\vert
B\cap C\right\vert ;\\
0, & \text{else}%
\end{cases}
\label{pf.thm.Nabla.prod0.c.2}%
\end{align}
(by Lemma \ref{lem.Nabla.prod0.b}, applied to $Z=w\left(  A\right)  \cap D$
and $k=\left\vert B\cap C\right\vert $).

Forget that we fixed $w$. We thus have proved (\ref{pf.thm.Nabla.prod0.c.2})
for each $w\in S_{n}$. Thus, (\ref{pf.thm.Nabla.prod0.c.1}) becomes%
\begin{align*}
&  \sum_{\substack{U\subseteq D,\\V\subseteq A;\\\left\vert U\right\vert
=\left\vert V\right\vert }}\left(  -1\right)  ^{\left\vert U\right\vert
-\left\vert B\cap C\right\vert }\dbinom{\left\vert U\right\vert }{\left\vert
B\cap C\right\vert }\nabla_{U,V}\\
&  =\sum_{w\in S_{n}}\underbrace{\left(  \sum_{\substack{U\subseteq
D,\\V\subseteq A;\\\left\vert U\right\vert =\left\vert V\right\vert
;\\w\left(  V\right)  =U}}\left(  -1\right)  ^{\left\vert U\right\vert
-\left\vert B\cap C\right\vert }\dbinom{\left\vert U\right\vert }{\left\vert
B\cap C\right\vert }\right)  }_{\substack{=%
\begin{cases}
1, & \text{if }\left\vert w\left(  A\right)  \cap D\right\vert =\left\vert
B\cap C\right\vert ;\\
0, & \text{else}%
\end{cases}
\\\text{(by (\ref{pf.thm.Nabla.prod0.c.2}))}}}w\\
&  =\sum_{w\in S_{n}}%
\begin{cases}
1, & \text{if }\left\vert w\left(  A\right)  \cap D\right\vert =\left\vert
B\cap C\right\vert ;\\
0, & \text{else}%
\end{cases}
w\\
&  =\sum_{\substack{w\in S_{n};\\\left\vert w\left(  A\right)  \cap
D\right\vert =\left\vert B\cap C\right\vert }}w.
\end{align*}
Multiplying this equality by $\omega_{B,C}$, we find%
\begin{align*}
&  \omega_{B,C}\sum_{\substack{U\subseteq D,\\V\subseteq A;\\\left\vert
U\right\vert =\left\vert V\right\vert }}\left(  -1\right)  ^{\left\vert
U\right\vert -\left\vert B\cap C\right\vert }\dbinom{\left\vert U\right\vert
}{\left\vert B\cap C\right\vert }\nabla_{U,V}\\
&  =\omega_{B,C}\sum_{\substack{w\in S_{n};\\\left\vert w\left(  A\right)
\cap D\right\vert =\left\vert B\cap C\right\vert }}w=\nabla_{D,C}\nabla
_{B,A}\ \ \ \ \ \ \ \ \ \ \left(  \text{by Theorem \ref{thm.Nabla.prod0}
\textbf{(a)}}\right)  .
\end{align*}
Thus, Theorem \ref{thm.Nabla.prod0} \textbf{(b)} is proved. \medskip

\textbf{(c)} Theorem \ref{thm.Nabla.prod0} \textbf{(b)} yields%
\begin{align*}
\nabla_{D,C}\nabla_{B,A}  &  =\omega_{B,C}\sum_{\substack{U\subseteq
D,\\V\subseteq A;\\\left\vert U\right\vert =\left\vert V\right\vert }}\left(
-1\right)  ^{\left\vert U\right\vert -\left\vert B\cap C\right\vert }%
\dbinom{\left\vert U\right\vert }{\left\vert B\cap C\right\vert }\nabla
_{U,V}\\
&  =\omega_{B,C}\sum_{\substack{U\subseteq D,\\V\subseteq A;\\\left\vert
U\right\vert =\left\vert V\right\vert }}\left(  -1\right)  ^{\left\vert
V\right\vert -\left\vert B\cap C\right\vert }\dbinom{\left\vert V\right\vert
}{\left\vert B\cap C\right\vert }\nabla_{U,V}\\
&  \ \ \ \ \ \ \ \ \ \ \ \ \ \ \ \ \ \ \ \ \left(  \text{due to the
}\left\vert U\right\vert =\left\vert V\right\vert \text{ condition under the
summation sign}\right) \\
&  =\omega_{B,C}\sum_{V\subseteq A}\left(  -1\right)  ^{\left\vert
V\right\vert -\left\vert B\cap C\right\vert }\dbinom{\left\vert V\right\vert
}{\left\vert B\cap C\right\vert }\underbrace{\sum_{\substack{U\subseteq
D;\\\left\vert U\right\vert =\left\vert V\right\vert }}\nabla_{U,V}%
}_{\substack{=\widetilde{\nabla}_{D,V}\\\text{(by Proposition
\ref{prop.Nabla.simple} \textbf{(c)})}}}\\
&  =\omega_{B,C}\sum_{V\subseteq A}\left(  -1\right)  ^{\left\vert
V\right\vert -\left\vert B\cap C\right\vert }\dbinom{\left\vert V\right\vert
}{\left\vert B\cap C\right\vert }\widetilde{\nabla}_{D,V}.
\end{align*}
This proves Theorem \ref{thm.Nabla.prod0} \textbf{(c)}.
\end{proof}

\subsection{The $D$-filtration}

We shall next derive some nilpotency-type consequences from the multiplication rule.

For the rest of this section, we fix a subset $D$ of $\left[  n\right]  $. We
define\footnote{Here and in the following, \textquotedblleft%
$\operatorname*{span}$\textquotedblright\ always means \textquotedblleft%
$\operatorname*{span}\nolimits_{\mathbf{k}}$\textquotedblright\ (that is, a
$\mathbf{k}$-linear span).}%
\[
\mathcal{F}_{k}:=\operatorname*{span}\left\{  \widetilde{\nabla}_{D,C}%
\ \mid\ C\subseteq\left[  n\right]  \text{ with }\left\vert C\right\vert \leq
k\right\}
\]
for each $k\in\mathbb{Z}$. Of course, $\mathcal{F}_{n}\supseteq\mathcal{F}%
_{n-1}\supseteq\cdots\supseteq\mathcal{F}_{0}\supseteq\mathcal{F}_{-1}=0$. It
is easy to see that $\mathcal{F}_{0}$ is spanned by $\widetilde{\nabla
}_{D,\varnothing}=\nabla_{\varnothing,\varnothing}=\sum_{w\in S_{n}}w$. (Note,
however, that $\bigcup_{k\in\mathbb{Z}}\mathcal{F}_{k}$ is usually a proper
subset of $\mathbf{k}\left[  S_{n}\right]  $, so that the $\mathcal{F}_{k}$ do
not form a decreasing filtration of $\mathbf{k}\left[  S_{n}\right]  $.)

\begin{definition}
For any subset $C\subseteq\left[  n\right]  $ and any $k\in\mathbb{N}$, we
define the integer%
\[
\delta_{D,C,k}:=\sum\limits_{\substack{B\subseteq D;\\\left\vert B\right\vert
=k}}\omega_{B,C}\left(  -1\right)  ^{k-\left\vert B\cap C\right\vert }%
\dbinom{k}{\left\vert B\cap C\right\vert }\in\mathbb{Z}.
\]

\end{definition}

Now, we note the following:

\begin{proposition}
\label{prop.Fk.1}Let $C\subseteq\left[  n\right]  $ satisfy $\left\vert
C\right\vert =\left\vert D\right\vert $. Let $k\in\mathbb{N}$. Then,%
\[
\left(  \nabla_{D,C}-\delta_{D,C,k}\right)  \mathcal{F}_{k}\subseteq
\mathcal{F}_{k-1}.
\]

\end{proposition}

\begin{proof}
By the definition of $\mathcal{F}_{k}$, it suffices to show that
\begin{equation}
\left(  \nabla_{D,C}-\delta_{D,C,k}\right)  \widetilde{\nabla}_{D,A}%
\in\mathcal{F}_{k-1} \label{pf.prop.Fk.1.goal}%
\end{equation}
for each $A\subseteq\left[  n\right]  $ with $\left\vert A\right\vert \leq k$.

To prove this, we fix $A\subseteq\left[  n\right]  $ with $\left\vert
A\right\vert \leq k$. Then, Proposition \ref{prop.Nabla.simple} \textbf{(c)}
yields%
\begin{equation}
\widetilde{\nabla}_{D,A}=\sum\limits_{\substack{U\subseteq D;\\\left\vert
U\right\vert =\left\vert A\right\vert }}\nabla_{U,A}=\sum
\limits_{\substack{B\subseteq D;\\\left\vert B\right\vert =\left\vert
A\right\vert }}\nabla_{B,A}. \label{pf.prop.Fk.1.3}%
\end{equation}
Multiplying this equality by $\nabla_{D,C}$ from the left, we obtain%
\begin{equation}
\nabla_{D,C}\widetilde{\nabla}_{D,A}=\sum\limits_{\substack{B\subseteq
D;\\\left\vert B\right\vert =\left\vert A\right\vert }}\nabla_{D,C}%
\nabla_{B,A}. \label{pf.prop.Fk.1.3m}%
\end{equation}

However, for each subset $B\subseteq D$ satisfying $\left\vert B\right\vert
=\left\vert A\right\vert $, we can use Theorem \ref{thm.Nabla.prod0}
\textbf{(c)} to obtain%
\begin{align}
\nabla_{D,C}\nabla_{B,A}  &  =\omega_{B,C}\sum_{V\subseteq A}\left(
-1\right)  ^{\left\vert V\right\vert -\left\vert B\cap C\right\vert }%
\dbinom{\left\vert V\right\vert }{\left\vert B\cap C\right\vert }%
\widetilde{\nabla}_{D,V}\nonumber\\
&  =\sum_{V\subseteq A}\omega_{B,C}\left(  -1\right)  ^{\left\vert
V\right\vert -\left\vert B\cap C\right\vert }\dbinom{\left\vert V\right\vert
}{\left\vert B\cap C\right\vert }\widetilde{\nabla}_{D,V}\nonumber\\
&  =\sum_{\substack{V\subseteq A;\\V\neq A}}\omega_{B,C}\left(  -1\right)
^{\left\vert V\right\vert -\left\vert B\cap C\right\vert }\dbinom{\left\vert
V\right\vert }{\left\vert B\cap C\right\vert }\underbrace{\widetilde{\nabla
}_{D,V}}_{\substack{\in\mathcal{F}_{k-1}\\\text{(since }V\subseteq A\text{ and
}V\neq A\\\text{entail }\left\vert V\right\vert <\left\vert A\right\vert \leq
k\text{ and}\\\text{thus }\left\vert V\right\vert \leq k-1\text{)}%
}}\nonumber\\
&  \ \ \ \ \ \ \ \ \ \ +\omega_{B,C}\left(  -1\right)  ^{\left\vert
A\right\vert -\left\vert B\cap C\right\vert }\dbinom{\left\vert A\right\vert
}{\left\vert B\cap C\right\vert }\widetilde{\nabla}_{D,A}\nonumber\\
&  \ \ \ \ \ \ \ \ \ \ \ \ \ \ \ \ \ \ \ \ \left(
\begin{array}
[c]{c}%
\text{here, we have split off the}\\
\text{addend for }V=A\text{ from the sum}%
\end{array}
\right) \nonumber\\
&  \equiv\omega_{B,C}\left(  -1\right)  ^{\left\vert A\right\vert -\left\vert
B\cap C\right\vert }\dbinom{\left\vert A\right\vert }{\left\vert B\cap
C\right\vert }\widetilde{\nabla}_{D,A}\operatorname{mod}\mathcal{F}_{k-1}.
\label{pf.prop.Fk.1.NN}%
\end{align}

Recall that $\left\vert A\right\vert \leq k$. Hence, we are in one of the
following two cases:

\textit{Case 1:} We have $\left\vert A\right\vert =k$.

\textit{Case 2:} We have $\left\vert A\right\vert <k$.

Let us first consider Case 1. In this case, we have $\left\vert A\right\vert
=k$. Hence, (\ref{pf.prop.Fk.1.3m}) becomes%
\begin{align*}
\nabla_{D,C}\widetilde{\nabla}_{D,A}  &  =\sum\limits_{\substack{B\subseteq
D;\\\left\vert B\right\vert =\left\vert A\right\vert }}\nabla_{D,C}%
\nabla_{B,A}\\
&  \equiv\sum\limits_{\substack{B\subseteq D;\\\left\vert B\right\vert
=\left\vert A\right\vert }}\omega_{B,C}\left(  -1\right)  ^{\left\vert
A\right\vert -\left\vert B\cap C\right\vert }\dbinom{\left\vert A\right\vert
}{\left\vert B\cap C\right\vert }\widetilde{\nabla}_{D,A}%
\ \ \ \ \ \ \ \ \ \ \left(  \text{by (\ref{pf.prop.Fk.1.NN})}\right) \\
&  =\underbrace{\sum\limits_{\substack{B\subseteq D;\\\left\vert B\right\vert
=k}}\omega_{B,C}\left(  -1\right)  ^{k-\left\vert B\cap C\right\vert }%
\dbinom{k}{\left\vert B\cap C\right\vert }}_{=\delta_{D,C,k}}\widetilde{\nabla
}_{D,A}\ \ \ \ \ \ \ \ \ \ \left(  \text{since }\left\vert A\right\vert
=k\right) \\
&  =\delta_{D,C,k}\widetilde{\nabla}_{D,A}\operatorname{mod}\mathcal{F}_{k-1}.
\end{align*}
In other words, $\nabla_{D,C}\widetilde{\nabla}_{D,A}-\delta_{D,C,k}%
\widetilde{\nabla}_{D,A}\in\mathcal{F}_{k-1}$. In other words, $\left(
\nabla_{D,C}-\delta_{D,C,k}\right)  \widetilde{\nabla}_{D,A}\in\mathcal{F}%
_{k-1}$. Hence, (\ref{pf.prop.Fk.1.goal}) is proved in Case 1.

Let us now consider Case 2. In this case, we have $\left\vert A\right\vert
<k$. Hence, $\left\vert A\right\vert \leq k-1$, so that $\widetilde{\nabla
}_{D,A}\in\mathcal{F}_{k-1}$. In other words, $\widetilde{\nabla}_{D,A}%
\equiv0\operatorname{mod}\mathcal{F}_{k-1}$. Now, (\ref{pf.prop.Fk.1.3m})
becomes%
\begin{align*}
\nabla_{D,C}\widetilde{\nabla}_{D,A}  &  =\sum\limits_{\substack{B\subseteq
D;\\\left\vert B\right\vert =\left\vert A\right\vert }}\nabla_{D,C}%
\nabla_{B,A}\\
&  \equiv\sum\limits_{\substack{B\subseteq D;\\\left\vert B\right\vert
=\left\vert A\right\vert }}\omega_{B,C}\left(  -1\right)  ^{\left\vert
A\right\vert -\left\vert B\cap C\right\vert }\dbinom{\left\vert A\right\vert
}{\left\vert B\cap C\right\vert }\underbrace{\widetilde{\nabla}_{D,A}}%
_{\equiv0\operatorname{mod}\mathcal{F}_{k-1}}\ \ \ \ \ \ \ \ \ \ \left(
\text{by (\ref{pf.prop.Fk.1.NN})}\right) \\
&  \equiv0\operatorname{mod}\mathcal{F}_{k-1}.
\end{align*}
Hence,%
\[
\left(  \nabla_{D,C}-\delta_{D,C,k}\right)  \widetilde{\nabla}_{D,A}%
=\underbrace{\nabla_{D,C}\widetilde{\nabla}_{D,A}}_{\equiv0\operatorname{mod}%
\mathcal{F}_{k-1}}-\,\delta_{D,C,k}\underbrace{\widetilde{\nabla}_{D,A}%
}_{\equiv0\operatorname{mod}\mathcal{F}_{k-1}}\equiv0-0=0\operatorname{mod}%
\mathcal{F}_{k-1}.
\]
In other words, $\left(  \nabla_{D,C}-\delta_{D,C,k}\right)  \widetilde{\nabla
}_{D,A}\in\mathcal{F}_{k-1}$. Hence, (\ref{pf.prop.Fk.1.goal}) is proved in
Case 2.

We have now proved (\ref{pf.prop.Fk.1.goal}) in both Cases 1 and 2. Thus,
(\ref{pf.prop.Fk.1.goal}) always holds, and Proposition \ref{prop.Fk.1} is proved.
\end{proof}

\begin{definition}
\label{def.NabDalpha}Let $\alpha=\left(  \alpha_{C}\right)  _{C\subseteq
\left[  n\right]  ;\ \left\vert C\right\vert =\left\vert D\right\vert }$ be a
family of scalars in $\mathbf{k}$ indexed by the $\left\vert D\right\vert
$-element subsets of $\left[  n\right]  $. Then, we set%
\[
\nabla_{D,\alpha}:=\sum_{\substack{C\subseteq\left[  n\right]  ;\\\left\vert
C\right\vert =\left\vert D\right\vert }}\alpha_{C}\nabla_{D,C}\in\mathcal{A}.
\]
Furthermore, for each $k\in\mathbb{N}$, we set%
\[
\delta_{D,\alpha,k}:=\sum_{\substack{C\subseteq\left[  n\right]  ;\\\left\vert
C\right\vert =\left\vert D\right\vert }}\alpha_{C}\delta_{D,C,k}\in
\mathbf{k}.
\]

\end{definition}

\begin{proposition}
\label{prop.Fk.2}Let $\alpha=\left(  \alpha_{C}\right)  _{C\subseteq\left[
n\right]  ;\ \left\vert C\right\vert =\left\vert D\right\vert }$ be a family
of scalars in $\mathbf{k}$ indexed by the $\left\vert D\right\vert $-element
subsets of $\left[  n\right]  $. Let $k\in\mathbb{N}$. Then,%
\[
\left(  \nabla_{D,\alpha}-\delta_{D,\alpha,k}\right)  \mathcal{F}_{k}%
\subseteq\mathcal{F}_{k-1}.
\]

\end{proposition}

\begin{proof}
Proposition \ref{prop.Fk.1} yields $\left(  \nabla_{D,C}-\delta_{D,C,k}%
\right)  \mathcal{F}_{k}\subseteq\mathcal{F}_{k-1}$ for each $C\subseteq
\left[  n\right]  $ satisfying $\left\vert C\right\vert =\left\vert
D\right\vert $. Multiply this relation by $\alpha_{C}$ and sum up over all
$C\subseteq\left[  n\right]  $ satisfying $\left\vert C\right\vert =\left\vert
D\right\vert $. The result is Proposition \ref{prop.Fk.2}.
\end{proof}

\begin{proposition}
\label{prop.Fk.3}Let $\alpha=\left(  \alpha_{C}\right)  _{C\subseteq\left[
n\right]  ;\ \left\vert C\right\vert =\left\vert D\right\vert }$ be a family
of scalars in $\mathbf{k}$ indexed by the $\left\vert D\right\vert $-element
subsets of $\left[  n\right]  $. Then, for each integer $m\geq-1$, we have%
\[
\left(  \prod_{k=0}^{m}\left(  \nabla_{D,\alpha}-\delta_{D,\alpha,k}\right)
\right)  \mathcal{F}_{m}=0.
\]

\end{proposition}

\begin{proof}
Induction on $m$. The \textit{base case} is obvious, since $\mathcal{F}%
_{-1}=0$. The \textit{induction step} (from $m-1$ to $m$) uses Proposition
\ref{prop.Fk.2} as follows:%
\begin{align*}
\left(  \prod_{k=0}^{m}\left(  \nabla_{D,\alpha}-\delta_{D,\alpha,k}\right)
\right)  \mathcal{F}_{m}  &  =\left(  \prod_{k=0}^{m-1}\left(  \nabla
_{D,\alpha}-\delta_{D,\alpha,k}\right)  \right)  \underbrace{\left(
\nabla_{D,\alpha}-\delta_{D,\alpha,m}\right)  \mathcal{F}_{m}}%
_{\substack{\subseteq\mathcal{F}_{m-1}\\\text{(by Proposition \ref{prop.Fk.2}%
)}}}\\
&  \subseteq\left(  \prod_{k=0}^{m-1}\left(  \nabla_{D,\alpha}-\delta
_{D,\alpha,k}\right)  \right)  \mathcal{F}_{m-1}=0
\end{align*}
and thus $\left(  \prod_{k=0}^{m}\left(  \nabla_{D,\alpha}-\delta_{D,\alpha
,k}\right)  \right)  \mathcal{F}_{m}=0$. Thus, Proposition \ref{prop.Fk.3} is proved.
\end{proof}

\subsection{The triangularity theorem}

We can now state our main theorem (still using Definition \ref{def.NabDalpha}):

\begin{theorem}
\label{thm.NablaDal.triangular}Let $D$ be a subset of $\left[  n\right]  $.
Let $\alpha=\left(  \alpha_{C}\right)  _{C\subseteq\left[  n\right]
;\ \left\vert C\right\vert =\left\vert D\right\vert }$ be a family of scalars
in $\mathbf{k}$ indexed by the $\left\vert D\right\vert $-element subsets of
$\left[  n\right]  $. Then,%
\[
\left(  \prod_{k=0}^{\left\vert D\right\vert }\left(  \nabla_{D,\alpha}%
-\delta_{D,\alpha,k}\right)  \right)  \nabla_{D,\alpha}=0.
\]

\end{theorem}

\begin{proof}
For each subset $C$ of $\left[  n\right]  $ satisfying $\left\vert
C\right\vert =\left\vert D\right\vert $, we have $\nabla_{D,C}%
=\widetilde{\nabla}_{D,C}$ (by Proposition \ref{prop.Nabla.simple}
\textbf{(e)}) and thus $\nabla_{D,C}=\widetilde{\nabla}_{D,C}\in
\mathcal{F}_{\left\vert D\right\vert }$ (since $\left\vert C\right\vert
=\left\vert D\right\vert \leq\left\vert D\right\vert $). Thus, $\nabla
_{D,\alpha}\in\mathcal{F}_{\left\vert D\right\vert }$ as well (since
$\nabla_{D,\alpha}$ is a $\mathbf{k}$-linear combination of such $\nabla
_{D,C}$'s). Hence,%
\[
\left(  \prod_{k=0}^{\left\vert D\right\vert }\left(  \nabla_{D,\alpha}%
-\delta_{D,\alpha,k}\right)  \right)  \nabla_{D,\alpha}=\left(  \prod
_{k=0}^{\left\vert D\right\vert }\left(  \nabla_{D,\alpha}-\delta_{D,\alpha
,k}\right)  \right)  \mathcal{F}_{\left\vert D\right\vert }=0
\]
(by Proposition \ref{prop.Fk.3}, applied to $m=\left\vert D\right\vert $).
This proves Theorem \ref{thm.NablaDal.triangular}.
\end{proof}

Using the antipode $S$ of $\mathcal{A}$, we can obtain a reflected version of
Theorem \ref{thm.NablaDal.triangular}:

\begin{theorem}
\label{thm.NablaDal.triangular2}Let $D$ be a subset of $\left[  n\right]  $.
Let $\alpha=\left(  \alpha_{C}\right)  _{C\subseteq\left[  n\right]
;\ \left\vert C\right\vert =\left\vert D\right\vert }$ be a family of scalars
in $\mathbf{k}$ indexed by the $\left\vert D\right\vert $-element subsets of
$\left[  n\right]  $. Set
\[
\nabla_{\alpha,D}:=\sum_{\substack{C\subseteq\left[  n\right]  ;\\\left\vert
C\right\vert =\left\vert D\right\vert }}\alpha_{C}\nabla_{C,D}\in\mathcal{A}.
\]
Then,%
\[
\left(  \prod_{k=0}^{\left\vert D\right\vert }\left(  \nabla_{\alpha,D}%
-\delta_{D,\alpha,k}\right)  \right)  \nabla_{\alpha,D}=0.
\]

\end{theorem}

\begin{proof}
The antipode $S$ is a $\mathbf{k}$-algebra anti-homomorphism, and sends
$\nabla_{D,\alpha}$ to $\nabla_{\alpha,D}$ (since it sends $\nabla_{D,C}$ to
$\nabla_{C,D}$ for each $C$). Thus, Theorem \ref{thm.NablaDal.triangular2}
follows easily by applying the antipode to Theorem
\ref{thm.NablaDal.triangular}. (Note that we don't have to reverse the order
of factors in the product, since all these factors commute with each other.)
\end{proof}

\begin{corollary}
\label{cor.Nablatil.triangular}Let $B$ and $D$ be two subsets of $\left[
n\right]  $. For each $k\in\mathbb{N}$, we set%
\[
\widetilde{\delta}_{D,B,k}:=\sum_{\substack{C\subseteq B;\\\left\vert
C\right\vert =\left\vert D\right\vert }}\delta_{D,C,k}\in\mathbb{Z}.
\]
Then,%
\[
\left(  \prod_{k=0}^{\left\vert D\right\vert }\left(  \widetilde{\nabla}%
_{B,D}-\widetilde{\delta}_{D,B,k}\right)  \right)  \widetilde{\nabla}%
_{B,D}=0.
\]

\end{corollary}

\begin{proof}
Define a family $\alpha=\left(  \alpha_{C}\right)  _{C\subseteq\left[
n\right]  ;\ \left\vert C\right\vert =\left\vert D\right\vert }$ of scalars in
$\mathbf{k}$ by setting%
\[
\alpha_{C}=%
\begin{cases}
1, & \text{if }C\subseteq B;\\
0, & \text{if }C\not \subseteq B
\end{cases}
\ \ \ \ \ \ \ \ \ \ \text{for each }C\subseteq\left[  n\right]  \text{.}%
\]

Then, Proposition \ref{prop.Nabla.simple} \textbf{(c)} yields
\[
\widetilde{\nabla}_{B,D}=\sum\limits_{\substack{U\subseteq B;\\\left\vert
U\right\vert =\left\vert D\right\vert }}\nabla_{U,D}=\sum
\limits_{\substack{C\subseteq B;\\\left\vert C\right\vert =\left\vert
D\right\vert }}\nabla_{C,D}=\sum\limits_{\substack{C\subseteq\left[  n\right]
;\\\left\vert C\right\vert =\left\vert D\right\vert }}\alpha_{C}\nabla
_{C,D}=\nabla_{\alpha,D},
\]
where $\nabla_{\alpha,D}$ is defined as in Theorem
\ref{thm.NablaDal.triangular2}. Hence, Corollary \ref{cor.Nablatil.triangular}
follows from Theorem \ref{thm.NablaDal.triangular2}, once we realize that the
$\delta_{D,\alpha,k}$ from Theorem \ref{thm.NablaDal.triangular2} is precisely
the $\widetilde{\delta}_{D,B,k}$.
\end{proof}

Corollary \ref{cor.Nablatil.triangular} shows that the element
$\widetilde{\nabla}_{B,D}$ has a minimal polynomial that factors entirely into
linear factors. Moreover, there are at most $\left\vert D\right\vert +2$
factors, and one of them is $X$ or else there are at most $\left\vert
D\right\vert +1$ of them.

\begin{question}
Can we simplify the formula for $\widetilde{\delta}_{D,B,k}$ ?
\end{question}

\subsection{A table of minimal polynomials}

For any subsets $A$ and $B$ of $\left[  n\right]  $, we let $\kappa_{A,B}$ be
the sum (in $\mathbb{Z}\left[  S_{n}\right]  $) of all permutations $w\in
S_{n}$ that satisfy $w\left(  A\right)  \cap B=\varnothing$ (that is,
$w\left(  a\right)  \notin B$ for all $a\in A$). Then, $\kappa_{A,B}$ is
simply $\widetilde{\nabla}_{\left[  n\right]  \setminus B,\ A}$. Thus,
Corollary \ref{cor.Nablatil.triangular} shows that the element $\kappa_{A,B}$
has a minimal polynomial that factors into at most $\left\vert A\right\vert
+2$ factors. Note that these factors will sometimes have multiplicities (e.g.,
the case of $n=6$ and $a=3$ and $b=2$ and $c=1$).

Let us collect a table of these minimal polynomials. We observe that the
minimal polynomial of $\kappa_{B,A}$ depends only on the three numbers
$a:=\left\vert A\right\vert $, $b:=\left\vert B\right\vert $ and
$c:=\left\vert A\cap B\right\vert $ (since any two pairs $\left(  A,B\right)
$ that agree in these three numbers can be obtained from each other by the
action of some permutation $\sigma\in S_{n}$, and therefore the corresponding
elements $\kappa_{B,A}$ are conjugate to each other in $\mathcal{A}$). Hence,
we can rename $\kappa_{B,A}$ as $\kappa_{a,b,c}$.

We also note that $\kappa_{a,b,c}=0$ if $c>a$ or $b>a$ or $a+b>n$. Hence, we
only need to consider the cases $a,b\in\left[  0,n\right]  $ and $a+b\leq n$
and $c\in\left[  0,\ \min\left\{  a,b\right\}  \right]  $.

Moreover, $\kappa_{B,A}$ is the antipode of $\kappa_{A,B}$ (by Proposition
\ref{prop.Nabla.simple} \textbf{(g)}), and the antipode preserves minimal
polynomials. Thus, we only need to consider the case $a\leq b$.

This being said, here is a table of minpols (= minimal polynomials) of
$\kappa_{a,b,c}$'s produced by SageMath:

------------------------------------------------------ \newline Let $n=1$.
\newline For $b=0$, the minpol is $x-1$. \newline

------------------------------------------------------ \newline Let $n=2$.
\newline For $b=0$, the minpol is $(x-2)x$. \newline For $a=1$ and $b=1$ and
$c=0$, the minpol is $x-1$. \newline For $a=1$ and $b=1$ and $c=1$, the minpol
is $(x-1)(x+1)$. \newline%
------------------------------------------------------ \newline Let $n=3$.
\newline For $b=0$, the minpol is $(x-6)x$. \newline For $a=1$ and $b=1$ and
$c=0$, the minpol is $(x-4)(x-1)x$. \newline For $a=1$ and $b=1$ and $c=1$,
the minpol is $(x-4)x(x+2)$. \newline For $a=2$ and $b=1$ and $c=0$, the
minpol is $(x-2)x$. \newline For $a=2$ and $b=1$ and $c=1$, the minpol is
$(x-2)x(x+1)$. \newline

------------------------------------------------------ \newline Let $n=4$.
\newline For $b=0$, the minpol is $(x-24)x$. \newline For $a=1$ and $b=1$ and
$c=0$, the minpol is $(x-18)(x-2)x$. \newline For $a=1$ and $b=1$ and $c=1$,
the minpol is $(x-18)x(x+6)$. \newline For $a=2$ and $b=1$ and $c=0$, the
minpol is $(x-12)(x-4)x$. \newline For $a=2$ and $b=1$ and $c=1$, the minpol
is $(x-12)x(x+4)$. \newline For $a=3$ and $b=1$ and $c=0$, the minpol is
$(x-6)x$. \newline For $a=3$ and $b=1$ and $c=1$, the minpol is $(x-6)x(x+2)$.
\newline For $a=2$ and $b=2$ and $c=0$, the minpol is $(x-4)x$. \newline For
$a=2$ and $b=2$ and $c=1$, the minpol is $(x-4)(x+2)x^{2}$. \newline For $a=2$
and $b=2$ and $c=2$, the minpol is $(x-4)x(x+4)$. \newline

------------------------------------------------------ \newline Let $n=5$.
\newline For $b=0$, the minpol is $(x-120)x$. \newline For $a=1$ and $b=1$ and
$c=0$, the minpol is $(x-96)(x-6)x$. \newline For $a=1$ and $b=1$ and $c=1$,
the minpol is $(x-96)x(x+24)$. \newline For $a=2$ and $b=1$ and $c=0$, the
minpol is $(x-72)(x-12)x$. \newline For $a=2$ and $b=1$ and $c=1$, the minpol
is $(x-72)x(x+18)$. \newline For $a=3$ and $b=1$ and $c=0$, the minpol is
$(x-48)(x-18)x$. \newline For $a=3$ and $b=1$ and $c=1$, the minpol is
$(x-48)x(x+12)$. \newline For $a=4$ and $b=1$ and $c=0$, the minpol is
$(x-24)x$. \newline For $a=4$ and $b=1$ and $c=1$, the minpol is
$(x-24)x(x+6)$. \newline For $a=2$ and $b=2$ and $c=0$, the minpol is
$(x-36)(x-16)(x-4)x$. \newline For $a=2$ and $b=2$ and $c=1$, the minpol is
$(x-36)x(x+4)$. \newline For $a=2$ and $b=2$ and $c=2$, the minpol is
$(x-36)(x-12)x(x+24)$. \newline For $a=3$ and $b=2$ and $c=0$, the minpol is
$(x-12)x$. \newline For $a=3$ and $b=2$ and $c=1$, the minpol is
$(x-12)(x-2)x(x+4)$. \newline For $a=3$ and $b=2$ and $c=2$, the minpol is
$(x-12)(x-4)x(x+8)$. \newline

------------------------------------------------------ \newline Let $n=6$.
\newline For $b=0$, the minpol is $(x-720)x$.\newline For $a=1$ and $b=1$ and
$c=0$, the minpol is $(x-600)(x-24)x$. \newline For $a=1$ and $b=1$ and $c=1$,
the minpol is $(x-600)x(x+120)$. \newline For $a=2$ and $b=1$ and $c=0$, the
minpol is $(x-480)(x-48)x$. \newline For $a=2$ and $b=1$ and $c=1$, the minpol
is $(x-480)x(x+96)$. \newline For $a=3$ and $b=1$ and $c=0$, the minpol is
$(x-360)(x-72)x$. \newline For $a=3$ and $b=1$ and $c=1$, the minpol is
$(x-360)x(x+72)$. \newline For $a=4$ and $b=1$ and $c=0$, the minpol is
$(x-240)(x-96)x$. \newline For $a=4$ and $b=1$ and $c=1$, the minpol is
$(x-240)x(x+48)$. \newline For $a=5$ and $b=1$ and $c=0$, the minpol is
$(x-120)x$. \newline For $a=5$ and $b=1$ and $c=1$, the minpol is
$(x-120)x(x+24)$. \newline For $a=2$ and $b=2$ and $c=0$, the minpol is
$(x-288)(x-72)(x-8)x$. \newline For $a=2$ and $b=2$ and $c=1$, the minpol is
$(x-288)x(x+12)(x+36)$. \newline For $a=2$ and $b=2$ and $c=2$, the minpol is
$(x-288)(x-48)x(x+144)$. \newline For $a=3$ and $b=2$ and $c=0$, the minpol is
$(x-144)(x-72)(x-24)x$. \newline For $a=3$ and $b=2$ and $c=1$, the minpol is
$(x-144)(x+16)x^{2}$. \newline For $a=3$ and $b=2$ and $c=2$, the minpol is
$(x-144)(x-24)x(x+72)$. \newline For $a=4$ and $b=2$ and $c=0$, the minpol is
$(x-48)x$. \newline For $a=4$ and $b=2$ and $c=1$, the minpol is
$(x-48)(x-12)x(x+12)$. \newline For $a=4$ and $b=2$ and $c=2$, the minpol is
$(x-48)(x-8)x(x+24)$. \newline For $a=3$ and $b=3$ and $c=0$, the minpol is
$(x-36)x$. \newline For $a=3$ and $b=3$ and $c=1$, the minpol is
$(x-36)(x-12)x(x+4)(x+12)$. \newline For $a=3$ and $b=3$ and $c=2$, the minpol
is $(x-36)(x-12)x(x+4)(x+12)$. \newline For $a=3$ and $b=3$ and $c=3$, the
minpol is $(x-36)x(x+36)$. \newline

\subsection{Aside: The abstract Nabla-algebra}

We take a tangent and address a question that is suggested by Theorem
\ref{thm.Nabla.prod0} \textbf{(b)} but takes us out of the symmetric group
algebra $\mathcal{A}$. Namely, let us see what happens if we take the
multiplication rule in Theorem \ref{thm.Nabla.prod0} \textbf{(b)} literally
while forgetting what the $\nabla_{B,A}$ are.

\begin{theorem}
\label{thm.Nabla-alg.ass}For any two subsets $A$ and $B$ of $\left[  n\right]
$ satisfying $\left\vert A\right\vert =\left\vert B\right\vert $, introduce a
formal symbol $\Delta_{B,A}$. Thus, we have introduced altogether $\sum
_{k=0}^{n}\dbinom{n}{k}^{2}=\dbinom{2n}{n}$ symbols $\Delta_{B,A}$. Let
$\mathcal{D}$ be the free $\mathbf{k}$-module with basis $\left(  \Delta
_{B,A}\right)  _{A,B\subseteq\left[  n\right]  \text{ with }\left\vert
A\right\vert =\left\vert B\right\vert }$. Define a multiplication on
$\mathcal{D}$ by%
\[
\Delta_{D,C}\Delta_{B,A}:=\omega_{B,C}\sum_{\substack{U\subseteq
D,\\V\subseteq A;\\\left\vert U\right\vert =\left\vert V\right\vert }}\left(
-1\right)  ^{\left\vert U\right\vert -\left\vert B\cap C\right\vert }%
\dbinom{\left\vert U\right\vert }{\left\vert B\cap C\right\vert }\Delta
_{U,V}.
\]
(Recall Definition \ref{def.omegaBC}, which defines the $\omega_{B,C}$ here.)
Then, $\mathcal{D}$ becomes a nonunital $\mathbf{k}$-algebra.
\end{theorem}

\textbf{Proof omitted due to excessive ugliness.}

\begin{noncompile}
New idea: Set $\Gamma_{B,A}:=\sum_{\substack{P\subseteq B\text{ and
}Q\subseteq A;\\\left\vert P\right\vert =\left\vert Q\right\vert }}\left(
-1\right)  ^{\left\vert A\right\vert -\left\vert Q\right\vert }\Delta_{P,Q}$.
Then,%
\begin{align*}
\Delta_{D,C}\Gamma_{B,A}  &  =\sum_{\substack{P\subseteq B\text{ and
}Q\subseteq A;\\\left\vert P\right\vert =\left\vert Q\right\vert }}\left(
-1\right)  ^{\left\vert A\right\vert -\left\vert Q\right\vert }\Delta
_{D,C}\Delta_{P,Q}\\
&  =\sum_{\substack{P\subseteq B\text{ and }Q\subseteq A;\\\left\vert
P\right\vert =\left\vert Q\right\vert }}\left(  -1\right)  ^{\left\vert
A\right\vert -\left\vert Q\right\vert }\omega_{P,C}\sum_{\substack{U\subseteq
D,\\V\subseteq Q;\\\left\vert U\right\vert =\left\vert V\right\vert }}\left(
-1\right)  ^{\left\vert U\right\vert -\left\vert P\cap C\right\vert }%
\dbinom{\left\vert U\right\vert }{\left\vert P\cap C\right\vert }\Delta
_{U,V}\\
&  =\sum_{\substack{U\subseteq D;\\V\subseteq A;\\\left\vert U\right\vert
=\left\vert V\right\vert }}\ \ \sum_{\substack{P\subseteq B\text{ and
}Q\subseteq A;\\\left\vert P\right\vert =\left\vert Q\right\vert }}\left(
-1\right)  ^{\left\vert A\right\vert -\left\vert P\right\vert }\omega
_{P,C}\left(  -1\right)  ^{\left\vert U\right\vert -\left\vert P\cap
C\right\vert }\dbinom{\left\vert U\right\vert }{\left\vert P\cap C\right\vert
}\Delta_{U,V}\\
&  =\sum_{\substack{U\subseteq D;\\V\subseteq A;\\\left\vert U\right\vert
=\left\vert V\right\vert }}\ \ \sum_{P\subseteq B}\left(  -1\right)
^{\left\vert A\right\vert -\left\vert P\right\vert }\omega_{P,C}\left(
-1\right)  ^{\left\vert U\right\vert -\left\vert P\cap C\right\vert }%
\dbinom{\left\vert U\right\vert }{\left\vert P\cap C\right\vert }%
\dbinom{\left\vert A\right\vert -\left\vert V\right\vert }{\left\vert
P\right\vert -\left\vert V\right\vert }\Delta_{U,V}.
\end{align*}
But%
\begin{align*}
&  \sum_{P\subseteq B}\left(  -1\right)  ^{\left\vert A\right\vert -\left\vert
P\right\vert }\omega_{P,C}\left(  -1\right)  ^{\left\vert U\right\vert
-\left\vert P\cap C\right\vert }\dbinom{\left\vert U\right\vert }{\left\vert
P\cap C\right\vert }\dbinom{\left\vert A\right\vert -\left\vert V\right\vert
}{\left\vert P\right\vert -\left\vert V\right\vert }\\
&  =\sum_{P\subseteq B}\left(  -1\right)  ^{\left\vert A\right\vert
+\left\vert U\right\vert -\left\vert P\setminus C\right\vert }\omega
_{P,C}\dbinom{\left\vert U\right\vert }{\left\vert P\cap C\right\vert }%
\dbinom{\left\vert A\right\vert -\left\vert U\right\vert }{\left\vert
P\right\vert -\left\vert U\right\vert }%
\end{align*}

\end{noncompile}

\begin{noncompile}
\begin{proof}
[Proof idea.]It would be nice if the $\nabla_{B,A}$'s were linearly
independent (so that $\mathcal{D}$ would just become a subalgebra of
$\mathcal{A}$ after renaming $\Delta_{B,A}$ as $\nabla_{B,A}$). Alas, they are
not (Proposition \ref{prop.Nabla.simple} \textbf{(d)} is a relation, and not
the only one). The following trick lets us nevertheless pretend that they are:

Fix $A,B,C,D,E,F\subseteq\left[  n\right]  $ with $\left\vert A\right\vert
=\left\vert B\right\vert $ and $\left\vert C\right\vert =\left\vert
D\right\vert $ and $\left\vert E\right\vert =\left\vert F\right\vert $. We
need to prove that
\[
\Delta_{F,E}\left(  \Delta_{D,C}\Delta_{B,A}\right)  =\left(  \Delta
_{F,E}\Delta_{D,C}\right)  \Delta_{B,A}.
\]
We have%
\begin{align*}
&  \Delta_{F,E}\left(  \Delta_{D,C}\Delta_{B,A}\right) \\
&  =\Delta_{F,E}\omega_{B,C}\sum_{\substack{U\subseteq D,\\V\subseteq
A;\\\left\vert U\right\vert =\left\vert V\right\vert }}\left(  -1\right)
^{\left\vert U\right\vert -\left\vert B\cap C\right\vert }\dbinom{\left\vert
U\right\vert }{\left\vert B\cap C\right\vert }\Delta_{U,V}\\
&  =\omega_{B,C}\sum_{\substack{U\subseteq D,\\V\subseteq A;\\\left\vert
U\right\vert =\left\vert V\right\vert }}\left(  -1\right)  ^{\left\vert
U\right\vert -\left\vert B\cap C\right\vert }\dbinom{\left\vert U\right\vert
}{\left\vert B\cap C\right\vert }\Delta_{F,E}\Delta_{U,V}\\
&  =\omega_{B,C}\sum_{\substack{U\subseteq D,\\V\subseteq A;\\\left\vert
U\right\vert =\left\vert V\right\vert }}\left(  -1\right)  ^{\left\vert
U\right\vert -\left\vert B\cap C\right\vert }\dbinom{\left\vert U\right\vert
}{\left\vert B\cap C\right\vert }\omega_{U,E}\sum_{\substack{U^{\prime
}\subseteq F,\\V^{\prime}\subseteq V;\\\left\vert U^{\prime}\right\vert
=\left\vert V^{\prime}\right\vert }}\left(  -1\right)  ^{\left\vert U^{\prime
}\right\vert -\left\vert U\cap E\right\vert }\dbinom{\left\vert U^{\prime
}\right\vert }{\left\vert U\cap E\right\vert }\Delta_{U^{\prime},V^{\prime}}\\
&  =\omega_{B,C}\sum_{\substack{U^{\prime}\subseteq F,\\V^{\prime}\subseteq
A;\\\left\vert U^{\prime}\right\vert =\left\vert V^{\prime}\right\vert
}}\ \ \sum_{\substack{U\subseteq D,\\V\subseteq A;\\\left\vert U\right\vert
=\left\vert V\right\vert ;\\V^{\prime}\subseteq V}}\left(  -1\right)
^{\left\vert U\right\vert -\left\vert B\cap C\right\vert }\dbinom{\left\vert
U\right\vert }{\left\vert B\cap C\right\vert }\omega_{U,E}\left(  -1\right)
^{\left\vert U^{\prime}\right\vert -\left\vert U\cap E\right\vert }%
\dbinom{\left\vert U^{\prime}\right\vert }{\left\vert U\cap E\right\vert
}\Delta_{U^{\prime},V^{\prime}}%
\end{align*}
and
\begin{align*}
&  \left(  \Delta_{F,E}\Delta_{D,C}\right)  \Delta_{B,A}\\
&  =\omega_{D,E}\sum_{\substack{U\subseteq F,\\V\subseteq C;\\\left\vert
U\right\vert =\left\vert V\right\vert }}\left(  -1\right)  ^{\left\vert
U\right\vert -\left\vert D\cap E\right\vert }\dbinom{\left\vert U\right\vert
}{\left\vert D\cap E\right\vert }\Delta_{U,V}\Delta_{B,A}\\
&  =\omega_{D,E}\sum_{\substack{U\subseteq F,\\V\subseteq C;\\\left\vert
U\right\vert =\left\vert V\right\vert }}\left(  -1\right)  ^{\left\vert
U\right\vert -\left\vert D\cap E\right\vert }\dbinom{\left\vert U\right\vert
}{\left\vert D\cap E\right\vert }\omega_{B,V}\sum_{\substack{U^{\prime
}\subseteq U,\\V^{\prime}\subseteq A;\\\left\vert U^{\prime}\right\vert
=\left\vert V^{\prime}\right\vert }}\left(  -1\right)  ^{\left\vert U^{\prime
}\right\vert -\left\vert B\cap V\right\vert }\dbinom{\left\vert U^{\prime
}\right\vert }{\left\vert B\cap V\right\vert }\Delta_{U^{\prime},V^{\prime}}\\
&  =\omega_{D,E}\sum_{\substack{U^{\prime}\subseteq F,\\V^{\prime}\subseteq
A;\\\left\vert U^{\prime}\right\vert =\left\vert V^{\prime}\right\vert
}}\ \ \sum_{\substack{U\subseteq F,\\V\subseteq C;\\\left\vert U\right\vert
=\left\vert V\right\vert ;\\U^{\prime}\subseteq U}}\left(  -1\right)
^{\left\vert U\right\vert -\left\vert D\cap E\right\vert }\dbinom{\left\vert
U\right\vert }{\left\vert D\cap E\right\vert }\omega_{B,V}\left(  -1\right)
^{\left\vert U^{\prime}\right\vert -\left\vert B\cap V\right\vert }%
\dbinom{\left\vert U^{\prime}\right\vert }{\left\vert B\cap V\right\vert
}\Delta_{U^{\prime},V^{\prime}}.
\end{align*}
Thus, we need to show that any two subsets $U^{\prime}\subseteq F$ and
$V^{\prime}\subseteq A$ satisfying $\left\vert U^{\prime}\right\vert
=\left\vert V^{\prime}\right\vert $ satisfy%
\begin{align*}
&  \omega_{B,C}\sum_{\substack{U\subseteq D,\\V\subseteq A;\\\left\vert
U\right\vert =\left\vert V\right\vert ;\\V^{\prime}\subseteq V}}\left(
-1\right)  ^{\left\vert U\right\vert -\left\vert B\cap C\right\vert }%
\dbinom{\left\vert U\right\vert }{\left\vert B\cap C\right\vert }\omega
_{U,E}\left(  -1\right)  ^{\left\vert U^{\prime}\right\vert -\left\vert U\cap
E\right\vert }\dbinom{\left\vert U^{\prime}\right\vert }{\left\vert U\cap
E\right\vert }\\
&  =\omega_{D,E}\sum_{\substack{U\subseteq F,\\V\subseteq C;\\\left\vert
U\right\vert =\left\vert V\right\vert ;\\U^{\prime}\subseteq U}}\left(
-1\right)  ^{\left\vert U\right\vert -\left\vert D\cap E\right\vert }%
\dbinom{\left\vert U\right\vert }{\left\vert D\cap E\right\vert }\omega
_{B,V}\left(  -1\right)  ^{\left\vert U^{\prime}\right\vert -\left\vert B\cap
V\right\vert }\dbinom{\left\vert U^{\prime}\right\vert }{\left\vert B\cap
V\right\vert }.
\end{align*}
To do so, we fix two such subsets $U^{\prime}\subseteq F$ and $V^{\prime
}\subseteq A$ satisfying $\left\vert U^{\prime}\right\vert =\left\vert
V^{\prime}\right\vert $. Note that the addend of the first sum does not depend
on $V$, whereas the addend of the second sum does not depend on $U$. Hence, we
can rewrite the equality we must prove as%
\begin{align*}
&  \omega_{B,C}\sum_{U\subseteq D}\dbinom{\left\vert A\right\vert -\left\vert
V^{\prime}\right\vert }{\left\vert U\right\vert -\left\vert V^{\prime
}\right\vert }\left(  -1\right)  ^{\left\vert U\right\vert -\left\vert B\cap
C\right\vert }\dbinom{\left\vert U\right\vert }{\left\vert B\cap C\right\vert
}\omega_{U,E}\left(  -1\right)  ^{\left\vert U^{\prime}\right\vert -\left\vert
U\cap E\right\vert }\dbinom{\left\vert U^{\prime}\right\vert }{\left\vert
U\cap E\right\vert }\\
&  =\omega_{D,E}\sum_{V\subseteq C}\dbinom{\left\vert F\right\vert -\left\vert
U^{\prime}\right\vert }{\left\vert V\right\vert -\left\vert U^{\prime
}\right\vert }\left(  -1\right)  ^{\left\vert U\right\vert -\left\vert D\cap
E\right\vert }\dbinom{\left\vert U\right\vert }{\left\vert D\cap E\right\vert
}\omega_{B,V}\left(  -1\right)  ^{\left\vert U^{\prime}\right\vert -\left\vert
B\cap V\right\vert }\dbinom{\left\vert U^{\prime}\right\vert }{\left\vert
B\cap V\right\vert }.
\end{align*}

TODO

What happens to $\mathcal{D}$ when we increment $n$ (that is, replace $n$ by
$n+1$) ? Of course, the $\mathbf{k}$-module $\mathcal{D}$ becomes larger, but
let us only focus on the existingTODO
\end{proof}
\end{noncompile}

\begin{question}
The above proof idea is clearly in bad taste. There should be a more
conceptual proof that identifies $\mathcal{D}$ as some existing (nonunital)
$\mathbf{k}$-algebra (what nonunital $\mathbf{k}$-algebra has dimension
$\dbinom{2n}{n}$ over $\mathbf{k}$ ?) or at least with a subquotient of a such.
\end{question}

\begin{example}
Let $n=1$. Then, the $\mathbf{k}$-module $\mathcal{D}$ in Theorem
\ref{thm.Nabla-alg.ass} has basis $\left(  u,v\right)  $ with $u=\Delta
_{\varnothing,\varnothing}$ and $v=\Delta_{\left\{  1\right\}  ,\left\{
1\right\}  }$. The multiplication on $\mathcal{D}$ defined ibidem is given by%
\[
uu=uv=vu=u,\ \ \ \ \ \ \ \ \ \ vv=v.
\]
Thus, the nonunital $\mathbf{k}$-algebra $\mathcal{D}$ is isomorphic to the
$\mathbf{k}$-algebra $\mathbf{k}\left[  x\right]  /\left(  x^{2}-x\right)  $,
and therefore has a unity (namely, $v$).
\end{example}

\begin{example}
\label{exa.Nabla-alg.n=2}Let $n=2$. Then, the $\mathbf{k}$-module
$\mathcal{D}$ in Theorem \ref{thm.Nabla-alg.ass} has basis $\left(
u,v_{11},v_{12},v_{21},v_{22},w\right)  $ with $u=\Delta_{\varnothing
,\varnothing}$ and $v_{ij}=\Delta_{\left\{  i\right\}  ,\left\{  j\right\}  }$
and $w=\Delta_{\left[  2\right]  ,\left[  2\right]  }$. The multiplication on
$\mathcal{D}$ defined ibidem is given by%
\begin{align*}
uu  &  =uw=wu=2u,\ \ \ \ \ \ \ \ \ \ uv_{ij}=v_{ij}u=u,\\
v_{dc}v_{ba}  &  =u-v_{da}\ \ \ \ \ \ \ \ \ \ \text{if }b\neq c;\\
v_{dc}v_{ba}  &  =v_{da}\ \ \ \ \ \ \ \ \ \ \text{if }b=c,\\
v_{ij}w  &  =v_{i1}+v_{i2},\ \ \ \ \ \ \ \ \ \ wv_{ij}=v_{1j}+v_{2j},\\
ww  &  =2w.
\end{align*}
This nonunital $\mathbf{k}$-algebra $\mathcal{D}$ has a unity if and only if
$2$ is invertible in $\mathbf{k}$. This unity is $\dfrac{1}{4}\left(
v_{11}+v_{22}-v_{12}-v_{21}+2w\right)  $.
\end{example}

\begin{example}
Let $n=3$. Then, the $\mathbf{k}$-module $\mathcal{D}$ has a basis consisting
of $\dbinom{6}{3}=20$ vectors of the form $\Delta_{B,A}$ with $A,B\subseteq
\left[  3\right]  $ satisfying $\left\vert A\right\vert =\left\vert
B\right\vert $. Its multiplication turns it into a nonunital algebra. When
$3!$ is invertible in $\mathbf{k}$, this algebra has a unity, namely%
\[
\dfrac{1}{18}\sum_{i=1}^{3}\Delta_{\left\{  i\right\}  ,\left\{  i\right\}
}-\dfrac{1}{36}\sum_{i\neq j}\Delta_{\left\{  i\right\}  ,\left\{  j\right\}
}+\dfrac{1}{6}\sum_{i<j}\Delta_{\left\{  i,j\right\}  ,\left\{  i,j\right\}
}-\dfrac{1}{12}\sum_{i\neq j\neq k\neq i}\Delta_{\left\{  i,j\right\}
,\left\{  i,k\right\}  }+\dfrac{1}{6}\Delta_{\left[  3\right]  ,\left[
3\right]  }.
\]

\end{example}

\begin{question}
Does the $\mathbf{k}$-algebra $\mathcal{D}$ in Theorem \ref{thm.Nabla-alg.ass}
have a unity if $n!$ is invertible in $\mathbf{k}$ ? (I suspect that the
answer is \textquotedblleft yes\textquotedblright. This has been checked with
SageMath for all $n\leq5$.)
\end{question}

\begin{noncompile}
We note that, despite the \textquotedblleft
elementary-matrix-like\textquotedblright\ look of the $v_{ba}$ products in
Example \ref{exa.Nabla-alg.n=2}, the $\mathbf{k}$-algebra $\mathcal{D}$ is not
a direct product of matrix algebras.
\end{noncompile}

\begin{question}
What does the representation theory of $\mathcal{D}$ look like?
\end{question}

Using SageMath, we computed some data for small $n$ and for $\mathbf{k}%
=\mathbb{Q}$:%
\[%
\begin{tabular}
[c]{|c||c|c|c|c|}\hline
& $n=2$ & $n=3$ & $n=4$ & $n=5$\\\hline\hline
$\dim\mathcal{D}$ & $6$ & $20$ & $70$ & $252$\\\hline
$\dim Z\left(  \mathcal{D}\right)  $ & $3$ & $4$ & $5$ & $6$\\\hline
$\dim J\left(  \mathcal{D}\right)  $ & $3$ & $5$ & $39$ & $84$\\\hline
Cartan invs & $\left(
\begin{array}
[c]{ccc}%
1 & 0 & 0\\
0 & 2 & 1\\
0 & 1 & 1
\end{array}
\right)  $ & $\left(
\begin{array}
[c]{cccc}%
1 & 0 & 0 & 0\\
0 & 1 & 0 & 1\\
0 & 0 & 1 & 0\\
0 & 1 & 0 & 2
\end{array}
\right)  $ & $\left(
\begin{array}
[c]{rrrrr}%
1 & 0 & 0 & 0 & 0\\
0 & 1 & 0 & 1 & 0\\
0 & 0 & 1 & 0 & 1\\
0 & 1 & 0 & 2 & 0\\
0 & 0 & 1 & 0 & 2
\end{array}
\right)  $ & \\\hline
\end{tabular}
\
\]
(where $Z\left(  \mathcal{D}\right)  $ and $J\left(  \mathcal{D}\right)  $
denote the center and the Jacobson radical of $\mathcal{D}$, respectively, and
where \textquotedblleft Cartan invs\textquotedblright\ means the matrix of
Cartan invariants). This shows that the algebra $\mathcal{D}$ for $n\geq2$ is
far from semisimple. For example, for $n=2$, it has (over $\mathbb{Q}$) a
$3$-dimensional Jacobson radical spanned by the vectors $u-v_{12}-v_{21}$,
$v_{12}-v_{21}$ and $v_{11}-v_{22}$. But $\mathcal{D}$ is not
\textquotedblleft too nilpotent\textquotedblright\ either; in particular, the
semisimple quotient of $\mathcal{D}$ for $n=3$ is not commutative.

\begin{question}
Does the center of $\mathcal{D}$ always have dimension $n+1$ when $n!$ is
invertible in $\mathbf{k}$ ?
\end{question}

\begin{noncompile}
Without the invertibility condition, the answer is definitely
\textquotedblleft no\textquotedblright, since the center of $\mathcal{D}$ for
$\mathbf{k}=\mathbb{F}_{2}$ and $n=3$ has dimension $12$ rather than $4$. (Not
sure about this -- Sage might be tacitly assuming unitality...)
\end{noncompile}

\begin{noncompile}
Assume that $n!$ is invertible in $\mathbf{k}$. Is the $\mathbf{k}$-algebra
$\mathcal{D}$ isomorphic to the planar rook algebra (\cite[\S 3]{FlHaHe09},
\cite[\S 2.2]{PotRub20}, \cite[\S 4.3]{Scrims22})? Note that the latter
algebra is isomorphic to the direct product $\prod_{k=0}^{n}\mathbf{k}%
^{\dbinom{n}{k}\times\dbinom{n}{k}}$ (by \cite[Corollary 3.4]{FlHaHe09}).

No, $\mathcal{D}$ is not isomorphic to the planar rook algebra, since it has a
nonzero radical.
\end{noncompile}

\begin{noncompile}
subsection: (OLD) The triangular subalgebra approach

The following is obsolete.

Try to find a filtration on which $\kappa_{A,B}$ acts triangularly.

\begin{enumerate}
\item Here is a baby example: $a=1$ and $b=1$ and $c=1$. Thus, we can take
$A=B=\left\{  1\right\}  $. Hence,
\[
\kappa_{A,B}=\sum_{\substack{w;\\w\left(  1\right)  \neq1}}w=\nabla-\omega
\]
where%
\[
\nabla:=\sum_{w}w\ \ \ \ \ \ \ \ \ \ \text{and}\ \ \ \ \ \ \ \ \ \ \omega
:=\sum_{\substack{w;\\w\left(  1\right)  =1}}w.
\]
Easily, $\nabla^{2}=n!\nabla$ and $\omega^{2}=\left(  n-1\right)  !\omega$ and
$\omega\nabla=\nabla\omega=\left(  n-1\right)  !\nabla$ (here, we use the fact
that $\nabla u=u\nabla=\nabla$ for every $u\in S_{n}$). Hence, the span
$\left\langle \nabla,\omega,1\right\rangle $ is a commutative subalgebra of
$\mathbf{k}\left[  S_{n}\right]  $, and each of its generators $\nabla
,\omega,1$ acts triangularly on it (sending generators to smaller-or-equal
generators). Thus, the case $\left(  a,b,c\right)  =\left(  1,1,1\right)  $ is solved.

\item Let us solve the case $\left(  a,b,c\right)  =\left(  1,1,0\right)  $
next. This corresponds to $A=\left\{  1\right\}  $ and $B=\left\{  2\right\}
$. Hence,%
\[
\kappa_{A,B}=\sum_{\substack{w;\\w\left(  1\right)  \neq2}}w=\nabla-\rho
\]
where $\rho:=\sum_{\substack{w;\\w\left(  1\right)  =2}}w$. Here we can easily
find $\rho^{2}=\left(  n-2\right)  !\left(  \nabla-\rho\right)  $. Thus,
again, the span $\left\langle \nabla,\rho,1\right\rangle $ is a commutative
subalgebra of $\mathbf{k}\left[  S_{n}\right]  $, and each of its generators
$\nabla,\rho,1$ acts triangularly on it (sending generators to
smaller-or-equal generators). Thus, the case $\left(  a,b,c\right)  =\left(
1,1,0\right)  $ is solved.

What do we learn from this? Do we always have a triangular subalgebra?

Let us say that a list $\left(  a_{1},a_{2},\ldots,a_{k}\right)  $ of vectors
in $\mathbf{k}\left[  S_{n}\right]  $ is \textbf{triangular} if it satisfies
\[
a_{i}a_{j}\in\left\langle a_{1},a_{2},\ldots,a_{j}\right\rangle
\ \ \ \ \ \ \ \ \ \ \text{for all }i,j\in\left[  k\right]  .
\]
A triangular list always spans a subalgebra of $\mathbf{k}\left[
S_{n}\right]  $, and all the elements of this subalgebra have factorable
minimal polynomials (since they act triangularly on it).

Thus, the list $\left(  \nabla,\rho,1\right)  $ is triangular. More generally,
any list of the form $\left(  \nabla,b_{1},b_{2},\ldots,b_{k},1\right)  $ is
triangular if we have
\[
b_{i}b_{j}\in\left\langle \nabla,b_{1},b_{2},\ldots,b_{j}\right\rangle
\ \ \ \ \ \ \ \ \ \ \text{for all }i,j\in\left[  k\right]
\]
(because $\nabla u$ and $u\nabla$ are always scalar multiples of $\nabla$).

\item Let us try to take this idea further. For any $a,b\in\left[  n\right]
$, set
\[
u_{a,b}:=\sum_{\substack{w;\\w\left(  a\right)  =b}}w.
\]
Then, it is not hard to see that%
\[
u_{c,d}u_{a,b}=%
\begin{cases}
\left(  n-2\right)  !\left(  \nabla-u_{a,d}\right)  , & \text{if }b\neq c;\\
\left(  n-1\right)  !u_{a,d}, & \text{if }b=c.
\end{cases}
\]
(Note the similarity to elementary matrices! Maybe there is a secret
morphism?) Thus, the span $\left\langle \nabla,u_{1,b},u_{2,b},\ldots
,u_{n,b},1\right\rangle $ is a triangular subalgebra for any $b\in\left[
n\right]  $ (but not commutative!). Note that $u_{a,b}=S\left(  u_{b,a}%
\right)  $.

\item We can use this to solve the case $a=\forall$ and $b=1$ and $c=\forall$.

That is, we take $B=\left\{  1\right\}  $ and any $A$. The two subcases are
$1\in A$ and $1\notin A$.

We have $\kappa_{A,B}=\kappa_{A,\left\{  1\right\}  }=\sum_{i\notin A}u_{i,1}%
$. But this is in the subalgebra $\left\langle \nabla,u_{1,b},u_{2,b}%
,\ldots,u_{n,b},1\right\rangle $ for $b=1$, which is triangular.

We can be more concrete here: We have
\begin{align*}
\kappa_{A,B}\nabla &  =\left(  n-\left\vert A\right\vert \right)  \left(
n-1\right)  !\nabla;\\
\kappa_{A,B}u_{k,1}  &  =\left(  n-\left\vert A\cup\left\{  1\right\}
\right\vert \right)  \left(  n-2\right)  !\nabla\\
&  \ \ \ \ \ \ \ \ \ \ +\left(  \left(  n-1\right)  !\left\vert \left\{
1\right\}  \setminus A\right\vert -\left(  n-2\right)  !\left(  n-\left\vert
A\cup\left\{  1\right\}  \right\vert \right)  \right)  u_{k,1};\\
\kappa_{A,B}1  &  =\sum_{i\notin A}u_{i,1}.
\end{align*}
Setting $a=\left\vert A\right\vert $, the eigenvalues of $\kappa_{A,B}$ are
thus%
\begin{align*}
&  \left(  n-\left\vert A\right\vert \right)  \left(  n-1\right)  !=\left(
n-a\right)  \left(  n-1\right)  !,\\
&  \left(  n-1\right)  !\left\vert \left\{  1\right\}  \setminus A\right\vert
-\left(  n-2\right)  !\left(  n-\left\vert A\cup\left\{  1\right\}
\right\vert \right) \\
&  \ \ \ \ \ \ \ \ \ \ =%
\begin{cases}
-\left(  n-2\right)  !\left(  n-a\right)  , & \text{if }1\in A;\\
\left(  n-1\right)  !-\left(  n-2\right)  !\left(  n-a-1\right)  , & \text{if
}1\notin A
\end{cases}
\\
&  \ \ \ \ \ \ \ \ \ \ =%
\begin{cases}
-\left(  n-2\right)  !\left(  n-a\right)  , & \text{if }1\in A;\\
\left(  n-2\right)  !a, & \text{if }1\notin A
\end{cases}
,\\
&  0.
\end{align*}

\item For $a=n-1$ and $b=1$ and $c=0$, we can do even better: In that case,
$\left(  \kappa_{A,B}-\left(  n-1\right)  !\right)  \kappa_{A,B}=0$. To see
this, observe that $\kappa_{A,B}=\nabla-u_{1,1}=\nabla-\omega$ in the above notation.

\item The case $b=0$ is even easier: here, $\kappa_{A,B}=\nabla$.

\item So it remains to study $b\geq2$.
\end{enumerate}

subsection: (OLD) The Specht approach

The following observations are no longer needed, but they make for interesting context.

\begin{enumerate}
\item We only need to understand the actions of $\kappa_{A,B}$ on each Specht
module $S^{\lambda}$.

\item We claim that $\kappa_{A,B}S^{\lambda}=0$ whenever the partition
$\lambda$ has length $\ell\left(  \lambda\right)  >2$.

\textit{Proof:} Recall that $S^{\lambda}=b_{\lambda}a_{\lambda}\mathbf{k}%
\left[  S_{n}\right]  $, where $b_{\lambda}$ is the Young antisymmetrizer of
$\lambda$. It thus suffices to show that $\kappa_{A,B}b_{\lambda}=0$. But this
is easy: The first column of the standard tableau $T^{\lambda}$ has at least
$3$ entries, and thus has two entries that either both belong to $A$ or both
belong to $\left[  n\right]  \setminus A$ (by the pigeonhole principle). The
transposition $\tau$ that swaps these two entries must then preserve
$\kappa_{A,B}$ from the right, i.e., must satisfy $\kappa_{A,B}\tau
=\kappa_{A,B}$. But this means that $\kappa_{A,B}\left(  \tau-1\right)  =0$,
so that $\kappa_{A,B}b_{\lambda}=0$ (since $b_{\lambda}=\left(  \tau-1\right)
\eta$ for some $\eta\in\mathbf{k}\left[  S_{n}\right]  $).

\item Thus, it suffices to study the action of $\kappa_{A,B}$ on each Specht
module $S^{\lambda}$ with $\ell\left(  \lambda\right)  \leq2$.

To this purpose, it suffices to study the action of $\kappa_{A,B}$ on each
Young permutation module $M^{\lambda}$ with $\ell\left(  \lambda\right)
\leq2$.

But these modules $M^{\lambda}=M^{\left(  k,n-k\right)  }$ are just
permutation modules where $S_{n}$ acts on $\left\{  k\text{-element subsets of
}\left[  n\right]  \right\}  $ in the obvious way.

So we must show, for each $k\in\left[  0,n\right]  $, that the action of
$\kappa_{A,B}$ on $\left\{  k\text{-element subsets of }\left[  n\right]
\right\}  $ has integer eigenvalues.

We can rewrite this action combinatorially, as a $\dbinom{n}{k}\times
\dbinom{n}{k}$-matrix, without thinking of it as a permutation problem.

\item Note that the action of $\mathbf{k}\left[  S_{n}\right]  $ on
$S^{\lambda}$ for $\ell\left(  \lambda\right)  \leq2$ factors through the
Temperley--Lieb algebra. Thus, we can also view $\kappa_{A,B}$ as an element
of the latter. Not sure if this helps...
\end{enumerate}
\end{noncompile}

\newpage

\section{\label{sec.row-to-row}Row-to-row sums in the symmetric group algebra}

\subsection{Definitions}

As we recall, $n$ is a nonnegative integer and $\mathbf{k}$ a commutative
ring. We work in the group algebra $\mathcal{A}=\mathbf{k}\left[
S_{n}\right]  $ of the symmetric group $S_{n}$.

\subsubsection{Row-to-row sums and the $\nabla_{\mathbf{B},\mathbf{A}}$}

A \emph{set decomposition} of a set $U$ shall mean a tuple $\left(
U_{1},U_{2},\ldots,U_{k}\right)  $ of disjoint subsets of $U$ such that
$U_{1}\cup U_{2}\cup\cdots\cup U_{k}=U$. The subsets $U_{1},U_{2},\ldots
,U_{k}$ are called the \emph{blocks} of this set decomposition $\left(
U_{1},U_{2},\ldots,U_{k}\right)  $. The number $k$ of these blocks is called
the \emph{length} of this set decomposition. The length of a set decomposition
$\mathbf{U}$ is called $\ell\left(  \mathbf{U}\right)  $.

A \emph{set composition} of a set $U$ shall mean a set decomposition of $U$
whose blocks are all nonempty. Clearly, any set decomposition of $U$ can be
transformed into a set composition of $U$ by removing all empty blocks.

For instance, $\left(  \left\{  1,3\right\}  ,\varnothing,\left\{  2\right\}
\right)  $ is a set decomposition of $\left[  3\right]  $, but not a set
composition (due to the presence of $\varnothing$). It has length $3$.
Removing the block $\varnothing$ from it yields the set composition $\left(
\left\{  1,3\right\}  ,\left\{  2\right\}  \right)  $ of $\left[  3\right]  $,
whose length is $2$.

Let $\operatorname*{SD}\left(  n\right)  $ denote the set of all set
decompositions of $\left[  n\right]  $.

Let $\operatorname*{SC}\left(  n\right)  $ denote the set of all set
compositions of $\left[  n\right]  $. Clearly, $\operatorname*{SC}\left(
n\right)  \subseteq\operatorname*{SD}\left(  n\right)  $.

If $\mathbf{A}=\left(  A_{1},A_{2},\ldots,A_{k}\right)  $ and $\mathbf{B}%
=\left(  B_{1},B_{2},\ldots,B_{k}\right)  $ are two set decompositions of
$\left[  n\right]  $ having the same length, then we define the element%
\begin{equation}
\nabla_{\mathbf{B},\mathbf{A}}:=\sum_{\substack{w\in S_{n};\\w\left(
A_{i}\right)  =B_{i}\text{ for all }i}}w\ \ \ \ \ \ \ \ \ \ \text{of
}\mathcal{A}. \label{eq.NabBA.def}%
\end{equation}
This will be called a \emph{row-to-row sum}. It has been denoted $\left(
\mathbf{A}\rightarrow\mathbf{B}\right)  $ in Canfield's and Williamson's work
\cite{CanWil89}, and also is the $q=1$ particular case of the
\textquotedblleft Murphy element\textquotedblright\ $x_{st}$ of the Hecke
algebra studied in \cite[\S 3]{Murphy92} and \cite[\S 4]{Murphy95} (if we
encode $\mathbf{A}$ and $\mathbf{B}$ as row-standard tableaux, not necessarily
of partition shape).

For instance, if $n=4$ and $\mathbf{A}=\left(  \left\{  3\right\}  ,\left\{
1,2\right\}  ,\left\{  4\right\}  \right)  $ and $\mathbf{B}=\left(  \left\{
1\right\}  ,\left\{  2,4\right\}  ,\left\{  3\right\}  \right)  $, then%
\[
\nabla_{\mathbf{B},\mathbf{A}}=\sum_{\substack{w\in S_{4};\\w\left(  \left\{
3\right\}  \right)  =\left\{  1\right\}  ;\\w\left(  \left\{  1,2\right\}
\right)  =\left\{  2,4\right\}  ;\\w\left(  \left\{  4\right\}  \right)
=\left\{  3\right\}  }}w=\operatorname*{oln}\left(  2413\right)
+\operatorname*{oln}\left(  4213\right)  ,
\]
where $\operatorname*{oln}\left(  i_{1}i_{2}\ldots i_{n}\right)  $ means the
permutation in $S_{n}$ with one-line notation $\left(  i_{1},i_{2}%
,\ldots,i_{n}\right)  $.

We observe some easy properties of row-to-row sums:

\begin{proposition}
\label{prop.row.simple}Let $\mathbf{A}=\left(  A_{1},A_{2},\ldots
,A_{k}\right)  $ and $\mathbf{B}=\left(  B_{1},B_{2},\ldots,B_{k}\right)  $ be
two set decompositions of $\left[  n\right]  $ having the same length. Then:

\begin{enumerate}
\item[\textbf{(a)}] We have $\nabla_{\mathbf{B},\mathbf{A}}=0$ unless each
$i\in\left[  k\right]  $ satisfies $\left\vert A_{i}\right\vert =\left\vert
B_{i}\right\vert $.

\item[\textbf{(b)}] The element $\nabla_{\mathbf{B},\mathbf{A}}$ does not
change if we permute the blocks of $\mathbf{A}$ and the blocks of $\mathbf{B}$
using the same permutation. In other words, for any permutation $\sigma\in
S_{k}$, we have $\nabla_{\mathbf{B},\mathbf{A}}=\nabla_{\mathbf{B}%
\sigma,\mathbf{A}\sigma}$, where $\mathbf{A}\sigma:=\left(  A_{\sigma\left(
1\right)  },A_{\sigma\left(  2\right)  },\ldots,A_{\sigma\left(  k\right)
}\right)  $ and $\mathbf{B}\sigma:=\left(  B_{\sigma\left(  1\right)
},B_{\sigma\left(  2\right)  },\ldots,B_{\sigma\left(  k\right)  }\right)  $.

\item[\textbf{(c)}] The element $\nabla_{\mathbf{B},\mathbf{A}}$ does not
change if we remove empty blocks from $\mathbf{A}$ and from $\mathbf{B}$,
provided that these blocks are in the same positions in both $\mathbf{A}$ and
$\mathbf{B}$.

\item[\textbf{(d)}] The antipode $S$ of $\mathcal{A}$ satisfies $S\left(
\nabla_{\mathbf{B},\mathbf{A}}\right)  =\nabla_{\mathbf{A},\mathbf{B}}$.
\end{enumerate}
\end{proposition}

\begin{proof}
\textbf{(a)} Assume that not every $i\in\left[  k\right]  $ satisfies
$\left\vert A_{i}\right\vert =\left\vert B_{i}\right\vert $. Then, there
exists no permutation $w\in S_{n}$ that satisfies $\left(  w\left(
A_{i}\right)  =B_{i}\text{ for all }i\right)  $ (since the injectivity of such
a permutation $w$ would imply $\left\vert w\left(  A_{i}\right)  \right\vert
=\left\vert A_{i}\right\vert $ and thus $\left\vert A_{i}\right\vert
=\left\vert w\left(  A_{i}\right)  \right\vert =\left\vert B_{i}\right\vert $
because of $w\left(  A_{i}\right)  =B_{i}$). Hence, the sum $\sum
_{\substack{w\in S_{n};\\w\left(  A_{i}\right)  =B_{i}\text{ for all }i}}w$ in
(\ref{eq.NabBA.def}) is empty and thus equals $0$. Therefore,
(\ref{eq.NabBA.def}) shows that $\nabla_{\mathbf{B},\mathbf{A}}=0$. This
proves Proposition \ref{prop.row.simple} \textbf{(a)}. \medskip

\textbf{(b)} Let $\sigma\in S_{k}$. Then, (\ref{eq.NabBA.def}) yields
\begin{align}
\nabla_{\mathbf{B},\mathbf{A}}  &  =\sum_{\substack{w\in S_{n};\\w\left(
A_{i}\right)  =B_{i}\text{ for all }i}}w\ \ \ \ \ \ \ \ \ \ \text{and}%
\label{pf.prop.row.simple.b.1}\\
\nabla_{\mathbf{B}\sigma,\mathbf{A}\sigma}  &  =\sum_{\substack{w\in
S_{n};\\w\left(  A_{\sigma\left(  i\right)  }\right)  =B_{\sigma\left(
i\right)  }\text{ for all }i}}w. \label{pf.prop.row.simple.b.2}%
\end{align}
But $\sigma$ is a permutation of $\left[  k\right]  $; thus the condition
\textquotedblleft$w\left(  A_{i}\right)  =B_{i}$ for all $i$\textquotedblright%
\ is equivalent to \textquotedblleft$w\left(  A_{\sigma\left(  i\right)
}\right)  =B_{\sigma\left(  i\right)  }$ for all $i$\textquotedblright. Hence,
the right hand sides of (\ref{pf.prop.row.simple.b.1}) and
(\ref{pf.prop.row.simple.b.2}) are equal. Thus, so are the left hand sides. In
other words, $\nabla_{\mathbf{B},\mathbf{A}}=\nabla_{\mathbf{B}\sigma
,\mathbf{A}\sigma}$. This proves Proposition \ref{prop.row.simple}
\textbf{(b)}. \medskip

\textbf{(c)} Let $\mathbf{A}=\left(  A_{1},A_{2},\ldots,A_{k}\right)  $ and
$\mathbf{B}=\left(  B_{1},B_{2},\ldots,B_{k}\right)  $. Assume that both
$\mathbf{A}$ and $\mathbf{B}$ have an empty block in the same position --
i.e., there exists some $r\in\left[  k\right]  $ such that $A_{r}=\varnothing$
and $B_{r}=\varnothing$. Consider this $r$. Let $\mathbf{A}^{\prime}:=\left(
A_{1},A_{2},\ldots,A_{r-1},A_{r+1},A_{r+2},\ldots,A_{k}\right)  $ and
$\mathbf{B}^{\prime}:=\left(  B_{1},B_{2},\ldots,B_{r-1},B_{r+1}%
,B_{r+2},\ldots,B_{k}\right)  $ be the set decompositions obtained from
$\mathbf{A}$ and $\mathbf{B}$ by removing the empty blocks $A_{r}$ and $B_{r}%
$. We must show that $\nabla_{\mathbf{B},\mathbf{A}}=\nabla_{\mathbf{B}%
^{\prime},\mathbf{A}^{\prime}}$.

Essentially, this is obvious from the definition of $\nabla_{\mathbf{B}%
,\mathbf{A}}$: The empty blocks $A_{r}$ and $B_{r}$ satisfy $w\left(
A_{r}\right)  =B_{r}$ for any permutation $w\in S_{n}$ (since $w\left(
\varnothing\right)  =\varnothing$). Hence, the condition \textquotedblleft%
$w\left(  A_{i}\right)  =B_{i}$ for all $i$\textquotedblright\ in
(\ref{eq.NabBA.def}) is tautologically satisfied for $i=r$. Thus, we can
replace \textquotedblleft for all $i$\textquotedblright\ by \textquotedblleft
for all $i\neq r$\textquotedblright\ in (\ref{eq.NabBA.def}) without changing
the sum. But this gives us precisely $\nabla_{\mathbf{B}^{\prime}%
,\mathbf{A}^{\prime}}$. Hence, $\nabla_{\mathbf{B},\mathbf{A}}=\nabla
_{\mathbf{B}^{\prime},\mathbf{A}^{\prime}}$. This proves Proposition
\ref{prop.row.simple} \textbf{(c)}. \medskip

\textbf{(d)} Recall that $S$ is a $\mathbf{k}$-linear map sending each
permutation $w\in S_{n}$ to $w^{-1}$. Hence, applying $S$ to the equality
(\ref{eq.NabBA.def}), we obtain%
\[
S\left(  \nabla_{\mathbf{B},\mathbf{A}}\right)  =\sum_{\substack{w\in
S_{n};\\w\left(  A_{i}\right)  =B_{i}\text{ for all }i}}w^{-1}=\sum
_{\substack{w\in S_{n};\\w^{-1}\left(  B_{i}\right)  =A_{i}\text{ for all }%
i}}w^{-1}%
\]
(here, we have rewritten the condition \textquotedblleft$w\left(
A_{i}\right)  =B_{i}$\textquotedblright\ in the form \textquotedblleft%
$w^{-1}\left(  B_{i}\right)  =A_{i}$\textquotedblright\ (which is equivalent,
since $w$ is a permutation)). Thus,
\begin{align*}
S\left(  \nabla_{\mathbf{B},\mathbf{A}}\right)   &  =\sum_{\substack{w\in
S_{n};\\w^{-1}\left(  B_{i}\right)  =A_{i}\text{ for all }i}}w^{-1}%
=\sum_{\substack{w\in S_{n};\\w\left(  B_{i}\right)  =A_{i}\text{ for all }%
i}}w\\
&  \ \ \ \ \ \ \ \ \ \ \left(
\begin{array}
[c]{c}%
\text{here, we have substituted }w\text{ for }w^{-1}\text{ in the sum,}\\
\text{since the map }S_{n}\rightarrow S_{n},\ w\mapsto w^{-1}\text{ is a
bijection}%
\end{array}
\right) \\
&  =\nabla_{\mathbf{A},\mathbf{B}}\ \ \ \ \ \ \ \ \ \ \left(  \text{by the
definition of }\nabla_{\mathbf{A},\mathbf{B}}\right)  .
\end{align*}
This proves Proposition \ref{prop.row.simple} \textbf{(d)}.
\end{proof}

Moreover, these row-to-row sums $\nabla_{\mathbf{B},\mathbf{A}}$ generalize
the rectangular rook sums $\nabla_{B,A}$ from Section \ref{sec.rooksum}:

\begin{proposition}
\label{prop.row.rook}Let $A$ and $B$ be two subsets of $\left[  n\right]  $.
Define the two set decompositions $\mathbf{A}:=\left(  A,\ \left[  n\right]
\setminus A\right)  $ and $\mathbf{B}:=\left(  B,\ \left[  n\right]  \setminus
B\right)  $ of $\left[  n\right]  $. Then, $\nabla_{\mathbf{B},\mathbf{A}%
}=\nabla_{B,A}$.
\end{proposition}

\begin{proof}
By the definition of $\nabla_{\mathbf{B},\mathbf{A}}$, we have%
\begin{equation}
\nabla_{\mathbf{B},\mathbf{A}}=\sum_{\substack{w\in S_{n};\\w\left(  A\right)
=B;\\w\left(  \left[  n\right]  \setminus A\right)  =\left[  n\right]
\setminus B}}w. \label{pf.prop.row.rook.1}%
\end{equation}
However, the condition \textquotedblleft$w\left(  \left[  n\right]  \setminus
A\right)  =\left[  n\right]  \setminus B$\textquotedblright\ under the sum
here is redundant, since it follows from \textquotedblleft$w\left(  A\right)
=B$\textquotedblright\ when $w$ is a permutation. Thus, we can remove this
condition. Hence, (\ref{pf.prop.row.rook.1}) simplifies to%
\[
\nabla_{\mathbf{B},\mathbf{A}}=\sum_{\substack{w\in S_{n};\\w\left(  A\right)
=B}}w=\nabla_{B,A}%
\]
(by the definition of $\nabla_{B,A}$). This proves Proposition
\ref{prop.row.rook}.
\end{proof}

\begin{remark}
The row-to-row sums can also be rewritten using colorings instead of set
(de)compositions. Namely, a \emph{coloring} of $\left[  n\right]  $ means a
map $f:\left[  n\right]  \rightarrow C$ to some set $C$. If $C=\left[
k\right]  $ for some $k\in\mathbb{N}$, then such a coloring $f$ can be
regarded as a set decomposition of $\left[  n\right]  $ of length $k$, where
the $i$-th block is $f^{-1}\left(  i\right)  $ for each $i\in\left[  k\right]
$. The image $f\left(  j\right)  $ of an element $j\in\left[  n\right]  $
under a coloring $f:\left[  n\right]  \rightarrow C$ is called the
\emph{color} of $j$ (under $f$). Now, the row-to-row sum $\nabla_{g,f}$
corresponding to two colorings $f$ and $g$ of $\left[  n\right]  $ is the sum
of all permutations $w\in S_{n}$ that satisfy $g\circ w=f$. (This is a
\textquotedblleft preservation of colors\textquotedblright\ condition.)
\end{remark}

\begin{remark}
Let $u\in S_{n}$ be any permutation. Let $\mathbf{A}$ be the set composition
$\left(  \left\{  1\right\}  ,\ \left\{  2\right\}  ,\ \ldots,\ \left\{
n\right\}  \right)  $ of $\left[  n\right]  $, and let $\mathbf{B}$ be the set
composition $\left(  \left\{  u\left(  1\right)  \right\}  ,\ \left\{
u\left(  2\right)  \right\}  ,\ \ldots,\ \left\{  u\left(  n\right)  \right\}
\right)  $ of $\left[  n\right]  $. Then, $\nabla_{\mathbf{B},\mathbf{A}}=u$.
Thus, the row-to-row sums $\nabla_{\mathbf{B},\mathbf{A}}$ in general are not
as special as their particular cases the rectangular rook sums $\nabla_{B,A}$.
In particular, the minimal polynomials of general row-to-row sums
$\nabla_{\mathbf{B},\mathbf{A}}$ cannot be factored into linear factors over
$\mathbb{Z}$.
\end{remark}

The symmetric group $S_{n}$ acts on the set $\operatorname*{SD}\left(
n\right)  =\left\{  \text{set decompositions of }\left[  n\right]  \right\}  $
(from the left) by the rule%
\begin{align*}
&  w\left(  B_{1},B_{2},\ldots,B_{k}\right)  =\left(  w\left(  B_{1}\right)
,\ w\left(  B_{2}\right)  ,\ \ldots,\ w\left(  B_{k}\right)  \right) \\
&  \ \ \ \ \ \ \ \ \ \ \text{for all }w\in S_{n}\text{ and all }\left(
B_{1},B_{2},\ldots,B_{k}\right)  \in\operatorname*{SD}\left(  n\right)  .
\end{align*}
The subset $\operatorname*{SC}\left(  n\right)  $ of $\operatorname*{SD}%
\left(  n\right)  $ is preserved under this $S_{n}$-action; thus $S_{n}$ acts
on $\operatorname*{SC}\left(  n\right)  $ as well.

The action of $S_{n}$ on $\operatorname*{SD}\left(  n\right)  $ we just
defined allows us to rewrite the equality (\ref{eq.NabBA.def}) as follows:%
\begin{equation}
\nabla_{\mathbf{B},\mathbf{A}}:=\sum_{\substack{w\in S_{n};\\w\mathbf{A}%
=\mathbf{B}}}w \label{eq.NabBA.def.rewr}%
\end{equation}
for any two set decompositions $\mathbf{A},\mathbf{B}\in\operatorname*{SD}%
\left(  n\right)  $ satisfying $\ell\left(  \mathbf{A}\right)  =\ell\left(
\mathbf{B}\right)  $. More importantly, the row-to-row sums $\nabla
_{\mathbf{B},\mathbf{A}}$ transform in a very simple way under this action:

\begin{proposition}
\label{prop.uNabv}Let $u,v\in S_{n}$ be any permutations. Let $\mathbf{A}%
,\mathbf{B}\in\operatorname*{SD}\left(  n\right)  $ be any two set
decompositions satisfying $\ell\left(  \mathbf{A}\right)  =\ell\left(
\mathbf{B}\right)  $. Then,%
\begin{equation}
u\nabla_{\mathbf{B},\mathbf{A}}v=\nabla_{u\mathbf{B},v^{-1}\mathbf{A}}.
\label{eq.uNabv}%
\end{equation}

\end{proposition}

\begin{proof}
Set $k=\ell\left(  \mathbf{A}\right)  =\ell\left(  \mathbf{B}\right)  $, and
write the set decompositions $\mathbf{A}$ and $\mathbf{B}$ as $\mathbf{A}%
=\left(  A_{1},A_{2},\ldots,A_{k}\right)  $ and $\mathbf{B}=\left(
B_{1},B_{2},\ldots,B_{k}\right)  $. Then, $u\mathbf{B}=\left(  u\left(
B_{1}\right)  ,u\left(  B_{2}\right)  ,\ldots,u\left(  B_{k}\right)  \right)
$ and $v^{-1}\mathbf{A}=\left(  v^{-1}\left(  A_{1}\right)  ,v^{-1}\left(
A_{2}\right)  ,\ldots,v^{-1}\left(  A_{k}\right)  \right)  $. Hence, the
definition of row-to-row sums yields%
\begin{align*}
\nabla_{u\mathbf{B},v^{-1}\mathbf{A}}  &  =\sum_{\substack{w\in S_{n}%
;\\w\left(  v^{-1}\left(  A_{i}\right)  \right)  =u\left(  B_{i}\right)
\text{ for all }i}}w\\
&  =\sum_{\substack{w\in S_{n};\\\left(  uwv\right)  \left(  v^{-1}\left(
A_{i}\right)  \right)  =u\left(  B_{i}\right)  \text{ for all }i}%
}uwv\ \ \ \ \ \ \ \ \ \ \left(
\begin{array}
[c]{c}%
\text{here, we have substituted }uwv\\
\text{for }w\text{ in the sum}%
\end{array}
\right) \\
&  =\sum_{\substack{w\in S_{n};\\w\left(  A_{i}\right)  =B_{i}\text{ for all
}i}}uwv\ \ \ \ \ \ \ \ \ \ \left(
\begin{array}
[c]{c}%
\text{since the condition}\\
\text{\textquotedblleft}\left(  uwv\right)  \left(  v^{-1}\left(
A_{i}\right)  \right)  =u\left(  B_{i}\right)  \text{\textquotedblright}\\
\text{is equivalent to \textquotedblleft}w\left(  A_{i}\right)  =B_{i}%
\text{\textquotedblright}%
\end{array}
\right) \\
&  =u\underbrace{\sum_{\substack{w\in S_{n};\\w\left(  A_{i}\right)
=B_{i}\text{ for all }i}}w}_{=\nabla_{\mathbf{B},\mathbf{A}}}v=u\nabla
_{\mathbf{B},\mathbf{A}}v.
\end{align*}
This proves Proposition \ref{prop.uNabv}.
\end{proof}

\subsubsection{Antisymmetrizers and the $\nabla_{U}^{-}$}

The sign of a permutation $w\in S_{n}$ shall be denoted by $\left(  -1\right)
^{w}$.

For each subset $U$ of $\left[  n\right]  $, we define the element%
\begin{equation}
\nabla_{U}^{-}:=\sum_{\substack{w\in S_{n};\\w\left(  i\right)  =i\text{ for
all }i\in\left[  n\right]  \setminus U}}\left(  -1\right)  ^{w}w\in
\mathcal{A}. \label{eq.Nab-U.1}%
\end{equation}
This is called the \emph{antisymmetrizer} of $U$ (aka the $U$%
\emph{-sign-integral} in the language of \cite[Definition 3.7.1]{sga}). Note
that $\nabla_{U}^{-}=1$ if $\left\vert U\right\vert \leq1$. Another way to
rephrase the definition of $\nabla_{U}^{-}$ is
\begin{equation}
\nabla_{U}^{-}:=\sum_{w\in S_{U}}\left(  -1\right)  ^{w}w\in\mathcal{A},
\label{eq.Nab-U.2}%
\end{equation}
where $S_{U}$ denotes the symmetric group on the set $U$ (embedded into
$S_{n}$ in the default way: each permutation $w\in S_{U}$ is extended to a
permutation of $\left[  n\right]  $ by letting it fix all elements of $\left[
n\right]  \setminus U$).

The antisymmetrizers $\nabla_{U}^{-}$ interact nicely with the permutations
$v\in S_{n}$:

\begin{proposition}
\label{prop.wNab-U}Let $v\in S_{n}$ be a permutation. Let $U$ be any subset of
$\left[  n\right]  $. Then,%
\begin{equation}
v\nabla_{U}^{-}=\nabla_{v\left(  U\right)  }^{-}v. \label{eq.wNab-U}%
\end{equation}

\end{proposition}

\begin{proof}
The definition of $\nabla_{v\left(  U\right)  }^{-}$ yields%
\begin{equation}
\nabla_{v\left(  U\right)  }^{-}=\sum_{\substack{w\in S_{n};\\w\left(
i\right)  =i\text{ for all }i\in\left[  n\right]  \setminus v\left(  U\right)
}}\left(  -1\right)  ^{w}w. \label{pf.prop.wNab-U.1}%
\end{equation}
The permutation $v\in S_{n}$ is a bijection from $\left[  n\right]  $ to
$\left[  n\right]  $, and thus induces a group isomorphism%
\begin{align*}
S_{\left[  n\right]  }  &  \rightarrow S_{\left[  n\right]  },\\
w  &  \mapsto vwv^{-1}.
\end{align*}
Of course, this isomorphism is just conjugation by $v$ in the group
$S_{\left[  n\right]  }=S_{n}$. But since it is induced by the bijection
$v:\left[  n\right]  \rightarrow\left[  n\right]  $, we immediately see from
functoriality that it sends the permutations $w\in S_{n}$ satisfying
\textquotedblleft$w\left(  i\right)  =i$ for all $i\in\left[  n\right]
\setminus U$\textquotedblright\ to the permutations $w\in S_{n}$ satisfying
\textquotedblleft$w\left(  i\right)  =i$ for all $i\in\left[  n\right]
\setminus v\left(  U\right)  $\textquotedblright. Thus, the latter
permutations are the images of the former permutations under the map $w\mapsto
vwv^{-1}$. Hence, we can substitute $vwv^{-1}$ for $w$ in the sum on the right
hand side of (\ref{pf.prop.wNab-U.1}). We thus obtain
\begin{align*}
\sum_{\substack{w\in S_{n};\\w\left(  i\right)  =i\text{ for all }i\in\left[
n\right]  \setminus v\left(  U\right)  }}\left(  -1\right)  ^{w}w  &
=\sum_{\substack{w\in S_{n};\\w\left(  i\right)  =i\text{ for all }i\in\left[
n\right]  \setminus U}}\underbrace{\left(  -1\right)  ^{vwv^{-1}}%
}_{\substack{=\left(  -1\right)  ^{w}\\\text{(since conjugation}%
\\\text{preserves the sign}\\\text{of a permutation)}}}vwv^{-1}\\
&  =\sum_{\substack{w\in S_{n};\\w\left(  i\right)  =i\text{ for all }%
i\in\left[  n\right]  \setminus U}}\left(  -1\right)  ^{w}vwv^{-1}\\
&  =v\underbrace{\sum_{\substack{w\in S_{n};\\w\left(  i\right)  =i\text{ for
all }i\in\left[  n\right]  \setminus U}}\left(  -1\right)  ^{w}w}%
_{\substack{=\nabla_{U}^{-}\\\text{(by the definition of }\nabla_{U}%
^{-}\text{)}}}v^{-1}=v\nabla_{U}^{-}v^{-1}.
\end{align*}
In view of (\ref{pf.prop.wNab-U.1}), this rewrites as $\nabla_{v\left(
U\right)  }^{-}=v\nabla_{U}^{-}v^{-1}$. In other words, $\nabla_{v\left(
U\right)  }^{-}v=v\nabla_{U}^{-}$. This proves Proposition \ref{prop.wNab-U}.
\end{proof}

The following fact will also be useful:

\begin{proposition}
\label{prop.Nab-UV}Let $U$ and $V$ be two subsets of $\left[  n\right]  $ such
that $V\subseteq U$. Then, $\nabla_{U}^{-}\mathcal{A}\subseteq\nabla_{V}%
^{-}\mathcal{A}$.
\end{proposition}

\begin{proof}
For any two distinct elements $p,q\in\left[  n\right]  $, let $t_{p,q}$ be the
transposition in $S_{n}$ that swaps $p$ with $q$. Then, a well-known formula
(\cite[Lemma 3.11.6]{sga}) says that every $X\subseteq\left[  n\right]  $ and
every $x\in X$ satisfy%
\[
\nabla_{X}^{-}=\nabla_{X\setminus\left\{  x\right\}  }^{-}\left(  1-\sum_{y\in
X\setminus\left\{  x\right\}  }t_{y,x}\right)  .
\]
Hence, every $X\subseteq\left[  n\right]  $ and every $x\in X$ satisfy%
\begin{align}
\nabla_{X}^{-}\mathcal{A}  &  =\nabla_{X\setminus\left\{  x\right\}  }%
^{-}\underbrace{\left(  1-\sum_{y\in X\setminus\left\{  x\right\}  }%
t_{y,x}\right)  \mathcal{A}}_{\subseteq\mathcal{A}}\nonumber\\
&  \subseteq\nabla_{X\setminus\left\{  x\right\}  }^{-}\mathcal{A}.
\label{pf.prop.Nab-UV.step}%
\end{align}
Now, $V\subseteq U$ shows that the set $V$ can be obtained from $U$ by
removing some elements (possibly none, if $V=U$). In other words,
$V=U\setminus\left\{  u_{1},u_{2},\ldots,u_{m}\right\}  $ for some distinct
elements $u_{1},u_{2},\ldots,u_{m}\in U$. Using these elements, we have%
\begin{align*}
\nabla_{U}^{-}\mathcal{A}  &  \subseteq\nabla_{U\setminus\left\{
u_{1}\right\}  }^{-}\mathcal{A}\ \ \ \ \ \ \ \ \ \ \left(  \text{by
(\ref{pf.prop.Nab-UV.step})}\right) \\
&  \subseteq\nabla_{\left(  U\setminus\left\{  u_{1}\right\}  \right)
\setminus\left\{  u_{2}\right\}  }^{-}\mathcal{A}\ \ \ \ \ \ \ \ \ \ \left(
\text{by (\ref{pf.prop.Nab-UV.step})}\right) \\
&  \subseteq\nabla_{\left(  \left(  U\setminus\left\{  u_{1}\right\}  \right)
\setminus\left\{  u_{2}\right\}  \right)  \setminus\left\{  u_{3}\right\}
}^{-}\mathcal{A}\ \ \ \ \ \ \ \ \ \ \left(  \text{by
(\ref{pf.prop.Nab-UV.step})}\right) \\
&  \subseteq\cdots\\
&  \subseteq\nabla_{\left(  \left(  \left(  U\setminus\left\{  u_{1}\right\}
\right)  \setminus\left\{  u_{2}\right\}  \right)  \setminus\cdots\right)
\setminus\left\{  u_{m}\right\}  }^{-}\mathcal{A}\ \ \ \ \ \ \ \ \ \ \left(
\text{by (\ref{pf.prop.Nab-UV.step})}\right) \\
&  =\nabla_{V}^{-}\mathcal{A}%
\end{align*}
(since $\left(  \left(  \left(  U\setminus\left\{  u_{1}\right\}  \right)
\setminus\left\{  u_{2}\right\}  \right)  \setminus\cdots\right)
\setminus\left\{  u_{m}\right\}  =U\setminus\left\{  u_{1},u_{2},\ldots
,u_{m}\right\}  =V$). This proves Proposition \ref{prop.Nab-UV}.
\end{proof}

\subsection{The two ideals}

This all was easy. Let us now move towards deeper waters. Recall that the
notation \textquotedblleft$\operatorname*{span}$\textquotedblright\ always
means a $\mathbf{k}$-linear span.

\begin{definition}
\label{def.IJ}Let $k\in\mathbb{N}$. We define two $\mathbf{k}$-submodules
$\mathcal{I}_{k}$ and $\mathcal{J}_{k}$ of $\mathcal{A}$ by
\[
\mathcal{I}_{k}:=\operatorname*{span}\left\{  \nabla_{\mathbf{B},\mathbf{A}%
}\ \mid\ \mathbf{A},\mathbf{B}\in\operatorname*{SC}\left(  n\right)  \text{
with }\ell\left(  \mathbf{A}\right)  =\ell\left(  \mathbf{B}\right)  \leq
k\right\}
\]
and%
\[
\mathcal{J}_{k}:=\mathcal{A}\cdot\operatorname*{span}\left\{  \nabla_{U}%
^{-}\ \mid\ U\text{ is a subset of }\left[  n\right]  \text{ having size
}k+1\right\}  \cdot\mathcal{A}.
\]

\end{definition}

Note that the set $\left\{  \nabla_{U}^{-}\ \mid\ U\text{ is a subset of
}\left[  n\right]  \text{ having size }k+1\right\}  $ is empty when $k\geq n$,
since no subsets of $\left[  n\right]  $ have size larger than $n$. The span
of an empty set is the zero submodule $\left\{  0\right\}  $.

\begin{proposition}
\label{prop.IJ.1}Let $k\in\mathbb{N}$. Then:

\begin{enumerate}
\item[\textbf{(a)}] Both $\mathcal{I}_{k}$ and $\mathcal{J}_{k}$ are ideals of
$\mathcal{A}$. (\textquotedblleft Ideal\textquotedblright\ always means
\textquotedblleft two-sided ideal\textquotedblright.)

\item[\textbf{(b)}] We have%
\begin{align*}
\mathcal{J}_{k}  &  =\mathcal{A}\cdot\operatorname*{span}\left\{  \nabla
_{U}^{-}\ \mid\ U\text{ is a subset of }\left[  n\right]  \text{ having size
}k+1\right\} \\
&  =\operatorname*{span}\left\{  \nabla_{U}^{-}\ \mid\ U\text{ is a subset of
}\left[  n\right]  \text{ having size }k+1\right\}  \cdot\mathcal{A}.
\end{align*}

\item[\textbf{(c)}] The antipode $S$ satisfies $S\left(  \mathcal{I}%
_{k}\right)  =\mathcal{I}_{k}$ and $S\left(  \mathcal{J}_{k}\right)
=\mathcal{J}_{k}$.

\item[\textbf{(d)}] We have
\[
\mathcal{I}_{k}=\operatorname*{span}\left\{  \nabla_{\mathbf{B},\mathbf{A}%
}\ \mid\ \mathbf{A},\mathbf{B}\in\operatorname*{SD}\left(  n\right)  \text{
with }\ell\left(  \mathbf{A}\right)  =\ell\left(  \mathbf{B}\right)  \leq
k\right\}  .
\]

\item[\textbf{(e)}] We have
\begin{align*}
\mathcal{J}_{k}  &  =\mathcal{A}\cdot\operatorname*{span}\left\{  \nabla
_{U}^{-}\ \mid\ U\text{ is a subset of }\left[  n\right]  \text{ having size
}>k\right\}  \cdot\mathcal{A}\\
&  =\mathcal{A}\cdot\operatorname*{span}\left\{  \nabla_{U}^{-}\ \mid\ U\text{
is a subset of }\left[  n\right]  \text{ having size }>k\right\} \\
&  =\operatorname*{span}\left\{  \nabla_{U}^{-}\ \mid\ U\text{ is a subset of
}\left[  n\right]  \text{ having size }>k\right\}  \cdot\mathcal{A}.
\end{align*}

\item[\textbf{(f)}] If $X$ is any subset of $\left[  n\right]  $ having size
$k+1$, then%
\[
\mathcal{J}_{k}=\mathcal{A}\cdot\nabla_{X}^{-}\cdot\mathcal{A}.
\]

\item[\textbf{(g)}] We have%
\[
\mathcal{I}_{k}=\operatorname*{span}\left\{  \nabla_{\mathbf{B},\mathbf{A}%
}\ \mid\ \mathbf{A},\mathbf{B}\in\operatorname*{SD}\left(  n\right)  \text{
with }\ell\left(  \mathbf{A}\right)  =\ell\left(  \mathbf{B}\right)
=k\right\}  .
\]

\end{enumerate}
\end{proposition}

\begin{proof}
\textbf{(a)} Clearly, $\mathcal{J}_{k}$ is an ideal of $\mathcal{A}$ (since
$\mathcal{J}_{k}$ has the form $\mathcal{J}_{k}=\mathcal{AXA}$ for some
$\mathbf{k}$-submodule $\mathcal{X}\subseteq\mathcal{A}$). It remains to show
that so is $\mathcal{I}_{k}$.

But
\[
\mathcal{I}_{k}=\operatorname*{span}\left\{  \nabla_{\mathbf{B},\mathbf{A}%
}\ \mid\ \mathbf{A},\mathbf{B}\in\operatorname*{SC}\left(  n\right)  \text{
with }\ell\left(  \mathbf{A}\right)  =\ell\left(  \mathbf{B}\right)  \leq
k\right\}  .
\]
Hence, $\mathcal{I}_{k}$ is a $\mathbf{k}$-submodule of $\mathcal{A}$.
Moreover, (\ref{eq.uNabv}) shows that for any two set compositions
$\mathbf{A},\mathbf{B}\in\operatorname*{SC}\left(  n\right)  $ with
$\ell\left(  \mathbf{A}\right)  =\ell\left(  \mathbf{B}\right)  \leq k$ and
any two permutations $u,v\in S_{n}$, we have $u\nabla_{\mathbf{B},\mathbf{A}%
}v=\nabla_{u\mathbf{B},v^{-1}\mathbf{A}}\in\mathcal{I}_{k}$ (since
$v^{-1}\mathbf{A}$ and $u\mathbf{B}$ are again two set compositions in
$\operatorname*{SC}\left(  n\right)  $ and satisfy $\ell\left(  v^{-1}%
\mathbf{A}\right)  =\ell\left(  u\mathbf{B}\right)  =\ell\left(
\mathbf{A}\right)  =\ell\left(  \mathbf{B}\right)  \leq k$). By linearity,
this shows that $\mathcal{AI}_{k}\mathcal{A\subseteq I}_{k}$ (since
$\mathcal{A}$ is spanned by the permutations $w\in S_{n}$, whereas
$\mathcal{I}_{k}$ is spanned by the $\nabla_{\mathbf{B},\mathbf{A}}$ for
$\mathbf{A},\mathbf{B}\in\operatorname*{SC}\left(  n\right)  $ with
$\ell\left(  \mathbf{A}\right)  =\ell\left(  \mathbf{B}\right)  \leq k$). In
other words, $\mathcal{I}_{k}$ is an ideal of $\mathcal{A}$. Thus, the proof
of part \textbf{(a)} is complete. \medskip

\textbf{(b)} Since $\mathcal{A}=\operatorname*{span}S_{n}$, we have%
\begin{align*}
&  \mathcal{A}\cdot\operatorname*{span}\left\{  \nabla_{U}^{-}\ \mid\ U\text{
is a subset of }\left[  n\right]  \text{ having size }k+1\right\} \\
&  =\operatorname*{span}S_{n}\cdot\operatorname*{span}\left\{  \nabla_{U}%
^{-}\ \mid\ U\text{ is a subset of }\left[  n\right]  \text{ having size
}k+1\right\} \\
&  =\operatorname*{span}\left\{  v\nabla_{U}^{-}\ \mid\ v\in S_{n}\text{,
while }U\text{ is a subset of }\left[  n\right]  \text{ having size
}k+1\right\} \\
&  =\operatorname*{span}\left\{  \nabla_{v\left(  U\right)  }^{-}v\ \mid\ v\in
S_{n}\text{, while }U\text{ is a subset of }\left[  n\right]  \text{ having
size }k+1\right\} \\
&  \ \ \ \ \ \ \ \ \ \ \ \ \ \ \ \ \ \ \ \ \left(  \text{by Proposition
\ref{prop.wNab-U}}\right) \\
&  =\operatorname*{span}\left\{  \nabla_{U}^{-}v\ \mid\ v\in S_{n}\text{,
while }U\text{ is a subset of }\left[  n\right]  \text{ having size
}k+1\right\} \\
&  \ \ \ \ \ \ \ \ \ \ \ \ \ \ \ \ \ \ \ \ \left(
\begin{array}
[c]{c}%
\text{here, we have substituted }U\text{ for }v\left(  U\right)  \text{, since
}v\in S_{n}\\
\text{permutes the subsets of }\left[  n\right]  \text{ having size }k+1
\end{array}
\right)
\end{align*}
and
\begin{align*}
&  \operatorname*{span}\left\{  \nabla_{U}^{-}\ \mid\ U\text{ is a subset of
}\left[  n\right]  \text{ having size }k+1\right\}  \cdot\mathcal{A}\\
&  =\operatorname*{span}\left\{  \nabla_{U}^{-}\ \mid\ U\text{ is a subset of
}\left[  n\right]  \text{ having size }k+1\right\}  \cdot\operatorname*{span}%
S_{n}\\
&  =\operatorname*{span}\left\{  \nabla_{U}^{-}v\ \mid\ v\in S_{n}\text{,
while }U\text{ is a subset of }\left[  n\right]  \text{ having size
}k+1\right\}  .
\end{align*}
The right hand sides of these two equalities are equal; hence, so are their
left hand sides. In other words,%
\begin{align*}
&  \mathcal{A}\cdot\operatorname*{span}\left\{  \nabla_{U}^{-}\ \mid\ U\text{
is a subset of }\left[  n\right]  \text{ having size }k+1\right\} \\
&  =\operatorname*{span}\left\{  \nabla_{U}^{-}\ \mid\ U\text{ is a subset of
}\left[  n\right]  \text{ having size }k+1\right\}  \cdot\mathcal{A}.
\end{align*}

Now, the definition of $\mathcal{J}_{k}$ yields%
\begin{align*}
\mathcal{J}_{k}  &  =\mathcal{A}\cdot\underbrace{\operatorname*{span}\left\{
\nabla_{U}^{-}\ \mid\ U\text{ is a subset of }\left[  n\right]  \text{ having
size }k+1\right\}  \cdot\mathcal{A}}_{\substack{=\mathcal{A}\cdot
\operatorname*{span}\left\{  \nabla_{U}^{-}\ \mid\ U\text{ is a subset of
}\left[  n\right]  \text{ having size }k+1\right\}  \\\text{(by the preceding
sentence)}}}\\
&  =\underbrace{\mathcal{A}\cdot\mathcal{A}}_{=\mathcal{A}}\cdot
\operatorname*{span}\left\{  \nabla_{U}^{-}\ \mid\ U\text{ is a subset of
}\left[  n\right]  \text{ having size }k+1\right\} \\
&  =\mathcal{A}\cdot\operatorname*{span}\left\{  \nabla_{U}^{-}\ \mid\ U\text{
is a subset of }\left[  n\right]  \text{ having size }k+1\right\} \\
&  =\operatorname*{span}\left\{  \nabla_{U}^{-}\ \mid\ U\text{ is a subset of
}\left[  n\right]  \text{ having size }k+1\right\}  \cdot\mathcal{A}.
\end{align*}
Thus, part \textbf{(b)} is proved. \medskip

\textbf{(c)} The equality $S\left(  \mathcal{I}_{k}\right)  =\mathcal{I}_{k}$
follows from Proposition \ref{prop.row.simple} \textbf{(d)}. It remains to
prove $S\left(  \mathcal{J}_{k}\right)  =\mathcal{J}_{k}$.

Let $\mathcal{X}_{k}=\operatorname*{span}\left\{  \nabla_{U}^{-}%
\ \mid\ U\text{ is a subset of }\left[  n\right]  \text{ having size
}k+1\right\}  $. Then, the definition of $\mathcal{J}_{k}$ rewrites as
$\mathcal{J}_{k}=\mathcal{AX}_{k}\mathcal{A}$. But we have $S\left(
\nabla_{U}^{-}\right)  =\nabla_{U}^{-}$ for each $U\subseteq\left[  n\right]
$ (see, e.g., \cite[Example 3.11.13 \textbf{(c)}]{sga}). Since the map $S$ is
$\mathbf{k}$-linear, we thus conclude that $S\left(  \mathcal{X}_{k}\right)
=\mathcal{X}_{k}$ (since $\mathcal{X}_{k}=\operatorname*{span}\left\{
\nabla_{U}^{-}\ \mid\ U\text{ is a subset of }\left[  n\right]  \text{ having
size }k+1\right\}  $). Now, from $\mathcal{J}_{k}=\mathcal{AX}_{k}\mathcal{A}%
$, we obtain%
\begin{align*}
S\left(  \mathcal{J}_{k}\right)   &  =S\left(  \mathcal{AX}_{k}\mathcal{A}%
\right)  =\underbrace{S\left(  \mathcal{A}\right)  }_{=\mathcal{A}%
}\underbrace{S\left(  \mathcal{X}_{k}\right)  }_{=\mathcal{X}_{k}%
}\underbrace{S\left(  \mathcal{A}\right)  }_{=\mathcal{A}}%
\ \ \ \ \ \ \ \ \ \ \left(
\begin{array}
[c]{c}%
\text{since }S\text{ is a }\mathbf{k}\text{-algebra}\\
\text{anti-automorphism}%
\end{array}
\right) \\
&  =\mathcal{AX}_{k}\mathcal{A=J}_{k}.
\end{align*}
This completes the proof of part \textbf{(c)}. \medskip

\textbf{(d)} Any set composition of $\left[  n\right]  $ is a set
decomposition of $\left[  n\right]  $. In other words, $\operatorname*{SC}%
\left(  n\right)  \subseteq\operatorname*{SD}\left(  n\right)  $. Thus,%
\[
\mathcal{I}_{k}\subseteq\operatorname*{span}\left\{  \nabla_{\mathbf{B}%
,\mathbf{A}}\ \mid\ \mathbf{A},\mathbf{B}\in\operatorname*{SD}\left(
n\right)  \text{ with }\ell\left(  \mathbf{A}\right)  =\ell\left(
\mathbf{B}\right)  \leq k\right\}  .
\]
It remains to prove the converse inclusion. To this purpose, we must show that
if $\mathbf{A}$ and $\mathbf{B}$ are set decompositions of $\left[  n\right]
$ satisfying $\ell\left(  \mathbf{A}\right)  =\ell\left(  \mathbf{B}\right)
\leq k$, then the row-to-row sum $\nabla_{\mathbf{B},\mathbf{A}}$ belongs to
$\mathcal{I}_{k}$. So let us show this.

Let $\mathbf{A}$ and $\mathbf{B}$ be set decompositions of $\left[  n\right]
$ satisfying $\ell\left(  \mathbf{A}\right)  =\ell\left(  \mathbf{B}\right)
\leq k$. We must prove that the row-to-row sum $\nabla_{\mathbf{B},\mathbf{A}%
}$ belongs to $\mathcal{I}_{k}$.

If any of the blocks of $\mathbf{A}$ is empty while the corresponding block of
$\mathbf{B}$ is not, then this is clear, since Proposition
\ref{prop.row.simple} \textbf{(a)} yields $\nabla_{\mathbf{B},\mathbf{A}}%
=0\in\mathcal{I}_{k}$. A similar argument applies if any of the blocks of
$\mathbf{B}$ is empty while the corresponding block of $\mathbf{A}$ is not. In
the remaining case, the empty blocks of $\mathbf{A}$ appear at the same
positions as the empty blocks of $\mathbf{B}$. Removing all these empty blocks
from both $\mathbf{A}$ and $\mathbf{B}$, we obtain two set compositions
$\mathbf{A}^{\prime},\mathbf{B}^{\prime}$ of $\left[  n\right]  $ satisfying
$\ell\left(  \mathbf{A}^{\prime}\right)  =\ell\left(  \mathbf{B}^{\prime
}\right)  \leq\ell\left(  \mathbf{A}\right)  =\ell\left(  \mathbf{B}\right)  $
and $\nabla_{\mathbf{B},\mathbf{A}}=\nabla_{\mathbf{B}^{\prime},\mathbf{A}%
^{\prime}}$ (by Proposition \ref{prop.row.simple} \textbf{(c)}). But the
definition of $\mathcal{I}_{k}$ yields $\nabla_{\mathbf{B}^{\prime}%
,\mathbf{A}^{\prime}}\in\mathcal{I}_{k}$ (since $\mathbf{A}^{\prime
},\mathbf{B}^{\prime}\in\operatorname*{SC}\left(  n\right)  $ and $\ell\left(
\mathbf{A}^{\prime}\right)  =\ell\left(  \mathbf{B}^{\prime}\right)  \leq
\ell\left(  \mathbf{A}\right)  =\ell\left(  \mathbf{B}\right)  \leq k$). Thus,
$\nabla_{\mathbf{B},\mathbf{A}}=\nabla_{\mathbf{B}^{\prime},\mathbf{A}%
^{\prime}}\in\mathcal{I}_{k}$. This completes the proof of part \textbf{(d)}.
\medskip

\textbf{(e)} The definition of $\mathcal{J}_{k}$ yields%
\begin{align*}
\mathcal{J}_{k}  &  =\mathcal{A}\cdot\operatorname*{span}\left\{  \nabla
_{U}^{-}\ \mid\ U\text{ is a subset of }\left[  n\right]  \text{ having size
}k+1\right\}  \cdot\mathcal{A}\\
&  \subseteq\mathcal{A}\cdot\operatorname*{span}\left\{  \nabla_{U}^{-}%
\ \mid\ U\text{ is a subset of }\left[  n\right]  \text{ having size
}>k\right\}  \cdot\mathcal{A}%
\end{align*}
(since any subset of size $k+1$ has size $>k$). Let us now prove the converse
inclusion. Since $\mathcal{J}_{k}$ is an ideal of $\mathcal{A}$, it suffices
to show that $\nabla_{U}^{-}\in\mathcal{J}_{k}$ whenever $U$ is a subset of
$\left[  n\right]  $ having size $>k$. So let $U$ be a subset of $\left[
n\right]  $ having size $>k$. Then, $U$ has size $\geq k+1$. Hence, $U$ has a
subset $V$ of size $k+1$. Consider this $V$. Proposition \ref{prop.Nab-UV}
yields $\nabla_{U}^{-}\mathcal{A}\subseteq\nabla_{V}^{-}\mathcal{A}$. But the
definition of $\mathcal{J}_{k}$ shows that $\nabla_{V}^{-}\in\mathcal{J}_{k}$
(since $V$ is a subset of $\left[  n\right]  $ having size $k+1$), and thus we
have $\nabla_{V}^{-}\mathcal{A}\subseteq\mathcal{J}_{k}\mathcal{A}%
\subseteq\mathcal{J}_{k}$ (since $\mathcal{J}_{k}$ is an ideal of
$\mathcal{A}$). Hence, $\nabla_{U}^{-}=\nabla_{U}^{-}\cdot\underbrace{1}%
_{\in\mathcal{A}}\in\nabla_{U}^{-}\mathcal{A}\subseteq\nabla_{V}%
^{-}\mathcal{A}\subseteq\mathcal{J}_{k}$. So we have shown that $\nabla
_{U}^{-}\in\mathcal{J}_{k}$ whenever $U$ is a subset of $\left[  n\right]  $
having size $>k$. This proves
\[
\mathcal{A}\cdot\operatorname*{span}\left\{  \nabla_{U}^{-}\ \mid\ U\text{ is
a subset of }\left[  n\right]  \text{ having size }>k\right\}  \cdot
\mathcal{A}\subseteq\mathcal{J}_{k}%
\]
(since $\mathcal{J}_{k}$ is an ideal of $\mathcal{A}$). Combining this with
the inclusion%
\[
\mathcal{J}_{k}\subseteq\mathcal{A}\cdot\operatorname*{span}\left\{
\nabla_{U}^{-}\ \mid\ U\text{ is a subset of }\left[  n\right]  \text{ having
size }>k\right\}  \cdot\mathcal{A}%
\]
(which we have already proved), we obtain%
\[
\mathcal{J}_{k}=\mathcal{A}\cdot\operatorname*{span}\left\{  \nabla_{U}%
^{-}\ \mid\ U\text{ is a subset of }\left[  n\right]  \text{ having size
}>k\right\}  \cdot\mathcal{A}.
\]
Similarly, we can show%
\[
\mathcal{J}_{k}=\mathcal{A}\cdot\operatorname*{span}\left\{  \nabla_{U}%
^{-}\ \mid\ U\text{ is a subset of }\left[  n\right]  \text{ having size
}>k\right\}
\]
and%
\[
\mathcal{J}_{k}=\operatorname*{span}\left\{  \nabla_{U}^{-}\ \mid\ U\text{ is
a subset of }\left[  n\right]  \text{ having size }>k\right\}  \cdot
\mathcal{A}%
\]
(using Proposition \ref{prop.IJ.1} \textbf{(b)} as a starting point). Thus,
part \textbf{(e)} is proved. \medskip

\textbf{(f)} Let $X$ be any subset of $\left[  n\right]  $ having size $k+1$.
Then, from $\mathcal{A}=\operatorname*{span}S_{n}$, we obtain%
\[
\mathcal{A}\cdot\nabla_{X}^{-}=\operatorname*{span}\left\{  v\nabla_{X}%
^{-}\ \mid\ v\in S_{n}\right\}  =\operatorname*{span}\left\{  \nabla_{v\left(
X\right)  }^{-}v\ \mid\ v\in S_{n}\right\}
\]
(since Proposition \ref{prop.wNab-U} yields $v\nabla_{X}^{-}=\nabla_{v\left(
X\right)  }^{-}v$). Hence, again using $\mathcal{A}=\operatorname*{span}S_{n}%
$, we obtain
\begin{align*}
&  \underbrace{\mathcal{A}\cdot\nabla_{X}^{-}}_{=\operatorname*{span}\left\{
\nabla_{v\left(  X\right)  }^{-}v\ \mid\ v\in S_{n}\right\}  }\cdot
\underbrace{\mathcal{A}}_{=\operatorname*{span}S_{n}}\\
&  =\operatorname*{span}\left\{  \nabla_{v\left(  X\right)  }^{-}v\ \mid\ v\in
S_{n}\right\}  \cdot\operatorname*{span}S_{n}\\
&  =\operatorname*{span}\left\{  \nabla_{v\left(  X\right)  }^{-}%
vw\ \mid\ v\in S_{n}\text{ and }w\in S_{n}\right\} \\
&  =\operatorname*{span}\left\{  \nabla_{v\left(  X\right)  }^{-}u\ \mid\ v\in
S_{n}\text{ and }u\in S_{n}\right\} \\
&  \ \ \ \ \ \ \ \ \ \ \ \ \ \ \ \ \ \ \ \ \left(  \text{here, we have
substituted }u\text{ for }vw\right) \\
&  =\operatorname*{span}\left\{  \nabla_{U}^{-}u\ \mid\ U\text{ is a subset of
}\left[  n\right]  \text{ having size }k+1\text{, and }u\in S_{n}\right\} \\
&  \ \ \ \ \ \ \ \ \ \ \ \ \ \ \ \ \ \ \ \ \left(
\begin{array}
[c]{c}%
\text{here, we have substituted }U\text{ for }v\left(  X\right)  \text{,}\\
\text{since each subset }U\text{ of }\left[  n\right]  \text{ having size
}k+1\\
\text{can be written as }v\left(  X\right)  \text{ for some }v\in S_{n}%
\end{array}
\right) \\
&  =\operatorname*{span}\left\{  \nabla_{U}^{-}\ \mid\ U\text{ is a subset of
}\left[  n\right]  \text{ having size }k+1\right\}  \cdot
\underbrace{\operatorname*{span}S_{n}}_{=\mathcal{A}}\\
&  =\operatorname*{span}\left\{  \nabla_{U}^{-}\ \mid\ U\text{ is a subset of
}\left[  n\right]  \text{ having size }k+1\right\}  \cdot\mathcal{A}\\
&  =\mathcal{J}_{k}\ \ \ \ \ \ \ \ \ \ \left(  \text{by part \textbf{(b)}%
}\right)  .
\end{align*}
This proves part \textbf{(f)}. \medskip

\textbf{(g)} This will follow from part \textbf{(d)}, once we have proved the
equality%
\begin{align*}
&  \left\{  \nabla_{\mathbf{B},\mathbf{A}}\ \mid\ \mathbf{A},\mathbf{B}%
\in\operatorname*{SD}\left(  n\right)  \text{ with }\ell\left(  \mathbf{A}%
\right)  =\ell\left(  \mathbf{B}\right)  \leq k\right\} \\
&  =\left\{  \nabla_{\mathbf{B},\mathbf{A}}\ \mid\ \mathbf{A},\mathbf{B}%
\in\operatorname*{SD}\left(  n\right)  \text{ with }\ell\left(  \mathbf{A}%
\right)  =\ell\left(  \mathbf{B}\right)  =k\right\}  .
\end{align*}
So let us prove this equality. The right hand side here is clearly a subset of
the left hand side. It remains to show the reverse inclusion (i.e., that the
left hand side is a subset of the right hand side). In other words, it remains
to check that each $\nabla_{\mathbf{B},\mathbf{A}}$ with $\mathbf{A}%
,\mathbf{B}\in\operatorname*{SD}\left(  n\right)  $ satisfying $\ell\left(
\mathbf{A}\right)  =\ell\left(  \mathbf{B}\right)  \leq k$ can be rewritten in
the form $\nabla_{\mathbf{B}^{\prime},\mathbf{A}^{\prime}}$ for some
$\mathbf{A}^{\prime},\mathbf{B}^{\prime}\in\operatorname*{SD}\left(  n\right)
$ satisfying $\ell\left(  \mathbf{A}^{\prime}\right)  =\ell\left(
\mathbf{B}^{\prime}\right)  =k$. But this is easy: Set $m:=\ell\left(
\mathbf{A}\right)  =\ell\left(  \mathbf{B}\right)  \leq k$, and write the set
decompositions $\mathbf{A}$ and $\mathbf{B}$ in the form $\mathbf{A}=\left(
A_{1},A_{2},\ldots,A_{m}\right)  $ and $\mathbf{B}=\left(  B_{1},B_{2}%
,\ldots,B_{m}\right)  $; then set%
\begin{align*}
\mathbf{A}^{\prime}  &  :=\left(  A_{1},A_{2},\ldots,A_{m}%
,\underbrace{\varnothing,\varnothing,\ldots,\varnothing}_{k-m\text{ empty
sets}}\right)  \ \ \ \ \ \ \ \ \ \ \text{and}\\
\mathbf{B}^{\prime}  &  :=\left(  B_{1},B_{2},\ldots,B_{m}%
,\underbrace{\varnothing,\varnothing,\ldots,\varnothing}_{k-m\text{ empty
sets}}\right)
\end{align*}
(this is allowed since $m\leq k$). Then, $\mathbf{A}^{\prime}$ and
$\mathbf{B}^{\prime}$ are set decompositions in $\operatorname*{SD}\left(
n\right)  $ satisfying $\ell\left(  \mathbf{A}^{\prime}\right)  =\ell\left(
\mathbf{B}^{\prime}\right)  =k$ and $\nabla_{\mathbf{B}^{\prime}%
,\mathbf{A}^{\prime}}=\nabla_{\mathbf{B},\mathbf{A}}$ (the latter follows from
Proposition \ref{prop.row.simple} \textbf{(c)}, since $\mathbf{A}$ and
$\mathbf{B}$ can be obtained from $\mathbf{A}^{\prime}$ and $\mathbf{B}%
^{\prime}$ by removing the $k-m$ empty blocks at the end). Thus, the proof of
part \textbf{(g)} is complete.
\end{proof}

The ideals $\mathcal{J}_{k}$ have been studied several times. In particular,
the ideal $\mathcal{J}_{k}$ is the kernel of the map in \cite[Theorem
4.2]{deCPro76} (where our $k$ and $n$ have been renamed as $n$ and $m$). Also,
the ideal $\mathcal{J}_{2}$ has recently appeared in quantum information
theory as the ideal $\widetilde{\mathcal{I}}_{n}^{\operatorname*{swap}}$ in
\cite[Lemma 3.4]{BCEHK23}.

\subsection{Annihilators and the bilinear form}

If $\mathcal{B}$ is any subset of $\mathcal{A}$, then we define the two
subsets%
\begin{align*}
\operatorname*{LAnn}\mathcal{B}  &  :=\left\{  a\in\mathcal{A}\ \mid
\ ab=0\text{ for all }b\in\mathcal{B}\right\}  \ \ \ \ \ \ \ \ \ \ \text{and}%
\\
\operatorname*{RAnn}\mathcal{B}  &  :=\left\{  a\in\mathcal{A}\ \mid
\ ba=0\text{ for all }b\in\mathcal{B}\right\}
\end{align*}
of $\mathcal{A}$. We call them the \emph{left annihilator} and the \emph{right
annihilator} of $\mathcal{B}$, respectively. These annihilators
$\operatorname*{LAnn}\mathcal{B}$ and $\operatorname*{RAnn}\mathcal{B}$ are
always $\mathbf{k}$-submodules of $\mathcal{A}$, even when $\mathcal{B}$ is not.

Moreover, we define the $\mathbf{k}$-bilinear form%
\[
\left\langle \cdot,\cdot\right\rangle :\mathcal{A}\times\mathcal{A}%
\rightarrow\mathbf{k},
\]
which sends the pair $\left(  u,v\right)  $ to $%
\begin{cases}
1, & \text{if }u=v;\\
0, & \text{if }u\neq v
\end{cases}
$ for any two permutations $u,v\in S_{n}$. This is the standard nondegenerate
symmetric bilinear form on $\mathcal{A}=\mathbf{k}\left[  S_{n}\right]  $
known from representation theory. We shall refer to this form as the \emph{dot
product}.

If $\mathcal{B}$ is any subset of $\mathcal{A}$, then we define the subset%
\[
\mathcal{B}^{\perp}:=\left\{  a\in\mathcal{A}\ \mid\ \left\langle
a,b\right\rangle =0\text{ for all }b\in\mathcal{B}\right\}
\]
of $\mathcal{A}$. This is called the \emph{orthogonal complement} of
$\mathcal{B}$ in $\mathcal{A}$. Note that it does not change if we replace
$\left\langle a,b\right\rangle $ by $\left\langle b,a\right\rangle $ in its
definition, since the form $\left\langle \cdot,\cdot\right\rangle $ is
symmetric. Note also that $\mathcal{B}^{\perp}$ is always a $\mathbf{k}%
$-submodule of $\mathcal{A}$, even when $\mathcal{B}$ is not.

\begin{definition}
\label{def.avoid.up}Let $k\in\mathbb{N}$.

\begin{enumerate}
\item[\textbf{(a)}] Let $w\in S_{n}$ be a permutation. We say that $w$
\emph{avoids }$12\cdots\left(  k+1\right)  $ if there exists no $\left(
k+1\right)  $-element subset $U$ of $\left[  n\right]  $ such that the
restriction $w\mid_{U}$ is increasing (i.e., if there exist no $k+1$ elements
$i_{1}<i_{2}<\cdots<i_{k+1}$ of $\left[  n\right]  $ such that $w\left(
i_{1}\right)  <w\left(  i_{2}\right)  <\cdots<w\left(  i_{k+1}\right)  $).

\item[\textbf{(b)}] We let $\operatorname*{Av}\nolimits_{n}\left(  k+1\right)
$ denote the set of all permutations $w\in S_{n}$ that avoid $12\cdots\left(
k+1\right)  $.
\end{enumerate}
\end{definition}

This notion of \textquotedblleft avoiding $12\cdots\left(  k+1\right)
$\textquotedblright\ is taken from the theory of pattern avoidance (see, e.g.,
\cite[Chapters 4--5]{Bona22}).

\subsection{The main theorem}

We now arrive at one of our main results, which will be proved in Subsection
\ref{sec.row.main-proof}:

\begin{theorem}
\label{thm.row.main}Let $k\in\mathbb{N}$. Then:

\begin{enumerate}
\item[\textbf{(a)}] We have $\mathcal{I}_{k}=\mathcal{J}_{k}^{\perp
}=\operatorname*{LAnn}\mathcal{J}_{k}=\operatorname*{RAnn}\mathcal{J}_{k}$.

\item[\textbf{(b)}] We have $\mathcal{J}_{k}=\mathcal{I}_{k}^{\perp
}=\operatorname*{LAnn}\mathcal{I}_{k}=\operatorname*{RAnn}\mathcal{I}_{k}$.

\item[\textbf{(c)}] The $\mathbf{k}$-module $\mathcal{I}_{k}$ is free of rank
$\left\vert \operatorname*{Av}\nolimits_{n}\left(  k+1\right)  \right\vert $.

\item[\textbf{(d)}] The $\mathbf{k}$-module $\mathcal{J}_{k}$ is free of rank
$\left\vert S_{n}\setminus\operatorname*{Av}\nolimits_{n}\left(  k+1\right)
\right\vert $.

\item[\textbf{(e)}] The $\mathbf{k}$-module $\mathcal{A}/\mathcal{I}_{k}$ is
free with basis $\left(  \overline{w}\right)  _{w\in S_{n}\setminus
\operatorname*{Av}\nolimits_{n}\left(  k+1\right)  }$. (Here, $\overline{w}$
denotes the projection of $w\in\mathcal{A}$ onto the quotient $\mathcal{A}%
/\mathcal{I}_{k}$.)

\item[\textbf{(f)}] The $\mathbf{k}$-module $\mathcal{A}/\mathcal{J}_{k}$ is
free with basis $\left(  \overline{w}\right)  _{w\in\operatorname*{Av}%
\nolimits_{n}\left(  k+1\right)  }$. (Here, $\overline{w}$ denotes the
projection of $w\in\mathcal{A}$ onto the quotient $\mathcal{A}/\mathcal{J}%
_{k}$.)

\item[\textbf{(g)}] Assume that $n!$ is invertible in $\mathbf{k}$. Then,
$\mathcal{A}=\mathcal{I}_{k}\oplus\mathcal{J}_{k}$ (internal direct sum) as
$\mathbf{k}$-module. Moreover, $\mathcal{I}_{k}$ and $\mathcal{J}_{k}$ are
nonunital subalgebras of $\mathcal{A}$ that have unities and satisfy
$\mathcal{A}\cong\mathcal{I}_{k}\times\mathcal{J}_{k}$ as $\mathbf{k}$-algebras.\footnotemark
\end{enumerate}
\end{theorem}

\footnotetext{The formulation \textquotedblleft nonunital subalgebras that
have unities\textquotedblright\ may sound paradoxical. But it is literally
true: A nonunital subalgebra of $\mathcal{A}$ can have a unity but still be
nonunital \textbf{as a subalgebra} because its unity is not the unity of
$\mathcal{A}$. This is precisely the situation that the nonunital subalgebras
$\mathcal{I}_{k}$ and $\mathcal{J}_{k}$ find themselves in.}Furthermore, the
ideals $\mathcal{I}_{k}$ and $\mathcal{J}_{k}$ of $\mathcal{A}$ have
representation-theoretical meanings:

\begin{itemize}
\item They are the annihilators of certain tensor power representations of
$S_{n}$: see Theorem \ref{thm.AnnVkn} and Theorem \ref{thm.AnnNnk}.

\item If $n!$ is invertible in $\mathbf{k}$ (for example, if $\mathbf{k}$ is a
field of characteristic $0$), then they are furthermore certain subproducts of
the Artin--Wedderburn decomposition of $\mathcal{A}$: see Theorem
\ref{thm.IJ.rep}.
\end{itemize}

Once Theorem \ref{thm.row.main} is proved, we will easily obtain the following
corollary regarding the $k=2$ case:

\begin{corollary}
\label{cor.Nabla.span}\ \ 

\begin{enumerate}
\item[\textbf{(a)}] We have%
\begin{equation}
\mathcal{I}_{2}=\operatorname*{span}\left\{  \nabla_{B,A}\ \mid\ A,B\subseteq
\left[  n\right]  \right\}  . \label{eq.cor.Nabla.span.eq}%
\end{equation}

\item[\textbf{(b)}] The $\mathbf{k}$-module $\mathcal{I}_{2}$ is free of rank
$\left\vert \operatorname*{Av}\nolimits_{n}\left(  3\right)  \right\vert $,
which is the Catalan number $C_{n}$.
\end{enumerate}
\end{corollary}

See Subsection \ref{sec.row.main-proof} for the proof of this corollary.

\begin{remark}
\label{rmk.IJ=murphy}The results in this section have a significant overlap
with existing literature, particularly with prior work on the Murphy cellular
bases and on annihilators of Young permutation modules. The latter will be
discussed in Subsection \ref{sec.row.ann}. Let us briefly outline the former,
insofar as it provides alternative proofs for parts of Theorem
\ref{thm.row.main}. Note that this remark is included for context only and
will only be used in the proof of Theorem \ref{thm.A/I+J}.

The Murphy cellular bases first appeared in the paper \cite{CanWil89} by
Canfield and Williamson; then, Murphy (\cite{Murphy92} and \cite{Murphy95})
generalized them to the Hecke algebra. The easiest way to define them (in
$\mathcal{A}=\mathbf{k}\left[  S_{n}\right]  $) is as follows (following the
notations of \cite[\S 6.8]{sga}): An $n$\emph{-bitableau} shall mean a triple
$\left(  \lambda,U,V\right)  $, where $\lambda$ is a partition of $n$ and
where $U$ and $V$ are two tableaux of shape $\lambda$ with entries
$1,2,\ldots,n$ each. Given an $n$-bitableau $\left(  \lambda,U,V\right)  $, we
define two elements%
\begin{align*}
\nabla_{V,U}^{\operatorname*{Row}}  &  :=\sum_{\substack{w\in S_{n};\\wU\text{
is row-equivalent to }V}}w\ \ \ \ \ \ \ \ \ \ \text{and}\\
\nabla_{V,U}^{-\operatorname*{Col}}  &  :=\sum_{\substack{w\in S_{n}%
;\\wU\text{ is column-equivalent to }V}}\left(  -1\right)  ^{w}%
w\ \ \ \ \ \ \ \ \ \ \text{of }\mathcal{A}.
\end{align*}
An $n$-bitableau $\left(  \lambda,U,V\right)  $ is said to be \emph{standard}
if both tableaux $U$ and $V$ are standard. We let $\operatorname*{SBT}\left(
n\right)  $ be the set of all standard $n$-bitableaux, while
$\operatorname*{BT}\left(  n\right)  $ denotes the set of all $n$-bitableaux
(standard or not). Then, both families%
\[
\left(  \nabla_{V,U}^{\operatorname*{Row}}\right)  _{\left(  \lambda
,U,V\right)  \in\operatorname*{SBT}\left(  n\right)  }%
\ \ \ \ \ \ \ \ \ \ \text{and}\ \ \ \ \ \ \ \ \ \ \left(  \nabla
_{V,U}^{-\operatorname*{Col}}\right)  _{\left(  \lambda,U,V\right)
\in\operatorname*{SBT}\left(  n\right)  }%
\]
are bases of $\mathcal{A}$ (see \cite[Theorems 3.12, 3.14 and 6.13]{CanWil89}
or \cite[Theorem 6.8.14]{sga}). These are the so-called \emph{Murphy cellular
bases}.

These bases have been used to define ideals (left, right, two-sided) of
$\mathcal{A}$. In particular, for any $k\in\mathbb{N}$, we can define the four
spans%
\begin{align*}
\MurpF_{\operatorname*{std},\operatorname*{len}\leq k}^{\operatorname*{Row}}
&  :=\operatorname*{span}\left\{  \nabla_{V,U}^{\operatorname*{Row}}%
\ \mid\ \left(  \lambda,U,V\right)  \in\operatorname*{SBT}\left(  n\right)
\text{ and }\ell\left(  \lambda\right)  \leq k\right\}  ;\\
\MurpF_{\operatorname*{std},\operatorname*{len}>k}^{-\operatorname*{Col}}  &
:=\operatorname*{span}\left\{  \nabla_{V,U}^{-\operatorname*{Col}}%
\ \mid\ \left(  \lambda,U,V\right)  \in\operatorname*{SBT}\left(  n\right)
\text{ and }\ell\left(  \lambda\right)  >k\right\}  ;\\
\MurpF_{\operatorname*{all},\operatorname*{len}\leq k}^{\operatorname*{Row}}
&  :=\operatorname*{span}\left\{  \nabla_{V,U}^{\operatorname*{Row}}%
\ \mid\ \left(  \lambda,U,V\right)  \in\operatorname*{BT}\left(  n\right)
\text{ and }\ell\left(  \lambda\right)  \leq k\right\}  ;\\
\MurpF_{\operatorname*{all},\operatorname*{len}>k}^{-\operatorname*{Col}}  &
:=\operatorname*{span}\left\{  \nabla_{V,U}^{-\operatorname*{Col}}%
\ \mid\ \left(  \lambda,U,V\right)  \in\operatorname*{BT}\left(  n\right)
\text{ and }\ell\left(  \lambda\right)  >k\right\}  .
\end{align*}
According to \cite[Theorem 6.8.47]{sga}, these are actually just two spans: we
have%
\[
\MurpF_{\operatorname*{std},\operatorname*{len}\leq k}^{\operatorname*{Row}%
}=\MurpF_{\operatorname*{all},\operatorname*{len}\leq k}^{\operatorname*{Row}%
}=\left(  \MurpF_{\operatorname*{std},\operatorname*{len}>k}%
^{-\operatorname*{Col}}\right)  ^{\perp}%
\]
and
\[
\MurpF_{\operatorname*{std},\operatorname*{len}>k}^{-\operatorname*{Col}%
}=\MurpF_{\operatorname*{all},\operatorname*{len}>k}^{-\operatorname*{Col}%
}=\left(  \MurpF_{\operatorname*{std},\operatorname*{len}\leq k}%
^{\operatorname*{Row}}\right)  ^{\perp},
\]
and furthermore, both spans $\MurpF_{\operatorname*{std},\operatorname*{len}%
\leq k}^{\operatorname*{Row}}=\MurpF_{\operatorname*{all},\operatorname*{len}%
\leq k}^{\operatorname*{Row}}$ and $\MurpF_{\operatorname*{std}%
,\operatorname*{len}>k}^{-\operatorname*{Col}}=\MurpF_{\operatorname*{all}%
,\operatorname*{len}>k}^{-\operatorname*{Col}}$ are two-sided ideals of
$\mathcal{A}$.

Now, we claim that these ideals $\MurpF_{\operatorname*{std}%
,\operatorname*{len}\leq k}^{\operatorname*{Row}}$ and
$\MurpF_{\operatorname*{std},\operatorname*{len}>k}^{-\operatorname*{Col}}$
are precisely our $\mathcal{I}_{k}$ and $\mathcal{J}_{k}$, respectively. Indeed:

\begin{itemize}
\item Our ideal $\mathcal{I}_{k}$ is spanned by the row-to-row sums of the
form $\nabla_{\mathbf{B},\mathbf{A}}$, where $\mathbf{A}=\left(  A_{1}%
,A_{2},\ldots,A_{p}\right)  $ and $\mathbf{B}=\left(  B_{1},B_{2},\ldots
,B_{p}\right)  $ are set compositions of $\left[  n\right]  $ satisfying
$\ell\left(  \mathbf{A}\right)  =\ell\left(  \mathbf{B}\right)  \leq k$. By
Proposition \ref{prop.row.simple} \textbf{(a)}, we can restrict ourselves to
those pairs $\left(  \mathbf{A},\mathbf{B}\right)  $ of set compositions that
satisfy $\left\vert A_{i}\right\vert =\left\vert B_{i}\right\vert $ for all
$i$ (since all other pairs produce row-to-row sums equal to $0$). Such pairs
$\left(  \mathbf{A},\mathbf{B}\right)  $ can be encoded as pairs of
\textquotedblleft weird-shaped\textquotedblright\ tableaux $A,B$, both of
shape $\alpha$, where $\alpha=\left(  \left\vert A_{1}\right\vert ,\left\vert
A_{2}\right\vert ,\ldots,\left\vert A_{p}\right\vert \right)  $ is the list of
sizes of the blocks of $\mathbf{A}$ (or equivalently $\mathbf{B}$), in the
simplest possible way: Let%
\begin{align*}
\left(  i\text{-th row of }A\right)   &  =\left(  \text{list of elements of
}A_{i}\text{ (in increasing order)}\right)  \ \ \ \ \ \ \ \ \ \ \text{and}\\
\left(  i\text{-th row of }B\right)   &  =\left(  \text{list of elements of
}B_{i}\text{ (in increasing order)}\right)
\end{align*}
for each $i\in\left[  p\right]  $. We can furthermore find a permutation
$\sigma\in S_{p}$ such that $\left\vert A_{\sigma\left(  1\right)
}\right\vert \geq\left\vert A_{\sigma\left(  2\right)  }\right\vert \geq
\cdots\geq\left\vert A_{\sigma\left(  p\right)  }\right\vert $. By permuting
the blocks of $\mathbf{A}$ and of $\mathbf{B}$ using this permutation
$\sigma\in S_{p}$ (so that $\nabla_{\mathbf{B},\mathbf{A}}$ stays unchanged),
we can furthermore ensure that $\alpha$ is a partition (of length $\ell\left(
\alpha\right)  =p\leq k$). Then, $\left(  \alpha,A,B\right)  $ is an
$n$-bitableau (although usually not a standard one), and we have
$\nabla_{\mathbf{B},\mathbf{A}}=\nabla_{B,A}^{\operatorname*{Row}}$.

Hence, each nonzero row-to-row sum $\nabla_{\mathbf{B},\mathbf{A}}$ with
$\mathbf{A},\mathbf{B}\in\operatorname*{SC}\left(  n\right)  $ satisfying
$\ell\left(  \mathbf{A}\right)  =\ell\left(  \mathbf{B}\right)  \leq k$ can be
rewritten as $\nabla_{B,A}^{\operatorname*{Row}}$ for some $n$-bitableau
$\left(  \lambda,A,B\right)  \in\operatorname*{BT}\left(  n\right)  $ with
$\ell\left(  \lambda\right)  \leq k$. Conversely, any $\nabla_{B,A}%
^{\operatorname*{Row}}$ can be rewritten as $\nabla_{\mathbf{B},\mathbf{A}}$
by reversing the above construction (let $\mathbf{A}$ be the set composition
of $\left[  n\right]  $ whose $i$-th block is the set of all entries of the
$i$-th row of $A$, and likewise construct $\mathbf{B}$ from $B$). Thus,%
\begin{align*}
&  \left\{  \nabla_{\mathbf{B},\mathbf{A}}\ \mid\ \mathbf{A},\mathbf{B}%
\in\operatorname*{SC}\left(  n\right)  \text{ with }\ell\left(  \mathbf{A}%
\right)  =\ell\left(  \mathbf{B}\right)  \leq k\right\}  \setminus\left\{
0\right\} \\
&  =\left\{  \nabla_{B,A}^{\operatorname*{Row}}\ \mid\ \left(  \lambda
,A,B\right)  \in\operatorname*{BT}\left(  n\right)  \text{ with }\ell\left(
\lambda\right)  \leq k\right\}  \setminus\left\{  0\right\}  .
\end{align*}
Therefore,%
\begin{align}
\mathcal{I}_{k}  &  =\operatorname*{span}\left\{  \nabla_{\mathbf{B}%
,\mathbf{A}}\ \mid\ \mathbf{A},\mathbf{B}\in\operatorname*{SC}\left(
n\right)  \text{ with }\ell\left(  \mathbf{A}\right)  =\ell\left(
\mathbf{B}\right)  \leq k\right\} \nonumber\\
&  =\operatorname*{span}\left\{  \nabla_{B,A}^{\operatorname*{Row}}%
\ \mid\ \left(  \lambda,A,B\right)  \in\operatorname*{BT}\left(  n\right)
\text{ with }\ell\left(  \lambda\right)  \leq k\right\} \nonumber\\
&  =\operatorname*{span}\left\{  \nabla_{V,U}^{\operatorname*{Row}}%
\ \mid\ \left(  \lambda,U,V\right)  \in\operatorname*{BT}\left(  n\right)
\text{ and }\ell\left(  \lambda\right)  \leq k\right\} \nonumber\\
&  =\MurpF_{\operatorname*{all},\operatorname*{len}\leq k}%
^{\operatorname*{Row}}=\MurpF_{\operatorname*{std},\operatorname*{len}\leq
k}^{\operatorname*{Row}}. \label{eq.rmk.IJ=murphy.I}%
\end{align}

\item Define a $\mathbf{k}$-submodule $\mathcal{U}_{k}$ of $\mathcal{A}$ by%
\[
\mathcal{U}_{k}:=\operatorname*{span}\left\{  \nabla_{X}^{-}\ \mid\ X\text{ is
a subset of }\left[  n\right]  \text{ with size }\left\vert X\right\vert
>k\right\}  .
\]
Then, \cite[Proposition 6.8.53]{sga} says that
\[
\MurpF_{\operatorname*{std},\operatorname*{len}>k}^{-\operatorname*{Col}%
}=\mathcal{AU}_{k}=\mathcal{U}_{k}\mathcal{A}=\mathcal{AU}_{k}\mathcal{A}.
\]
But Proposition \ref{prop.IJ.1} \textbf{(e)} yields%
\begin{align*}
\mathcal{J}_{k}  &  =\mathcal{A}\cdot\operatorname*{span}\underbrace{\left\{
\nabla_{U}^{-}\ \mid\ U\text{ is a subset of }\left[  n\right]  \text{ having
size }>k\right\}  }_{=\left\{  \nabla_{X}^{-}\ \mid\ X\text{ is a subset of
}\left[  n\right]  \text{ with size }\left\vert X\right\vert >k\right\}
}\cdot\,\mathcal{A}\\
&  =\mathcal{A}\cdot\underbrace{\operatorname*{span}\left\{  \nabla_{X}%
^{-}\ \mid\ X\text{ is a subset of }\left[  n\right]  \text{ with size
}\left\vert X\right\vert >k\right\}  }_{=\mathcal{U}_{k}}\cdot\,\mathcal{A}\\
&  =\mathcal{AU}_{k}\mathcal{A}.
\end{align*}
Comparing these two equalities, we find%
\begin{equation}
\mathcal{J}_{k}=\MurpF_{\operatorname*{std},\operatorname*{len}>k}%
^{-\operatorname*{Col}}. \label{eq.rmk.IJ=murphy.J}%
\end{equation}

\end{itemize}

With the equalities (\ref{eq.rmk.IJ=murphy.I}) and (\ref{eq.rmk.IJ=murphy.J})
in hand, we can easily derive some parts of Theorem \ref{thm.row.main} from
known properties of the Murphy cellular bases. In particular,
(\ref{eq.rmk.IJ=murphy.I}) shows that%
\[
\mathcal{I}_{k}=\MurpF_{\operatorname*{std},\operatorname*{len}\leq
k}^{\operatorname*{Row}}=\operatorname*{span}\left\{  \nabla_{V,U}%
^{\operatorname*{Row}}\ \mid\ \left(  \lambda,U,V\right)  \in
\operatorname*{SBT}\left(  n\right)  \text{ and }\ell\left(  \lambda\right)
\leq k\right\}  ;
\]
thus, $\mathcal{I}_{k}$ is spanned by the family $\left(  \nabla
_{V,U}^{\operatorname*{Row}}\right)  _{\left(  \lambda,U,V\right)
\in\operatorname*{SBT}\left(  n\right)  \text{ and }\ell\left(  \lambda
\right)  \leq k}$, which is a subfamily of the Murphy cellular basis $\left(
\nabla_{V,U}^{\operatorname*{Row}}\right)  _{\left(  \lambda,U,V\right)
\in\operatorname*{SBT}\left(  n\right)  }$ of $\mathcal{A}$. Hence,
$\mathcal{I}_{k}$ is a direct addend of $\mathcal{A}$ as a $\mathbf{k}%
$-module, and both $\mathcal{I}_{k}$ and $\mathcal{A}/\mathcal{I}_{k}$ are
free $\mathbf{k}$-modules, with respective bases $\left(  \nabla
_{V,U}^{\operatorname*{Row}}\right)  _{\left(  \lambda,U,V\right)
\in\operatorname*{SBT}\left(  n\right)  \text{ and }\ell\left(  \lambda
\right)  \leq k}$ and $\left(  \overline{\nabla_{V,U}^{\operatorname*{Row}}%
}\right)  _{\left(  \lambda,U,V\right)  \in\operatorname*{SBT}\left(
n\right)  \text{ and }\ell\left(  \lambda\right)  >k}$. Likewise, we can get
analogous claims about $\mathcal{J}_{k}$ and $\mathcal{A}/\mathcal{J}_{k}$
from (\ref{eq.rmk.IJ=murphy.J}). Using (\ref{eq.rmk.IJ=murphy.I}) and
(\ref{eq.rmk.IJ=murphy.J}), we can furthermore rewrite the equality
$\MurpF_{\operatorname*{std},\operatorname*{len}\leq k}^{\operatorname*{Row}%
}=\left(  \MurpF_{\operatorname*{std},\operatorname*{len}>k}%
^{-\operatorname*{Col}}\right)  ^{\perp}$ (which is part of \cite[Theorem
6.8.47]{sga}) as $\mathcal{I}_{k}=\mathcal{J}_{k}^{\perp}$, and similarly we
obtain $\mathcal{J}_{k}=\mathcal{I}_{k}^{\perp}$.

However, not all of Theorem \ref{thm.row.main} follows this easily from this
approach. Anyway, we shall now take a more elementary point of view.
\end{remark}

\subsection{Lemmas for the proof of the main theorem}

First, however, we prepare for the proof of Theorem \ref{thm.row.main}. We
start with a litany of lemmas. Our first lemma is a collection of basic
properties of the bilinear form $\left\langle \cdot,\cdot\right\rangle $:

\begin{lemma}
\label{lem.coeff1-form}\ \ 

\begin{enumerate}
\item[\textbf{(a)}] Let $\operatorname*{coeff}\nolimits_{1}:\mathcal{A}%
\rightarrow\mathbf{k}$ be the map that sends each element $\mathbf{a}$ of
$\mathcal{A}=\mathbf{k}\left[  S_{n}\right]  $ to the coefficient of the
identity permutation $\operatorname*{id}\in S_{n}$ in $\mathbf{a}$. In other
words, $\operatorname*{coeff}\nolimits_{1}:\mathcal{A}\rightarrow\mathbf{k}$
is the $\mathbf{k}$-linear map that sends the permutation $\operatorname*{id}%
\in S_{n}$ to $1$ while sending any non-identity permutation $w\in S_{n}$ to
$0$.

Then, the bilinear form $\left\langle \cdot,\cdot\right\rangle $ can be
expressed as follows: For any $a,b\in\mathcal{A}$, we have
\begin{align}
\left\langle a,b\right\rangle  &  =\operatorname*{coeff}\nolimits_{1}\left(
S\left(  a\right)  b\right) \label{pf.lem.row.LAnn.scalp.Sab}\\
&  =\operatorname*{coeff}\nolimits_{1}\left(  bS\left(  a\right)  \right)
\label{pf.lem.row.LAnn.scalp.bSa}\\
&  =\operatorname*{coeff}\nolimits_{1}\left(  S\left(  b\right)  a\right)
\label{pf.lem.row.LAnn.scalp.Sba}\\
&  =\operatorname*{coeff}\nolimits_{1}\left(  aS\left(  b\right)  \right)  .
\label{pf.lem.row.LAnn.scalp.aSb}%
\end{align}

\item[\textbf{(b)}] The bilinear form $\left\langle \cdot,\cdot\right\rangle $
is $S$-invariant: That is, for all $a,b\in\mathcal{A}$, we have%
\begin{equation}
\left\langle a,b\right\rangle =\left\langle S\left(  a\right)  ,S\left(
b\right)  \right\rangle . \label{eq.lem.coeff1-form.b.1}%
\end{equation}

\item[\textbf{(c)}] Let $\mathcal{B}$ be any subset of $\mathcal{A}$. Then,
$\left(  S\left(  \mathcal{B}\right)  \right)  ^{\perp}=S\left(
\mathcal{B}^{\perp}\right)  $.
\end{enumerate}
\end{lemma}

\begin{proof}
\textbf{(a)} This is \cite[Proposition 6.8.17]{sga}. \medskip

\textbf{(b)} Let $a,b\in\mathcal{A}$. Then, $S\left(  S\left(  a\right)
\right)  =a$ (since $S\circ S=\operatorname*{id}$). But
(\ref{pf.lem.row.LAnn.scalp.Sab}) (applied to $S\left(  a\right)  $ and
$S\left(  b\right)  $ instead of $a$ and $b$) yields%
\[
\left\langle S\left(  a\right)  ,S\left(  b\right)  \right\rangle
=\operatorname*{coeff}\nolimits_{1}\left(  \underbrace{S\left(  S\left(
a\right)  \right)  }_{=a}S\left(  b\right)  \right)  =\operatorname*{coeff}%
\nolimits_{1}\left(  aS\left(  b\right)  \right)  =\left\langle
a,b\right\rangle
\]
(by (\ref{pf.lem.row.LAnn.scalp.aSb})). In other words, $\left\langle
a,b\right\rangle =\left\langle S\left(  a\right)  ,S\left(  b\right)
\right\rangle $. This proves (\ref{eq.lem.coeff1-form.b.1}) and thus Lemma
\ref{lem.coeff1-form} \textbf{(b)}. \medskip

\textbf{(c)} By the definition of $\left(  S\left(  \mathcal{B}\right)
\right)  ^{\perp}$, we have%
\begin{align}
\left(  S\left(  \mathcal{B}\right)  \right)  ^{\perp}  &  =\left\{
a\in\mathcal{A}\ \mid\ \left\langle a,b\right\rangle =0\text{ for all }b\in
S\left(  \mathcal{B}\right)  \right\} \nonumber\\
&  =\left\{  a\in\mathcal{A}\ \mid\ \left\langle a,c\right\rangle =0\text{ for
all }c\in S\left(  \mathcal{B}\right)  \right\} \nonumber\\
&  =\left\{  a\in\mathcal{A}\ \mid\ \left\langle a,S\left(  b\right)
\right\rangle =0\text{ for all }b\in\mathcal{B}\right\}
\label{pf.lem.coeff1-form.c.1}%
\end{align}
(since $S\left(  \mathcal{B}\right)  =\left\{  S\left(  b\right)  \ \mid
\ b\in\mathcal{B}\right\}  $). On the other hand,%
\begin{align*}
S\left(  \mathcal{B}^{\perp}\right)   &  =\left\{  S\left(  a\right)
\ \mid\ a\in\mathcal{B}^{\perp}\right\} \\
&  =\left\{  S\left(  a\right)  \ \mid\ a\in\mathcal{A}\text{ satisfying
}\left\langle a,b\right\rangle =0\text{ for all }b\in\mathcal{B}\right\} \\
&  \ \ \ \ \ \ \ \ \ \ \ \ \ \ \ \ \ \ \ \ \left(  \text{since }%
\mathcal{B}^{\perp}=\left\{  a\in\mathcal{A}\ \mid\ \left\langle
a,b\right\rangle =0\text{ for all }b\in\mathcal{B}\right\}  \right) \\
&  =\left\{  S\left(  a\right)  \ \mid\ a\in\mathcal{A}\text{ satisfying
}\left\langle S\left(  a\right)  ,S\left(  b\right)  \right\rangle =0\text{
for all }b\in\mathcal{B}\right\} \\
&  \ \ \ \ \ \ \ \ \ \ \ \ \ \ \ \ \ \ \ \ \left(  \text{since
(\ref{eq.lem.coeff1-form.b.1}) yields }\left\langle a,b\right\rangle
=\left\langle S\left(  a\right)  ,S\left(  b\right)  \right\rangle \right) \\
&  =\left\{  a\ \mid\ a\in\mathcal{A}\text{ satisfying }\left\langle
a,S\left(  b\right)  \right\rangle =0\text{ for all }b\in\mathcal{B}\right\}
\end{align*}
(here, we have substituted $a$ for $S\left(  a\right)  $ in our set, since
$S:\mathcal{A}\rightarrow\mathcal{A}$ is a bijection). In other words,%
\[
S\left(  \mathcal{B}^{\perp}\right)  =\left\{  a\in\mathcal{A}\ \mid
\ \left\langle a,S\left(  b\right)  \right\rangle =0\text{ for all }%
b\in\mathcal{B}\right\}  .
\]
Comparing this with (\ref{pf.lem.coeff1-form.c.1}), we obtain $\left(
S\left(  \mathcal{B}\right)  \right)  ^{\perp}=S\left(  \mathcal{B}^{\perp
}\right)  $. This proves Lemma \ref{lem.coeff1-form} \textbf{(c)}.
\end{proof}

\begin{verlong}

\begin{proof}

\item[Proof of Lemma \ref{lem.coeff1-form} \textbf{(a)}.] Let $a,b\in
\mathcal{A}$. We must then prove the four equalities
(\ref{pf.lem.row.LAnn.scalp.Sab}), (\ref{pf.lem.row.LAnn.scalp.bSa}),
(\ref{pf.lem.row.LAnn.scalp.Sba}) and (\ref{pf.lem.row.LAnn.scalp.aSb}). Each
of these four equalities is $\mathbf{k}$-linear in each of $a$ and $b$. Thus,
we WLOG assume that both $a$ and $b$ belong to the basis $S_{n}$ of
$\mathcal{A}$. In other words, $a$ and $b$ are permutations in $S_{n}$. Then,%
\[
S\left(  a\right)  =a^{-1}\ \ \ \ \ \ \ \ \ \ \text{and}%
\ \ \ \ \ \ \ \ \ \ S\left(  b\right)  =b^{-1}\ \ \ \ \ \ \ \ \ \ \text{and}%
\ \ \ \ \ \ \ \ \ \ \left\langle a,b\right\rangle =%
\begin{cases}
1, & \text{if }a=b;\\
0, & \text{if }a\neq b.
\end{cases}
\]
Hence,
\begin{align*}
\operatorname*{coeff}\nolimits_{1}\left(  S\left(  a\right)  b\right)   &
=\operatorname*{coeff}\nolimits_{1}\left(  a^{-1}b\right)  =%
\begin{cases}
1, & \text{if }a^{-1}b=\operatorname*{id};\\
0, & \text{if }a^{-1}b\neq\operatorname*{id}%
\end{cases}
\ \ \ \ \ \ \ \ \ \ \left(  \text{since }a^{-1}b\in S_{n}\right) \\
&  =%
\begin{cases}
1, & \text{if }a=b;\\
0, & \text{if }a\neq b
\end{cases}
\ \ \ \ \ \ \ \ \ \ \left(  \text{since }a^{-1}b=\operatorname*{id}\text{ is
equivalent to }a=b\right) \\
&  =\left\langle a,b\right\rangle .
\end{align*}
Thus, (\ref{pf.lem.row.LAnn.scalp.Sab}) is proved. The other three equalities
are proved similarly.
\end{proof}
\end{verlong}

The next lemma connects orthogonal spaces with left/right annihilators:

\begin{lemma}
\label{lem.row.LAnn}Let $\mathcal{B}$ be a left ideal of $\mathcal{A}$. Then:

\begin{enumerate}
\item[\textbf{(a)}] We have $\mathcal{B}^{\perp}=\operatorname*{LAnn}\left(
S\left(  \mathcal{B}\right)  \right)  $.

\item[\textbf{(b)}] If $S\left(  \mathcal{B}\right)  =\mathcal{B}$, then
$\mathcal{B}^{\perp}=\operatorname*{LAnn}\mathcal{B}=\operatorname*{RAnn}%
\mathcal{B}$.
\end{enumerate}
\end{lemma}

\begin{proof}
\textbf{(a)} Define the $\mathbf{k}$-linear map $\operatorname*{coeff}%
\nolimits_{1}:\mathcal{A}\rightarrow\mathbf{k}$ as in Lemma
\ref{lem.coeff1-form} \textbf{(a)}. It is easy to see that each $a\in
\mathcal{A}$ and each $w\in S_{n}$ satisfy the equality%
\begin{equation}
\operatorname*{coeff}\nolimits_{1}\left(  aw^{-1}\right)  =\left(  \text{the
coefficient of }w\text{ in }a\right)  \label{pf.lem.row.LAnn.a.1}%
\end{equation}
(indeed, if we write $a$ as $\sum_{u\in S_{n}}\alpha_{u}u$ with $\alpha_{u}%
\in\mathbf{k}$, then both sides of this equality equal $\alpha_{w}$).

Let $a\in\operatorname*{LAnn}\left(  S\left(  \mathcal{B}\right)  \right)  $.
Then, $aS\left(  b\right)  =0$ for all $b\in\mathcal{B}$. Hence,
(\ref{pf.lem.row.LAnn.scalp.aSb}) yields $\left\langle a,b\right\rangle
=\operatorname*{coeff}\nolimits_{1}\left(  \underbrace{aS\left(  b\right)
}_{=0}\right)  =\operatorname*{coeff}\nolimits_{1}0=0$ for all $b\in
\mathcal{B}$. In other words, $a\in\mathcal{B}^{\perp}$. Thus, we have shown
that $\operatorname*{LAnn}\left(  S\left(  \mathcal{B}\right)  \right)
\subseteq\mathcal{B}^{\perp}$.

Conversely, let $c\in\mathcal{B}^{\perp}$. Then, $\left\langle
c,b\right\rangle =0$ for all $b\in\mathcal{B}$. Now, let $b\in\mathcal{B}$ be
arbitrary. Then, for every $w\in S_{n}$, we have $wb\in\mathcal{B}$ (since
$\mathcal{B}$ is a left ideal of $\mathcal{A}$) and therefore $\left\langle
c,wb\right\rangle =0$ (since $c\in\mathcal{B}^{\perp}$), so that%
\begin{align*}
0  &  =\left\langle c,wb\right\rangle =\operatorname*{coeff}\nolimits_{1}%
\left(  cS\left(  wb\right)  \right)  \ \ \ \ \ \ \ \ \ \ \left(  \text{by
(\ref{pf.lem.row.LAnn.scalp.aSb})}\right) \\
&  =\operatorname*{coeff}\nolimits_{1}\left(  cS\left(  b\right)
w^{-1}\right)  \ \ \ \ \ \ \ \ \ \ \left(  \text{since }S\left(  wb\right)
=S\left(  b\right)  \underbrace{S\left(  w\right)  }_{=w^{-1}}=S\left(
b\right)  w^{-1}\right) \\
&  =\left(  \text{the coefficient of }w\text{ in }cS\left(  b\right)  \right)
\ \ \ \ \ \ \ \ \ \ \left(  \text{by (\ref{pf.lem.row.LAnn.a.1}), applied to
}a=cS\left(  b\right)  \right)  .
\end{align*}
Since this holds for each $w\in S_{n}$, we thus have shown that all
coefficients of $cS\left(  b\right)  $ equal $0$. In other words, $cS\left(
b\right)  =0$. Since this holds for each $b\in\mathcal{B}$, we conclude that
$c\in\operatorname*{LAnn}\left(  S\left(  \mathcal{B}\right)  \right)  $.
Thus, we have shown that $\mathcal{B}^{\perp}\subseteq\operatorname*{LAnn}%
\left(  S\left(  \mathcal{B}\right)  \right)  $. Combining this with
$\operatorname*{LAnn}\left(  S\left(  \mathcal{B}\right)  \right)
\subseteq\mathcal{B}^{\perp}$, we obtain $\mathcal{B}^{\perp}%
=\operatorname*{LAnn}\left(  S\left(  \mathcal{B}\right)  \right)  $. Thus,
Lemma \ref{lem.row.LAnn} \textbf{(a)} is proved. \medskip

\textbf{(b)} Assume that $S\left(  \mathcal{B}\right)  =\mathcal{B}$. Then,
Lemma \ref{lem.row.LAnn} \textbf{(a)} yields $\mathcal{B}^{\perp
}=\operatorname*{LAnn}\left(  S\left(  \mathcal{B}\right)  \right)
=\operatorname*{LAnn}\mathcal{B}$ (since $S\left(  \mathcal{B}\right)
=\mathcal{B}$). Furthermore, Lemma \ref{lem.coeff1-form} \textbf{(c)} yields
$\left(  S\left(  \mathcal{B}\right)  \right)  ^{\perp}=S\left(
\mathcal{B}^{\perp}\right)  $. In view of $S\left(  \mathcal{B}\right)
=\mathcal{B}$, we can rewrite this as $\mathcal{B}^{\perp}=S\left(
\mathcal{B}^{\perp}\right)  $. However, $S$ is a $\mathbf{k}$-algebra
anti-automorphism. Thus, $\operatorname*{RAnn}\left(  S\left(  \mathcal{B}%
\right)  \right)  =S\left(  \operatorname*{LAnn}\mathcal{B}\right)  $. In view
of $S\left(  \mathcal{B}\right)  =\mathcal{B}$, we can rewrite this as
$\operatorname*{RAnn}\mathcal{B}=S\left(  \underbrace{\operatorname*{LAnn}%
\mathcal{B}}_{=\mathcal{B}^{\perp}}\right)  =S\left(  \mathcal{B}^{\perp
}\right)  =\mathcal{B}^{\perp}$. Thus, $\mathcal{B}^{\perp}%
=\operatorname*{RAnn}\mathcal{B}$. Combined with $\mathcal{B}^{\perp
}=\operatorname*{LAnn}\mathcal{B}$, this completes the proof of Lemma
\ref{lem.row.LAnn} \textbf{(b)}.
\end{proof}

\begin{lemma}
\label{lem.mod.sub-basis}Let $\mathcal{M}$ be a free $\mathbf{k}$-module with
a basis $\left(  m_{i}\right)  _{i\in I}$. Let $J$ and $K$ be two disjoint
subsets of $I$ such that $J\cup K=I$. Let $\mathcal{N}$ be a $\mathbf{k}%
$-submodule of $\mathcal{M}$ such that the quotient module $\mathcal{M}%
/\mathcal{N}$ has a basis $\left(  \overline{m_{i}}\right)  _{i\in J}$. (Here,
as usual, $\overline{m}$ denotes the projection of any vector $m\in
\mathcal{M}$ onto the quotient $\mathcal{M}/\mathcal{N}$.)

Then:

\begin{enumerate}
\item[\textbf{(a)}] The $\mathbf{k}$-module $\mathcal{N}$ is free of rank
$\left\vert K\right\vert $.

\item[\textbf{(b)}] There exists a $\mathbf{k}$-linear projection
$\pi:\mathcal{M}\rightarrow\mathcal{N}$ (that is, a $\mathbf{k}$-linear map
$\pi:\mathcal{M}\rightarrow\mathcal{N}$ such that $\left.  \pi\mid
_{\mathcal{N}}\right.  =\operatorname*{id}$).
\end{enumerate}
\end{lemma}

\begin{proof}
For each $k\in K$, the vector $\overline{m_{k}}\in\mathcal{M}/\mathcal{N}$ can
be written as a $\mathbf{k}$-linear combination of the family $\left(
\overline{m_{i}}\right)  _{i\in J}$ (since this family is a basis of
$\mathcal{M}/\mathcal{N}$). In other words, there exist coefficients $c_{k,j}$
for all $k\in K$ and $j\in J$ such that each $k\in K$ satisfies%
\begin{equation}
\overline{m_{k}}=\sum_{j\in J}c_{k,j}\overline{m_{j}}.
\label{pf.lem.mod.sub-basis.1}%
\end{equation}
Consider these $c_{k,j}$. Now, let us set%
\begin{equation}
v_{k}:=m_{k}-\sum_{j\in J}c_{k,j}m_{j} \label{pf.lem.mod.sub-basis.2}%
\end{equation}
for each $k\in K$. This element $v_{k}$ belongs to $\mathcal{N}$ (since
$\overline{v_{k}}=\overline{m_{k}-\sum_{j\in J}c_{k,j}m_{j}}=\overline{m_{k}%
}-\sum_{j\in J}c_{k,j}\overline{m_{j}}=\overline{0}$ by
(\ref{pf.lem.mod.sub-basis.1})). Thus, $\left(  v_{k}\right)  _{k\in K}$ is a
family of vectors in $\mathcal{N}$. This family is easily seen to be
$\mathbf{k}$-linearly independent\footnote{\textit{Proof.} Let $\left(
\lambda_{k}\right)  _{k\in K}\in\mathbf{k}^{K}$ be a family of scalars such
that $\sum_{k\in K}\lambda_{k}v_{k}=0$. We must show that all $\lambda_{k}$
are $0$.
\par
We have%
\begin{align}
0  &  =\sum_{k\in K}\lambda_{k}v_{k}=\sum_{k\in K}\lambda_{k}\left(
m_{k}-\sum_{j\in J}c_{k,j}m_{j}\right)  \ \ \ \ \ \ \ \ \ \ \left(  \text{by
(\ref{pf.lem.mod.sub-basis.2})}\right) \nonumber\\
&  =\sum_{k\in K}\lambda_{k}m_{k}-\sum_{k\in K}\lambda_{k}\sum_{j\in J}%
c_{k,j}m_{j}\nonumber\\
&  =\sum_{k\in K}\lambda_{k}m_{k}-\sum_{j\in J}\left(  \sum_{k\in K}%
\lambda_{k}c_{k,j}\right)  m_{j}. \label{pf.lem.mod.sub-basis.li.pf.1}%
\end{align}
Note that there is no overlap between the $m_{k}$'s in the first sum here and
the $m_{j}$'s in the second sum, since the sets $J$ and $K$ are disjoint.
Since the family $\left(  m_{i}\right)  _{i\in I}$ is $\mathbf{k}$-linearly
independent (being a basis of $\mathcal{M}$), we thus conclude from
(\ref{pf.lem.mod.sub-basis.li.pf.1}) that all the coefficients $\lambda_{k}$
and $\sum_{k\in K}\lambda_{k}c_{k,j}$ are $0$. Hence, in particular, all
$\lambda_{k}$ are $0$. Qed.}. Moreover, it spans the $\mathbf{k}$-module
$\mathcal{N}$ (this, too, is not hard to see\footnote{\textit{Proof.} Let
$w\in\mathcal{N}$. We must show that $w\in\operatorname*{span}\left\{
v_{k}\mid k\in K\right\}  $.
\par
First, we expand $w\in\mathcal{N}\subseteq\mathcal{M}$ in the basis $\left(
m_{i}\right)  _{i\in I}$ of $\mathcal{M}$ as follows:%
\[
w=\sum_{i\in I}w_{i}m_{i}\ \ \ \ \ \ \ \ \ \ \text{with }w_{i}\in\mathbf{k}.
\]
Since the set $I$ is the union of its disjoint subsets $J$ and $K$, we can
break this sum up as follows:%
\begin{align}
w  &  =\sum_{i\in J}w_{i}m_{i}+\sum_{i\in K}w_{i}m_{i}=\sum_{j\in J}w_{j}%
m_{j}+\sum_{k\in K}w_{k}\underbrace{m_{k}}_{\substack{=\sum_{j\in J}%
c_{k,j}m_{j}+v_{k}\\\text{(by (\ref{pf.lem.mod.sub-basis.2}))}}}\nonumber\\
&  =\sum_{j\in J}w_{j}m_{j}+\underbrace{\sum_{k\in K}w_{k}\left(  \sum_{j\in
J}c_{k,j}m_{j}+v_{k}\right)  }_{=\sum_{j\in J}\ \ \sum_{k\in K}w_{k}%
c_{k,j}m_{j}+\sum_{k\in K}w_{k}v_{k}}=\underbrace{\sum_{j\in J}w_{j}m_{j}%
+\sum_{j\in J}\ \ \sum_{k\in K}w_{k}c_{k,j}m_{j}}_{=\sum_{j\in J}\left(
w_{j}+\sum_{k\in K}w_{k}c_{k,j}\right)  m_{j}}+\sum_{k\in K}w_{k}%
v_{k}\nonumber\\
&  =\sum_{j\in J}\left(  w_{j}+\sum_{k\in K}w_{k}c_{k,j}\right)  m_{j}%
+\sum_{k\in K}w_{k}v_{k}. \label{pf.lem.mod.sub-basis.sp.pf.1}%
\end{align}
Projecting this onto $\mathcal{M}/\mathcal{N}$, we obtain%
\begin{align*}
\overline{w}  &  =\overline{\sum_{j\in J}\left(  w_{j}+\sum_{k\in K}%
w_{k}c_{k,j}\right)  m_{j}+\sum_{k\in K}w_{k}v_{k}}\\
&  =\sum_{j\in J}\left(  w_{j}+\sum_{k\in K}w_{k}c_{k,j}\right)
\overline{m_{j}}+\sum_{k\in K}w_{k}\underbrace{\overline{v_{k}}}%
_{\substack{=0\\\text{(since }v_{k}\in\mathcal{N}\text{)}}}=\sum_{j\in
J}\left(  w_{j}+\sum_{k\in K}w_{k}c_{k,j}\right)  \overline{m_{j}}.
\end{align*}
Hence,%
\[
\sum_{j\in J}\left(  w_{j}+\sum_{k\in K}w_{k}c_{k,j}\right)  \overline{m_{j}%
}=\overline{w}=0\ \ \ \ \ \ \ \ \ \ \left(  \text{since }w\in\mathcal{N}%
\right)  .
\]
Since the family $\left(  \overline{m_{i}}\right)  _{i\in J}$ is $\mathbf{k}%
$-linearly independent (being a basis of $\mathcal{N}$), we thus conclude that
all the coefficients $w_{j}+\sum_{k\in K}w_{k}c_{k,j}$ here are $0$. Thus,
(\ref{pf.lem.mod.sub-basis.sp.pf.1}) becomes
\[
w=\sum_{j\in J}\underbrace{\left(  w_{j}+\sum_{k\in K}w_{k}c_{k,j}\right)
}_{=0}m_{j}+\sum_{k\in K}w_{k}v_{k}=\sum_{k\in K}w_{k}v_{k}\in
\operatorname*{span}\left\{  v_{k}\mid k\in K\right\}  ,
\]
qed.}). Thus, the family $\left(  v_{k}\right)  _{k\in K}$ is a basis of
$\mathcal{N}$. Therefore, $\mathcal{N}$ is free of rank $\left\vert
K\right\vert $. This proves Lemma \ref{lem.mod.sub-basis} \textbf{(a)}.
\medskip

\textbf{(b)} Recall that $v_{k}\in\mathcal{N}$ for each $k\in K$. Let
$\pi:\mathcal{M}\rightarrow\mathcal{N}$ be the $\mathbf{k}$-linear map that
sends each basis element $m_{i}$ of $\mathcal{M}$ to
\[%
\begin{cases}
v_{i}, & \text{if }i\in K;\\
0, & \text{if }i\in J.
\end{cases}
\]
This is well-defined (since $\left(  m_{i}\right)  _{i\in I}$ is a basis of
$\mathcal{M}$, and since each $i\in I$ belongs to exactly one of the sets $K$
and $J$). It is easy to see that this map $\pi$ sends $v_{k}$ to $v_{k}$ for
each $k\in K$ (because applying $\pi$ to the right hand side of
(\ref{pf.lem.mod.sub-basis.2}) kills all $m_{j}$ with $j\in J$ but sends
$m_{k}$ to $v_{k}$). Thus, by linearity, we obtain $\left.  \pi\mid
_{\mathcal{N}}\right.  =\operatorname*{id}$ (since $\left(  v_{k}\right)
_{k\in K}$ is a basis of $\mathcal{N}$). Thus, $\pi$ is a projection. This
proves Lemma \ref{lem.mod.sub-basis} \textbf{(b)}.
\end{proof}

\begin{lemma}
\label{lem.row.IJ=0}Let $k\in\mathbb{N}$. Then, $\mathcal{I}_{k}%
\mathcal{J}_{k}=\mathcal{J}_{k}\mathcal{I}_{k}=0$.
\end{lemma}

\begin{proof}
Let us first show that $\mathcal{I}_{k}\mathcal{J}_{k}=0$. Indeed, we have%
\[
\mathcal{I}_{k}=\operatorname*{span}\left\{  \nabla_{\mathbf{B},\mathbf{A}%
}\ \mid\ \mathbf{A},\mathbf{B}\in\operatorname*{SC}\left(  n\right)  \text{
with }\ell\left(  \mathbf{A}\right)  =\ell\left(  \mathbf{B}\right)  \leq
k\right\}
\]
and%
\[
\mathcal{J}_{k}=\operatorname*{span}\left\{  \nabla_{U}^{-}\ \mid\ U\text{ is
a subset of }\left[  n\right]  \text{ having size }k+1\right\}  \cdot
\mathcal{A}%
\]
(by Proposition \ref{prop.IJ.1} \textbf{(b)}). Thus, in order to prove that
$\mathcal{I}_{k}\mathcal{J}_{k}=0$, it suffices to show that $\nabla
_{\mathbf{B},\mathbf{A}}\nabla_{U}^{-}=0$ for all set compositions
$\mathbf{A},\mathbf{B}\in\operatorname*{SC}\left(  n\right)  $ satisfying
$\ell\left(  \mathbf{A}\right)  =\ell\left(  \mathbf{B}\right)  \leq k$ and
all subsets $U$ of $\left[  n\right]  $ having size $k+1$. So let us show
this. We fix two set compositions $\mathbf{A},\mathbf{B}\in\operatorname*{SC}%
\left(  n\right)  $ satisfying $\ell\left(  \mathbf{A}\right)  =\ell\left(
\mathbf{B}\right)  \leq k$ and a subset $U$ of $\left[  n\right]  $ having
size $k+1$. We must show that $\nabla_{\mathbf{B},\mathbf{A}}\nabla_{U}^{-}=0$.

\begin{noncompile}
Write the set compositions $\mathbf{A}$ and $\mathbf{B}$ as $\mathbf{A}%
=\left(  A_{1},A_{2},\ldots,A_{p}\right)  $ and $\mathbf{B}=\left(
B_{1},B_{2},\ldots,B_{p}\right)  $, where $p=\ell\left(  \mathbf{A}\right)
=\ell\left(  \mathbf{B}\right)  $. Then, $p=\ell\left(  \mathbf{A}\right)
=\ell\left(  \mathbf{B}\right)  \leq k<k+1=\left\vert U\right\vert $.
\end{noncompile}

We have $\ell\left(  \mathbf{A}\right)  =\ell\left(  \mathbf{B}\right)  \leq
k<k+1=\left\vert U\right\vert $ (since $U$ has size $k+1$). In other words,
$\mathbf{A}$ has fewer blocks than $U$ has elements. Hence, by the pigeonhole
principle, there exist two distinct elements $u$ and $v$ of $U$ that belong to
the same block of $\mathbf{A}$. Pick such $u$ and $v$. Let $\tau\in S_{U}$ be
the transposition that swaps $u$ and $v$. Then, $\nabla_{U}^{-}=\left(
1-\tau\right)  q$ for some $q\in\mathbf{k}\left[  S_{U}\right]  $ (since
$\left\langle \tau\right\rangle =\left\{  1,\tau\right\}  $ is a subgroup of
the group $S_{U}$ and since the permutation $\tau$ is odd\footnote{In more
details: We have $\nabla_{U}^{-}=\left(  1-\tau\right)  \nabla_{U}%
^{\operatorname*{even}}$, where $\nabla_{U}^{\operatorname*{even}}%
=\sum_{\substack{w\in S_{U};\\\left(  -1\right)  ^{w}=1}}w$. This follows from
\cite[Proposition 3.7.4 \textbf{(d)}]{sga}, applied to $X=U$ and $i=u$ and
$j=v$ (our $\tau$ is thus the $t_{i,j}$ of \cite[Proposition 3.7.4
\textbf{(d)}]{sga}).}). Consider this $q$. But (\ref{eq.uNabv}) (applied to
$u=\operatorname*{id}$ and $v=\tau$) yields%
\[
\nabla_{\mathbf{B},\mathbf{A}}\tau=\nabla_{\mathbf{B},\tau^{-1}\mathbf{A}%
}=\nabla_{\mathbf{B},\mathbf{A}}%
\]
(since $\tau^{-1}\mathbf{A}=\mathbf{A}$ (because $u$ and $v$ belong to the
same block of $\mathbf{A}$, and thus the transposition $\tau$ preserves each
block of $\mathbf{A}$)). Hence,%
\[
\nabla_{\mathbf{B},\mathbf{A}}\underbrace{\nabla_{U}^{-}}_{=\left(
1-\tau\right)  q}=\underbrace{\nabla_{\mathbf{B},\mathbf{A}}\left(
1-\tau\right)  }_{\substack{=\nabla_{\mathbf{B},\mathbf{A}}-\nabla
_{\mathbf{B},\mathbf{A}}\tau\\=0\\\text{(since }\nabla_{\mathbf{B},\mathbf{A}%
}\tau=\nabla_{\mathbf{B},\mathbf{A}}\text{)}}}q=0.
\]
This completes our proof of $\mathcal{I}_{k}\mathcal{J}_{k}=0$.

It remains to prove that $\mathcal{J}_{k}\mathcal{I}_{k}=0$. This can be done
similarly, but can also be derived from $\mathcal{I}_{k}\mathcal{J}_{k}=0$
easily: Since $S$ is an algebra anti-automorphism, we have%
\[
S\left(  \mathcal{I}_{k}\mathcal{J}_{k}\right)  =\underbrace{S\left(
\mathcal{J}_{k}\right)  }_{\substack{=\mathcal{J}_{k}\\\text{(by Proposition
\ref{prop.IJ.1} \textbf{(c)})}}}\underbrace{S\left(  \mathcal{I}_{k}\right)
}_{\substack{=\mathcal{I}_{k}\\\text{(by Proposition \ref{prop.IJ.1}
\textbf{(c)})}}}=\mathcal{J}_{k}\mathcal{I}_{k}.
\]
Thus, $\mathcal{J}_{k}\mathcal{I}_{k}=S\left(  \mathcal{I}_{k}\mathcal{J}%
_{k}\right)  =0$ (since $\mathcal{I}_{k}\mathcal{J}_{k}=0$). The proof of
Lemma \ref{lem.row.IJ=0} is thus complete.
\end{proof}

The following purely combinatorial lemma is a variant of the
Erd\"{o}s--Szekeres theorem:

\begin{lemma}
\label{lem.erdos-szek-var}Let $k\in\mathbb{N}$. Let $v\in\operatorname*{Av}%
\nolimits_{n}\left(  k+1\right)  $. Then, there exists a set decomposition
$\mathbf{A}=\left(  A_{1},A_{2},\ldots,A_{k}\right)  $ of $\left[  n\right]  $
such that all restrictions $v\mid_{A_{1}},\ v\mid_{A_{2}},\ \ldots
,\ v\mid_{A_{k}}$ are decreasing.
\end{lemma}

\begin{proof}
Let $\mathbf{v}$ be the sequence $\left(  v\left(  1\right)  ,v\left(
2\right)  ,\ldots,v\left(  n\right)  \right)  $. Then, $\mathbf{v}$ has no
increasing subsequence of length $k+1$ (because $v\in\operatorname*{Av}%
\nolimits_{n}\left(  k+1\right)  $). Thus, $\mathbf{v}$ has no increasing
subsequence of length $>k$. Hence, the length of any nonempty increasing
subsequence of $\mathbf{v}$ must be some $i\in\left[  k\right]  $.

For each $i\in\left[  k\right]  $, we define a set%
\begin{align*}
A_{i}  &  :=\left\{  j\in\left[  n\right]  \ \mid\ \text{the longest
increasing subsequence of }\mathbf{v}\right. \\
&  \ \ \ \ \ \ \ \ \ \ \ \ \ \ \ \ \ \ \ \ \ \ \ \ \ \ \ \ \ \ \left.
\text{ending with }v\left(  j\right)  \text{ has length }i\right\}  .
\end{align*}
These $k$ sets $A_{1},A_{2},\ldots,A_{k}$ are clearly disjoint, and their
union is $\left[  n\right]  $ (since the length of any nonempty increasing
subsequence of $\mathbf{v}$ must be some $i\in\left[  k\right]  $). In other
words, $\mathbf{A}:=\left(  A_{1},A_{2},\ldots,A_{k}\right)  $ is a set
decomposition of $\left[  n\right]  $. It remains to show that all
restrictions $v\mid_{A_{1}},\ v\mid_{A_{2}},\ \ldots,\ v\mid_{A_{k}}$ are decreasing.

To do this, we assume the contrary. Thus, there exists some $i\in\left[
k\right]  $ such that $v\mid_{A_{i}}$ is not decreasing. Consider this $i$.
Then, there exist two elements $p<q$ of $A_{i}$ such that $v\left(  p\right)
<v\left(  q\right)  $. Consider these $p$ and $q$. The longest increasing
subsequence of $\mathbf{v}$ ending with $v\left(  q\right)  $ has length $i$
(since $q\in A_{i}$). Now, consider the longest increasing subsequence of
$\mathbf{v}$ ending with $v\left(  p\right)  $. This subsequence, too, has
length $i$ (since $p\in A_{i}$), and thus can be written as $\left(  v\left(
p_{1}\right)  <v\left(  p_{2}\right)  <\cdots<v\left(  p_{i}\right)  \right)
$ with $p_{i}=p$ (since it ends with $v\left(  p\right)  $). By appending
$v\left(  q\right)  $ to it, we obtain an increasing subsequence $\left(
v\left(  p_{1}\right)  <v\left(  p_{2}\right)  <\cdots<v\left(  p_{i}\right)
<v\left(  q\right)  \right)  $ of $\mathbf{v}$ (since $p_{i}=p$ and thus
$v\left(  p_{i}\right)  =v\left(  p\right)  <v\left(  q\right)  $) that ends
with $v\left(  q\right)  $ and has length $i+1$. But this contradicts the fact
that the longest increasing subsequence of $\mathbf{v}$ ending with $v\left(
q\right)  $ has length $i$. This contradiction shows that our assumption was
false. Thus, all restrictions $v\mid_{A_{1}},\ v\mid_{A_{2}},\ \ldots
,\ v\mid_{A_{k}}$ are decreasing. This completes the proof of Lemma
\ref{lem.erdos-szek-var}.
\end{proof}

In the following few lemmas, we will use the lexicographic order on $S_{n}$.
This is a total order on $S_{n}$, defined by setting%
\begin{align*}
\left(  u<v\right)  \  &  \Longleftrightarrow\ \left(  \text{the smallest
}i\in\left[  n\right]  \text{ satisfying }u\left(  i\right)  \neq v\left(
i\right)  \right. \\
&  \ \ \ \ \ \ \ \ \ \ \ \ \ \ \ \ \ \ \ \ \left.  \text{satisfies }u\left(
i\right)  <v\left(  i\right)  \right)  \ \ \ \ \ \ \ \ \ \ \text{for all
}u,v\in S_{n}.
\end{align*}

\begin{lemma}
\label{lem.lex-less-1}Let $k\in\mathbb{N}$. Let $\mathbf{A}=\left(
A_{1},A_{2},\ldots,A_{k}\right)  $ be a set decomposition of $\left[
n\right]  $. Let $v\in S_{n}$ be a permutation such that all restrictions
$v\mid_{A_{1}},\ v\mid_{A_{2}},\ \ldots,\ v\mid_{A_{k}}$ are decreasing. Let
$w\in S_{n}$ be a further permutation such that all $i\in\left[  k\right]  $
satisfy $w\left(  A_{i}\right)  =v\left(  A_{i}\right)  $. Then, $w\leq v$ in
lexicographic order.
\end{lemma}

\begin{proof}
Intuitively, this is clear: The condition \textquotedblleft all $i\in\left[
k\right]  $ satisfy $w\left(  A_{i}\right)  =v\left(  A_{i}\right)
$\textquotedblright\ shows that $w$ can be obtained from $v$ by permuting the
values of $v$ on $A_{1}$, permuting the values of $v$ on $A_{2}$, and so on.
But any such permutation (unless it is the identity) decreases $v$ in
lexicographic order (because the restrictions $v\mid_{A_{1}},\ v\mid_{A_{2}%
},\ \ldots,\ v\mid_{A_{k}}$ are decreasing, and thus any nontrivial
permutation \textquotedblleft puts elements in a more natural
order\textquotedblright). Hence, $w\leq v$.

\begin{fineprint}
Here is a rigorous proof: If $w=v$, then the claim is obvious. So we WLOG
assume that $w\neq v$. Then, there exists some $p\in\left[  n\right]  $ such
that $w\left(  p\right)  \neq v\left(  p\right)  $. Consider the
\textbf{smallest} such $p$. Then,%
\begin{equation}
w\left(  r\right)  =v\left(  r\right)  \ \ \ \ \ \ \ \ \ \ \text{for each
}r<p. \label{pf.lem.lex-less-1.r}%
\end{equation}

Since $\mathbf{A}$ is a set decomposition of $\left[  n\right]  $, there
exists some $i\in\left[  k\right]  $ such that $p\in A_{i}$. Consider this $i$.

We have $p\in A_{i}$, thus $w\left(  p\right)  \in w\left(  A_{i}\right)
=v\left(  A_{i}\right)  $ (by assumption). In other words, $w\left(  p\right)
=v\left(  r\right)  $ for some $r\in A_{i}$. Consider this $r$.

If we had $r<p$, then (\ref{pf.lem.lex-less-1.r}) would yield $w\left(
r\right)  =v\left(  r\right)  =w\left(  p\right)  $, which would yield $r=p$
(since $w$ is injective), which would contradict $r<p$. Hence, $r<p$ is
impossible. Thus, $r\geq p$.

Moreover, $w\left(  p\right)  \neq v\left(  p\right)  $ and thus $v\left(
p\right)  \neq w\left(  p\right)  =v\left(  r\right)  $, so that $p\neq r$.
Combining this with $r\geq p$, we obtain $r>p$. Since the restriction
$v\mid_{A_{i}}$ is decreasing (by assumption), this entails $v\left(
r\right)  <v\left(  p\right)  $ (since $r$ and $p$ belong to $A_{i}$). Thus,
$w\left(  p\right)  =v\left(  r\right)  <v\left(  p\right)  $.

So we have shown that the smallest $p\in\left[  n\right]  $ that satisfies
$w\left(  p\right)  \neq v\left(  p\right)  $ must satisfy $w\left(  p\right)
<v\left(  p\right)  $. In other words, $w<v$ in lexicographic order. Hence,
$w\leq v$ follows, and Lemma \ref{lem.lex-less-1} is proved.
\end{fineprint}
\end{proof}

\begin{lemma}
\label{lem.row.A/I-span}Let $k\in\mathbb{N}$. Then, the quotient $\mathbf{k}%
$-module $\mathcal{A}/\mathcal{I}_{k}$ is spanned by the family $\left(
\overline{w}\right)  _{w\in S_{n}\setminus\operatorname*{Av}\nolimits_{n}%
\left(  k+1\right)  }$.
\end{lemma}

\begin{proof}
Clearly, $\mathcal{A}/\mathcal{I}_{k}$ is spanned by the $\overline{u}$ for
$u\in S_{n}$. Hence, it suffices to prove that
\begin{equation}
\overline{u}\in\operatorname*{span}\left(  \left(  \overline{w}\right)  _{w\in
S_{n}\setminus\operatorname*{Av}\nolimits_{n}\left(  k+1\right)  }\right)
\ \ \ \ \ \ \ \ \ \ \text{for each }u\in S_{n}.
\label{pf.lem.row.A/I-span.goal}%
\end{equation}

To prove this, we proceed by induction on $u$ in lexicographic order. Thus, we
fix a permutation $v\in S_{n}$, and we assume (as the induction hypothesis)
that (\ref{pf.lem.row.A/I-span.goal}) holds for every $u<v$ in lexicographic
order. We must now prove (\ref{pf.lem.row.A/I-span.goal}) for $u=v$.

If $v\in S_{n}\setminus\operatorname*{Av}\nolimits_{n}\left(  k+1\right)  $,
then this is trivial. Thus, we WLOG assume that $v\notin S_{n}\setminus
\operatorname*{Av}\nolimits_{n}\left(  k+1\right)  $. Hence, $v\in
\operatorname*{Av}\nolimits_{n}\left(  k+1\right)  $. Therefore, by Lemma
\ref{lem.erdos-szek-var}, there exists a set decomposition $\mathbf{A}=\left(
A_{1},A_{2},\ldots,A_{k}\right)  $ of $\left[  n\right]  $ such that all
restrictions $v\mid_{A_{1}},\ v\mid_{A_{2}},\ \ldots,\ v\mid_{A_{k}}$ are
decreasing. Consider this set decomposition $\mathbf{A}=\left(  A_{1}%
,A_{2},\ldots,A_{k}\right)  $. Define a further set decomposition
$\mathbf{B}=\left(  B_{1},B_{2},\ldots,B_{k}\right)  $ of $\left[  n\right]  $
by $\mathbf{B}=v\mathbf{A}$ (using the action of $S_{n}$ on
$\operatorname*{SD}\left(  n\right)  $), that is, by%
\[
B_{i}:=v\left(  A_{i}\right)  \ \ \ \ \ \ \ \ \ \ \text{for each }i\in\left[
k\right]  .
\]
Thus, $v\in S_{n}$ is a permutation satisfying $v\left(  A_{i}\right)  =B_{i}$
for all $i$. Hence, the row-to-row sum
\[
\nabla_{\mathbf{B},\mathbf{A}}=\sum_{\substack{w\in S_{n};\\w\left(
A_{i}\right)  =B_{i}\text{ for all }i}}w
\]
contains the permutation $v$ as one of its addends. All its remaining addends
are permutations $w$ that satisfy $w<v$ in lexicographic order (by Lemma
\ref{lem.lex-less-1}\footnote{In more detail: We must show that any
permutation $w\in S_{n}$ that satisfies $\left(  w\left(  A_{i}\right)
=B_{i}\text{ for all }i\right)  $ and is distinct from $v$ must satisfy $w<v$
in lexicographic order. So let $w$ be such a permutation. Then, $w\left(
A_{i}\right)  =B_{i}=v\left(  A_{i}\right)  $ for each $i\in\left[  k\right]
$. Hence, Lemma \ref{lem.lex-less-1} shows that $w\leq v$ is lexicographic
order. Since $w$ is distinct from $v$, we thus obtain $w<v$.}). Thus, all the
addends of $\nabla_{\mathbf{B},\mathbf{A}}$ except for $v$ are
lexicographically smaller than $v$. Hence,%
\[
\nabla_{\mathbf{B},\mathbf{A}}=v+\left(  \text{some permutations }w<v\right)
.
\]
Therefore,%
\begin{equation}
v=\nabla_{\mathbf{B},\mathbf{A}}-\left(  \text{some permutations }w<v\right)
. \label{pf.lem.row.A/I-span.4}%
\end{equation}
But the set decompositions $\mathbf{A},\mathbf{B}\in\operatorname*{SD}\left(
n\right)  $ satisfy $\ell\left(  \mathbf{A}\right)  =\ell\left(
\mathbf{B}\right)  =k$. Hence, Proposition \ref{prop.IJ.1} \textbf{(d)} yields
$\nabla_{\mathbf{B},\mathbf{A}}\in\mathcal{I}_{k}$. Thus, projecting the
equality (\ref{pf.lem.row.A/I-span.4}) onto the quotient $\mathcal{A}%
/\mathcal{I}_{k}$, we obtain%
\begin{align*}
\overline{v}  &  =\underbrace{\overline{\nabla_{\mathbf{B},\mathbf{A}}}%
}_{\substack{=0\\\text{(since }\nabla_{\mathbf{B},\mathbf{A}}\in
\mathcal{I}_{k}\text{)}}}-\,\overline{\left(  \text{some permutations
}w<v\right)  }\\
&  =-\underbrace{\overline{\left(  \text{some permutations }w<v\right)  }%
}_{\substack{\in\operatorname*{span}\left(  \left(  \overline{w}\right)
_{w\in S_{n}\setminus\operatorname*{Av}\nolimits_{n}\left(  k+1\right)
}\right)  \\\text{(by our induction hypothesis)}}}\in\operatorname*{span}%
\left(  \left(  \overline{w}\right)  _{w\in S_{n}\setminus\operatorname*{Av}%
\nolimits_{n}\left(  k+1\right)  }\right)  .
\end{align*}
In other words, (\ref{pf.lem.row.A/I-span.goal}) holds for $u=v$. This
completes the induction. Thus, Lemma \ref{lem.row.A/I-span} is proved.
\end{proof}

\begin{lemma}
\label{lem.lex-less-2}Let $v\in S_{n}$ be a permutation. Let $U$ be a subset
of $\left[  n\right]  $ such that the restriction $v\mid_{U}$ is increasing.
Let $w\in S_{n}$ be a permutation. Assume that $w$ agrees with $v$ on all
elements outside of $U$ (that is, we have $w\left(  i\right)  =v\left(
i\right)  $ for each $i\in\left[  n\right]  \setminus U$). Then, $w\geq v$ in
lexicographic order.
\end{lemma}

\begin{proof}
Essentially, this is because $w$ can be obtained from $v$ by permuting the
values of $v$ on $U$ (since $w$ agrees with $v$ on all elements outside of
$U$), and any such permutation increases $v$ in lexicographic order (because
the restriction $v\mid_{U}$ is increasing).

\begin{fineprint}
Here is a rigorous proof: If $w=v$, then the claim is obvious. So we WLOG
assume that $w\neq v$. Then, there exists some $p\in\left[  n\right]  $ such
that $w\left(  p\right)  \neq v\left(  p\right)  $. Consider the
\textbf{smallest} such $p$. Then,%
\begin{equation}
w\left(  r\right)  =v\left(  r\right)  \ \ \ \ \ \ \ \ \ \ \text{for each
}r<p. \label{pf.lem.lex-less-2.r}%
\end{equation}

From $w\left(  p\right)  \neq v\left(  p\right)  $, we obtain $p\notin\left[
n\right]  \setminus U$ (since $w\left(  i\right)  =v\left(  i\right)  $ for
each $i\in\left[  n\right]  \setminus U$). Thus, $p\in U$.

But the permutations $w$ and $v$ agree on all elements outside of $U$. Hence,
in particular, $w\left(  \left[  n\right]  \setminus U\right)  =v\left(
\left[  n\right]  \setminus U\right)  $. However, $w$ and $v$ are permutations
of $\left[  n\right]  $, so we have $w\left(  \left[  n\right]  \setminus
U\right)  =\left[  n\right]  \setminus w\left(  U\right)  $ and $v\left(
\left[  n\right]  \setminus U\right)  =\left[  n\right]  \setminus v\left(
U\right)  $, and therefore%
\[
\left[  n\right]  \setminus w\left(  U\right)  =w\left(  \left[  n\right]
\setminus U\right)  =v\left(  \left[  n\right]  \setminus U\right)  =\left[
n\right]  \setminus v\left(  U\right)  .
\]
Since $w\left(  U\right)  $ and $v\left(  U\right)  $ are subsets of $\left[
n\right]  $, we can take complements in this equality, and conclude that
$w\left(  U\right)  =v\left(  U\right)  $.

From $p\in U$, we obtain $w\left(  p\right)  \in w\left(  U\right)  =v\left(
U\right)  $. In other words, $w\left(  p\right)  =v\left(  r\right)  $ for
some $r\in U$. Consider this $r$. If we had $r<p$, then
(\ref{pf.lem.lex-less-2.r}) would yield $w\left(  r\right)  =v\left(
r\right)  =w\left(  p\right)  $, which would entail $r=p$ (since $w$ is
injective), and this would contradict $r<p$. Thus we cannot have $r<p$. Hence,
$r\geq p$. But we also have $w\left(  p\right)  \neq v\left(  p\right)  $,
hence $v\left(  p\right)  \neq w\left(  p\right)  =v\left(  r\right)  $ and
thus $p\neq r$. Combined with $r\geq p$, this shows that $r>p$. Since the
restriction $v\mid_{U}$ is increasing, we thus conclude that $v\left(
r\right)  >v\left(  p\right)  $ (since $r$ and $p$ belong to $U$). Thus,
$w\left(  p\right)  =v\left(  r\right)  >v\left(  p\right)  $.

So we have shown that the smallest $p\in\left[  n\right]  $ that satisfies
$w\left(  p\right)  \neq v\left(  p\right)  $ must satisfy $w\left(  p\right)
>v\left(  p\right)  $. In other words, $w>v$ in lexicographic order. Hence,
$w\geq v$ follows, and Lemma \ref{lem.lex-less-2} is proved.
\end{fineprint}
\end{proof}

\begin{lemma}
\label{lem.row.A/J-span}Let $k\in\mathbb{N}$. Then, the quotient $\mathbf{k}%
$-module $\mathcal{A}/\mathcal{J}_{k}$ is spanned by the family $\left(
\overline{w}\right)  _{w\in\operatorname*{Av}\nolimits_{n}\left(  k+1\right)
}$.
\end{lemma}

\begin{proof}
Clearly, $\mathcal{A}/\mathcal{J}_{k}$ is spanned by the $\overline{u}$ for
$u\in S_{n}$. Hence, it suffices to prove that
\begin{equation}
\overline{u}\in\operatorname*{span}\left(  \left(  \overline{w}\right)
_{w\in\operatorname*{Av}\nolimits_{n}\left(  k+1\right)  }\right)
\ \ \ \ \ \ \ \ \ \ \text{for each }u\in S_{n}.
\label{pf.lem.row.A/J-span.goal}%
\end{equation}

To prove this, we proceed by induction on $u$ in reverse lexicographic order.
Thus, we fix a permutation $v\in S_{n}$, and we assume (as the induction
hypothesis) that (\ref{pf.lem.row.A/J-span.goal}) holds for every $u>v$ in
lexicographic order. We must now prove (\ref{pf.lem.row.A/J-span.goal}) for
$u=v$.

If $v\in\operatorname*{Av}\nolimits_{n}\left(  k+1\right)  $, then this is
trivial. Thus, we WLOG assume that $v\notin\operatorname*{Av}\nolimits_{n}%
\left(  k+1\right)  $. Hence, $v\in S_{n}\setminus\operatorname*{Av}%
\nolimits_{n}\left(  k+1\right)  $. Therefore, there exists a $\left(
k+1\right)  $-element subset $U$ of $\left[  n\right]  $ such that the
restriction $v\mid_{U}$ is increasing. Consider this $U$. Thus, the sum
\[
v\nabla_{U}^{-}=\sum_{\substack{w\in S_{n}\text{ agrees with }v\text{
on}\\\text{all elements outside of }U}}\pm w
\]
contains the permutation $v$ as one of its addends. All its remaining addends
have the form $\pm w$ where the permutation $w\in S_{n}$ satisfies $w>v$ in
lexicographic order (because if $w\in S_{n}$ agrees with $v$ on all elements
outside of $U$ but is distinct from $v$, then Lemma \ref{lem.lex-less-2}
yields $w\geq v$ and therefore $w>v$). Hence,%
\[
v\nabla_{U}^{-}=v\pm\left(  \text{some permutations }w>v\right)  .
\]
Therefore,%
\begin{equation}
v=v\nabla_{U}^{-}\pm\left(  \text{some permutations }w>v\right)  .
\label{pf.lem.row.A/J-span.4}%
\end{equation}
But the definition of $\mathcal{J}_{k}$ yields $v\nabla_{U}^{-}\in
\mathcal{J}_{k}$. Thus, projecting the equality (\ref{pf.lem.row.A/J-span.4})
onto the quotient $\mathcal{A}/\mathcal{J}_{k}$, we obtain%
\begin{align*}
\overline{v}  &  =\underbrace{\overline{v\nabla_{U}^{-}}}%
_{\substack{=0\\\text{(since }v\nabla_{U}^{-}\in\mathcal{J}_{k}\text{)}}%
}\pm\,\overline{\left(  \text{some permutations }w>v\right)  }\\
&  =-\underbrace{\overline{\left(  \text{some permutations }w>v\right)  }%
}_{\substack{\in\operatorname*{span}\left(  \left(  \overline{w}\right)
_{w\in\operatorname*{Av}\nolimits_{n}\left(  k+1\right)  }\right)  \\\text{(by
our induction hypothesis)}}}\in\operatorname*{span}\left(  \left(
\overline{w}\right)  _{w\in\operatorname*{Av}\nolimits_{n}\left(  k+1\right)
}\right)  .
\end{align*}
In other words, (\ref{pf.lem.row.A/J-span.goal}) holds for $u=v$. This
completes the induction. Thus, Lemma \ref{lem.row.A/J-span} is proved.
\end{proof}

\begin{lemma}
\label{lem.row.I-c1}Let $k\in\mathbb{N}$. Let $\left(  \alpha_{w}\right)
_{w\in\operatorname*{Av}\nolimits_{n}\left(  k+1\right)  }\in\mathbf{k}%
^{\operatorname*{Av}\nolimits_{n}\left(  k+1\right)  }$ be a family of scalars
satisfying%
\[
\sum_{w\in\operatorname*{Av}\nolimits_{n}\left(  k+1\right)  }\alpha_{w}%
w\in\mathcal{I}_{k}^{\perp}.
\]
Then, $\alpha_{w}=0$ for all $w\in\operatorname*{Av}\nolimits_{n}\left(
k+1\right)  $.
\end{lemma}

\begin{proof}
Assume the contrary. Thus, there exist some $w\in\operatorname*{Av}%
\nolimits_{n}\left(  k+1\right)  $ such that $\alpha_{w}\neq0$. Let $v$ be the
lexicographically smallest such $w$. Thus, $\alpha_{v}\neq0$, but%
\begin{equation}
\alpha_{w}=0\ \ \ \ \ \ \ \ \ \ \text{for every }w\in\operatorname*{Av}%
\nolimits_{n}\left(  k+1\right)  \text{ satisfying }w<v.
\label{pf.lem.row.I-c1.1}%
\end{equation}

As in the proof of Lemma \ref{lem.row.A/I-span}, we can construct set
decompositions $\mathbf{A},\mathbf{B}\in\operatorname*{SD}\left(  n\right)  $
such that $\ell\left(  \mathbf{A}\right)  =\ell\left(  \mathbf{B}\right)  =k$
and $\nabla_{\mathbf{B},\mathbf{A}}\in\mathcal{I}_{k}$ and%
\begin{equation}
\nabla_{\mathbf{B},\mathbf{A}}=v+\left(  \text{some permutations }w<v\right)
\label{pf.lem.row.I-c1.2}%
\end{equation}
hold. Consider these $\mathbf{A}$ and $\mathbf{B}$. The $\mathbf{k}$-bilinear
form $\left\langle \cdot,\cdot\right\rangle $ has the property that for every
$w\in S_{n}$ and every $\mathbf{a}\in\mathcal{A}$, the coefficient of $w$ in
$\mathbf{a}$ is $\left\langle w,\ \mathbf{a}\right\rangle $. Hence, the
equality (\ref{pf.lem.row.I-c1.2}) shows that%
\begin{equation}
\left\langle v,\ \nabla_{\mathbf{B},\mathbf{A}}\right\rangle =1
\label{pf.lem.row.I-c1.2a}%
\end{equation}
and%
\begin{equation}
\left\langle w,\ \nabla_{\mathbf{B},\mathbf{A}}\right\rangle
=0\ \ \ \ \ \ \ \ \ \ \text{for each permutation }w>v.
\label{pf.lem.row.I-c1.2b}%
\end{equation}

From $\sum_{w\in\operatorname*{Av}\nolimits_{n}\left(  k+1\right)  }\alpha
_{w}w\in\mathcal{I}_{k}^{\perp}$ and $\nabla_{\mathbf{B},\mathbf{A}}%
\in\mathcal{I}_{k}$, we conclude that
\[
\left\langle \sum_{w\in\operatorname*{Av}\nolimits_{n}\left(  k+1\right)
}\alpha_{w}w,\ \nabla_{\mathbf{B},\mathbf{A}}\right\rangle =0.
\]
Thus,%
\begin{align*}
0  &  =\left\langle \sum_{w\in\operatorname*{Av}\nolimits_{n}\left(
k+1\right)  }\alpha_{w}w,\ \nabla_{\mathbf{B},\mathbf{A}}\right\rangle \\
&  =\sum_{w\in\operatorname*{Av}\nolimits_{n}\left(  k+1\right)
}\ \ \underbrace{\underbrace{\alpha_{w}}_{\substack{=0\text{ if }%
w<v\\\text{(by (\ref{pf.lem.row.I-c1.1}))}}}\ \ \underbrace{\left\langle
w,\ \nabla_{\mathbf{B},\mathbf{A}}\right\rangle }_{\substack{=0\text{ if
}w>v\\\text{(by (\ref{pf.lem.row.I-c1.2b}))}}}}_{\substack{=0\text{ if }w\neq
v\\\text{(since }w\neq v\text{ entails }w<v\text{ or }w>v\text{)}}}\\
&  =\alpha_{v}\underbrace{\left\langle v,\ \nabla_{\mathbf{B},\mathbf{A}%
}\right\rangle }_{\substack{=1\\\text{(by (\ref{pf.lem.row.I-c1.2a}))}%
}}=\alpha_{v}\neq0,
\end{align*}
which is absurd. This completes the proof of Lemma \ref{lem.row.I-c1}.
\end{proof}

\begin{lemma}
\label{lem.row.J-c1}Let $k\in\mathbb{N}$. Let $\left(  \alpha_{w}\right)
_{w\in S_{n}\setminus\operatorname*{Av}\nolimits_{n}\left(  k+1\right)  }%
\in\mathbf{k}^{S_{n}\setminus\operatorname*{Av}\nolimits_{n}\left(
k+1\right)  }$ be a family of scalars satisfying%
\[
\sum_{w\in S_{n}\setminus\operatorname*{Av}\nolimits_{n}\left(  k+1\right)
}\alpha_{w}w\in\mathcal{J}_{k}^{\perp}.
\]
Then, $\alpha_{w}=0$ for all $w\in S_{n}\setminus\operatorname*{Av}%
\nolimits_{n}\left(  k+1\right)  $.
\end{lemma}

\begin{proof}
Assume the contrary. Thus, there exist some $w\in S_{n}\setminus
\operatorname*{Av}\nolimits_{n}\left(  k+1\right)  $ such that $\alpha_{w}%
\neq0$. Let $v$ be the lexicographically largest such $w$. Thus, $\alpha
_{v}\neq0$, but%
\begin{equation}
\alpha_{w}=0\ \ \ \ \ \ \ \ \ \ \text{for every }w\in S_{n}\setminus
\operatorname*{Av}\nolimits_{n}\left(  k+1\right)  \text{ satisfying }w>v.
\label{pf.lem.row.J-c1.1}%
\end{equation}

As in the proof of Lemma \ref{lem.row.A/J-span}, we can construct a $\left(
k+1\right)  $-element subset $U$ of $\left[  n\right]  $ such that
$v\nabla_{U}^{-}\in\mathcal{J}_{k}$ and%
\begin{equation}
v\nabla_{U}^{-}=v\pm\left(  \text{some permutations }w>v\right)  .
\label{pf.lem.row.J-c1.2}%
\end{equation}
Consider this $U$. The $\mathbf{k}$-bilinear form $\left\langle \cdot
,\cdot\right\rangle $ has the property that for every $w\in S_{n}$ and every
$\mathbf{a}\in\mathcal{A}$, the coefficient of $w$ in $\mathbf{a}$ is
$\left\langle w,\ \mathbf{a}\right\rangle $. Hence, the equality
(\ref{pf.lem.row.J-c1.2}) shows that%
\begin{equation}
\left\langle v,\ v\nabla_{U}^{-}\right\rangle =1 \label{pf.lem.row.J-c1.2a}%
\end{equation}
and%
\begin{equation}
\left\langle w,\ v\nabla_{U}^{-}\right\rangle =0\ \ \ \ \ \ \ \ \ \ \text{for
each permutation }w<v. \label{pf.lem.row.J-c1.2b}%
\end{equation}

From $\sum_{w\in S_{n}\setminus\operatorname*{Av}\nolimits_{n}\left(
k+1\right)  }\alpha_{w}w\in\mathcal{J}_{k}^{\perp}$ and $v\nabla_{U}^{-}%
\in\mathcal{J}_{k}$, we conclude that
\[
\left\langle \sum_{w\in S_{n}\setminus\operatorname*{Av}\nolimits_{n}\left(
k+1\right)  }\alpha_{w}w,\ v\nabla_{U}^{-}\right\rangle =0.
\]
Thus,%
\begin{align*}
0  &  =\left\langle \sum_{w\in S_{n}\setminus\operatorname*{Av}\nolimits_{n}%
\left(  k+1\right)  }\alpha_{w}w,\ v\nabla_{U}^{-}\right\rangle \\
&  =\sum_{w\in S_{n}\setminus\operatorname*{Av}\nolimits_{n}\left(
k+1\right)  }\ \ \underbrace{\underbrace{\alpha_{w}}_{\substack{=0\text{ if
}w>v\\\text{(by (\ref{pf.lem.row.J-c1.1}))}}}\ \ \underbrace{\left\langle
w,\ v\nabla_{U}^{-}\right\rangle }_{\substack{=0\text{ if }w<v\\\text{(by
(\ref{pf.lem.row.J-c1.2b}))}}}}_{\substack{=0\text{ if }w\neq v\\\text{(since
}w\neq v\text{ entails }w<v\text{ or }w>v\text{)}}}\\
&  =\alpha_{v}\underbrace{\left\langle v,\ v\nabla_{U}^{-}\right\rangle
}_{\substack{=1\\\text{(by (\ref{pf.lem.row.J-c1.2a}))}}}=\alpha_{v}\neq0,
\end{align*}
which is absurd. This completes the proof of Lemma \ref{lem.row.J-c1}.
\end{proof}

\begin{lemma}
\label{lem.row.I-basis}Let $k\in\mathbb{N}$. Then, the $\mathbf{k}$-module
$\mathcal{A}/\mathcal{I}_{k}$ is free with basis $\left(  \overline{w}\right)
_{w\in S_{n}\setminus\operatorname*{Av}\nolimits_{n}\left(  k+1\right)  }$.
\end{lemma}

\begin{proof}
The family $\left(  \overline{w}\right)  _{w\in S_{n}\setminus
\operatorname*{Av}\nolimits_{n}\left(  k+1\right)  }$ spans this $\mathbf{k}%
$-module $\mathcal{A}/\mathcal{I}_{k}$, as we know from Lemma
\ref{lem.row.A/I-span}. It remains to prove that it is $\mathbf{k}$-linearly independent.

Let $\left(  \alpha_{w}\right)  _{w\in S_{n}\setminus\operatorname*{Av}%
\nolimits_{n}\left(  k+1\right)  }\in\mathbf{k}^{S_{n}\setminus
\operatorname*{Av}\nolimits_{n}\left(  k+1\right)  }$ be a family of scalars
satisfying%
\begin{equation}
\sum_{w\in S_{n}\setminus\operatorname*{Av}\nolimits_{n}\left(  k+1\right)
}\alpha_{w}\overline{w}=0. \label{pf.lem.row.I-basis.1}%
\end{equation}
We thus need to show that $\alpha_{w}=0$ for all $w\in S_{n}\setminus
\operatorname*{Av}\nolimits_{n}\left(  k+1\right)  $.

However, (\ref{pf.lem.row.I-basis.1}) means that $\sum_{w\in S_{n}%
\setminus\operatorname*{Av}\nolimits_{n}\left(  k+1\right)  }\alpha_{w}%
w\in\mathcal{I}_{k}$. But Lemma \ref{lem.row.IJ=0} yields $\mathcal{I}%
_{k}\mathcal{J}_{k}=0$. Thus, $\mathcal{I}_{k}\subseteq\operatorname*{LAnn}%
\mathcal{J}_{k}$. However, Proposition \ref{prop.IJ.1} \textbf{(c)} yields
$S\left(  \mathcal{J}_{k}\right)  =\mathcal{J}_{k}$. Furthermore,
$\mathcal{J}_{k}$ is an ideal of $\mathcal{A}$ (by Proposition \ref{prop.IJ.1}
\textbf{(a)}), hence a left ideal of $\mathcal{A}$. Thus, Lemma
\ref{lem.row.LAnn} \textbf{(b)} (applied to $\mathcal{B}=\mathcal{J}_{k}$)
yields $\mathcal{J}_{k}^{\perp}=\operatorname*{LAnn}\left(  \mathcal{J}%
_{k}\right)  =\operatorname*{RAnn}\left(  \mathcal{J}_{k}\right)  $. Thus,
\[
\sum_{w\in S_{n}\setminus\operatorname*{Av}\nolimits_{n}\left(  k+1\right)
}\alpha_{w}w\in\mathcal{I}_{k}\subseteq\operatorname*{LAnn}\mathcal{J}%
_{k}=\mathcal{J}_{k}^{\perp}.
\]
Lemma \ref{lem.row.J-c1} thus yields that $\alpha_{w}=0$ for all $w\in
S_{n}\setminus\operatorname*{Av}\nolimits_{n}\left(  k+1\right)  $. This
completes the proof of Lemma \ref{lem.row.I-basis}.
\end{proof}

\begin{lemma}
\label{lem.row.J-basis}Let $k\in\mathbb{N}$. Then, the $\mathbf{k}$-module
$\mathcal{A}/\mathcal{J}_{k}$ is free with basis $\left(  \overline{w}\right)
_{w\in\operatorname*{Av}\nolimits_{n}\left(  k+1\right)  }$.
\end{lemma}

\begin{proof}
Analogous to the proof of Lemma \ref{lem.row.I-basis}. (Of course, use Lemma
\ref{lem.row.A/J-span} and Lemma \ref{lem.row.I-c1} instead of Lemma
\ref{lem.row.A/I-span} and Lemma \ref{lem.row.J-c1} now.)
\end{proof}

\begin{lemma}
\label{lem.row.Ikperp}Let $k\in\mathbb{N}$. Then,
\[
\mathcal{I}_{k}=\mathcal{J}_{k}^{\perp}=\operatorname*{LAnn}\mathcal{J}%
_{k}=\operatorname*{RAnn}\mathcal{J}_{k}.
\]

\end{lemma}

\begin{proof}
Proposition \ref{prop.IJ.1} \textbf{(c)} yields $S\left(  \mathcal{J}%
_{k}\right)  =\mathcal{J}_{k}$. Furthermore, $\mathcal{J}_{k}$ is an ideal of
$\mathcal{A}$ (by Proposition \ref{prop.IJ.1} \textbf{(a)}). Thus, Lemma
\ref{lem.row.LAnn} \textbf{(b)} (applied to $\mathcal{B}=\mathcal{J}_{k}$)
yields $\mathcal{J}_{k}^{\perp}=\operatorname*{LAnn}\mathcal{J}_{k}%
=\operatorname*{RAnn}\mathcal{J}_{k}$. Thus, it remains to prove that
$\mathcal{I}_{k}=\mathcal{J}_{k}^{\perp}$.

Lemma \ref{lem.row.IJ=0} yields $\mathcal{I}_{k}\mathcal{J}_{k}=0$. Thus,
$\mathcal{I}_{k}\subseteq\operatorname*{LAnn}\mathcal{J}_{k}=\mathcal{J}%
_{k}^{\perp}$. Thus, we only need to show that $\mathcal{J}_{k}^{\perp
}\subseteq\mathcal{I}_{k}$.

Let $a\in\mathcal{J}_{k}^{\perp}$. We must prove that $a\in\mathcal{I}_{k}$.

Lemma \ref{lem.row.A/I-span} shows that the quotient module $\mathcal{A}%
/\mathcal{I}_{k}$ is spanned by the family $\left(  \overline{w}\right)
_{w\in S_{n}\setminus\operatorname*{Av}\nolimits_{n}\left(  k+1\right)  }$.
Hence, the projection $\overline{a}\in\mathcal{A}/\mathcal{I}_{k}$ can be
written as a $\mathbf{k}$-linear combination of this family. In other words,
we can write $\overline{a}$ as%
\begin{equation}
\overline{a}=\sum_{w\in S_{n}\setminus\operatorname*{Av}\nolimits_{n}\left(
k+1\right)  }\alpha_{w}\overline{w} \label{pf.lem.row.Ikperp.4}%
\end{equation}
for some family $\left(  \alpha_{w}\right)  _{w\in S_{n}\setminus
\operatorname*{Av}\nolimits_{n}\left(  k+1\right)  }\in\mathbf{k}%
^{S_{n}\setminus\operatorname*{Av}\nolimits_{n}\left(  k+1\right)  }$ of
scalars. Consider this family. We can rewrite (\ref{pf.lem.row.Ikperp.4}) as%
\[
a-\sum_{w\in S_{n}\setminus\operatorname*{Av}\nolimits_{n}\left(  k+1\right)
}\alpha_{w}w\in\mathcal{I}_{k}\subseteq\mathcal{J}_{k}^{\perp}.
\]
Since $a\in\mathcal{J}_{k}^{\perp}$, this yields $\sum_{w\in S_{n}%
\setminus\operatorname*{Av}\nolimits_{n}\left(  k+1\right)  }\alpha_{w}%
w\in\mathcal{J}_{k}^{\perp}$. By Lemma \ref{lem.row.J-c1}, we thus conclude
that $\alpha_{w}=0$ for all $w\in S_{n}\setminus\operatorname*{Av}%
\nolimits_{n}\left(  k+1\right)  $. Thus, (\ref{pf.lem.row.Ikperp.4}) rewrites
as $\overline{a}=\sum_{w\in S_{n}\setminus\operatorname*{Av}\nolimits_{n}%
\left(  k+1\right)  }0\overline{w}=0$, so that $a\in\mathcal{I}_{k}$. This
completes our proof of Lemma \ref{lem.row.Ikperp}.
\end{proof}

\begin{lemma}
\label{lem.row.Jkperp}Let $k\in\mathbb{N}$. Then,
\[
\mathcal{J}_{k}=\mathcal{I}_{k}^{\perp}=\operatorname*{LAnn}\mathcal{I}%
_{k}=\operatorname*{RAnn}\mathcal{I}_{k}.
\]

\end{lemma}

\begin{proof}
Analogous to the proof of Lemma \ref{lem.row.Ikperp}. (Of course, use Lemma
\ref{lem.row.A/J-span} and Lemma \ref{lem.row.I-c1} instead of Lemma
\ref{lem.row.A/I-span} and Lemma \ref{lem.row.J-c1} now.)
\end{proof}

\begin{lemma}
\label{lem.maschke.unital}Assume that $n!$ is invertible in $\mathbf{k}$. Let
$\mathcal{I}$ be an ideal of $\mathcal{A}$. Assume that there exists a
$\mathbf{k}$-linear projection $\pi:\mathcal{A}\rightarrow\mathcal{I}$. Then,
$\mathcal{I}$ is a nonunital subalgebra of $\mathcal{A}$ that has a unity.
\end{lemma}

\begin{proof}
Note that $\mathcal{I}$ is an ideal of $\mathcal{A}$, thus a left ideal of
$\mathcal{A}$, hence a left $\mathcal{A}$-submodule of $\mathcal{A}$.
Moreover, $\left\vert S_{n}\right\vert =n!$ is invertible in $\mathbf{k}$.
Hence, the standard proof of the Maschke theorem (via averaging the projection
$\pi$ over $S_{n}$) yields that there exists a $\mathbf{k}$-linear projection
$\pi^{\prime}:\mathcal{A}\rightarrow\mathcal{I}$ that is a left $\mathcal{A}%
$-module homomorphism\footnote{Explicitly, $\pi^{\prime}$ can be constructed
as follows:%
\[
\pi^{\prime}\left(  a\right)  =\dfrac{1}{\left\vert S_{n}\right\vert }%
\sum_{\sigma\in S_{n}}\sigma\pi\left(  \sigma^{-1}a\right)
\ \ \ \ \ \ \ \ \ \ \text{for each }a\in\mathcal{A}.
\]
\par
For a concrete reference, see \cite[Theorem 4.4.14]{sga}. Note that the
existence of a $\mathbf{k}$-linear projection from $\mathcal{A}$ onto
$\mathcal{I}$ is equivalent to saying that $\mathcal{I}$ is a direct addend of
$\mathcal{A}$ as a $\mathbf{k}$-module; furthermore, the same holds when each
appearance of \textquotedblleft$\mathbf{k}$-\textquotedblright\ is replaced by
\textquotedblleft left $\mathcal{A}$-\textquotedblright.}. Consider this
$\pi^{\prime}$.

Let $e:=\pi^{\prime}\left(  1\right)  \in\mathcal{I}$. Then, we claim that
\begin{equation}
ue=u\ \ \ \ \ \ \ \ \ \ \text{for each }u\in\mathcal{I}.
\label{pf.lem.maschke.unital.1}%
\end{equation}

[\textit{Proof of (\ref{pf.lem.maschke.unital.1}):} Let $u\in\mathcal{I}$.
Then, $\pi^{\prime}\left(  u\right)  =u$ (since $\pi^{\prime}$ is a
projection). However, $\pi^{\prime}$ is a left $\mathcal{A}$-module
homomorphism. Thus, $\pi^{\prime}\left(  u1\right)  =u\underbrace{\pi^{\prime
}\left(  1\right)  }_{=e}=ue$. Since $u1=u$, we can rewrite this as
$\pi^{\prime}\left(  u\right)  =ue$. Hence, $ue=\pi^{\prime}\left(  u\right)
=u$. This proves (\ref{pf.lem.maschke.unital.1}).] \medskip

Clearly, $\mathcal{I}$ is a nonunital subalgebra of $\mathcal{A}$ (since
$\mathcal{I}$ is an ideal of $\mathcal{A}$). From
(\ref{pf.lem.maschke.unital.1}), we see that this algebra $\mathcal{I}$ has a
right unity (namely, $e$). A similar argument (using right instead of left
$\mathcal{A}$-modules) yields that $\mathcal{I}$ has a left unity. Thus, a
standard argument shows that $\mathcal{I}$ has a unity (since any binary
operation that has a left neutral element and a right neutral element has a
neutral element). Thus, Lemma \ref{lem.maschke.unital} is proved.
\end{proof}

\begin{lemma}
\label{lem.IxJ.gen}Let $\mathcal{I}$ and $\mathcal{J}$ be two ideals of
$\mathcal{A}$ such that $\mathcal{I}=\operatorname*{LAnn}\mathcal{J}$ and
$\mathcal{J}=\operatorname*{LAnn}\mathcal{I}$. Assume that $\mathcal{I}$ is a
nonunital subalgebra of $\mathcal{A}$ that has a unity. Then, $\mathcal{A}%
=\mathcal{I}\oplus\mathcal{J}$ (internal direct sum) as $\mathbf{k}$-module.
Moreover, $\mathcal{I}$ and $\mathcal{J}$ are nonunital subalgebras of
$\mathcal{A}$ that have unities and satisfy $\mathcal{A}\cong\mathcal{I}%
\times\mathcal{J}$ as $\mathbf{k}$-algebras.
\end{lemma}

\begin{proof}
Clearly, $\mathcal{I}$ and $\mathcal{J}$ are nonunital subalgebras of
$\mathcal{A}$ (since any ideal of $\mathcal{A}$ is a nonunital subalgebra).

From $\mathcal{I}=\operatorname*{LAnn}\mathcal{J}$, we obtain $\mathcal{IJ}%
=0$. Similarly, $\mathcal{JI}=0$.

We have assumed that $\mathcal{I}$ is a nonunital subalgebra of $\mathcal{A}$
that has a unity. Let $1_{\mathcal{I}}$ denote its unity.

Set $g:=1-1_{\mathcal{I}}$. Then, each $u\in\mathcal{I}$ satisfies%
\[
gu=\left(  1-1_{\mathcal{I}}\right)  u=u-\underbrace{1_{\mathcal{I}}%
u}_{\substack{=u\\\text{(since }1_{\mathcal{I}}\text{ is the}\\\text{unity of
}\mathcal{I}\text{)}}}=u-u=0.
\]
In other words, $g\in\operatorname*{LAnn}\mathcal{I}=\mathcal{J}$. Moreover,
each $v\in\mathcal{J}$ satisfies $1_{\mathcal{I}}v=0$ (since
$\underbrace{1_{\mathcal{I}}}_{\in\mathcal{I}}\underbrace{v}_{\in\mathcal{J}%
}\in\mathcal{IJ}=0$) and thus%
\[
\underbrace{g}_{=1-1_{\mathcal{I}}}v=\left(  1-1_{\mathcal{I}}\right)
v=v-\underbrace{1_{\mathcal{I}}v}_{=0}=v.
\]
Hence, $g$ is a left unity of the algebra $\mathcal{J}$ (since $g\in
\mathcal{J}$). A similar computation shows that $g$ is a right unity of
$\mathcal{J}$. Hence, $g$ is a unity of $\mathcal{J}$. We shall thus rename
$g$ as $1_{\mathcal{J}}$ now. Of course, this shows that the nonunital algebra
$\mathcal{J}$ has a unity.

Moreover, each $u\in\mathcal{I}\cap\mathcal{J}$ satisfies $u\in\mathcal{I}%
\cap\mathcal{J}\subseteq\mathcal{J}$ and therefore%
\begin{align*}
u  &  =\underbrace{u}_{\in\mathcal{I}\cap\mathcal{J}\subseteq\mathcal{I}%
}\underbrace{1_{\mathcal{J}}}_{\in\mathcal{J}}\ \ \ \ \ \ \ \ \ \ \left(
\text{since }1_{\mathcal{J}}\text{ is the unity of }\mathcal{J}\right) \\
&  \in\mathcal{IJ}=0
\end{align*}
and thus $u=0$. In other words, $\mathcal{I}\cap\mathcal{J}=0$. Furthermore,
each $a\in\mathcal{A}$ satisfies%
\[
a=1_{\mathcal{I}}a+\underbrace{\left(  1-1_{\mathcal{I}}\right)
}_{=g=1_{\mathcal{J}}}a=\underbrace{1_{\mathcal{I}}a}_{\substack{\in
\mathcal{I}\\\text{(since }\mathcal{I}\text{ is an ideal}\\\text{and
}1_{\mathcal{I}}\in\mathcal{I}\text{)}}}+\underbrace{1_{\mathcal{J}}%
a}_{\substack{\in\mathcal{J}\\\text{(since }\mathcal{J}\text{ is an
ideal}\\\text{and }1_{\mathcal{J}}\in\mathcal{J}\text{)}}}\in\mathcal{I}%
+\mathcal{J}.
\]
This shows that $\mathcal{I}+\mathcal{J}=\mathcal{A}$. Combining this with
$\mathcal{I}\cap\mathcal{J}=0$, we conclude that $\mathcal{A}=\mathcal{I}%
\oplus\mathcal{J}$ (internal direct sum) as $\mathbf{k}$-module. Hence, the
$\mathbf{k}$-linear map%
\begin{align*}
\mathcal{I}\times\mathcal{J}  &  \rightarrow\mathcal{A},\\
\left(  i,j\right)   &  \mapsto i+j
\end{align*}
is a $\mathbf{k}$-module isomorphism. This isomorphism furthermore respects
the multiplication (since $\mathcal{IJ}=\mathcal{JI}=0$, and thus every
$\left(  i,j\right)  ,\left(  i^{\prime},j^{\prime}\right)  \in\mathcal{I}%
\times\mathcal{J}$ satisfy $\left(  i+j\right)  \left(  i^{\prime}+j^{\prime
}\right)  =ii^{\prime}+\underbrace{ij^{\prime}}_{\substack{=0\\\text{(since
}\mathcal{IJ}=0\text{)}}}+\underbrace{ji^{\prime}}_{\substack{=0\\\text{(since
}\mathcal{JI}=0\text{)}}}+\,jj^{\prime}=ii^{\prime}+jj^{\prime}$), and thus is
a nonunital $\mathbf{k}$-algebra isomorphism. Hence, it must also respect the
unity (since it is an isomorphism), and thus is a $\mathbf{k}$-algebra
isomorphism. We thus conclude that $\mathcal{A}\cong\mathcal{I}\times
\mathcal{J}$ as $\mathbf{k}$-algebras. This completes the proof of Lemma
\ref{lem.IxJ.gen}.
\end{proof}

\begin{lemma}
\label{lem.maschke.dirprod}Assume that $n!$ is invertible in $\mathbf{k}$. Let
$J$ be a subset of $S_{n}$. Let $\mathcal{I}$ and $\mathcal{J}$ be two ideals
of $\mathcal{A}$ such that $\mathcal{I}=\operatorname*{LAnn}\mathcal{J}$ and
$\mathcal{J}=\operatorname*{LAnn}\mathcal{I}$. Assume that the family $\left(
\overline{w}\right)  _{w\in J}$ is a basis of the $\mathbf{k}$-module
$\mathcal{A}/\mathcal{I}$. Then, $\mathcal{A}=\mathcal{I}\oplus\mathcal{J}$
(internal direct sum) as $\mathbf{k}$-module. Moreover, $\mathcal{I}$ and
$\mathcal{J}$ are nonunital subalgebras of $\mathcal{A}$ that have unities and
satisfy $\mathcal{A}\cong\mathcal{I}\times\mathcal{J}$ as $\mathbf{k}$-algebras.
\end{lemma}

\begin{proof}
Clearly, $\mathcal{I}$ and $\mathcal{J}$ are nonunital subalgebras of
$\mathcal{A}$ (since any ideal of $\mathcal{A}$ is a nonunital subalgebra).

Let $K:=S_{n}\setminus J$\text{. Thus, }$J$ and $K$ are two disjoint subsets
of $S_{n}$ such that $J\cup K=S_{n}$. Moreover, $\left(  w\right)  _{w\in
S_{n}}$ is a basis of the $\mathbf{k}$-module $\mathcal{A}$, whereas $\left(
\overline{w}\right)  _{w\in J}$ is a basis of the $\mathbf{k}$-module
$\mathcal{A}/\mathcal{I}$. Hence, Lemma \ref{lem.mod.sub-basis} \textbf{(b)}
(applied to $\mathcal{M}=\mathcal{A}$ and $\mathcal{N}=\mathcal{I}$ and
$I=S_{n}$ and $\left(  m_{i}\right)  _{i\in I}=\left(  w\right)  _{w\in S_{n}%
}$) yields that there exists a $\mathbf{k}$-linear projection $\pi
:\mathcal{A}\rightarrow\mathcal{I}$ (that is, a $\mathbf{k}$-linear map
$\pi:\mathcal{A}\rightarrow\mathcal{I}$ such that $\left.  \pi\mid
_{\mathcal{I}}\right.  =\operatorname*{id}$). Hence, Lemma
\ref{lem.maschke.unital} shows that $\mathcal{I}$ is a nonunital subalgebra of
$\mathcal{A}$ that has a unity.

Thus, Lemma \ref{lem.IxJ.gen} shows that $\mathcal{A}=\mathcal{I}%
\oplus\mathcal{J}$ (internal direct sum) as $\mathbf{k}$-module, and
furthermore, it shows that $\mathcal{I}$ and $\mathcal{J}$ are nonunital
subalgebras of $\mathcal{A}$ that have unities and satisfy $\mathcal{A}%
\cong\mathcal{I}\times\mathcal{J}$ as $\mathbf{k}$-algebras. This proves Lemma
\ref{lem.maschke.dirprod}.
\end{proof}

\subsection{\label{sec.row.main-proof}Proof of the main theorem}

We can now prove Theorem \ref{thm.row.main} and Corollary \ref{cor.Nabla.span}
by combining what we have shown so far:

\begin{proof}
[Proof of Theorem \ref{thm.row.main}.]\textbf{(a)} This is just Lemma
\ref{lem.row.Ikperp}. \medskip

\textbf{(b)} This is just Lemma \ref{lem.row.Jkperp}. \medskip

\textbf{(c)} Lemma \ref{lem.row.I-basis} yields that the $\mathbf{k}$-module
$\mathcal{A}/\mathcal{I}_{k}$ is free with basis $\left(  \overline{w}\right)
_{w\in S_{n}\setminus\operatorname*{Av}\nolimits_{n}\left(  k+1\right)  }$.
Hence, Lemma \ref{lem.mod.sub-basis} \textbf{(a)} (applied to $\mathcal{M}%
=\mathcal{A}$ and $\mathcal{N}=\mathcal{I}_{k}$ and $I=S_{n}$ and
$J=S_{n}\setminus\operatorname*{Av}\nolimits_{n}\left(  k+1\right)  $ and
$K=\operatorname*{Av}\nolimits_{n}\left(  k+1\right)  $ and $\left(
m_{i}\right)  _{i\in I}=\left(  w\right)  _{w\in S_{n}}$) yields that the
$\mathbf{k}$-module $\mathcal{I}_{k}$ is free of rank $\left\vert
\operatorname*{Av}\nolimits_{n}\left(  k+1\right)  \right\vert $. This proves
Theorem \ref{thm.row.main} \textbf{(c)}. \medskip

\textbf{(d)} This is proved similarly to part \textbf{(c)}, but using Lemma
\ref{lem.row.J-basis} instead of Lemma \ref{lem.row.I-basis}. (This time,
Lemma \ref{lem.mod.sub-basis} \textbf{(a)} must be applied to
$J=\operatorname*{Av}\nolimits_{n}\left(  k+1\right)  $ and $K=S_{n}%
\setminus\operatorname*{Av}\nolimits_{n}\left(  k+1\right)  $.) \medskip

\textbf{(e)} This is just Lemma \ref{lem.row.I-basis}. \medskip

\textbf{(f)} This is just Lemma \ref{lem.row.J-basis}. \medskip

\textbf{(g)} Proposition \ref{prop.IJ.1} \textbf{(a)} yields that both
$\mathcal{I}_{k}$ and $\mathcal{J}_{k}$ are ideals of $\mathcal{A}$. Theorem
\ref{thm.row.main} \textbf{(a)} yields $\mathcal{I}_{k}=\operatorname*{LAnn}%
\mathcal{J}_{k}$. Theorem \ref{thm.row.main} \textbf{(b)} yields
$\mathcal{J}_{k}=\operatorname*{LAnn}\mathcal{I}_{k}$. Clearly, $S_{n}%
\setminus\operatorname*{Av}\nolimits_{n}\left(  k+1\right)  $ and
$\operatorname*{Av}\nolimits_{n}\left(  k+1\right)  $ are two disjoint subsets
of $S_{n}$ such that $\left(  S_{n}\setminus\operatorname*{Av}\nolimits_{n}%
\left(  k+1\right)  \right)  \cup\operatorname*{Av}\nolimits_{n}\left(
k+1\right)  =S_{n}$. Theorem \ref{thm.row.main} \textbf{(e)} says that the
$\mathbf{k}$-module $\mathcal{A}/\mathcal{I}_{k}$ is free with basis $\left(
\overline{w}\right)  _{w\in S_{n}\setminus\operatorname*{Av}\nolimits_{n}%
\left(  k+1\right)  }$. Hence, Lemma \ref{lem.maschke.dirprod} (applied to
$J=S_{n}\setminus\operatorname*{Av}\nolimits_{n}\left(  k+1\right)  $ and
$\mathcal{I}=\mathcal{I}_{k}$ and $\mathcal{J}=\mathcal{J}_{k}$) yields that
$\mathcal{A}=\mathcal{I}_{k}\oplus\mathcal{J}_{k}$ (internal direct sum) as
$\mathbf{k}$-module, and moreover, $\mathcal{I}_{k}$ and $\mathcal{J}_{k}$ are
nonunital subalgebras of $\mathcal{A}$ that have unities and satisfy
$\mathcal{A}\cong\mathcal{I}_{k}\times\mathcal{J}_{k}$ as $\mathbf{k}%
$-algebras. This proves Theorem \ref{thm.row.main} \textbf{(g)}.
\end{proof}

\begin{proof}
[Proof of Corollary \ref{cor.Nabla.span}.]\textbf{(a)} Proposition
\ref{prop.IJ.1} \textbf{(g)} (applied to $k=2$) yields%
\begin{align}
\mathcal{I}_{2}  &  =\operatorname*{span}\left\{  \nabla_{\mathbf{B}%
,\mathbf{A}}\ \mid\ \mathbf{A},\mathbf{B}\in\operatorname*{SD}\left(
n\right)  \text{ with }\ell\left(  \mathbf{A}\right)  =\ell\left(
\mathbf{B}\right)  =2\right\} \nonumber\\
&  =\operatorname*{span}\left\{  \nabla_{\left(  B,\left[  n\right]  \setminus
B\right)  ,\left(  A,\left[  n\right]  \setminus A\right)  }\ \mid
\ A,B\subseteq\left[  n\right]  \right\}  \label{pf.cor.Nabla.span.1}%
\end{align}
(since the set decompositions $\mathbf{A}\in\operatorname*{SD}\left(
n\right)  $ with $\ell\left(  \mathbf{A}\right)  =2$ are precisely the pairs
of the form $\left(  A,\left[  n\right]  \setminus A\right)  $ for
$A\subseteq\left[  n\right]  $). But Proposition \ref{prop.row.rook} allows us
to rewrite the $\nabla_{\left(  B,\left[  n\right]  \setminus B\right)
,\left(  A,\left[  n\right]  \setminus A\right)  }$ on the right hand side
here as $\nabla_{B,A}$. Thus, (\ref{pf.cor.Nabla.span.1}) rewrites as
$\mathcal{I}_{2}=\operatorname*{span}\left\{  \nabla_{B,A}\ \mid
\ A,B\subseteq\left[  n\right]  \right\}  $. This proves part \textbf{(a)}.
\medskip

\textbf{(b)} Theorem \ref{thm.row.main} \textbf{(c)} (applied to $k=2$) shows
that the $\mathbf{k}$-module $\mathcal{I}_{2}$ is free of rank $\left\vert
\operatorname*{Av}\nolimits_{n}\left(  3\right)  \right\vert $. But it is
known (see, e.g., \cite[Corollary 4.8]{Bona22}) that $\left\vert
\operatorname*{Av}\nolimits_{n}\left(  3\right)  \right\vert =C_{n}$.
Corollary \ref{cor.Nabla.span} \textbf{(b)} follows from these two observations.
\end{proof}

\subsection{Opposite avoidance}

Parts of Theorem \ref{thm.row.main} can be transformed into \textquotedblleft
twin\textquotedblright\ forms by replacing the relevant permutations with
their complements (i.e., multiplying them with the permutation with one-line
notation $\left(  n,n-1,\ldots,1\right)  $). To state this, we need an
analogue of Definition \ref{def.avoid.up}:

\begin{definition}
\label{def.avoid.down}Let $k\in\mathbb{N}$.

\begin{enumerate}
\item[\textbf{(a)}] Let $w\in S_{n}$ be a permutation. We say that $w$
\emph{avoids }$\left(  k+1\right)  k\cdots1$ if there exists no $\left(
k+1\right)  $-element subset $U$ of $\left[  n\right]  $ such that the
restriction $w\mid_{U}$ is decreasing (i.e., if there exist no $k+1$ elements
$i_{1}<i_{2}<\cdots<i_{k+1}$ of $\left[  n\right]  $ such that $w\left(
i_{1}\right)  >w\left(  i_{2}\right)  >\cdots>w\left(  i_{k+1}\right)  $).

\item[\textbf{(b)}] We let $\operatorname*{Av}\nolimits_{n}^{\prime}\left(
k+1\right)  $ denote the set of all permutations $w\in S_{n}$ that avoid
$\left(  k+1\right)  k\cdots1$.
\end{enumerate}
\end{definition}

\begin{proposition}
\label{prop.avoid.up-down}Let $w_{0}\in S_{n}$ be the permutation that sends
each $i\in\left[  n\right]  $ to $n+1-i$. Let $k\in\mathbb{N}$ and $w\in
S_{n}$. Then, $w\in\operatorname*{Av}\nolimits_{n}^{\prime}\left(  k+1\right)
$ if and only if $w_{0}w\in\operatorname*{Av}\nolimits_{n}\left(  k+1\right)
$.
\end{proposition}

\begin{proof}
Given a subset $U$ of $\left[  n\right]  $, it is clear that the restriction
$\left(  w_{0}w\right)  \mid_{U}$ is increasing if and only if the
corresponding restriction $w\mid_{U}$ is decreasing (since the permutation
$w_{0}$ is strictly decreasing). Hence, $w_{0}w$ avoids $12\cdots\left(
k+1\right)  $ if and only if $w$ avoids $\left(  k+1\right)  k\cdots1$. Thus,
Proposition \ref{prop.avoid.up-down} follows.
\end{proof}

The set $\operatorname*{Av}\nolimits_{n}^{\prime}\left(  k+1\right)  $ allows
us to easily obtain a \textquotedblleft twin\textquotedblright\ to parts
\textbf{(c)}--\textbf{(f)} of Theorem \ref{thm.row.main}:

\begin{corollary}
\label{cor.row.twin}Let $k\in\mathbb{N}$. Then:

\begin{enumerate}
\item[\textbf{(a)}] The $\mathbf{k}$-module $\mathcal{I}_{k}$ is free of rank
$\left\vert \operatorname*{Av}\nolimits_{n}^{\prime}\left(  k+1\right)
\right\vert $.

\item[\textbf{(b)}] The $\mathbf{k}$-module $\mathcal{J}_{k}$ is free of rank
$\left\vert S_{n}\setminus\operatorname*{Av}\nolimits_{n}^{\prime}\left(
k+1\right)  \right\vert $.

\item[\textbf{(c)}] The $\mathbf{k}$-module $\mathcal{A}/\mathcal{I}_{k}$ is
free with basis $\left(  \overline{w}\right)  _{w\in S_{n}\setminus
\operatorname*{Av}\nolimits_{n}^{\prime}\left(  k+1\right)  }$. (Here,
$\overline{w}$ denotes the projection of $w\in\mathcal{A}$ onto the quotient
$\mathcal{A}/\mathcal{I}_{k}$.)

\item[\textbf{(d)}] The $\mathbf{k}$-module $\mathcal{A}/\mathcal{J}_{k}$ is
free with basis $\left(  \overline{w}\right)  _{w\in\operatorname*{Av}%
\nolimits_{n}^{\prime}\left(  k+1\right)  }$. (Here, $\overline{w}$ denotes
the projection of $w\in\mathcal{A}$ onto the quotient $\mathcal{A}%
/\mathcal{J}_{k}$.)
\end{enumerate}
\end{corollary}

\begin{proof}
\textbf{(a)} Proposition \ref{prop.avoid.up-down} shows that the map%
\begin{align*}
S_{n}  &  \rightarrow S_{n},\\
w  &  \mapsto w_{0}w
\end{align*}
restricts to a bijection from $\operatorname*{Av}\nolimits_{n}^{\prime}\left(
k+1\right)  $ to $\operatorname*{Av}\nolimits_{n}\left(  k+1\right)  $.
Hence,
\begin{equation}
\left\vert \operatorname*{Av}\nolimits_{n}\left(  k+1\right)  \right\vert
=\left\vert \operatorname*{Av}\nolimits_{n}^{\prime}\left(  k+1\right)
\right\vert . \label{pf.cor.row.twin.a.=}%
\end{equation}
But Theorem \ref{thm.row.main} \textbf{(c)} shows that the $\mathbf{k}$-module
$\mathcal{I}_{k}$ is free of rank $\left\vert \operatorname*{Av}%
\nolimits_{n}\left(  k+1\right)  \right\vert =\left\vert \operatorname*{Av}%
\nolimits_{n}^{\prime}\left(  k+1\right)  \right\vert $. This proves Corollary
\ref{cor.row.twin} \textbf{(a)}. \medskip

\textbf{(b)} Analogous to Corollary \ref{cor.row.twin} \textbf{(a)}, but using
Theorem \ref{thm.row.main} \textbf{(d)} instead of Theorem \ref{thm.row.main}
\textbf{(c)}. \medskip

\textbf{(c)} Proposition \ref{prop.avoid.up-down} shows that a permutation
$w\in S_{n}$ satisfies $w\in\operatorname*{Av}\nolimits_{n}^{\prime}\left(
k+1\right)  $ if and only if $w_{0}w\in\operatorname*{Av}\nolimits_{n}\left(
k+1\right)  $. Thus, the contrapositive also holds: A permutation $w\in S_{n}$
satisfies $w\notin\operatorname*{Av}\nolimits_{n}^{\prime}\left(  k+1\right)
$ if and only if $w_{0}w\notin\operatorname*{Av}\nolimits_{n}\left(
k+1\right)  $. Therefore, the map%
\begin{align}
S_{n}\setminus\operatorname*{Av}\nolimits_{n}^{\prime}\left(  k+1\right)   &
\rightarrow S_{n}\setminus\operatorname*{Av}\nolimits_{n}\left(  k+1\right)
,\nonumber\\
w  &  \mapsto w_{0}w \label{pf.cor.row.twin.c.b1}%
\end{align}
is a bijection. The inverse of this map must also send each $w$ to $w_{0}w$
(because $w_{0}$ is an involution, i.e., we have $w_{0}w_{0}%
=\operatorname*{id}$). Thus, this inverse is the map%
\begin{align}
S_{n}\setminus\operatorname*{Av}\nolimits_{n}\left(  k+1\right)   &
\rightarrow S_{n}\setminus\operatorname*{Av}\nolimits_{n}^{\prime}\left(
k+1\right)  ,\nonumber\\
w  &  \mapsto w_{0}w. \label{pf.cor.row.twin.c.b2}%
\end{align}
As a consequence, this latter map is a bijection.

Theorem \ref{thm.row.main} \textbf{(e)} shows that the $\mathbf{k}$-module
$\mathcal{A}/\mathcal{I}_{k}$ is free with basis $\left(  \overline{w}\right)
_{w\in S_{n}\setminus\operatorname*{Av}\nolimits_{n}\left(  k+1\right)  }$.
But $\mathcal{I}_{k}$ is a left ideal of $\mathcal{A}$ and thus fixed under
left multiplication by $w_{0}$. Hence, the map%
\begin{align*}
\mathcal{A}/\mathcal{I}_{k}  &  \rightarrow\mathcal{A}/\mathcal{I}_{k},\\
\overline{\mathbf{a}}  &  \mapsto\overline{w_{0}\mathbf{a}}%
\end{align*}
is well-defined and is an automorphism of the $\mathbf{k}$-module
$\mathcal{A}/\mathcal{I}_{k}$ (being invertible because $w_{0}^{2}%
=\operatorname*{id}$). Applying this map to the basis $\left(  \overline
{w}\right)  _{w\in S_{n}\setminus\operatorname*{Av}\nolimits_{n}\left(
k+1\right)  }$ of $\mathcal{A}/\mathcal{I}_{k}$, we thus obtain a new basis
$\left(  \overline{w_{0}w}\right)  _{w\in S_{n}\setminus\operatorname*{Av}%
\nolimits_{n}\left(  k+1\right)  }$ of $\mathcal{A}/\mathcal{I}_{k}$ (since
the image of a basis under a $\mathbf{k}$-module isomorphism is again a
basis). But the latter basis $\left(  \overline{w_{0}w}\right)  _{w\in
S_{n}\setminus\operatorname*{Av}\nolimits_{n}\left(  k+1\right)  }$ can be
reindexed as $\left(  \overline{w}\right)  _{w\in S_{n}\setminus
\operatorname*{Av}\nolimits_{n}^{\prime}\left(  k+1\right)  }$, since the map
(\ref{pf.cor.row.twin.c.b2}) is a bijection. Hence, we have shown that
$\left(  \overline{w}\right)  _{w\in S_{n}\setminus\operatorname*{Av}%
\nolimits_{n}^{\prime}\left(  k+1\right)  }$ is a basis of the $\mathbf{k}%
$-module $\mathcal{A}/\mathcal{I}_{k}$. This proves part \textbf{(c)}.
\medskip

\textbf{(d)} Analogous to Corollary \ref{cor.row.twin} \textbf{(c)}, but using
Theorem \ref{thm.row.main} \textbf{(f)} instead of Theorem \ref{thm.row.main}
\textbf{(e)}.
\end{proof}

\subsection{\label{sec.row.ann}$\mathcal{J}_{k}$ and $\mathcal{I}_{n-k-1}$ as
annihilators of tensor modules}

We shall now discuss how the ideals $\mathcal{J}_{k}$ and $\mathcal{I}_{k}$
(more precisely, $\mathcal{I}_{n-k-1}$, but this is just a matter of indexing)
can be interpreted as annihilators of certain left $\mathcal{A}$-modules. This
breaks no new ground, but rather recovers results by de Concini and Procesi
\cite[Theorem 4.2]{deCPro76} and Bowman, Doty and Martin \cite{BoDoMa22}; we
hope that our elementary approach makes these results more accessible.

\subsubsection{$\mathcal{J}_{k}$ as annihilator of $V_{k}^{\otimes n}$ (action
on places)}

Recall that $\mathcal{A}=\mathbf{k}\left[  S_{n}\right]  $. Thus, the
representations of $S_{n}$ over $\mathbf{k}$ are precisely the left
$\mathcal{A}$-modules. If $M$ is any left $\mathcal{A}$-module, then
$\operatorname*{Ann}M$ shall denote its annihilator, i.e., the ideal
\[
\left\{  a\in\mathcal{A}\ \mid\ aM=0\right\}  =\left\{  a\in\mathcal{A}%
\ \mid\ am=0\text{ for all }m\in M\right\}
\]
of $\mathcal{A}$. Note that if $M$ is a left ideal of $\mathcal{A}$, then this
annihilator $\operatorname*{Ann}M$ is just its left annihilator
$\operatorname*{LAnn}M$.

For each $k\in\mathbb{N}$, we consider the free $\mathbf{k}$-module $V_{k}$
with basis $\left(  e_{1},e_{2},\ldots,e_{k}\right)  $. On its $n$-th tensor
power $V_{k}^{\otimes n}$, the symmetric group $S_{n}$ acts by permuting the
tensor factors:%
\begin{align*}
\sigma\cdot\left(  v_{1}\otimes v_{2}\otimes\cdots\otimes v_{n}\right)   &
=v_{\sigma^{-1}\left(  1\right)  }\otimes v_{\sigma^{-1}\left(  2\right)
}\otimes\cdots\otimes v_{\sigma^{-1}\left(  n\right)  }\\
&  \ \ \ \ \ \ \ \ \ \ \text{for all }\sigma\in S_{n}\text{ and }v_{1}%
,v_{2},\ldots,v_{n}\in V_{k}.
\end{align*}
Thus, $V_{k}^{\otimes n}$ is a left $\mathcal{A}$-module for each
$k\in\mathbb{N}$. This action of $\mathcal{A}$ (or of $S_{n}$) is called
\emph{action on places} or \emph{action by place permutation}.

The $\mathbf{k}$-module $V_{k}^{\otimes n}$ can also be identified with the
$\mathbf{k}$-module of homogeneous polynomials of degree $n$ in $k$
noncommutative indeterminates $x_{1},x_{2},\ldots,x_{k}$ (via the isomorphism
that sends each pure tensor $e_{i_{1}}\otimes e_{i_{2}}\otimes\cdots\otimes
e_{i_{n}}$ to the noncommutative monomial $x_{i_{1}}x_{i_{2}}\cdots x_{i_{n}}%
$). Under this identification, the action of $S_{n}$ permutes the order of factors.

Now we shall prove the following result of de Concini and Procesi
\cite[Theorem 4.2]{deCPro76}:

\begin{theorem}
\label{thm.AnnVkn}Let $k\in\mathbb{N}$. Then,%
\[
\mathcal{J}_{k}=\operatorname*{Ann}\left(  V_{k}^{\otimes n}\right)  .
\]

\end{theorem}

\begin{proof}
We say that two set decompositions $\mathbf{A}$ and $\mathbf{B}$ of $\left[
n\right]  $ are \emph{equal-shaped} if they satisfy $\ell\left(
\mathbf{A}\right)  =\ell\left(  \mathbf{B}\right)  $ and if each block of
$\mathbf{B}$ has the same size as the corresponding block of $\mathbf{A}$. In
other words, two set decompositions $\left(  A_{1},A_{2},\ldots,A_{p}\right)
$ and $\left(  B_{1},B_{2},\ldots,B_{q}\right)  $ of $\left[  n\right]  $ are
\emph{equal-shaped} if and only if $p=q$ and $\left\vert A_{i}\right\vert
=\left\vert B_{i}\right\vert $ for each $i\in\left[  p\right]  =\left[
q\right]  $.

An equivalent characterization of equal-shapedness is the following: Two set
decompositions $\mathbf{A}$ and $\mathbf{B}$ of $\left[  n\right]  $ are
equal-shaped if and only if there exists some $w\in S_{n}$ such that
$w\mathbf{A}=\mathbf{B}$. (Indeed, the \textquotedblleft if\textquotedblright%
\ part is obvious, whereas the \textquotedblleft only if\textquotedblright%
\ part is easily shown by choosing a permutation $w\in S_{n}$ that sends each
block of $\mathbf{A}$ to the respective block of $\mathbf{B}$.)

Recall that $\operatorname*{SD}\left(  n\right)  $ is the set of all set
decompositions of $\left[  n\right]  $. For each $\mathbf{A}\in
\operatorname*{SD}\left(  n\right)  $, we define a $\mathbf{k}$-submodule%
\[
\mathcal{I}_{\mathbf{A}}:=\operatorname*{span}\left\{  \nabla_{\mathbf{B}%
,\mathbf{A}}\ \mid\ \mathbf{B}\in\operatorname*{SD}\left(  n\right)  \text{
with }\ell\left(  \mathbf{A}\right)  =\ell\left(  \mathbf{B}\right)  \right\}
\]
of $\mathcal{A}$. This $\mathbf{k}$-submodule $\mathcal{I}_{\mathbf{A}}$ is
actually a left ideal of $\mathcal{A}$, since (\ref{eq.uNabv}) yields%
\begin{equation}
w\nabla_{\mathbf{B},\mathbf{A}}=\nabla_{w\mathbf{B},\mathbf{A}}
\label{pf.thm.AnnVkn.wNab}%
\end{equation}
for all $w\in S_{n}$ and $\mathbf{B}\in\operatorname*{SD}\left(  n\right)  $
with $\ell\left(  \mathbf{A}\right)  =\ell\left(  \mathbf{B}\right)  $. It is
easy to see that as a $\mathbf{k}$-module, $\mathcal{I}_{\mathbf{A}}$ has a
basis consisting of the elements $\nabla_{\mathbf{B},\mathbf{A}}$, where
$\mathbf{B}$ ranges over all those set decompositions of $\left[  n\right]  $
that are equal-shaped to $\mathbf{A}$. (Indeed, these elements span
$\mathcal{I}_{\mathbf{A}}$ because Proposition \ref{prop.row.simple}
\textbf{(a)} shows that all other $\nabla_{\mathbf{B},\mathbf{A}}$'s are $0$.
And they are linearly independent because they are nonzero and have disjoint
supports, i.e., because they don't have any addends in common.)

It is furthermore clear from Proposition \ref{prop.IJ.1} \textbf{(d)} that%
\begin{equation}
\mathcal{I}_{k}=\sum_{\substack{\mathbf{A}\in\operatorname*{SD}\left(
n\right)  ;\\\ell\left(  \mathbf{A}\right)  \leq k}}\mathcal{I}_{\mathbf{A}}.
\label{pf.thm.AnnVkn.Ik=sum}%
\end{equation}
Hence,%
\begin{equation}
\operatorname*{Ann}\left(  \mathcal{I}_{k}\right)  =\bigcap
_{\substack{\mathbf{A}\in\operatorname*{SD}\left(  n\right)  ;\\\ell\left(
\mathbf{A}\right)  \leq k}}\operatorname*{Ann}\left(  \mathcal{I}_{\mathbf{A}%
}\right)  \label{pf.thm.AnnVkn.Ik=sum2}%
\end{equation}
(since the annihilator of a sum of $\mathcal{A}$-submodules is the
intersection of their individual annihilators).

On the other hand, let $\mathbf{A}=\left(  A_{1},A_{2},\ldots,A_{m}\right)
\in\operatorname*{SD}\left(  n\right)  $ be a set decomposition. Then, we can
define a diagram $D\left(  \mathbf{A}\right)  $ of size $n$ (see
\cite[Definition 5.1.2]{sga} for the notion of a diagram) by setting%
\[
D\left(  \mathbf{A}\right)  :=\bigcup_{i=1}^{m}\left\{  \left(  i,1\right)
,\ \left(  i,2\right)  ,\ \ldots,\ \left(  i,\left\vert A_{i}\right\vert
\right)  \right\}  .
\]
This is almost a Young diagram, except that its shape is the weak composition
$\left(  \left\vert A_{1}\right\vert ,\left\vert A_{2}\right\vert
,\ldots,\left\vert A_{m}\right\vert \right)  $ instead of a partition. We
consider the Young module $\mathcal{M}^{D\left(  \mathbf{A}\right)  }$, which
is a left $\mathcal{A}$-module (i.e., an $S_{n}$-representation) that (as a
$\mathbf{k}$-module) has basis $\left\{  n\text{-tabloids of shape }D\left(
\mathbf{A}\right)  \right\}  $ (see \cite[\S 5.3]{sga} for the relevant
definitions). If $\overline{T}$ is an $n$-tabloid of shape $D\left(
\mathbf{A}\right)  $, and if $i\in\left[  m\right]  $, then
$\operatorname*{Row}\left(  i,\overline{T}\right)  $ shall denote the set of
all entries in the $i$-th row of $\overline{T}$. Conversely, if $\overline{T}$
is an $n$-tabloid of shape $D\left(  \mathbf{A}\right)  $, and if $k\in\left[
n\right]  $, then $r_{\overline{T}}\left(  k\right)  $ shall denote the number
of the row of $\overline{T}$ that contains the entry $k$.

It is easy to see that if $\ell\left(  \mathbf{A}\right)  \leq k$, then the
$\mathbf{k}$-linear map%
\begin{align*}
\omega_{\mathbf{A}}:\mathcal{M}^{D\left(  \mathbf{A}\right)  }  &  \rightarrow
V_{k}^{\otimes n},\\
\overline{T}  &  \mapsto e_{r_{\overline{T}}\left(  1\right)  }\otimes
e_{r_{\overline{T}}\left(  2\right)  }\otimes\cdots\otimes e_{r_{\overline{T}%
}\left(  n\right)  }%
\end{align*}
is an injective left $\mathcal{A}$-module morphism\footnote{Indeed,
\par
\begin{itemize}
\item this map is a left $\mathcal{A}$-module morphism because any $n$-tabloid
$\overline{T}$ and any permutation $w\in S_{n}$ satisfy $r_{\overline{wT}%
}\left(  i\right)  =r_{\overline{T}}\left(  w^{-1}\left(  i\right)  \right)  $
for all $i\in\left[  n\right]  $;
\par
\item this map is injective because an $n$-tabloid $\overline{T}$ of a given
shape (in our case, $D\left(  \mathbf{A}\right)  $) is uniquely determined by
the $n$-tuple $\left(  r_{\overline{T}}\left(  1\right)  ,r_{\overline{T}%
}\left(  2\right)  ,\ldots,r_{\overline{T}}\left(  n\right)  \right)  $.
\end{itemize}
}.

On the other hand, the $\mathbf{k}$-linear map%
\begin{align*}
\gamma_{\mathbf{A}}:\mathcal{M}^{D\left(  \mathbf{A}\right)  }  &
\rightarrow\mathcal{I}_{\mathbf{A}},\\
\overline{T}  &  \mapsto\nabla_{\left(  \operatorname*{Row}\left(
1,\overline{T}\right)  ,\ \operatorname*{Row}\left(  2,\overline{T}\right)
,\ \ldots,\ \operatorname*{Row}\left(  m,\overline{T}\right)  \right)
,\ \mathbf{A}}%
\end{align*}
is a left $\mathcal{A}$-module morphism as well\footnote{Indeed, for any
$n$-tabloid $\overline{T}$ of shape $D\left(  \mathbf{A}\right)  $ and any
permutation $w\in S_{n}$, we have
\begin{align*}
\gamma_{\mathbf{A}}\left(  w\overline{T}\right)   &  =\gamma_{\mathbf{A}%
}\left(  \overline{wT}\right)  =\nabla_{\left(  \operatorname*{Row}\left(
1,\overline{wT}\right)  ,\ \operatorname*{Row}\left(  2,\overline{wT}\right)
,\ \ldots,\ \operatorname*{Row}\left(  m,\overline{wT}\right)  \right)
,\ \mathbf{A}}\\
&  =\nabla_{w\left(  \operatorname*{Row}\left(  1,\overline{T}\right)
,\ \operatorname*{Row}\left(  2,\overline{T}\right)  ,\ \ldots
,\ \operatorname*{Row}\left(  m,\overline{T}\right)  \right)  ,\ \mathbf{A}}\\
&  \ \ \ \ \ \ \ \ \ \ \ \ \ \ \ \ \ \ \ \ \left(
\begin{array}
[c]{c}%
\text{since }\operatorname*{Row}\left(  i,\overline{wT}\right)  =w\left(
\operatorname*{Row}\left(  i,\overline{T}\right)  \right)  \text{ for each
}i\in\left[  m\right]  \text{,}\\
\text{and thus }\left(  \operatorname*{Row}\left(  1,\overline{wT}\right)
,\ \operatorname*{Row}\left(  2,\overline{wT}\right)  ,\ \ldots
,\ \operatorname*{Row}\left(  m,\overline{wT}\right)  \right) \\
=w\left(  \operatorname*{Row}\left(  1,\overline{T}\right)
,\ \operatorname*{Row}\left(  2,\overline{T}\right)  ,\ \ldots
,\ \operatorname*{Row}\left(  m,\overline{T}\right)  \right)
\end{array}
\right) \\
&  =w\underbrace{\nabla_{\left(  \operatorname*{Row}\left(  1,\overline
{T}\right)  ,\ \operatorname*{Row}\left(  2,\overline{T}\right)
,\ \ldots,\ \operatorname*{Row}\left(  m,\overline{T}\right)  \right)
,\ \mathbf{A}}}_{=\gamma_{\mathbf{A}}\left(  \overline{T}\right)
}\ \ \ \ \ \ \ \ \ \ \left(  \text{by (\ref{pf.thm.AnnVkn.wNab})}\right) \\
&  =w\gamma_{\mathbf{A}}\left(  \overline{T}\right)  .
\end{align*}
}. Next we shall show that this morphism $\gamma_{\mathbf{A}}$ is invertible.
Indeed, recall that the $\mathbf{k}$-module $\mathcal{I}_{\mathbf{A}}$ has a
basis consisting of the elements $\nabla_{\mathbf{B},\mathbf{A}}$, where
$\mathbf{B}$ ranges over all those set decompositions of $\left[  n\right]  $
that are equal-shaped to $\mathbf{A}$. In other words,%
\begin{equation}
\left(  \nabla_{\mathbf{B},\mathbf{A}}\right)  _{\mathbf{B}\text{ is a set
decomposition of }\left[  n\right]  \text{ that is equal-shaped to }%
\mathbf{A}} \label{pf.thm.AnnVkn.bas1}%
\end{equation}
is a basis of the $\mathbf{k}$-module $\mathcal{I}_{\mathbf{A}}$. But there is
a bijection%
\begin{align*}
&  \left\{  n\text{-tabloids of shape }D\left(  \mathbf{A}\right)  \right\} \\
&  \rightarrow\left\{  \text{set decompositions }\mathbf{B}\text{ of }\left[
n\right]  \text{ that are equal-shaped to }\mathbf{A}\right\}
\end{align*}
that sends each $n$-tabloid $\overline{T}$ to $\left(  \operatorname*{Row}%
\left(  1,\overline{T}\right)  ,\ \operatorname*{Row}\left(  2,\overline
{T}\right)  ,\ \ldots,\ \operatorname*{Row}\left(  m,\overline{T}\right)
\right)  $ (indeed, the lengths of the rows of $D\left(  \mathbf{A}\right)  $
are the sizes of the respective blocks of $\mathbf{A}$, and thus are exactly
the right size to fit the blocks of a set decomposition $\mathbf{B}$ that is
equal-shaped to $\mathbf{A}$). Hence, we can reindex our basis
(\ref{pf.thm.AnnVkn.bas1}) of $\mathcal{I}_{\mathbf{A}}$ as
\begin{equation}
\left(  \nabla_{\left(  \operatorname*{Row}\left(  1,\overline{T}\right)
,\ \operatorname*{Row}\left(  2,\overline{T}\right)  ,\ \ldots
,\ \operatorname*{Row}\left(  m,\overline{T}\right)  \right)  ,\ \mathbf{A}%
}\right)  _{\overline{T}\text{ is an }n\text{-tabloid of shape }D\left(
\mathbf{A}\right)  }. \label{pf.thm.AnnVkn.bas2}%
\end{equation}
Now, the $\mathbf{k}$-linear map $\gamma_{\mathbf{A}}$ sends the basis
$\left(  \overline{T}\right)  _{\overline{T}\text{ is an }n\text{-tabloid of
shape }D\left(  \mathbf{A}\right)  }$ of $\mathcal{M}^{D\left(  \mathbf{A}%
\right)  }$ to the basis (\ref{pf.thm.AnnVkn.bas2}) of $\mathcal{I}%
_{\mathbf{A}}$ (by the definition of $\gamma_{\mathbf{A}}$). Thus, it is
invertible (since any $\mathbf{k}$-linear map that sends a basis of its domain
to a basis of its target must be invertible). Therefore, it has an inverse
$\gamma_{\mathbf{A}}^{-1}:\mathcal{I}_{\mathbf{A}}\rightarrow\mathcal{M}%
^{D\left(  \mathbf{A}\right)  }$, which is also a left $\mathcal{A}$-module
isomorphism (since $\gamma_{\mathbf{A}}$ is a left $\mathcal{A}$-module morphism).

Composing this inverse with the injective left $\mathcal{A}$-module morphism
$\omega_{\mathbf{A}}:\mathcal{M}^{D\left(  \mathbf{A}\right)  }\rightarrow
V_{k}^{\otimes n}$ (when $\ell\left(  \mathbf{A}\right)  \leq k$), we obtain
the injective left $\mathcal{A}$-module morphism%
\[
\psi_{\mathbf{A}}:=\omega_{\mathbf{A}}\circ\gamma_{\mathbf{A}}^{-1}%
:\mathcal{I}_{\mathbf{A}}\rightarrow V_{k}^{\otimes n}.
\]

Forget that we fixed $\mathbf{A}$. We thus have found an injective left
$\mathcal{A}$-module morphism
\[
\psi_{\mathbf{A}}:\mathcal{I}_{\mathbf{A}}\rightarrow V_{k}^{\otimes n}%
\]
for each set decomposition $\mathbf{A}\in\operatorname*{SD}\left(  n\right)  $
with $\ell\left(  \mathbf{A}\right)  \leq k$.

It is easy to see that%
\begin{equation}
V_{k}^{\otimes n}=\sum_{\substack{\mathbf{A}\in\operatorname*{SD}\left(
n\right)  ;\\\ell\left(  \mathbf{A}\right)  \leq k}}\psi_{\mathbf{A}}\left(
\mathcal{I}_{\mathbf{A}}\right)  \label{pf.thm.AnnVkn.Vkn=sum}%
\end{equation}
(indeed, each basis vector $e_{i_{1}}\otimes e_{i_{2}}\otimes\cdots\otimes
e_{i_{n}}$ of $V_{k}^{\otimes n}$ can be written as the image of an
appropriate $n$-tabloid $\overline{T}$ under the map $\omega_{\mathbf{A}}$
(where $\mathbf{A}\in\operatorname*{SD}\left(  n\right)  $ is the set
decomposition $\left(  A_{1},A_{2},\ldots,A_{k}\right)  $ defined by
$A_{j}:=\left\{  p\in\left[  n\right]  \ \mid\ i_{p}=j\right\}  $), and thus
also as an image under $\omega_{\mathbf{A}}\circ\gamma_{\mathbf{A}}^{-1}%
=\psi_{\mathbf{A}}$). Hence,%
\begin{equation}
\operatorname*{Ann}\left(  V_{k}^{\otimes n}\right)  =\bigcap
_{\substack{\mathbf{A}\in\operatorname*{SD}\left(  n\right)  ;\\\ell\left(
\mathbf{A}\right)  \leq k}}\operatorname*{Ann}\left(  \psi_{\mathbf{A}}\left(
\mathcal{I}_{\mathbf{A}}\right)  \right)  \label{pf.thm.AnnVkn.Vkn=sum2}%
\end{equation}
(since the annihilator of a sum of $\mathcal{A}$-submodules is the
intersection of their individual annihilators). But each $\mathbf{A}%
\in\operatorname*{SD}\left(  n\right)  $ satisfying $\ell\left(
\mathbf{A}\right)  \leq k$ must satisfy $\psi_{\mathbf{A}}\left(
\mathcal{I}_{\mathbf{A}}\right)  \cong\mathcal{I}_{\mathbf{A}}$ as left
$\mathcal{A}$-modules (since the left $\mathcal{A}$-module morphism
$\psi_{\mathbf{A}}$ is injective) and thus $\operatorname*{Ann}\left(
\psi_{\mathbf{A}}\left(  \mathcal{I}_{\mathbf{A}}\right)  \right)
=\operatorname*{Ann}\left(  \mathcal{I}_{\mathbf{A}}\right)  $. Hence,
(\ref{pf.thm.AnnVkn.Vkn=sum2}) rewrites as%
\[
\operatorname*{Ann}\left(  V_{k}^{\otimes n}\right)  =\bigcap
_{\substack{\mathbf{A}\in\operatorname*{SD}\left(  n\right)  ;\\\ell\left(
\mathbf{A}\right)  \leq k}}\operatorname*{Ann}\left(  \mathcal{I}_{\mathbf{A}%
}\right)  .
\]
Comparing this with (\ref{pf.thm.AnnVkn.Ik=sum2}), we obtain%
\[
\operatorname*{Ann}\left(  V_{k}^{\otimes n}\right)  =\operatorname*{Ann}%
\left(  \mathcal{I}_{k}\right)  =\operatorname*{LAnn}\left(  \mathcal{I}%
_{k}\right)
\]
(since $\operatorname*{Ann}M=\operatorname*{LAnn}M$ for any left ideal $M$ of
$\mathcal{A}$). Since Theorem \ref{thm.row.main} \textbf{(b)} yields
$\mathcal{J}_{k}=\operatorname*{LAnn}\left(  \mathcal{I}_{k}\right)  $, we can
rewrite this as $\operatorname*{Ann}\left(  V_{k}^{\otimes n}\right)
=\mathcal{J}_{k}$. This proves Theorem \ref{thm.AnnVkn}.
\end{proof}

\begin{corollary}
\label{cor.actonVkn}Let $k\in\mathbb{N}$. Consider the action of $\mathcal{A}$
on $V_{k}^{\otimes n}$ as a $\mathbf{k}$-algebra morphism $\rho_{k}%
:\mathcal{A}\rightarrow\operatorname*{End}\nolimits_{\mathbf{k}}\left(
V_{k}^{\otimes n}\right)  $. Then, the image $\rho_{k}\left(  \mathcal{A}%
\right)  $ of this map $\rho_{k}$ has two bases $\left(  \rho_{k}\left(
w\right)  \right)  _{w\in\operatorname*{Av}\nolimits_{n}\left(  k+1\right)  }$
and $\left(  \rho_{k}\left(  w\right)  \right)  _{w\in\operatorname*{Av}%
\nolimits_{n}^{\prime}\left(  k+1\right)  }$ (as a $\mathbf{k}$-module).
\end{corollary}

\begin{proof}
We have $\operatorname*{Ker}\rho_{k}=\operatorname*{Ann}\left(  V_{k}^{\otimes
n}\right)  =\mathcal{J}_{k}$ by Theorem \ref{thm.AnnVkn}. Hence, by the first
isomorphism theorem, there is a $\mathbf{k}$-algebra isomorphism
$\mathcal{A}/\mathcal{J}_{k}\rightarrow\rho_{k}\left(  \mathcal{A}\right)  $
that sends each residue class $\overline{a}$ to $\rho_{k}\left(  a\right)  $.
Hence, we only need to show that the quotient $\mathbf{k}$-module
$\mathcal{A}/\mathcal{J}_{k}$ has two bases $\left(  \overline{w}\right)
_{w\in\operatorname*{Av}\nolimits_{n}\left(  k+1\right)  }$ and $\left(
\overline{w}\right)  _{w\in\operatorname*{Av}\nolimits_{n}^{\prime}\left(
k+1\right)  }$. But this is Theorem \ref{thm.row.main} \textbf{(f)} and
Corollary \ref{cor.row.twin} \textbf{(d)}.
\end{proof}

Theorem \ref{thm.AnnVkn} is \cite[Theorem 4.2]{deCPro76}, and also appears in
\cite[Example 2.11]{Donkin24}. It is also a particular case of H\"{a}rterich's
\cite[\S 3]{Harter99}. Indeed, H\"{a}rterich considers a sign-twisted version
of the $S_{n}$-representation $V_{k}^{\otimes n}$. The \emph{sign-twist} is
the $\mathbf{k}$-linear map $T_{\operatorname*{sign}}:\mathcal{A}%
\rightarrow\mathcal{A}$ that sends each permutation $w\in S_{n}$ to $\left(
-1\right)  ^{w}w$. This map is a $\mathbf{k}$-algebra automorphism (see
\cite[Theorem 3.11.5]{sga}), so it transforms any $S_{n}$-representation $M$
into a new $S_{n}$-representation $M^{\operatorname*{sign}}$, which is called
the \emph{sign-twist} of $M$ (see \cite[Definition 5.18.2]{sga} for the
precise definition) and which satisfies
\begin{equation}
\operatorname*{Ann}\left(  M^{\operatorname*{sign}}\right)
=T_{\operatorname*{sign}}\left(  \operatorname*{Ann}M\right)  .
\label{eq.harterich.AnnMsign}%
\end{equation}
H\"{a}rterich's $V^{\otimes n}$ (for $m=k$) is the sign-twist of our $S_{n}%
$-representation $V_{k}^{\otimes n}$ (since his $S_{n}$-action on $V^{\otimes
n}$ is given by permuting the factors and multiplying with the sign of the
permutation). Thus, the ideal $\mathcal{J}=\operatorname*{ann}%
\nolimits_{\mathcal{H}_{n}}V^{\otimes n}$ from \cite[\S 3]{Harter99} for $m=k$
and $q=1$ is the image of our annihilator $\operatorname*{Ann}\left(
V_{k}^{\otimes n}\right)  $ under $T_{\operatorname*{sign}}$ (by
(\ref{eq.harterich.AnnMsign})). Therefore, our Theorem \ref{thm.AnnVkn}
entails that this ideal $\mathcal{J}$ is spanned by the
$T_{\operatorname*{sign}}\left(  \nabla_{\mathbf{B},\mathbf{A}}\right)  $ for
$\mathbf{A},\mathbf{B}\in\operatorname*{SC}\left(  n\right)  $ satisfying
$\ell\left(  \mathbf{A}\right)  =\ell\left(  \mathbf{B}\right)  \leq k$. This
is implicitly what \cite[\S 3]{Harter99} (for $m=k$ and $q=1$) is saying as
well, although the latter theorem additionally picks out a basis of
$\mathcal{J}$ from these spanning elements.

We also note that Corollary \ref{cor.actonVkn} recovers \cite[Theorem
3.3]{Proces21} and also appears in \cite[Corollary 1.4 (i)]{BoDoMa22}, in
\cite[Example 2.11]{Donkin24} and in \cite[Theorem 1]{RaSaSu12}. Implicitly,
\cite[Lemma 8.1 and the remark after it]{BaiRai} also boils down to our
Theorem \ref{thm.AnnVkn}.

The image $\rho_{k}\left(  \mathcal{A}\right)  $ of $\mathcal{A}$ that was
considered in Corollary \ref{cor.actonVkn} is called the \textquotedblleft%
$k$-swap algebra\textquotedblright\ in \cite[\S 1.3]{KlStVo25}. Thus, as we
saw above, this algebra is isomorphic to the quotient algebra $\mathcal{A}%
/\mathcal{J}_{k}$.

\subsubsection{$T_{\operatorname*{sign}}\left(  \mathcal{I}_{n-k-1}\right)  $
as annihilator of $N_{n}^{\otimes k}$ (action on entries)}

An interesting counterpart to Theorem \ref{thm.AnnVkn} was recently proved by
Bowman, Doty and Martin (\cite[Theorem 7.4 (a)]{BoDoMa18}; see also
\cite[Example 2.13]{Donkin24}). We can prove this using our methods, too. Let
us first restate this result in our language:

Recall that the \emph{sign-twist} is the $\mathbf{k}$-linear map
$T_{\operatorname*{sign}}:\mathcal{A}\rightarrow\mathcal{A}$ that sends each
permutation $w\in S_{n}$ to $\left(  -1\right)  ^{w}w$. This is a $\mathbf{k}%
$-algebra automorphism of $\mathcal{A}$.

Consider a free $\mathbf{k}$-module $N_{n}$ with basis $\left(  e_{1}%
,e_{2},\ldots,e_{n}\right)  $, and let the group $S_{n}$ act on it by
permuting the basis vectors (that is, $\sigma\cdot e_{i}=e_{\sigma\left(
i\right)  }$ for all $\sigma\in S_{n}$ and $i\in\left[  n\right]  $). This
$S_{n}$-representation $N_{n}$ is called the \emph{natural representation} of
$S_{n}$.

Let $k\in\mathbb{N}$, and consider the $k$-th tensor power $N_{n}^{\otimes k}$
of this $S_{n}$-representation (where $S_{n}$ acts diagonally, i.e., by the
formula $\sigma\cdot\left(  v_{1}\otimes v_{2}\otimes\cdots\otimes
v_{k}\right)  =\sigma v_{1}\otimes\sigma v_{2}\otimes\cdots\otimes\sigma
v_{k}$ for all $\sigma\in S_{n}$ and all $v_{1},v_{2},\ldots,v_{k}\in V$).
This is again an $S_{n}$-representation, i.e., a left $\mathcal{A}$-module.
The action of $\mathcal{A}$ (or $S_{n}$) on it is called the \emph{action on
entries} (or \emph{diagonal action}).

The $\mathbf{k}$-module $N_{n}^{\otimes k}$ can also be identified with the
$\mathbf{k}$-module of homogeneous polynomials of degree $k$ in $n$
noncommutative indeterminates $x_{1},x_{2},\ldots,x_{n}$ (via the isomorphism
that sends each pure tensor $e_{i_{1}}\otimes e_{i_{2}}\otimes\cdots\otimes
e_{i_{k}}$ to the noncommutative monomial $x_{i_{1}}x_{i_{2}}\cdots x_{i_{k}}%
$). Under this identification, the action of $S_{n}$ permutes the
indeterminates (i.e., a permutation $\sigma\in S_{n}$ sends each $x_{i}$ to
$x_{\sigma\left(  i\right)  }$, and acts on products factor by factor).

Now, \cite[Theorem 7.4 (a)]{BoDoMa18} characterizes the annihilator of this
left $\mathcal{A}$-module $N_{n}^{\otimes k}$ as follows:

\begin{theorem}
\label{thm.AnnNnk}Let $k\in\mathbb{N}$. Then,%
\[
T_{\operatorname*{sign}}\left(  \mathcal{I}_{n-k-1}\right)
=\operatorname*{Ann}\left(  N_{n}^{\otimes k}\right)  .
\]
Here, we understand $\mathcal{I}_{m}$ to mean $0$ when $m<0$ (this is
consistent with Definition \ref{def.IJ}).
\end{theorem}

(Once again, this is written in terms of the Murphy cellular basis in
\cite[Theorem 7.4 (a)]{BoDoMa18}, but boils down to what we just said. The
same claim appears in a different disguise in \cite[Example 2.13]{Donkin24}.)

The proof of this theorem will be enabled by the following two lemmas:

\begin{lemma}
\label{lem.AnnNnk.1}Let $a_{1},a_{2},\ldots,a_{k}\in\left[  n\right]  $. Then,%
\[
\sum_{\substack{w\in S_{n};\\w\left(  a_{i}\right)  =a_{i}\text{ for all }%
i}}w=T_{\operatorname*{sign}}\left(  \nabla_{\left[  n\right]  \setminus
\left\{  a_{1},a_{2},\ldots,a_{k}\right\}  }^{-}\right)  .
\]

\end{lemma}

\begin{proof}
The definition of $\nabla_{\left[  n\right]  \setminus\left\{  a_{1}%
,a_{2},\ldots,a_{k}\right\}  }^{-}$ yields%
\begin{align*}
\nabla_{\left[  n\right]  \setminus\left\{  a_{1},a_{2},\ldots,a_{k}\right\}
}^{-}  &  =\sum_{\substack{w\in S_{n};\\w\left(  i\right)  =i\text{ for all
}i\in\left[  n\right]  \setminus\left(  \left[  n\right]  \setminus\left\{
a_{1},a_{2},\ldots,a_{k}\right\}  \right)  }}\left(  -1\right)  ^{w}w\\
&  =\sum_{\substack{w\in S_{n};\\w\left(  i\right)  =i\text{ for all }%
i\in\left\{  a_{1},a_{2},\ldots,a_{k}\right\}  }}\left(  -1\right)  ^{w}%
w=\sum_{\substack{w\in S_{n};\\w\left(  a_{i}\right)  =a_{i}\text{ for all }%
i}}\left(  -1\right)  ^{w}w.
\end{align*}
Applying the $\mathbf{k}$-linear map $T_{\operatorname*{sign}}$ to this
equality, we obtain%
\[
T_{\operatorname*{sign}}\left(  \nabla_{\left[  n\right]  \setminus\left\{
a_{1},a_{2},\ldots,a_{k}\right\}  }^{-}\right)  =\sum_{\substack{w\in
S_{n};\\w\left(  a_{i}\right)  =a_{i}\text{ for all }i}}\underbrace{\left(
-1\right)  ^{w}T_{\operatorname*{sign}}\left(  w\right)  }%
_{\substack{=w\\\text{(since }T_{\operatorname*{sign}}\left(  w\right)
=\left(  -1\right)  ^{w}w\text{)}}}=\sum_{\substack{w\in S_{n};\\w\left(
a_{i}\right)  =a_{i}\text{ for all }i}}w.
\]
This proves Lemma \ref{lem.AnnNnk.1}.
\end{proof}

\begin{lemma}
\label{lem.AnnNnk.2}Let $k\in\mathbb{N}$. For any $k$-tuple $\mathbf{a}%
=\left(  a_{1},a_{2},\ldots,a_{k}\right)  \in\left[  n\right]  ^{k}$, we
define the element%
\begin{equation}
\nabla_{\mathbf{a}}:=\sum_{\substack{w\in S_{n};\\w\left(  a_{i}\right)
=a_{i}\text{ for all }i}}w\in\mathcal{A}. \label{eq.lem.AnnNnk.2.defNab}%
\end{equation}
Then,
\[
T_{\operatorname*{sign}}\left(  \mathcal{J}_{n-k-1}\right)  =\mathcal{A}%
\cdot\operatorname*{span}\left\{  \nabla_{\mathbf{a}}\ \mid\ \mathbf{a}%
\in\left[  n\right]  ^{k}\right\}  .
\]
Here, we understand $\mathcal{J}_{m}$ to mean $\mathcal{A}$ when $m<0$.
\end{lemma}

\begin{proof}
We are in one of the following two cases:

\textit{Case 1:} We have $k<n$.

\textit{Case 2:} We have $k\geq n$.

Let us first consider Case 1. In this case, $k<n$, so that $n-k-1\in
\mathbb{N}$. Hence, Proposition \ref{prop.IJ.1} \textbf{(b)} (applied to
$n-k-1$ instead of $k$) yields%
\[
\mathcal{J}_{n-k-1}=\mathcal{A}\cdot\operatorname*{span}\left\{  \nabla
_{U}^{-}\ \mid\ U\text{ is a subset of }\left[  n\right]  \text{ having size
}n-k\right\}  .
\]
Applying the $\mathbf{k}$-algebra isomorphism $T_{\operatorname*{sign}%
}:\mathcal{A}\rightarrow\mathcal{A}$ to this equality, we find%
\[
T_{\operatorname*{sign}}\left(  \mathcal{J}_{n-k-1}\right)  =\mathcal{A}%
\cdot\operatorname*{span}\left\{  T_{\operatorname*{sign}}\left(  \nabla
_{U}^{-}\right)  \ \mid\ U\text{ is a subset of }\left[  n\right]  \text{
having size }n-k\right\}  .
\]

However, if $U$ is a subset of $\left[  n\right]  $ having size $n-k$, and if
we denote the $k$ elements of $\left[  n\right]  \setminus U$ by $a_{1}%
,a_{2},\ldots,a_{k}$, then%
\begin{align*}
\nabla_{\left(  a_{1},a_{2},\ldots,a_{k}\right)  }  &  =\sum_{\substack{w\in
S_{n};\\w\left(  a_{i}\right)  =a_{i}\text{ for all }i}%
}w\ \ \ \ \ \ \ \ \ \ \left(  \text{by (\ref{eq.lem.AnnNnk.2.defNab})}\right)
\\
&  =T_{\operatorname*{sign}}\left(  \nabla_{\left[  n\right]  \setminus
\left\{  a_{1},a_{2},\ldots,a_{k}\right\}  }^{-}\right)
\ \ \ \ \ \ \ \ \ \ \left(  \text{by Lemma \ref{lem.AnnNnk.1}}\right) \\
&  =T_{\operatorname*{sign}}\left(  \nabla_{U}^{-}\right)
\end{align*}
(since the definition of $a_{1},a_{2},\ldots,a_{k}$ yields $\left[  n\right]
\setminus\left\{  a_{1},a_{2},\ldots,a_{k}\right\}  =U$) and thus%
\[
T_{\operatorname*{sign}}\left(  \nabla_{U}^{-}\right)  =\nabla_{\left(
a_{1},a_{2},\ldots,a_{k}\right)  }\in\left\{  \nabla_{\mathbf{a}}%
\ \mid\ \mathbf{a}\in\left[  n\right]  ^{k}\right\}  .
\]
Hence,%
\[
\left\{  T_{\operatorname*{sign}}\left(  \nabla_{U}^{-}\right)  \ \mid
\ U\text{ is a subset of }\left[  n\right]  \text{ having size }n-k\right\}
\subseteq\left\{  \nabla_{\mathbf{a}}\ \mid\ \mathbf{a}\in\left[  n\right]
^{k}\right\}  .
\]
Thus,%
\begin{align}
T_{\operatorname*{sign}}\left(  \mathcal{J}_{n-k-1}\right)   &  =\mathcal{A}%
\cdot\operatorname*{span}\underbrace{\left\{  T_{\operatorname*{sign}}\left(
\nabla_{U}^{-}\right)  \ \mid\ U\text{ is a subset of }\left[  n\right]
\text{ having size }n-k\right\}  }_{\subseteq\left\{  \nabla_{\mathbf{a}%
}\ \mid\ \mathbf{a}\in\left[  n\right]  ^{k}\right\}  }\nonumber\\
&  \subseteq\mathcal{A}\cdot\operatorname*{span}\left\{  \nabla_{\mathbf{a}%
}\ \mid\ \mathbf{a}\in\left[  n\right]  ^{k}\right\}  .
\label{pf.lem.AnnNnk.2.c1.l}%
\end{align}

On the other hand, each $k$-tuple $\mathbf{a}=\left(  a_{1},a_{2},\ldots
,a_{k}\right)  \in\left[  n\right]  ^{k}$ satisfies%
\begin{align*}
\nabla_{\mathbf{a}}  &  =\sum_{\substack{w\in S_{n};\\w\left(  a_{i}\right)
=a_{i}\text{ for all }i}}w=T_{\operatorname*{sign}}\left(  \nabla_{\left[
n\right]  \setminus\left\{  a_{1},a_{2},\ldots,a_{k}\right\}  }^{-}\right)
\ \ \ \ \ \ \ \ \ \ \left(  \text{by Lemma \ref{lem.AnnNnk.1}}\right) \\
&  \in\left\{  T_{\operatorname*{sign}}\left(  \nabla_{U}^{-}\right)
\ \mid\ U\text{ is a subset of }\left[  n\right]  \text{ having size
}>n-k-1\right\}
\end{align*}
(since $\left[  n\right]  \setminus\left\{  a_{1},a_{2},\ldots,a_{k}\right\}
$ is a subset of $\left[  n\right]  $ having size $\geq n-k>n-k-1$). Thus,%
\[
\left\{  \nabla_{\mathbf{a}}\ \mid\ \mathbf{a}\in\left[  n\right]
^{k}\right\}  \subseteq\left\{  T_{\operatorname*{sign}}\left(  \nabla_{U}%
^{-}\right)  \ \mid\ U\text{ is a subset of }\left[  n\right]  \text{ having
size }>n-k-1\right\}  .
\]
Hence,%
\begin{align*}
&  \mathcal{A}\cdot\operatorname*{span}\left\{  \nabla_{\mathbf{a}}%
\ \mid\ \mathbf{a}\in\left[  n\right]  ^{k}\right\} \\
&  \subseteq\mathcal{A}\cdot\operatorname*{span}\left\{
T_{\operatorname*{sign}}\left(  \nabla_{U}^{-}\right)  \ \mid\ U\text{ is a
subset of }\left[  n\right]  \text{ having size }>n-k-1\right\} \\
&  =T_{\operatorname*{sign}}\left(  \mathcal{A}\cdot\operatorname*{span}%
\left\{  \nabla_{U}^{-}\ \mid\ U\text{ is a subset of }\left[  n\right]
\text{ having size }>n-k-1\right\}  \right) \\
&  \ \ \ \ \ \ \ \ \ \ \ \ \ \ \ \ \ \ \ \ \left(  \text{since }%
T_{\operatorname*{sign}}:\mathcal{A}\rightarrow\mathcal{A}\text{ is a
}\mathbf{k}\text{-algebra isomorphism}\right)  .
\end{align*}
In view of%
\[
\mathcal{J}_{n-k-1}=\mathcal{A}\cdot\operatorname*{span}\left\{  \nabla
_{U}^{-}\ \mid\ U\text{ is a subset of }\left[  n\right]  \text{ having size
}>n-k-1\right\}
\]
(by Proposition \ref{prop.IJ.1} \textbf{(e)}, applied to $n-k-1$ instead of
$k$), we can rewrite this as%
\[
\mathcal{A}\cdot\operatorname*{span}\left\{  \nabla_{\mathbf{a}}%
\ \mid\ \mathbf{a}\in\left[  n\right]  ^{k}\right\}  \subseteq
T_{\operatorname*{sign}}\left(  \mathcal{J}_{n-k-1}\right)  .
\]
Combining this with (\ref{pf.lem.AnnNnk.2.c1.l}), we obtain%
\[
T_{\operatorname*{sign}}\left(  \mathcal{J}_{n-k-1}\right)  =\mathcal{A}%
\cdot\operatorname*{span}\left\{  \nabla_{\mathbf{a}}\ \mid\ \mathbf{a}%
\in\left[  n\right]  ^{k}\right\}  .
\]
Thus, Lemma \ref{lem.AnnNnk.2} is proved in Case 1.

Now, let us consider Case 2. In this case, we have $k\geq n$, so that
$n-k-1<0$ and therefore $\mathcal{J}_{n-k-1}=\mathcal{A}$ (by our convention
that $\mathcal{J}_{m}=\mathcal{A}$ for $m<0$). Hence, $T_{\operatorname*{sign}%
}\left(  \mathcal{J}_{n-k-1}\right)  =T_{\operatorname*{sign}}\left(
\mathcal{A}\right)  =\mathcal{A}$. On the other hand, the set $\left[
n\right]  ^{k}$ contains the $k$-tuple $\mathbf{a}_{0}:=\left(  1,2,3,\ldots
,n,n,n,\ldots,n\right)  $ (since $k\geq n$), and the corresponding element
$\nabla_{\mathbf{a}_{0}}$ is $\operatorname*{id}$ (because if we apply
(\ref{eq.lem.AnnNnk.2.defNab}) to $\mathbf{a}=\mathbf{a}_{0}$, then the only
permutation $w\in S_{n}$ under the sum is $\operatorname*{id}$). Hence,
$\operatorname*{id}=\nabla_{\mathbf{a}_{0}}\in\operatorname*{span}\left\{
\nabla_{\mathbf{a}}\ \mid\ \mathbf{a}\in\left[  n\right]  ^{k}\right\}  $.
Thus, the left ideal $\mathcal{A}\cdot\operatorname*{span}\left\{
\nabla_{\mathbf{a}}\ \mid\ \mathbf{a}\in\left[  n\right]  ^{k}\right\}  $ of
$\mathcal{A}$ contains the element $\operatorname*{id}=1_{\mathcal{A}}$,
whence it must be the whole $\mathcal{A}$. In other words,%
\[
\mathcal{A}\cdot\operatorname*{span}\left\{  \nabla_{\mathbf{a}}%
\ \mid\ \mathbf{a}\in\left[  n\right]  ^{k}\right\}  =\mathcal{A}.
\]
Comparing this with $T_{\operatorname*{sign}}\left(  \mathcal{J}%
_{n-k-1}\right)  =\mathcal{A}$, we obtain%
\[
T_{\operatorname*{sign}}\left(  \mathcal{J}_{n-k-1}\right)  =\mathcal{A}%
\cdot\operatorname*{span}\left\{  \nabla_{\mathbf{a}}\ \mid\ \mathbf{a}%
\in\left[  n\right]  ^{k}\right\}  .
\]
Thus, Lemma \ref{lem.AnnNnk.2} is proved in Case 2. We have now proved Lemma
\ref{lem.AnnNnk.2} in both cases.
\end{proof}

\begin{proof}
[Proof of Theorem \ref{thm.AnnNnk}.]We extend the definition of $\mathcal{J}%
_{m}$ to negative $m$ by setting $\mathcal{J}_{m}:=\mathcal{A}$ when $m<0$.

Theorem \ref{thm.row.main} \textbf{(a)} shows that each $m\in\mathbb{N}$
satisfies $\mathcal{I}_{m}=\operatorname*{LAnn}\mathcal{J}_{m}$. This equality
also holds for negative $m$ (since $0=\operatorname*{LAnn}\mathcal{A}$). Thus,
each integer $m$ satisfies $\mathcal{I}_{m}=\operatorname*{LAnn}%
\mathcal{J}_{m}$. Applying $T_{\operatorname*{sign}}$ to both sides of this,
we obtain%
\[
T_{\operatorname*{sign}}\left(  \mathcal{I}_{m}\right)
=T_{\operatorname*{sign}}\left(  \operatorname*{LAnn}\mathcal{J}_{m}\right)
=\operatorname*{LAnn}\left(  T_{\operatorname*{sign}}\left(  \mathcal{J}%
_{m}\right)  \right)
\]
(since $T_{\operatorname*{sign}}$ is a $\mathbf{k}$-algebra isomorphism).
Applying this to $m=n-k-1$, we obtain%
\[
T_{\operatorname*{sign}}\left(  \mathcal{I}_{n-k-1}\right)
=\operatorname*{LAnn}\left(  T_{\operatorname*{sign}}\left(  \mathcal{J}%
_{n-k-1}\right)  \right)  =\operatorname*{Ann}\left(  T_{\operatorname*{sign}%
}\left(  \mathcal{J}_{n-k-1}\right)  \right)
\]
(since $\operatorname*{LAnn}\mathcal{K}=\operatorname*{Ann}\mathcal{K}$
whenever $\mathcal{K}$ is a left ideal of $\mathcal{A}$). It thus remains to
show that%
\begin{equation}
\operatorname*{Ann}\left(  T_{\operatorname*{sign}}\left(  \mathcal{J}%
_{n-k-1}\right)  \right)  =\operatorname*{Ann}\left(  N_{n}^{\otimes
k}\right)  . \label{pf.thm.AnnNnk.goal2}%
\end{equation}

We let the symmetric group $S_{n}$ act on the set $\left[  n\right]  ^{k}$ of
all $k$-tuples of elements of $\left[  n\right]  $ by the rule%
\begin{align*}
w\left(  a_{1},a_{2},\ldots,a_{k}\right)   &  :=\left(  w\left(  a_{1}\right)
,w\left(  a_{2}\right)  ,\ldots,w\left(  a_{k}\right)  \right) \\
&  \ \ \ \ \ \ \ \ \ \ \text{for all }w\in S_{n}\text{ and }\left(
a_{1},a_{2},\ldots,a_{k}\right)  \in\left[  n\right]  ^{k}.
\end{align*}
For any $k$-tuple $\mathbf{a}=\left(  a_{1},a_{2},\ldots,a_{k}\right)
\in\left[  n\right]  ^{k}$, we define the element%
\[
\nabla_{\mathbf{a}}:=\sum_{\substack{w\in S_{n};\\w\mathbf{a}=\mathbf{a}%
}}w=\sum_{\substack{w\in S_{n};\\w\left(  a_{i}\right)  =a_{i}\text{ for all
}i}}w\in\mathcal{A}.
\]
This is just the element $T_{\operatorname*{sign}}\left(  \nabla_{\left[
n\right]  \setminus\left\{  a_{1},a_{2},\ldots,a_{k}\right\}  }^{-}\right)  $
(by Lemma \ref{lem.AnnNnk.1}). Lemma \ref{lem.AnnNnk.2} yields%
\[
T_{\operatorname*{sign}}\left(  \mathcal{J}_{n-k-1}\right)  =\mathcal{A}%
\cdot\operatorname*{span}\left\{  \nabla_{\mathbf{a}}\ \mid\ \mathbf{a}%
\in\left[  n\right]  ^{k}\right\}  =\sum_{\mathbf{a}\in\left[  n\right]  ^{k}%
}\mathcal{A}\nabla_{\mathbf{a}}.
\]
Hence,%
\begin{equation}
\operatorname*{Ann}\left(  T_{\operatorname*{sign}}\left(  \mathcal{J}%
_{n-k-1}\right)  \right)  =\bigcap_{\mathbf{a}\in\left[  n\right]  ^{k}%
}\operatorname*{Ann}\left(  \mathcal{A}\nabla_{\mathbf{a}}\right)
\label{pf.thm.AnnNnk.Ann12}%
\end{equation}
(since the annihilator of a sum of left $\mathcal{A}$-submodules is the
intersection of their annihilators).

On the other hand, let us set $e_{\mathbf{a}}:=e_{a_{1}}\otimes e_{a_{2}%
}\otimes\cdots\otimes e_{a_{k}}\in N_{n}^{\otimes k}$ for any $k$-tuple
$\mathbf{a}=\left(  a_{1},a_{2},\ldots,a_{k}\right)  \in\left[  n\right]
^{k}$. Then, the family $\left(  e_{\mathbf{a}}\right)  _{\mathbf{a}\in\left[
n\right]  ^{k}}$ is a basis of the $\mathbf{k}$-module $N_{n}^{\otimes k}$,
and thus also generates $N_{n}^{\otimes k}$ as a left $\mathcal{A}$-module. In
other words,%
\[
N_{n}^{\otimes k}=\sum_{\mathbf{a}\in\left[  n\right]  ^{k}}\mathcal{A}%
e_{\mathbf{a}}.
\]
Hence,%
\begin{equation}
\operatorname*{Ann}\left(  N_{n}^{\otimes k}\right)  =\bigcap_{\mathbf{a}%
\in\left[  n\right]  ^{k}}\operatorname*{Ann}\left(  \mathcal{A}e_{\mathbf{a}%
}\right)  \label{pf.thm.AnnNnk.Ann22}%
\end{equation}
(since the annihilator of a sum of left $\mathcal{A}$-submodules is the
intersection of their annihilators).

We want to show that the left hand sides of (\ref{pf.thm.AnnNnk.Ann12}) and
(\ref{pf.thm.AnnNnk.Ann22}) are equal (because this will yield
(\ref{pf.thm.AnnNnk.goal2})). Of course, it suffices to show that the right
hand sides are equal. For this purpose, we will show that $\mathcal{A}%
\nabla_{\mathbf{a}}\cong\mathcal{A}e_{\mathbf{a}}$ as left $\mathcal{A}%
$-modules for every $\mathbf{a}\in\left[  n\right]  ^{k}$ (this will suffice,
since isomorphic $\mathcal{A}$-modules have the same annihilator).

But this is indeed quite easy: Let $\mathbf{a}=\left(  a_{1},a_{2}%
,\ldots,a_{k}\right)  \in\left[  n\right]  ^{k}$. Then, we can define a
$\mathbf{k}$-linear map%
\begin{align*}
\phi:N_{n}^{\otimes k}  &  \rightarrow\mathcal{A},\\
e_{\mathbf{c}}  &  \mapsto\sum_{\substack{w\in S_{n};\\w\mathbf{a}=\mathbf{c}%
}}w\ \ \ \ \ \ \ \ \ \ \text{for all }\mathbf{c}\in\left[  n\right]  ^{k}%
\end{align*}
(this is well-defined, since $\left(  e_{\mathbf{c}}\right)  _{\mathbf{c}%
\in\left[  n\right]  ^{k}}$ is a basis of the $\mathbf{k}$-module
$N_{n}^{\otimes k}$). It is easy to see that this map $\phi$ is $S_{n}%
$-equivariant (because every $u\in S_{n}$ and $\mathbf{c}\in\left[  n\right]
^{k}$ satisfy $ue_{\mathbf{c}}=e_{u\mathbf{c}}$ and $u\sum_{\substack{w\in
S_{n};\\w\mathbf{a}=\mathbf{c}}}w=\sum_{\substack{w\in S_{n};\\w\mathbf{a}%
=u\mathbf{c}}}w$) and thus left $\mathcal{A}$-linear. Moreover, it sends
$e_{\mathbf{a}}$ to $\sum_{\substack{w\in S_{n};\\w\mathbf{a}=\mathbf{a}%
}}w=\nabla_{\mathbf{a}}$ (by the definition of $\nabla_{\mathbf{a}}$). Hence,
$\phi\left(  \mathcal{A}e_{\mathbf{a}}\right)  =\mathcal{A}\nabla_{\mathbf{a}%
}$ (since $\phi$ is left $\mathcal{A}$-linear).

The map $\phi$ itself is not usually injective. However, we claim that its
restriction to $\mathcal{A}e_{\mathbf{a}}$ is injective. Indeed, recall again
that $ue_{\mathbf{c}}=e_{u\mathbf{c}}$ for all $u\in S_{n}$ and $\mathbf{c}%
\in\left[  n\right]  ^{k}$. In particular, $ue_{\mathbf{a}}=e_{u\mathbf{a}}$
for all $u\in S_{n}$. Hence, the $\mathbf{k}$-module $\mathcal{A}%
e_{\mathbf{a}}$ has a basis $\left(  e_{\mathbf{c}}\right)  _{\mathbf{c}\in
S_{n}\mathbf{a}}$ (where $S_{n}\mathbf{a}$ is the orbit of $\mathbf{a}%
\in\left[  n\right]  ^{k}$ under the $S_{n}$-action on $\left[  n\right]
^{k}$)\ \ \ \ \footnote{\textit{Proof:} We have $\mathcal{A}%
=\operatorname*{span}\left\{  u\ \mid\ u\in S_{n}\right\}  $. Thus,
\begin{align*}
\mathcal{A}e_{\mathbf{a}}  &  =\left(  \operatorname*{span}\left\{
u\ \mid\ u\in S_{n}\right\}  \right)  e_{\mathbf{a}}=\operatorname*{span}%
\left\{  ue_{\mathbf{a}}\ \mid\ u\in S_{n}\right\} \\
&  =\operatorname*{span}\underbrace{\left\{  e_{u\mathbf{a}}\ \mid\ u\in
S_{n}\right\}  }_{=\left\{  e_{\mathbf{c}}\ \mid\ \mathbf{c}\in S_{n}%
\mathbf{a}\right\}  }\ \ \ \ \ \ \ \ \ \ \left(  \text{since }ue_{\mathbf{a}%
}=e_{u\mathbf{a}}\text{ for all }u\in S_{n}\right) \\
&  =\operatorname*{span}\left\{  e_{\mathbf{c}}\ \mid\ \mathbf{c}\in
S_{n}\mathbf{a}\right\}  .
\end{align*}
Hence, the family $\left(  e_{\mathbf{c}}\right)  _{\mathbf{c}\in
S_{n}\mathbf{a}}$ spans the $\mathbf{k}$-module $\mathcal{A}e_{\mathbf{a}}$.
Since this family is furthermore $\mathbf{k}$-linearly independent (being a
subfamily of the basis $\left(  e_{\mathbf{c}}\right)  _{\mathbf{c}\in\left[
n\right]  ^{k}}$ of $N_{n}^{\otimes k}$), we thus conclude that it is a basis
of $\mathcal{A}e_{\mathbf{a}}$.}. The images of all these basis vectors
$e_{\mathbf{c}}$ under $\phi$ are the sums $\sum_{\substack{w\in
S_{n};\\w\mathbf{a}=\mathbf{c}}}w$, which are $\mathbf{k}$-linearly
independent (since they are sums of disjoint sets of permutations, and each of
them is nonempty because $\mathbf{c}\in S_{n}\mathbf{a}$ guarantees the
existence of at least one $w\in S_{n}$ satisfying $w\mathbf{a}=\mathbf{c}$).
Thus, the $\mathbf{k}$-linear map $\phi$ sends the basis vectors
$e_{\mathbf{c}}$ of $\mathcal{A}e_{\mathbf{a}}$ to $\mathbf{k}$-linearly
independent vectors in $\mathcal{A}$. Consequently, the restriction of this
map $\phi$ to $\mathcal{A}e_{\mathbf{a}}$ is injective.

Hence, $\phi\left(  \mathcal{A}e_{\mathbf{a}}\right)  \cong\mathcal{A}%
e_{\mathbf{a}}$. In view of $\phi\left(  \mathcal{A}e_{\mathbf{a}}\right)
=\mathcal{A}\nabla_{\mathbf{a}}$, we can rewrite this as $\mathcal{A}%
\nabla_{\mathbf{a}}\cong\mathcal{A}e_{\mathbf{a}}$.

\begin{noncompile}
Old argumentation: $\mathcal{A}e_{\mathbf{a}}=\mathcal{A}\left(  e_{a_{1}%
}\otimes e_{a_{2}}\otimes\cdots\otimes e_{a_{k}}\right)  $ is a permutation
representation of $S_{n}$, with a basis
\[
\left\{  e_{c_{1}}\otimes e_{c_{2}}\otimes\cdots\otimes e_{c_{k}}%
\ \mid\ \left(  c_{1},c_{2},\ldots,c_{k}\right)  =w\mathbf{a}\text{ for some
}w\in S_{n}\right\}
\]
and with the stabilizer of $e_{\mathbf{a}}$ being the set of all permutations
$w\in S_{n}$ that satisfy $w\mathbf{a}=\mathbf{a}$. On the other hand,
$\mathcal{A}\nabla_{\mathbf{a}}=\mathcal{A}\sum\limits_{\substack{w\in
S_{n};\\w\left(  a_{i}\right)  =a_{i}\text{ for all }i}}w$ is also a
permutation representation of $S_{n}$, with a basis
\[
\left\{  \sum\limits_{\substack{w\in S_{n};\\w\left(  a_{i}\right)
=c_{i}\text{ for all }i}}w\ \mid\ \left(  c_{1},c_{2},\ldots,c_{k}\right)
=w\mathbf{a}\text{ for some }w\in S_{n}\right\}
\]
(this family is $\mathbf{k}$-linearly independent because...) and with the
stabilizer of $\nabla_{\mathbf{a}}$ being again the set of all permutations
$w\in S_{n}$ that satisfy $w\mathbf{a}=\mathbf{a}$. Two permutation
representations that have the same stabilizer must be isomorphic. Thus, the
two permutation representations $\mathcal{A}\nabla_{\mathbf{a}}$ and
$\mathcal{A}e_{\mathbf{a}}$ of $S_{n}$ are isomorphic (since they have the
same stabilizer).
\end{noncompile}

As explained above, this completes the proof of Theorem \ref{thm.AnnNnk}.
\end{proof}

\begin{corollary}
\label{cor.actonNnk}Let $k\in\mathbb{N}$. Consider the action of $\mathcal{A}$
on $N_{n}^{\otimes k}$ as a $\mathbf{k}$-algebra morphism $\rho_{k}%
:\mathcal{A}\rightarrow\operatorname*{End}\nolimits_{\mathbf{k}}\left(
N_{n}^{\otimes k}\right)  $. Then, the image $\rho_{k}\left(  \mathcal{A}%
\right)  $ of this map $\rho_{k}$ has two bases $\left(  \rho_{k}\left(
w\right)  \right)  _{w\in S_{n}\setminus\operatorname*{Av}\nolimits_{n}\left(
n-k\right)  }$ and $\left(  \rho_{k}\left(  w\right)  \right)  _{w\in
S_{n}\setminus\operatorname*{Av}\nolimits_{n}^{\prime}\left(  n-k\right)  }$
(as a $\mathbf{k}$-module). Here, we understand $\operatorname*{Av}%
\nolimits_{n}\left(  m\right)  $ to be $\varnothing$ when $m<0$.
\end{corollary}

\begin{proof}
We have $\operatorname*{Ker}\rho_{k}=\operatorname*{Ann}\left(  N_{n}^{\otimes
k}\right)  =T_{\operatorname*{sign}}\left(  \mathcal{I}_{n-k-1}\right)  $ by
Theorem \ref{thm.AnnNnk}. Hence, by the first isomorphism theorem, there is a
$\mathbf{k}$-algebra isomorphism $\mathcal{A}/T_{\operatorname*{sign}}\left(
\mathcal{I}_{n-k-1}\right)  \rightarrow\rho_{k}\left(  \mathcal{A}\right)  $
that sends each residue class $\overline{a}$ to $\rho_{k}\left(  a\right)  $.
Hence, we only need to show that the quotient $\mathbf{k}$-module
$\mathcal{A}/T_{\operatorname*{sign}}\left(  \mathcal{I}_{n-k-1}\right)  $ has
two bases $\left(  \overline{w}\right)  _{w\in S_{n}\setminus
\operatorname*{Av}\nolimits_{n}\left(  n-k\right)  }$ and $\left(
\overline{w}\right)  _{w\in S_{n}\setminus\operatorname*{Av}\nolimits_{n}%
^{\prime}\left(  n-k\right)  }$. Upon applying the $\mathbf{k}$-algebra
automorphism $T_{\operatorname*{sign}}$ of $\mathcal{A}$ to this statement, it
takes the following simpler form: The quotient $\mathbf{k}$-module
$\mathcal{A}/\mathcal{I}_{n-k-1}$ has two bases $\left(  \overline{w}\right)
_{w\in S_{n}\setminus\operatorname*{Av}\nolimits_{n}\left(  n-k\right)  }$ and
$\left(  \overline{w}\right)  _{w\in S_{n}\setminus\operatorname*{Av}%
\nolimits_{n}^{\prime}\left(  n-k\right)  }$ (indeed, the bases still look the
same, because $T_{\operatorname*{sign}}$ sends each permutation $w\in S_{n}$
to $\left(  -1\right)  ^{w}w=\pm w$, and thus the bases are transformed only
by scaling some of the basis vectors by $-1$). In order to prove this
statement, we only need to apply Theorem \ref{thm.row.main} \textbf{(e)} and
Corollary \ref{cor.row.twin} \textbf{(c)} to $n-k-1$ instead of $k$ (at least
when $k\leq n-1$; but the other case is trivial because in that case, we have
$\mathcal{I}_{n-k-1}=0$ and $\operatorname*{Av}\nolimits_{n}\left(
n-k\right)  =\varnothing$ and $\operatorname*{Av}\nolimits_{n}^{\prime}\left(
n-k\right)  =\varnothing$).
\end{proof}

Corollary \ref{cor.actonNnk} appears in \cite[Example 2.13]{Donkin24} and (in
a slightly restated form) in \cite[Corollary 1.4 (ii)]{BoDoMa22}.

Further related results are found in \cite{DotNym07}, \cite{BoDoMa18},
\cite{BoDoMa22}, \cite{RaSaSu12}, \cite{Donkin24} and \cite[\S 8]{BaiRai}.

\subsection{The Specht module connection}

\subsubsection{$\mathcal{I}_{k}$ and $\mathcal{J}_{k}$ after
Artin--Wedderburn}

The following theorem discusses the representation-theoretical significance of
the ideals $\mathcal{I}_{k}$ and $\mathcal{J}_{k}$. We use some basic
representation theory, including the concept of a Specht module (see, e.g.,
\cite[\S 5.12]{Etingof} or \cite[Definition 5.11.1 \textbf{(a)}]%
{sga}\footnote{Note that the Specht module corresponding to a partition
$\lambda$ is called $V_{\lambda}$ in \cite[\S 5.12]{Etingof} and is called
$\mathcal{S}^{\lambda}$ in \cite[Definition 5.11.1 \textbf{(a)}]{sga}. The
definitions are not identical, but they define isomorphic modules (because of
\cite[Theorem 5.5.13 \textbf{(b)}]{sga}).}).

In what follows, the notation \textquotedblleft$\lambda\vdash n$%
\textquotedblright\ shall always mean \textquotedblleft$\lambda$ is a
partition of $n$\textquotedblright. In particular, the set of all partitions
of $n$ shall be denoted by $\left\{  \lambda\ \mid\ \lambda\vdash n\right\}  $
or just $\left\{  \lambda\vdash n\right\}  $. If $\lambda$ is a partition of
$n$, then the entries of $\lambda$ will be denoted by $\lambda_{1},\lambda
_{2},\lambda_{3},\ldots$, whereas the length of $\lambda$ will be written as
$\ell\left(  \lambda\right)  $.

\begin{theorem}
\label{thm.IJ.rep}Assume that $n!$ is invertible in $\mathbf{k}$. For each
partition $\lambda$ of $n$, let $S^{\lambda}$ denote the corresponding Specht
module (a left $\mathcal{A}$-module). For each $a\in\mathcal{A}$ and each
partition $\lambda$ of $n$, we let $a_{\lambda}\in\operatorname*{End}\left(
S^{\lambda}\right)  $ denote the action of $a$ on the Specht module
$S^{\lambda}$.

Consider the map%
\begin{align*}
\operatorname*{AW}:\mathcal{A}  &  \rightarrow\prod\limits_{\lambda\vdash
n}\operatorname*{End}\left(  S^{\lambda}\right)  ,\\
a  &  \mapsto\left(  a_{\lambda}\right)  _{\lambda\vdash n}.
\end{align*}
This map $\operatorname*{AW}$ is known to be a $\mathbf{k}$-algebra
isomorphism. (When $\mathbf{k}$ is a field, this follows from the
Artin--Wedderburn decomposition of $\mathcal{A}$, since the $S^{\lambda}$ are
the absolutely irreducible $\mathcal{A}$-modules; alternatively, this can be
derived from \cite[\S 17, Theorem 12]{Ruther48}. For a detailed proof, see
\cite[Theorem 5.14.1]{sga}.)

For each subset $U$ of $\left\{  \lambda\ \mid\ \lambda\vdash n\right\}  $, we
consider the subproduct $\prod\limits_{\lambda\in U}\operatorname*{End}\left(
S^{\lambda}\right)  $ of $\prod\limits_{\lambda\vdash n}\operatorname*{End}%
\left(  S^{\lambda}\right)  $. This is an ideal of $\prod\limits_{\lambda
\vdash n}\operatorname*{End}\left(  S^{\lambda}\right)  $. The preimage of
this subproduct under $\operatorname*{AW}$ is thus an ideal of $\mathcal{A}$,
and will be denoted by $\mathcal{A}_{U}$.

Now, let $k\in\mathbb{N}$. Then,%
\[
\mathcal{I}_{k}=\mathcal{A}_{\left\{  \lambda\vdash n\ \mid\ \ell\left(
\lambda\right)  \leq k\right\}  }\ \ \ \ \ \ \ \ \ \ \text{and}%
\ \ \ \ \ \ \ \ \ \ \mathcal{J}_{k}=\mathcal{A}_{\left\{  \lambda\vdash
n\ \mid\ \ell\left(  \lambda\right)  >k\right\}  }.
\]

\end{theorem}

The proof of this theorem will rely on the following general fact:

\begin{lemma}
\label{lem.IUJV}Let $\mathcal{M}$ be a $\mathbf{k}$-module. Let $\mathcal{I}$
and $\mathcal{J}$ be two $\mathbf{k}$-submodules of $\mathcal{M}$ such that
$\mathcal{M}=\mathcal{I}+\mathcal{J}$. Let $\mathcal{U}$ and $\mathcal{V}$ be
two $\mathbf{k}$-submodules of $\mathcal{M}$ such that $\mathcal{I}%
\subseteq\mathcal{U}$ and $\mathcal{J}\subseteq\mathcal{V}$ and $\mathcal{U}%
\cap\mathcal{V}=0$. Then, $\mathcal{I}=\mathcal{U}$ and $\mathcal{J}%
=\mathcal{V}$.
\end{lemma}

\begin{proof}
Let $u\in\mathcal{U}$. We shall show that $u\in\mathcal{I}$.

We have $u\in\mathcal{U}\subseteq\mathcal{M}=\mathcal{I}+\mathcal{J}$. Thus,
$u=i+j$ for some $i\in\mathcal{I}$ and $j\in\mathcal{J}$. Consider these $i$
and $j$. From $u=i+j$, we obtain $j=\underbrace{u}_{\in\mathcal{U}%
}-\underbrace{i}_{\in\mathcal{I}\subseteq\mathcal{U}}\in\mathcal{U}%
-\mathcal{U}\subseteq\mathcal{U}$. Combining this with $j\in\mathcal{J}%
\subseteq\mathcal{V}$, we obtain $j\in\mathcal{U}\cap\mathcal{V}=0$. In other
words, $j=0$. Hence, $u=i+\underbrace{j}_{=0}=i\in\mathcal{I}$.

Forget that we fixed $u$. We thus have shown that $u\in\mathcal{I}$ for each
$u\in\mathcal{U}$. In other words, $\mathcal{U}\subseteq\mathcal{I}$. Combined
with $\mathcal{I}\subseteq\mathcal{U}$, this yields $\mathcal{I}=\mathcal{U}$.
Similarly, we can show $\mathcal{J}=\mathcal{V}$. This proves Lemma
\ref{lem.IUJV}.
\end{proof}

\begin{proof}
[Proof of Theorem \ref{thm.IJ.rep}.]For each $U\subseteq\left\{  \lambda
\ \mid\ \lambda\vdash n\right\}  $, we have
\begin{align}
\mathcal{A}_{U}  &  =\operatorname*{AW}\nolimits^{-1}\left(  \text{the
subproduct }\prod\limits_{\lambda\in U}\operatorname*{End}\left(  S^{\lambda
}\right)  \text{ of }\prod\limits_{\lambda\vdash n}\operatorname*{End}\left(
S^{\lambda}\right)  \right) \nonumber\\
&  =\left\{  a\in\mathcal{A}\ \mid\text{ the }\lambda\text{-th entry of
}\operatorname*{AW}\left(  a\right)  \text{ is }0\text{ for all }\lambda\notin
U\right\} \nonumber\\
&  =\left\{  a\in\mathcal{A}\ \mid\ a_{\lambda}=0\text{ for all }\lambda\notin
U\right\} \nonumber\\
&  =\left\{  a\in\mathcal{A}\ \mid\ aS^{\lambda}=0\text{ for all }%
\lambda\notin U\right\}  \label{pf.prop.IJ.rep.AU=}%
\end{align}
(since $aS^{\lambda}$ is the image of $a_{\lambda}\in\operatorname*{End}%
\left(  S^{\lambda}\right)  $).

We shall first prove the following two claims:

\begin{statement}
\textit{Claim 1:} We have $\mathcal{I}_{k}\subseteq\mathcal{A}_{\left\{
\lambda\vdash n\ \mid\ \ell\left(  \lambda\right)  \leq k\right\}  }$.
\end{statement}

\begin{statement}
\textit{Claim 2:} We have $\mathcal{J}_{k}\subseteq\mathcal{A}_{\left\{
\lambda\vdash n\ \mid\ \ell\left(  \lambda\right)  >k\right\}  }$.
\end{statement}

\begin{proof}
[Proof of Claim 1.]Applying (\ref{pf.prop.IJ.rep.AU=}) to $U=\left\{
\lambda\vdash n\ \mid\ \ell\left(  \lambda\right)  \leq k\right\}  $, we
obtain%
\begin{align*}
&  \mathcal{A}_{\left\{  \lambda\vdash n\ \mid\ \ell\left(  \lambda\right)
\leq k\right\}  }\\
&  =\left\{  a\in\mathcal{A}\ \mid\ aS^{\lambda}=0\text{ for all }%
\lambda\vdash n\text{ that don't satisfy }\ell\left(  \lambda\right)  \leq
k\right\}  .
\end{align*}
Hence, in order to prove that $\mathcal{I}_{k}\subseteq\mathcal{A}_{\left\{
\lambda\vdash n\ \mid\ \ell\left(  \lambda\right)  \leq k\right\}  }$, it
suffices to show that all $a\in\mathcal{I}_{k}$ satisfy $aS^{\lambda}=0$ for
all partitions $\lambda\vdash n$ that don't satisfy $\ell\left(
\lambda\right)  \leq k$. Let us prove this.

Let $\lambda\vdash n$ be a partition that doesn't satisfy $\ell\left(
\lambda\right)  \leq k$. Thus, $\ell\left(  \lambda\right)  >k$. We must prove
that all $a\in\mathcal{I}_{k}$ satisfy $aS^{\lambda}=0$.

Since $\mathcal{I}_{k}=\operatorname*{span}\left\{  \nabla_{\mathbf{B}%
,\mathbf{A}}\ \mid\ \mathbf{A},\mathbf{B}\in\operatorname*{SC}\left(
n\right)  \text{ with }\ell\left(  \mathbf{A}\right)  =\ell\left(
\mathbf{B}\right)  \leq k\right\}  $, it suffices to prove that $\nabla
_{\mathbf{B},\mathbf{A}}S^{\lambda}=0$ for any two set compositions
$\mathbf{A},\mathbf{B}\in\operatorname*{SC}\left(  n\right)  $ with
$\ell\left(  \mathbf{A}\right)  =\ell\left(  \mathbf{B}\right)  \leq k$. So
let us consider two set compositions $\mathbf{A},\mathbf{B}\in
\operatorname*{SC}\left(  n\right)  $ with $\ell\left(  \mathbf{A}\right)
=\ell\left(  \mathbf{B}\right)  \leq k$. We must prove that $\nabla
_{\mathbf{B},\mathbf{A}}S^{\lambda}=0$.

If $T$ is any Young tableau of shape $\lambda$ filled with the entries
$1,2,\ldots,n$ (not necessarily standard), then

\begin{itemize}
\item we let $\mathcal{R}\left(  T\right)  $ denote the row group of $T$ (that
is, the group of all permutations $\sigma\in S_{n}$ that preserve the rows of
$T$ as sets);

\item we let $\mathbf{a}_{T}:=\sum\limits_{\sigma\in\mathcal{R}\left(
T\right)  }\sigma\in\mathcal{A}$ denote the row symmetrizer of $T$;

\item we let $\mathcal{C}\left(  T\right)  $ denote the column group of $T$
(that is, the group of all permutations $\sigma\in S_{n}$ that preserve the
columns of $T$ as sets);

\item we let $\mathbf{b}_{T}:=\sum\limits_{\sigma\in\mathcal{C}\left(
T\right)  }\left(  -1\right)  ^{\sigma}\sigma\in\mathcal{A}$ denote the column
antisymmetrizer of $T$.
\end{itemize}

Note that $\mathbf{a}_{T}$ and $\mathbf{b}_{T}$ are called $\nabla
_{\operatorname*{Row}T}$ and $\nabla_{\operatorname*{Col}T}^{-}$ in
\cite[\S 5.5.1]{sga}.

Let $T_{\lambda}$ be the Young tableau of shape $\lambda$ filled with the
entries $1,2,\ldots,n$ in the order \textquotedblleft row by row, starting
with the top row and proceeding down the rows\textquotedblright.\footnote{For
instance, if $\lambda=\left(  4,2\right)  $, then $T_{\lambda}%
=\ytableaushort{1234,56}$.} We denote the elements $\mathbf{a}_{T_{\lambda}}$
and $\mathbf{b}_{T_{\lambda}}$ of $\mathcal{A}$ by $\mathbf{a}_{\lambda}$ and
$\mathbf{b}_{\lambda}$. Thus, \cite[Theorem 5.12.2]{Etingof} or
\cite[Proposition 5.11.19 \textbf{(c)}]{sga} shows that $S^{\lambda}%
\cong\mathcal{A}\mathbf{a}_{\lambda}\mathbf{b}_{\lambda}$ as left
$\mathcal{A}$-modules\footnote{Specifically, this follows from \cite[Theorem
5.12.2]{Etingof} after observing that the $a_{\lambda}$ and $b_{\lambda}$ in
\cite[Theorem 5.12.2]{Etingof} are our $\dfrac{1}{\left\vert \mathcal{R}%
\left(  T\right)  \right\vert }\mathbf{a}_{\lambda}$ and $\dfrac{1}{\left\vert
\mathcal{C}\left(  T\right)  \right\vert }\mathbf{b}_{\lambda}$ (the fractions
here are well-defined, since $\left\vert \mathcal{R}\left(  T\right)
\right\vert $ and $\left\vert \mathcal{C}\left(  T\right)  \right\vert $ are
divisors of $n!$ and thus invertible in $\mathbf{k}$). Alternatively, this
follows by applying \cite[Theorem 5.5.13 \textbf{(b)}]{sga} to $T=T_{\lambda}$
and noticing that $\mathbf{a}_{\lambda}\mathbf{b}_{\lambda}=\mathbf{a}%
_{T}\mathbf{b}_{T}=\nabla_{\operatorname*{Row}T}\nabla_{\operatorname*{Col}%
T}^{-}=\mathbf{F}_{T}$ in the notations of \cite[Theorem 5.5.13 \textbf{(b)}%
]{sga} and that the number $h^{\lambda}$ defined in \cite[Theorem 5.5.13
\textbf{(b)}]{sga} is invertible in $\mathbf{k}$ (since it is a divisor of
$n!$).}.

Hence, $S^{\lambda}\cong\underbrace{\mathcal{A}\mathbf{a}_{\lambda}%
}_{\subseteq\mathcal{A}}\mathbf{b}_{\lambda}\subseteq\mathcal{A}%
\mathbf{b}_{\lambda}$. It thus suffices to show that $\nabla_{\mathbf{B}%
,\mathbf{A}}\mathcal{A}\mathbf{b}_{\lambda}=0$ (since $\nabla_{\mathbf{B}%
,\mathbf{A}}S^{\lambda}=0$ will then follow). In other words, it suffices to
show that $\nabla_{\mathbf{B},\mathbf{A}}w\mathbf{b}_{\lambda}=0$ for any
$w\in S_{n}$. But this is not hard:

Let $w\in S_{n}$. Let $T$ be the Young tableau of shape $\lambda$ obtained
from $T_{\lambda}$ by applying the permutation $w$ to each entry. (This is
$w\rightharpoonup T_{\lambda}$ in the notations of \cite[Definition
5.3.9]{sga}.) Thus, $\mathbf{b}_{T}=w\mathbf{b}_{\lambda}w^{-1}$ (by
\cite[Proposition 5.5.11]{sga}, applied to $T_{\lambda}$ instead of $T$), so
that $w\mathbf{b}_{\lambda}=\mathbf{b}_{T}w$.

The first column of the tableau $T$ contains $\ell\left(  \lambda\right)  $
entries, and thus contains more than $k$ entries (since $\ell\left(
\lambda\right)  >k$). Hence, at least two entries of this column belong to the
same block of $\mathbf{A}$ (by the pigeonhole principle, since $\mathbf{A}$
has only $\ell\left(  \mathbf{A}\right)  \leq k$ blocks). Pick two such
entries. Let $\tau\in S_{n}$ be the transposition that swaps these two
entries. This transposition $\tau$ thus preserves the blocks of $\mathbf{A}$
(since it swaps two numbers from the same block), and therefore preserves
$\nabla_{\mathbf{B},\mathbf{A}}$ from the right (i.e., satisfies
$\nabla_{\mathbf{B},\mathbf{A}}\tau=\nabla_{\mathbf{B},\mathbf{A}}$) because
of (\ref{eq.uNabv}). On the other hand, this transposition $\tau$ swaps two
entries in the first column of $T$, and thus belongs to the column group
$\mathcal{C}\left(  T\right)  $ of $T$. Hence, $\mathbf{b}_{T}=\left(
1-\tau\right)  \eta$ for some $\eta\in\mathcal{A}$ (since $\left\langle
\tau\right\rangle =\left\{  1,\tau\right\}  $ is a subgroup of the column
group $\mathcal{C}\left(  T\right)  $, and since the permutation $\tau$ is
odd; see \cite[Proposition 5.5.9 \textbf{(b)}]{sga} for details). Thus,%
\[
\nabla_{\mathbf{B},\mathbf{A}}\underbrace{w\mathbf{b}_{\lambda}}%
_{=\mathbf{b}_{T}w}=\nabla_{\mathbf{B},\mathbf{A}}\underbrace{\mathbf{b}_{T}%
}_{=\left(  1-\tau\right)  \eta}w=\underbrace{\nabla_{\mathbf{B},\mathbf{A}%
}\left(  1-\tau\right)  }_{\substack{=\nabla_{\mathbf{B},\mathbf{A}}%
-\nabla_{\mathbf{B},\mathbf{A}}\tau\\=0\\\text{(since }\nabla_{\mathbf{B}%
,\mathbf{A}}\tau=\nabla_{\mathbf{B},\mathbf{A}}\text{)}}}\eta w=0.
\]
This completes our proof of $\nabla_{\mathbf{B},\mathbf{A}}S^{\lambda}=0$.
Thus, as explained above, we have proved Claim 1.
\end{proof}

\begin{proof}
[Proof of Claim 2.]This is fairly similar to the above proof of Claim 1. Here
are the main milestones of the proof:

Let $\lambda\vdash n$ be a partition that satisfies $\ell\left(
\lambda\right)  \leq k$. We must prove that all $a\in\mathcal{J}_{k}$ satisfy
$aS^{\lambda}=0$.

Recall that $\mathcal{J}_{k}=\mathcal{A}\cdot\operatorname*{span}\left\{
\nabla_{U}^{-}\ \mid\ U\text{ is a subset of }\left[  n\right]  \text{ having
size }k+1\right\}  $ (by Proposition \ref{prop.IJ.1} \textbf{(b)}). Hence, it
suffices to show that $\nabla_{U}^{-}S^{\lambda}=0$ whenever $U$ is a subset
of $\left[  n\right]  $ having size $k+1$ (because then, $\mathcal{A}%
\underbrace{\nabla_{U}^{-}S^{\lambda}}_{=0}=0$ will automatically follow).

So let us fix a subset $U$ of $\left[  n\right]  $ having size $k+1$. Consider
the Young tableau $T_{\lambda}$ and the elements $\mathbf{a}_{\lambda}$ and
$\mathbf{b}_{\lambda}$ defined as in the above proof of Claim 1. Then,
$S^{\lambda}\cong\mathcal{A}\mathbf{a}_{\lambda}\mathbf{b}_{\lambda}$. Hence,
in order to prove that $\nabla_{U}^{-}S^{\lambda}=0$, it suffices to show that
$\nabla_{U}^{-}\mathcal{A}\mathbf{a}_{\lambda}=0$. For this, in turn, it
suffices to show that $\nabla_{U}^{-}w\mathbf{a}_{\lambda}=0$ for each $w\in
S_{n}$.

So let $w\in S_{n}$ be arbitrary. Let $T$ be the Young tableau of shape
$\lambda$ obtained from $T_{\lambda}$ by applying the permutation $w$ to each
entry. Then, $\mathbf{a}_{T}=w\mathbf{a}_{\lambda}w^{-1}$ (by
\cite[Proposition 5.5.11]{sga}), so that $w\mathbf{a}_{\lambda}=\mathbf{a}%
_{T}w$.

But the pigeonhole principle shows that there are two distinct elements of $U$
that belong to the same row of our tableau $T$ (since the set $U$ has $k+1$
elements, but the tableau $T$ has only $\ell\left(  \lambda\right)  \leq k$
rows). Let $\tau$ be the transposition that swaps these two elements. Then,
$\nabla_{U}^{-}\tau=-\nabla_{U}^{-}$ (by \cite[Proposition 3.7.4 \textbf{(c)}%
]{sga}), but also $\mathbf{a}_{T}=\left(  1+\tau\right)  \eta$ for some
$\eta\in\mathcal{A}$ (since $\left\langle \tau\right\rangle =\left\{
1,\tau\right\}  $ is a subgroup of the row group $\mathcal{R}\left(  T\right)
$; see \cite[Proposition 5.5.8 \textbf{(b)}]{sga} for details). Thus,%
\[
\nabla_{U}^{-}\underbrace{w\mathbf{a}_{\lambda}}_{=\mathbf{a}_{T}w}=\nabla
_{U}^{-}\underbrace{\mathbf{a}_{T}}_{=\left(  1+\tau\right)  \eta
}w=\underbrace{\nabla_{U}^{-}\left(  1+\tau\right)  }_{\substack{=\nabla
_{U}^{-}+\nabla_{U}^{-}\tau\\=0\\\text{(since }\nabla_{U}^{-}\tau=-\nabla
_{U}^{-}\text{)}}}\eta w=0.
\]
As we explained above, this completes the proof of Claim 2.
\end{proof}

However, if $X$ and $Y$ are two disjoint subsets of $\left\{  \lambda
\ \mid\ \lambda\vdash n\right\}  $, then $\mathcal{A}_{X}\cap\mathcal{A}%
_{Y}=0$ (since the subproducts $\prod\limits_{\lambda\in X}\operatorname*{End}%
\left(  S^{\lambda}\right)  $ and $\prod\limits_{\lambda\in Y}%
\operatorname*{End}\left(  S^{\lambda}\right)  $ of $\prod\limits_{\lambda
\vdash n}\operatorname*{End}\left(  S^{\lambda}\right)  $ have intersection
$0$). Thus,%
\[
\mathcal{A}_{\left\{  \lambda\vdash n\ \mid\ \ell\left(  \lambda\right)  \leq
k\right\}  }\cap\mathcal{A}_{\left\{  \lambda\vdash n\ \mid\ \ell\left(
\lambda\right)  >k\right\}  }=0.
\]
Meanwhile, Claim 1 and Claim 2 yield%
\[
\mathcal{I}_{k}\subseteq\mathcal{A}_{\left\{  \lambda\vdash n\ \mid
\ \ell\left(  \lambda\right)  \leq k\right\}  }\ \ \ \ \ \ \ \ \ \ \text{and}%
\ \ \ \ \ \ \ \ \ \ \mathcal{J}_{k}\subseteq\mathcal{A}_{\left\{
\lambda\vdash n\ \mid\ \ell\left(  \lambda\right)  >k\right\}  }.
\]
Furthermore, Theorem \ref{thm.row.main} \textbf{(g)} yields $\mathcal{A}%
=\mathcal{I}_{k}\oplus\mathcal{J}_{k}$ (internal direct sum), so that
$\mathcal{A}=\mathcal{I}_{k}+\mathcal{J}_{k}$. Thus, Lemma \ref{lem.IUJV}
(applied to $\mathcal{M}=\mathcal{A}$ and $\mathcal{I}=\mathcal{I}_{k}$ and
$\mathcal{J}=\mathcal{J}_{k}$ and $\mathcal{U}=\mathcal{A}_{\left\{
\lambda\vdash n\ \mid\ \ell\left(  \lambda\right)  \leq k\right\}  }$ and
$\mathcal{V}=\mathcal{A}_{\left\{  \lambda\vdash n\ \mid\ \ell\left(
\lambda\right)  >k\right\}  }$) yields
\[
\mathcal{I}_{k}=\mathcal{A}_{\left\{  \lambda\vdash n\ \mid\ \ell\left(
\lambda\right)  \leq k\right\}  }\ \ \ \ \ \ \ \ \ \ \text{and}%
\ \ \ \ \ \ \ \ \ \ \mathcal{J}_{k}=\mathcal{A}_{\left\{  \lambda\vdash
n\ \mid\ \ell\left(  \lambda\right)  >k\right\}  }.
\]
This proves Theorem \ref{thm.IJ.rep}.
\end{proof}

Note that part of Theorem \ref{thm.IJ.rep} for $k=3$ appears in
\cite[Proposition 3.5 (a)]{BCEHK23}. Moreover, \cite[Theorem 2.2]{KlStVo25}
can also be viewed as a corollary of Theorem \ref{thm.IJ.rep} (indeed, we
already saw that the $k$-swap algebra from \cite[Theorem 2.2]{KlStVo25} is
isomorphic to our $\mathcal{A}/\mathcal{J}_{k}$, which according to Theorem
\ref{thm.IJ.rep} is in turn isomorphic to $\prod\limits_{\substack{\lambda
\vdash n;\\\ell\left(  \lambda\right)  \leq k}}\operatorname*{End}\left(
S^{\lambda}\right)  $).

\subsubsection{Application: Counting avoiding permutations}

As a consequence of Theorem \ref{thm.IJ.rep}, we recover a classical
enumerative result that is commonly proved using the RSK correspondence (see,
e.g., \cite[(7.2)]{Bona22}, \cite[Corollary 7.23.12]{Stanley-EC2}):

\begin{corollary}
\label{cor.num-avoid}For each partition $\lambda$ of $n$, let $f^{\lambda}$ be
the number of standard tableaux of shape $\lambda$. Let $k\in\mathbb{N}$.
Then, the number of all permutations $w\in S_{n}$ that avoid $12\cdots\left(
k+1\right)  $ is $\sum\limits_{\substack{\lambda\vdash n;\\\ell\left(
\lambda\right)  \leq k}}\left(  f^{\lambda}\right)  ^{2}=\sum
\limits_{\substack{\lambda\vdash n;\\\lambda_{1}\leq k}}\left(  f^{\lambda
}\right)  ^{2}$.
\end{corollary}

\begin{proof}
Let $\mathbf{k}=\mathbb{Q}$. Let us use the notations of Theorem
\ref{thm.IJ.rep}. Then, the standard basis theorem for Specht modules (see,
e.g., \cite[Lemma 5.9.17]{sga}) says that
\begin{equation}
\dim\left(  S^{\lambda}\right)  =f^{\lambda} \label{pf.cor.num-avoid.dimSlam}%
\end{equation}
for each partition $\lambda$ of $n$.

However, Theorem \ref{thm.IJ.rep} shows that%
\[
\mathcal{I}_{k}=\mathcal{A}_{\left\{  \lambda\vdash n\ \mid\ \ell\left(
\lambda\right)  \leq k\right\}  }\cong\prod\limits_{\substack{\lambda\vdash
n;\\\ell\left(  \lambda\right)  \leq k}}\operatorname*{End}\left(  S^{\lambda
}\right)  \ \ \ \ \ \ \ \ \ \ \text{as }\mathbf{k}\text{-vector spaces}%
\]
(since $\mathcal{A}_{\left\{  \lambda\vdash n\ \mid\ \ell\left(
\lambda\right)  \leq k\right\}  }$ was defined as the preimage of
$\prod\limits_{\substack{\lambda\vdash n;\\\ell\left(  \lambda\right)  \leq
k}}\operatorname*{End}\left(  S^{\lambda}\right)  $ under the isomorphism
$\operatorname*{AW}$). Thus,%
\[
\dim\left(  \mathcal{I}_{k}\right)  =\dim\left(  \prod
\limits_{\substack{\lambda\vdash n;\\\ell\left(  \lambda\right)  \leq
k}}\operatorname*{End}\left(  S^{\lambda}\right)  \right)  =\sum
\limits_{\substack{\lambda\vdash n;\\\ell\left(  \lambda\right)  \leq
k}}\ \ \underbrace{\dim\left(  \operatorname*{End}\left(  S^{\lambda}\right)
\right)  }_{\substack{=\left(  \dim\left(  S^{\lambda}\right)  \right)
^{2}=\left(  f^{\lambda}\right)  ^{2}\\\text{(by
(\ref{pf.cor.num-avoid.dimSlam}))}}}=\sum\limits_{\substack{\lambda\vdash
n;\\\ell\left(  \lambda\right)  \leq k}}\left(  f^{\lambda}\right)  ^{2}.
\]
On the other hand, Theorem \ref{thm.row.main} \textbf{(c)} yields
\[
\dim\left(  \mathcal{I}_{k}\right)  =\left\vert \operatorname*{Av}%
\nolimits_{n}\left(  k+1\right)  \right\vert .
\]
Comparing these two equalities, we find $\left\vert \operatorname*{Av}%
\nolimits_{n}\left(  k+1\right)  \right\vert =\sum\limits_{\substack{\lambda
\vdash n;\\\ell\left(  \lambda\right)  \leq k}}\left(  f^{\lambda}\right)
^{2}$. In other words, the number of all permutations $w\in S_{n}$ that avoid
$12\cdots\left(  k+1\right)  $ is $\sum\limits_{\substack{\lambda\vdash
n;\\\ell\left(  \lambda\right)  \leq k}}\left(  f^{\lambda}\right)  ^{2}$.

It remains to prove that this sum also equals $\sum\limits_{\substack{\lambda
\vdash n;\\\lambda_{1}\leq k}}\left(  f^{\lambda}\right)  ^{2}$. But this is
easy: The bijection $\left\{  \lambda\vdash n\right\}  \rightarrow\left\{
\lambda\vdash n\right\}  $ that sends each partition $\lambda$ to its
transpose $\lambda^{t}$ swaps the roles of $\ell\left(  \lambda\right)  $ and
$\lambda_{1}$ (that is, we have $\ell\left(  \lambda^{t}\right)  =\lambda_{1}$
and $\lambda_{1}^{t}=\ell\left(  \lambda\right)  $). Thus, the sums
$\sum\limits_{\substack{\lambda\vdash n;\\\ell\left(  \lambda\right)  \leq
k}}\left(  f^{\lambda}\right)  ^{2}$ and $\sum\limits_{\substack{\lambda\vdash
n;\\\lambda_{1}\leq k}}\left(  f^{\lambda}\right)  ^{2}$ are the same up to
reindexing using this bijection. Hence, Corollary \ref{cor.num-avoid} is proved.
\end{proof}

\subsubsection{$\mathcal{A}/\left(  \mathcal{I}_{k}+T_{\operatorname*{sign}%
}\left(  \mathcal{J}_{\ell}\right)  \right)  $}

Corollary \ref{cor.num-avoid} is not the end of the line. There is a more
general result (\cite[Theorem 3]{Schens60}, \cite[Proposition 17.5]{Stanle71})
saying the following:

\begin{theorem}
\label{thm.num-avoid2}Let $k,\ell\in\mathbb{N}$. Then, the number of all
permutations $w\in S_{n}$ that avoid $12\cdots\left(  k+1\right)  $ and
$\left(  \ell+1\right)  \ell\cdots1$ is $\sum\limits_{\substack{\lambda\vdash
n;\\\ell\left(  \lambda\right)  \leq k\text{ and }\lambda_{1}\leq\ell}}\left(
f^{\lambda}\right)  ^{2}$.
\end{theorem}

\begin{proof}
From \cite[Theorem 3]{Schens60} or \cite[Proposition 17.5]{Stanle71}, we know
that for any $c,d\in\mathbb{N}$, we have%
\begin{align*}
&  \left(  \text{the number of permutations }w\in S_{n}\text{ such
that}\right. \\
&  \ \ \ \ \ \ \ \ \ \ \left.  \text{the longest increasing subsequence of
}w\text{ has length }c\right. \\
&  \ \ \ \ \ \ \ \ \ \ \left.  \text{and the longest decreasing subsequence of
}w\text{ has length }d\right) \\
&  =\sum\limits_{\substack{\lambda\vdash n;\\\ell\left(  \lambda\right)
=d\text{ and }\lambda_{1}=c}}\left(  f^{\lambda}\right)  ^{2}.
\end{align*}
Summing this over all $c\in\left\{  0,1,\ldots,\ell\right\}  $ and all
$d\in\left\{  0,1,\ldots,k\right\}  $, we obtain%
\begin{align*}
&  \left(  \text{the number of permutations }w\in S_{n}\text{ such
that}\right. \\
&  \ \ \ \ \ \ \ \ \ \ \left.  \text{the longest increasing subsequence of
}w\text{ has length }\leq\ell\right. \\
&  \ \ \ \ \ \ \ \ \ \ \left.  \text{and the longest decreasing subsequence of
}w\text{ has length }\leq k\right) \\
&  =\sum\limits_{\substack{\lambda\vdash n;\\\ell\left(  \lambda\right)  \leq
k\text{ and }\lambda_{1}\leq\ell}}\left(  f^{\lambda}\right)  ^{2}.
\end{align*}
But the permutations counted on the left hand side are just the permutations
in $w\in S_{n}$ that avoid $12\cdots\left(  \ell+1\right)  $ and $\left(
k+1\right)  \ell\cdots1$. Upon multiplying them by $w_{0}$ from the left, they
become the permutations $w\in S_{n}$ that avoid $12\cdots\left(  k+1\right)  $
and $\left(  \ell+1\right)  \ell\cdots1$ (by Proposition
\ref{prop.avoid.up-down}). Thus, the number of the latter permutations is
$\sum\limits_{\substack{\lambda\vdash n;\\\ell\left(  \lambda\right)  \leq
k\text{ and }\lambda_{1}\leq\ell}}\left(  f^{\lambda}\right)  ^{2}$. This
proves Theorem \ref{thm.num-avoid2}.
\end{proof}

We may thus wonder:

\begin{question}
\label{quest.num-avoid-gen}Can Theorem \ref{thm.num-avoid2} be proved using
the ideals $\mathcal{I}_{k}$ and $\mathcal{J}_{k}$ ?
\end{question}

One way to approach this is as follows. Let $T_{\operatorname*{sign}%
}:\mathcal{A}\rightarrow\mathcal{A}$ be the $\mathbf{k}$-algebra automorphism
of $\mathcal{A}$ sending each permutation $w\in S_{n}$ to $\left(  -1\right)
^{w}w$. If $n!$ is invertible in $\mathbf{k}$, then $\mathcal{I}_{k}\cap
T_{\operatorname*{sign}}\left(  \mathcal{I}_{\ell}\right)  $ is a free
$\mathbf{k}$-submodule of $\mathcal{A}$ having dimension $\sum
\limits_{\substack{\lambda\vdash n;\\\ell\left(  \lambda\right)  \leq k\text{
and }\lambda_{1}\leq\ell}}\left(  f^{\lambda}\right)  ^{2}$. Thus, we could
answer Question \ref{quest.num-avoid-gen} by finding a $\mathbf{k}$-module
basis of $\mathcal{I}_{k}\cap T_{\operatorname*{sign}}\left(  \mathcal{I}%
_{\ell}\right)  $ indexed by permutations that avoid $12\cdots\left(
k+1\right)  $ and $\left(  \ell+1\right)  \ell\cdots1$. However, such a basis
cannot be independent on $\operatorname*{char}\mathbf{k}$ like our above basis
of $\mathcal{I}_{k}$ was, since (for instance) the subspace $\mathcal{I}%
_{2}\cap T_{\operatorname*{sign}}\left(  \mathcal{I}_{2}\right)  $ for $n=3$
has dimension $4$ when $\mathbf{k}=\mathbb{Q}$ but dimension $5$ when
$\mathbf{k}=\mathbb{F}_{2}$. While this does not strictly rule out a proof
using this subspace, it suggests that new methods are required, since the
techniques we have used to prove Theorem \ref{thm.row.main} above cannot
create a dependence on the characteristic of $\mathbf{k}$. The dimension of
the quotient $\mathbf{k}$-module $\mathcal{A}/\left(  \mathcal{J}%
_{k}+T_{\operatorname*{sign}}\left(  \mathcal{J}_{\ell}\right)  \right)  $ can
depend on $\operatorname*{char}\mathbf{k}$ as well (for $n=3$ and $k=\ell=2$,
it equals $4$ when $\mathbf{k}=\mathbb{Q}$ but equals $5$ when $\mathbf{k}%
=\mathbb{F}_{2}$).

However, the \textquotedblleft mixed\textquotedblright\ quotient
$\mathcal{I}_{k}/\left(  \mathcal{I}_{k}\cap T_{\operatorname*{sign}}\left(
\mathcal{J}_{\ell}\right)  \right)  \cong\left(  \mathcal{I}_{k}%
+T_{\operatorname*{sign}}\left(  \mathcal{J}_{\ell}\right)  \right)
/T_{\operatorname*{sign}}\left(  \mathcal{J}_{\ell}\right)  $ behaves better.
Equivalently, the quotient $\mathcal{A}/\left(  \mathcal{I}_{k}%
+T_{\operatorname*{sign}}\left(  \mathcal{J}_{\ell}\right)  \right)  $ has a
basis of size $\operatorname*{Av}\nolimits_{n}^{\prime}\left(  \ell+1\right)
\setminus\operatorname*{Av}\nolimits_{n}\left(  k+1\right)  $ for all
$\mathbf{k}$:

\begin{theorem}
\label{thm.A/I+J}Let $k,\ell\in\mathbb{N}$. Then, the family $\left(
\overline{w}\right)  _{\operatorname*{Av}\nolimits_{n}^{\prime}\left(
\ell+1\right)  \setminus\operatorname*{Av}\nolimits_{n}\left(  k+1\right)  }$
is a basis of the $\mathbf{k}$-module $\mathcal{A}/\left(  \mathcal{I}%
_{k}+T_{\operatorname*{sign}}\left(  \mathcal{J}_{\ell}\right)  \right)  $.
\end{theorem}

We shall now outline a proof of this theorem (which might not be new; it is
probably a particular case of Donkin's \cite[Corollary 2.12]{Donkin24}).
However, the proof itself relies on Theorem \ref{thm.num-avoid2}, so it
produces no new proof of Theorem \ref{thm.num-avoid2}. Thus, Question
\ref{quest.num-avoid-gen} remains open.

\begin{proof}
[Outline of the proof of Theorem \ref{thm.A/I+J}.]The proof consists of two
steps: \textbf{(a)} showing that the family $\left(  \overline{w}\right)
_{\operatorname*{Av}\nolimits_{n}^{\prime}\left(  \ell+1\right)
\setminus\operatorname*{Av}\nolimits_{n}\left(  k+1\right)  }$ spans the
$\mathbf{k}$-module $\mathcal{A}/\left(  \mathcal{I}_{k}%
+T_{\operatorname*{sign}}\left(  \mathcal{J}_{\ell}\right)  \right)  $, and
\textbf{(b)} showing that this $\mathbf{k}$-module is free of rank $\left\vert
\operatorname*{Av}\nolimits_{n}^{\prime}\left(  \ell+1\right)  \setminus
\operatorname*{Av}\nolimits_{n}\left(  k+1\right)  \right\vert $. Once these
two claims are proved, we will be easily able to conclude the proof using
\cite[Lemma 5.21.9]{sga}.

We begin with step \textbf{(a)}. Thus, we set out to prove that the family
$\left(  \overline{w}\right)  _{\operatorname*{Av}\nolimits_{n}^{\prime
}\left(  \ell+1\right)  \setminus\operatorname*{Av}\nolimits_{n}\left(
k+1\right)  }$ spans the $\mathbf{k}$-module $\mathcal{A}/\left(
\mathcal{I}_{k}+T_{\operatorname*{sign}}\left(  \mathcal{J}_{\ell}\right)
\right)  $. In other words, we must prove that%
\begin{equation}
\overline{u}\in\operatorname*{span}\left(  \left(  \overline{w}\right)
_{\operatorname*{Av}\nolimits_{n}^{\prime}\left(  \ell+1\right)
\setminus\operatorname*{Av}\nolimits_{n}\left(  k+1\right)  }\right)
\ \ \ \ \ \ \ \ \ \ \text{for each }u\in S_{n}. \label{pf.thm.A/I+J.span}%
\end{equation}
To prove this, we proceed by induction on $u$ as in lexicographic order, as in
the proof of (\ref{pf.lem.row.A/I-span.goal}) above. Thus, we let $v\in S_{n}$
be arbitrary, and we assume that (\ref{pf.thm.A/I+J.span}) holds for all
$u<v$; we must then prove that (\ref{pf.thm.A/I+J.span}) holds for $u=v$. We
distinguish between three cases:

\begin{enumerate}
\item If $v\in\operatorname*{Av}\nolimits_{n}^{\prime}\left(  \ell+1\right)
\setminus\operatorname*{Av}\nolimits_{n}\left(  k+1\right)  $, then this claim
is obvious.

\item If $v\in\operatorname*{Av}\nolimits_{n}\left(  k+1\right)  $, then this
claim is proved just as in our above proof of (\ref{pf.lem.row.A/I-span.goal}%
), since $\mathcal{I}_{k}\subseteq\mathcal{I}_{k}+T_{\operatorname*{sign}%
}\left(  \mathcal{J}_{\ell}\right)  $.

\item If $v\notin\operatorname*{Av}\nolimits_{n}^{\prime}\left(
\ell+1\right)  $, then we proceed as follows: Set $v^{\prime}:=w_{0}v$. Then,
$v\notin\operatorname*{Av}\nolimits_{n}^{\prime}\left(  \ell+1\right)  $
entails $v^{\prime}\notin\operatorname*{Av}\nolimits_{n}\left(  \ell+1\right)
$ (by Proposition \ref{prop.avoid.up-down}). Hence, the argument made in the
proof of (\ref{pf.lem.row.A/J-span.goal}) (but applied to $\ell$ and
$v^{\prime}$ instead of $k$ and $v$) shows that
\[
\overline{v^{\prime}}=-\overline{\left(  \text{some permutations }w>v^{\prime
}\right)  }\ \ \ \ \ \ \ \ \ \ \text{in }\mathcal{A}/\mathcal{J}_{\ell}.
\]
Multiplying by $\overline{w_{0}}$ on the left, we transform this into
\begin{align*}
\overline{w_{0}v^{\prime}}  &  =-\overline{\left(  \text{some permutations
}w_{0}w\text{ with }w>v^{\prime}\right)  }\\
&  =-\overline{\left(  \text{some permutations }w^{\prime}<w_{0}v^{\prime
}\right)  }\ \ \ \ \ \ \ \ \ \ \text{in }\mathcal{A}/\mathcal{J}_{\ell}%
\end{align*}
(because multiplying by $w_{0}$ on the left flips the lexicographic order: if
$x>y$ in $S_{n}$, then $w_{0}x<w_{0}y$). Since $w_{0}v^{\prime}=v$ (because
$v^{\prime}=w_{0}v$ but $w_{0}$ is an involution), we can rewrite this as%
\[
\overline{v}=-\overline{\left(  \text{some permutations }w^{\prime}<v\right)
}\ \ \ \ \ \ \ \ \ \ \text{in }\mathcal{A}/\mathcal{J}_{\ell}.
\]
Applying the $\mathbf{k}$-algebra automorphism $T_{\operatorname*{sign}}$ (or,
rather, the $\mathbf{k}$-algebra isomorphism $\mathcal{A}/\mathcal{J}_{\ell
}\rightarrow\mathcal{A}/T_{\operatorname*{sign}}\left(  \mathcal{J}_{\ell
}\right)  $ it induces) to this equality, we obtain%
\[
\overline{v}=-\overline{\left(  \text{some permutations }w^{\prime}<v\right)
}\ \ \ \ \ \ \ \ \ \ \text{in }\mathcal{A}/T_{\operatorname*{sign}}\left(
\mathcal{J}_{\ell}\right)
\]
(we have suppressed the signs $\left(  -1\right)  ^{v}$ and $\left(
-1\right)  ^{w^{\prime}}$ here, since we are dealing with a linear combination
anyway). Thus, this equality also holds in $\mathcal{A}/\left(  \mathcal{I}%
_{k}+T_{\operatorname*{sign}}\left(  \mathcal{J}_{\ell}\right)  \right)  $.
From here, we easily obtain (\ref{pf.thm.A/I+J.span}) using the induction hypothesis.
\end{enumerate}

Thus, (\ref{pf.thm.A/I+J.span}) is proved, and we conclude that the family
$\left(  \overline{w}\right)  _{\operatorname*{Av}\nolimits_{n}^{\prime
}\left(  \ell+1\right)  \setminus\operatorname*{Av}\nolimits_{n}\left(
k+1\right)  }$ spans the $\mathbf{k}$-module $\mathcal{A}/\left(
\mathcal{I}_{k}+T_{\operatorname*{sign}}\left(  \mathcal{J}_{\ell}\right)
\right)  $.

Now we come to step \textbf{(b)} of our plan: showing that this $\mathbf{k}%
$-module is free of rank $\left\vert \operatorname*{Av}\nolimits_{n}^{\prime
}\left(  \ell+1\right)  \setminus\operatorname*{Av}\nolimits_{n}\left(
k+1\right)  \right\vert $.

Here we will need the Murphy cellular bases (see Remark \ref{rmk.IJ=murphy}).
Namely, applying (\ref{eq.rmk.IJ=murphy.J}) to $\ell$ instead of $k$, we
obtain $\mathcal{J}_{\ell}=\MurpF_{\operatorname*{std},\operatorname*{len}%
>\ell}^{-\operatorname*{Col}}$. Thus,%
\begin{align*}
T_{\operatorname*{sign}}\left(  \mathcal{J}_{\ell}\right)   &
=T_{\operatorname*{sign}}\left(  \MurpF_{\operatorname*{std}%
,\operatorname*{len}>\ell}^{-\operatorname*{Col}}\right) \\
&  =\operatorname*{span}\left\{  T_{\operatorname*{sign}}\left(  \nabla
_{V,U}^{\operatorname*{Row}}\right)  \ \mid\ \left(  \lambda,U,V\right)
\in\operatorname*{SBT}\left(  n\right)  \text{ and }\ell\left(  \lambda
\right)  >\ell\right\} \\
&  =\operatorname*{span}\left\{  \nabla_{V\mathbf{r},U\mathbf{r}%
}^{\operatorname*{Row}}\ \mid\ \left(  \lambda,U,V\right)  \in
\operatorname*{SBT}\left(  n\right)  \text{ and }\ell\left(  \lambda\right)
>\ell\right\} \\
&  \ \ \ \ \ \ \ \ \ \ \ \ \ \ \ \ \ \ \ \ \left(
\begin{array}
[c]{c}%
\text{here, }T\mathbf{r}\text{ denotes the transpose of a tableau }T\text{,}\\
\text{and we have used }T_{\operatorname*{sign}}\left(  \nabla_{V,U}%
^{\operatorname*{Row}}\right)  =\nabla_{V\mathbf{r},U\mathbf{r}}%
^{\operatorname*{Row}}\\
\text{(which is proved in \cite[Proposition 6.8.12]{sga})}%
\end{array}
\right) \\
&  =\operatorname*{span}\left\{  \nabla_{V,U}^{\operatorname*{Row}}%
\ \mid\ \left(  \lambda,U,V\right)  \in\operatorname*{SBT}\left(  n\right)
\text{ and }\ell\left(  \lambda^{t}\right)  >\ell\right\} \\
&  \ \ \ \ \ \ \ \ \ \ \ \ \ \ \ \ \ \ \ \ \left(  \text{here, we have
substituted }\left(  \lambda^{t},U\mathbf{r},V\mathbf{r}\right)  \text{ for
}\left(  \lambda,U,V\right)  \right) \\
&  =\operatorname*{span}\left\{  \nabla_{V,U}^{\operatorname*{Row}}%
\ \mid\ \left(  \lambda,U,V\right)  \in\operatorname*{SBT}\left(  n\right)
\text{ and }\lambda_{1}>\ell\right\}
\end{align*}
(since $\ell\left(  \lambda^{t}\right)  =\lambda_{1}$ for each partition
$\lambda$). Adding this to
\[
\mathcal{I}_{k}=\MurpF_{\operatorname*{std},\operatorname*{len}\leq
k}^{\operatorname*{Row}}=\operatorname*{span}\left\{  \nabla_{V,U}%
^{\operatorname*{Row}}\ \mid\ \left(  \lambda,U,V\right)  \in
\operatorname*{SBT}\left(  n\right)  \text{ and }\ell\left(  \lambda\right)
\leq k\right\}  ,
\]
we find%
\begin{align*}
&  \mathcal{I}_{k}+T_{\operatorname*{sign}}\left(  \mathcal{J}_{\ell}\right)
\\
&  =\operatorname*{span}\left\{  \nabla_{V,U}^{\operatorname*{Row}}%
\ \mid\ \left(  \lambda,U,V\right)  \in\operatorname*{SBT}\left(  n\right)
\text{ and }\left(  \ell\left(  \lambda\right)  \leq k\text{ or }\lambda
_{1}>\ell\right)  \right\} \\
&  =\operatorname*{span}\left\{  \nabla_{V,U}^{\operatorname*{Row}}%
\ \mid\ \left(  \lambda,U,V\right)  \in\operatorname*{SBT}\left(  n\right)
\text{ and }\lambda\in X\right\}  ,
\end{align*}
where
\[
X:=\left\{  \lambda\vdash n\ \mid\ \ell\left(  \lambda\right)  \leq k\text{ or
}\lambda_{1}>\ell\right\}  .
\]
This clearly shows that $\mathcal{I}_{k}+T_{\operatorname*{sign}}\left(
\mathcal{J}_{\ell}\right)  $ is a direct addend of $\mathcal{A}$ as a
$\mathbf{k}$-module (since it is spanned by a subfamily of the row Murphy
basis), and is free of rank
\[
\left(  \text{\# of }\left(  \lambda,U,V\right)  \in\operatorname*{SBT}\left(
n\right)  \text{ such that }\lambda\in X\right)  =\sum_{\lambda\in X}\left(
f^{\lambda}\right)  ^{2},
\]
whereas the quotient $\mathbf{k}$-module $\mathcal{A}/\left(  \mathcal{I}%
_{k}+T_{\operatorname*{sign}}\left(  \mathcal{J}_{\ell}\right)  \right)  $ is
free of rank
\begin{align*}
\sum\limits_{\substack{\lambda\vdash n;\\\lambda\notin X}}\left(  f^{\lambda
}\right)  ^{2}  &  =\sum\limits_{\substack{\lambda\vdash n;\\\ell\left(
\lambda\right)  >k\text{ and }\lambda_{1}\leq\ell}}\left(  f^{\lambda}\right)
^{2}=\underbrace{\sum_{\substack{\lambda\vdash n;\\\lambda_{1}\leq\ell
}}\left(  f^{\lambda}\right)  ^{2}}_{\substack{=\left\vert \operatorname*{Av}%
\nolimits_{n}\left(  \ell+1\right)  \right\vert \\\text{(by Corollary
\ref{cor.num-avoid},}\\\text{applied to }\ell\text{ instead of }k\text{)}%
}}-\underbrace{\sum\limits_{\substack{\lambda\vdash n;\\\ell\left(
\lambda\right)  \leq k\text{ and }\lambda_{1}\leq\ell}}\left(  f^{\lambda
}\right)  ^{2}}_{\substack{=\left\vert \operatorname*{Av}\nolimits_{n}\left(
k+1\right)  \cap\operatorname*{Av}\nolimits_{n}^{\prime}\left(  \ell+1\right)
\right\vert \\\text{(by Theorem \ref{thm.num-avoid2})}}}\\
&  =\underbrace{\left\vert \operatorname*{Av}\nolimits_{n}\left(
\ell+1\right)  \right\vert }_{\substack{=\left\vert \operatorname*{Av}%
\nolimits_{n}^{\prime}\left(  \ell+1\right)  \right\vert \\\text{(by
(\ref{pf.cor.row.twin.a.=}))}}}-\left\vert \operatorname*{Av}\nolimits_{n}%
\left(  k+1\right)  \cap\operatorname*{Av}\nolimits_{n}^{\prime}\left(
\ell+1\right)  \right\vert \\
&  =\left\vert \operatorname*{Av}\nolimits_{n}^{\prime}\left(  \ell+1\right)
\right\vert -\left\vert \operatorname*{Av}\nolimits_{n}\left(  k+1\right)
\cap\operatorname*{Av}\nolimits_{n}^{\prime}\left(  \ell+1\right)  \right\vert
\\
&  =\left\vert \operatorname*{Av}\nolimits_{n}^{\prime}\left(  \ell+1\right)
\setminus\operatorname*{Av}\nolimits_{n}\left(  k+1\right)  \right\vert .
\end{align*}

Finally, we are ready for everything to come together. We have shown that the
$\mathbf{k}$-module $\mathcal{A}/\left(  \mathcal{I}_{k}%
+T_{\operatorname*{sign}}\left(  \mathcal{J}_{\ell}\right)  \right)  $ is free
of rank $\left\vert \operatorname*{Av}\nolimits_{n}^{\prime}\left(
\ell+1\right)  \setminus\operatorname*{Av}\nolimits_{n}\left(  k+1\right)
\right\vert $. Hence, \cite[Lemma 5.21.9]{sga} shows that every family of
$\left\vert \operatorname*{Av}\nolimits_{n}^{\prime}\left(  \ell+1\right)
\setminus\operatorname*{Av}\nolimits_{n}\left(  k+1\right)  \right\vert $
vectors that spans this $\mathbf{k}$-module must be a basis of this
$\mathbf{k}$-module. Hence, the family $\left(  \overline{w}\right)
_{\operatorname*{Av}\nolimits_{n}^{\prime}\left(  \ell+1\right)
\setminus\operatorname*{Av}\nolimits_{n}\left(  k+1\right)  }$ is a basis of
this $\mathbf{k}$-module (since we know that it spans it). This completes the
proof of Theorem \ref{thm.A/I+J}.
\end{proof}

\begin{noncompile}
\label{rmk.A/I+J.murphy}Note that the set $X$ in the above proof is a
\textquotedblleft dominance upset\textquotedblright\ (i.e., if a partition
$\lambda\vdash n$ belongs to $X$ and another partition $\mu\vdash n$ dominates
$\lambda$, then $\mu$ also belongs to $X$). This is called a \textquotedblleft
co-saturated subset\textquotedblright\ of $\operatorname*{Par}\left(
n\right)  $ in \cite{Donkin24}. Thus, Conjecture \ref{conj.A/I+J} is likely to
be a particular case of Donkin's conjecture \cite[Remark 2.14]{Donkin24},
which has been verified for all $n\leq7$.
\end{noncompile}

\subsubsection{$T_{\operatorname*{sign}}\left(  \mathcal{I}_{n-k-1}\right)  $
after Artin--Wedderburn}

We end with a brief excursion into some related work by Hamaker and Rhoades
\cite{HamRho25}.

Let $k\in\mathbb{N}$. For any two $k$-tuples $\mathbf{a}=\left(  a_{1}%
,a_{2},\ldots,a_{k}\right)  \in\left[  n\right]  ^{k}$ and $\mathbf{b}=\left(
b_{1},b_{2},\ldots,b_{k}\right)  \in\left[  n\right]  ^{k}$, we define the
element%
\begin{equation}
\nabla_{\mathbf{b},\mathbf{a}}:=\sum_{\substack{w\in S_{n};\\w\left(
a_{i}\right)  =b_{i}\text{ for all }i}}w\in\mathcal{A}. \label{eq.Nabba}%
\end{equation}
Note that the $\nabla_{\mathbf{a}}$ in Lemma \ref{lem.AnnNnk.2} can thus be
rewritten as $\nabla_{\mathbf{a},\mathbf{a}}$. If we equip the set $\left[
n\right]  ^{k}$ with the left $S_{n}$-action given by%
\begin{align*}
w\left(  a_{1},a_{2},\ldots,a_{k}\right)   &  :=\left(  w\left(  a_{1}\right)
,w\left(  a_{2}\right)  ,\ldots,w\left(  a_{k}\right)  \right) \\
&  \ \ \ \ \ \ \ \ \ \ \text{for all }w\in S_{n}\text{ and }\left(
a_{1},a_{2},\ldots,a_{k}\right)  \in\left[  n\right]  ^{k}%
\end{align*}
(as we have already done in the proof of Theorem \ref{thm.AnnNnk}), then we
can rewrite the definition (\ref{eq.Nabba}) of $\nabla_{\mathbf{b},\mathbf{a}%
}$ as%
\begin{equation}
\nabla_{\mathbf{b},\mathbf{a}}:=\sum_{\substack{w\in S_{n};\\w\mathbf{a}%
=\mathbf{b}}}w\in\mathcal{A}. \label{eq.Nabba.2}%
\end{equation}
Thus,%
\begin{equation}
\nabla_{\mathbf{b},\mathbf{a}}=0\ \ \ \ \ \ \ \ \ \ \text{unless }%
\mathbf{b}=w\mathbf{a}\text{ for some }w\in S_{n}. \label{eq.Nabba.0}%
\end{equation}
Now, Lemma \ref{lem.AnnNnk.2} can be restated as follows:

\begin{proposition}
\label{prop.Nabbaspan}Let $k\in\mathbb{N}$. Then,
\[
T_{\operatorname*{sign}}\left(  \mathcal{J}_{n-k-1}\right)
=\operatorname*{span}\left\{  \nabla_{\mathbf{b},\mathbf{a}}\ \mid
\ \mathbf{a},\mathbf{b}\in\left[  n\right]  ^{k}\right\}  .
\]
Here, we understand $\mathcal{J}_{m}$ to mean $\mathcal{A}$ when $m<0$.
\end{proposition}

\begin{proof}
It is easy to see that $w\nabla_{\mathbf{b},\mathbf{a}}=\nabla_{w\mathbf{b}%
,\mathbf{a}}$ for any $w\in S_{n}$ and any $\mathbf{a},\mathbf{b}\in\left[
n\right]  ^{k}$. Applying this to $\mathbf{b}=\mathbf{a}$, we obtain
\begin{equation}
w\nabla_{\mathbf{a}}=\nabla_{w\mathbf{a},\mathbf{a}}
\label{pf.prop.Nabbaspan.1}%
\end{equation}
for any $w\in S_{n}$ and any $\mathbf{a}\in\left[  n\right]  ^{k}$ (since
$\nabla_{\mathbf{a}}=\nabla_{\mathbf{a},\mathbf{a}}$).

But Lemma \ref{lem.AnnNnk.2} yields%
\begin{align*}
T_{\operatorname*{sign}}\left(  \mathcal{J}_{n-k-1}\right)   &
=\underbrace{\mathcal{A}}_{=\operatorname*{span}\left\{  w\ \mid\ w\in
S_{n}\right\}  }\cdot\operatorname*{span}\left\{  \nabla_{\mathbf{a}}%
\ \mid\ \mathbf{a}\in\left[  n\right]  ^{k}\right\} \\
&  =\operatorname*{span}\left\{  w\ \mid\ w\in S_{n}\right\}  \cdot
\operatorname*{span}\left\{  \nabla_{\mathbf{a}}\ \mid\ \mathbf{a}\in\left[
n\right]  ^{k}\right\} \\
&  =\operatorname*{span}\left\{  w\nabla_{\mathbf{a}}\ \mid\ \mathbf{a}%
\in\left[  n\right]  ^{k}\text{ and }w\in S_{n}\right\} \\
&  =\operatorname*{span}\left\{  \nabla_{w\mathbf{a},\mathbf{a}}%
\ \mid\ \mathbf{a}\in\left[  n\right]  ^{k}\text{ and }w\in S_{n}\right\}
\ \ \ \ \ \ \ \ \ \ \left(  \text{by (\ref{pf.prop.Nabbaspan.1})}\right) \\
&  =\operatorname*{span}\left\{  \nabla_{\mathbf{b},\mathbf{a}}\ \mid
\ \mathbf{a},\mathbf{b}\in\left[  n\right]  ^{k}\text{ such that }%
\mathbf{b}=w\mathbf{a}\text{ for some }w\in S_{n}\right\} \\
&  =\operatorname*{span}\left\{  \nabla_{\mathbf{b},\mathbf{a}}\ \mid
\ \mathbf{a},\mathbf{b}\in\left[  n\right]  ^{k}\right\}
\ \ \ \ \ \ \ \ \ \ \left(  \text{by (\ref{eq.Nabba.0})}\right)  .
\end{align*}
This proves Proposition \ref{prop.Nabbaspan}.
\end{proof}

We conclude with the following coda to Theorem \ref{thm.IJ.rep}:

\begin{theorem}
\label{thm.TJ.rep}Assume that $n!$ is invertible in $\mathbf{k}$. Let
$k\in\mathbb{N}$. Let all notations be as in Theorem \ref{thm.IJ.rep}. Then,%
\[
T_{\operatorname*{sign}}\left(  \mathcal{J}_{n-k-1}\right)  =\mathcal{A}%
_{\left\{  \lambda\vdash n\ \mid\ \lambda_{1}\geq n-k\right\}  }.
\]
Here, $\lambda_{1}$ denotes the first entry of $\lambda$.
\end{theorem}

To prove this, we need a general fact about how $T_{\operatorname*{sign}}$
interacts with the Artin--Wedderburn isomorphism $\operatorname*{AW}$:

\begin{proposition}
\label{prop.AW.dual-A}Assume that $n!$ is invertible in $\mathbf{k}$. Let all
notations be as in Theorem \ref{thm.IJ.rep}. Let $X$ be a subset of $\left\{
\lambda\vdash n\right\}  $. Let $\lambda^{t}$ denote the transpose (i.e.,
conjugate) of any partition $\lambda$. Then,%
\[
T_{\operatorname*{sign}}\left(  \mathcal{A}_{X}\right)  =\mathcal{A}_{X^{t}},
\]
where $X^{t}$ denotes the subset $\left\{  \lambda^{t}\ \mid\ \lambda\in
X\right\}  =\left\{  \lambda\vdash n\ \mid\ \lambda^{t}\in X\right\}  $ of
$\left\{  \lambda\vdash n\right\}  $.
\end{proposition}

\begin{proof}
[Proof of Proposition \ref{prop.AW.dual-A}.]Let $\lambda$ be a partition of
$n$. It is well-known (e.g., \cite[Theorem 5.18.13]{sga}) that the sign-twist
of the Specht module $S^{\lambda}$ satisfies $\left(  S^{\lambda}\right)
^{\operatorname*{sign}}\cong S^{\lambda^{t}}$. Thus, in turn, $S^{\lambda
}\cong\left(  S^{\lambda^{t}}\right)  ^{\operatorname*{sign}}$ (since each
$S_{n}$-representation $V$ satisfies $\left(  V^{\operatorname*{sign}}\right)
^{\operatorname*{sign}}=V$). Hence, for any $a\in\mathcal{A}$, we have the
equivalence%
\begin{align}
\left(  aS^{\lambda}=0\right)  \  &  \Longleftrightarrow\ \left(  a\left(
S^{\lambda^{t}}\right)  ^{\operatorname*{sign}}=0\right) \nonumber\\
&  \Longleftrightarrow\ \left(  T_{\operatorname*{sign}}\left(  a\right)
S^{\lambda^{t}}=0\right)  \label{pf.prop.AW.dual-A.equiv}%
\end{align}
(because the definition of $\left(  S^{\lambda^{t}}\right)
^{\operatorname*{sign}}$ reveals that $a\left(  S^{\lambda^{t}}\right)
^{\operatorname*{sign}}=T_{\operatorname*{sign}}\left(  a\right)
S^{\lambda^{t}}$ as $\mathbf{k}$-modules). Now, (\ref{pf.prop.IJ.rep.AU=})
yields%
\begin{align*}
\mathcal{A}_{X}  &  =\left\{  a\in\mathcal{A}\ \mid\ aS^{\lambda}=0\text{ for
all }\lambda\notin X\right\} \\
&  =\left\{  a\in\mathcal{A}\ \mid\ T_{\operatorname*{sign}}\left(  a\right)
S^{\lambda^{t}}=0\text{ for all }\lambda\notin X\right\}
\ \ \ \ \ \ \ \ \ \ \left(  \text{by (\ref{pf.prop.AW.dual-A.equiv})}\right)
\\
&  =\left\{  a\in\mathcal{A}\ \mid\ T_{\operatorname*{sign}}\left(  a\right)
S^{\lambda}=0\text{ for all }\lambda\notin X^{t}\right\} \\
&  \ \ \ \ \ \ \ \ \ \ \ \ \ \ \ \ \ \ \ \ \left(
\begin{array}
[c]{c}%
\text{here, we have substituted }\lambda\text{ for }\lambda^{t}\text{,}\\
\text{since }X^{t}=\left\{  \lambda^{t}\ \mid\ \lambda\in X\right\}
\end{array}
\right) \\
&  =T_{\operatorname*{sign}}^{-1}\left(  \underbrace{\left\{  a\in
\mathcal{A}\ \mid\ aS^{\lambda}=0\text{ for all }\lambda\notin X^{t}\right\}
}_{\substack{=\mathcal{A}_{X^{t}}\\\text{(by (\ref{pf.prop.IJ.rep.AU=}),
applied to }U=X^{t}\text{)}}}\right) \\
&  =T_{\operatorname*{sign}}^{-1}\left(  \mathcal{A}_{X^{t}}\right)  .
\end{align*}
Since $T_{\operatorname*{sign}}$ is an isomorphism, we thus conclude that
$T_{\operatorname*{sign}}\left(  \mathcal{A}_{X}\right)  =\mathcal{A}_{X^{t}}%
$. This proves Proposition \ref{prop.AW.dual-A}.
\end{proof}

\begin{proof}
[Proof of Theorem \ref{thm.TJ.rep}.]The case $k\geq n$ is left to the reader
(both sides are $0$). In the case $k<n$, we have $n-k-1\in\mathbb{N}$, so that
we can apply Theorem \ref{thm.IJ.rep} to $n-k-1$ instead of $k$. Thus, from
the last equality of Theorem \ref{thm.IJ.rep}, we find%
\begin{equation}
\mathcal{J}_{n-k-1}=\mathcal{A}_{\left\{  \lambda\vdash n\ \mid\ \ell\left(
\lambda\right)  >n-k-1\right\}  }=\mathcal{A}_{\left\{  \lambda\vdash
n\ \mid\ \ell\left(  \lambda\right)  \geq n-k\right\}  }.
\label{pf.thm.TJ.rep.1}%
\end{equation}

Now, let us use the notations of Proposition \ref{prop.AW.dual-A}. Applying
Proposition \ref{prop.AW.dual-A} to $X=\left\{  \lambda\vdash n\ \mid
\ \ell\left(  \lambda\right)  \geq n-k\right\}  $, we obtain%
\begin{align*}
T_{\operatorname*{sign}}\left(  \mathcal{A}_{\left\{  \lambda\vdash
n\ \mid\ \ell\left(  \lambda\right)  \geq n-k\right\}  }\right)   &
=\mathcal{A}_{\left\{  \lambda\vdash n\ \mid\ \ell\left(  \lambda\right)  \geq
n-k\right\}  ^{t}}\\
&  =\mathcal{A}_{\left\{  \lambda\vdash n\ \mid\ \ell\left(  \lambda
^{t}\right)  \geq n-k\right\}  }\ \ \ \ \ \ \ \ \ \ \left(
\begin{array}
[c]{c}%
\text{since }X^{t}=\left\{  \lambda\vdash n\ \mid\ \lambda^{t}\in X\right\} \\
\text{for each }X\subseteq\left\{  \lambda\vdash n\right\}
\end{array}
\right) \\
&  =\mathcal{A}_{\left\{  \lambda\vdash n\ \mid\ \lambda_{1}\geq n-k\right\}
},
\end{align*}
because it is well-known (see, e.g., \cite[Theorem 5.1.10 \textbf{(c)}]{sga})
that every partition $\lambda$ satisfies $\ell\left(  \lambda^{t}\right)
=\lambda_{1}$. In light of (\ref{pf.thm.TJ.rep.1}), we can rewrite this as%
\[
T_{\operatorname*{sign}}\left(  \mathcal{J}_{n-k-1}\right)  =\mathcal{A}%
_{\left\{  \lambda\vdash n\ \mid\ \lambda_{1}\geq n-k\right\}  }.
\]
Theorem \ref{thm.TJ.rep} is thus proved.
\end{proof}

Combining Theorem \ref{thm.TJ.rep} with Proposition \ref{prop.Nabbaspan}, we
conclude that%
\begin{equation}
\operatorname*{span}\left\{  \nabla_{\mathbf{b},\mathbf{a}}\ \mid
\ \mathbf{a},\mathbf{b}\in\left[  n\right]  ^{k}\right\}  =\mathcal{A}%
_{\left\{  \lambda\vdash n\ \mid\ \lambda_{1}\geq n-k\right\}  }
\label{eq.thm.TJ.rep.cor}%
\end{equation}
for any $k\in\mathbb{N}$, under the assumption that $n!$ is invertible in
$\mathbf{k}$. This is essentially \cite[Theorem 3.5]{HamRho25}.
(\textquotedblleft Essentially\textquotedblright\ because the
$\operatorname*{Loc}\nolimits_{k}\left(  \mathfrak{S}_{n},\mathbb{C}\right)  $
from \cite{HamRho25} is not $\operatorname*{span}\left\{  \nabla
_{\mathbf{b},\mathbf{a}}\ \mid\ \mathbf{a},\mathbf{b}\in\left[  n\right]
^{k}\right\}  $ but rather $\operatorname*{span}\left\{  \nabla_{\mathbf{b}%
,\mathbf{a}}\ \mid\ \mathbf{a},\mathbf{b}\in\left[  n\right]  ^{k}\text{
injective}\right\}  $, where a $k$-tuple is said to be \emph{injective} if its
$k$ entries are distinct.) The same result appears in \cite[Theorem
7]{ElFrPi11} (where, again, the \textquotedblleft$k$-cosets\textquotedblright%
\ are the $\nabla_{\mathbf{b},\mathbf{a}}$ for injective $k$-tuples
$\mathbf{a},\mathbf{b}\in\left[  n\right]  ^{k}$). Note that the condition
\textquotedblleft$\lambda_{1}\geq n-k$\textquotedblright\ on a partition
$\lambda$ of $n$ is equivalent to \textquotedblleft$\lambda\geq\left(
n-k,1^{k}\right)  $ in lexicographic order\textquotedblright, and this is the
condition used in \cite[Theorem 7]{ElFrPi11}.

It is perhaps also worth mentioning that \cite[Lemma 3]{EveZoh20} is saying
that if $\mathbf{k}$ is a field of characteristic $0$, and if $l\in\left[
n\right]  $, then%
\begin{equation}
\mathcal{A}_{\left\{  \lambda\vdash n\ \mid\ \lambda_{1}<l\right\}  }=\left(
\mathcal{A}\nabla_{\left[  l\right]  }\mathcal{A}\right)  ^{\perp},
\end{equation}
where $\nabla_{\left[  l\right]  }$ is the sum of all permutations $w\in
S_{n}$ that fix the elements $l+1,l+2,\ldots,n$. This can be derived from
Theorem \ref{thm.IJ.rep}, Theorem \ref{thm.row.main} \textbf{(a)} and
Proposition \ref{prop.IJ.1} \textbf{(f)} using Proposition
\ref{prop.AW.dual-A} (and the fact that $\nabla_{\left[  l\right]
}=T_{\operatorname*{sign}}\left(  \nabla_{\left[  l\right]  }^{-}\right)  $).
We leave the details to the reader.

\bigskip

\begin{noncompile}
Henry Potts-Rubin, \textit{The Planar Rook Algebra}, senior thesis at the
College of Wooster, 2020.\newline\url{https://openworks.wooster.edu/independentstudy/9095/}
\end{noncompile}

\begin{noncompile}
\href{https://www.jmlr.org/papers/volume10/huang09a/huang09a.pdf}{Jonathan
Huang, Carlos Guestrin, Leonidas Guibas, \textit{Fourier Theoretic
Probabilistic Inference over Permutations}, Journal of Machine Learning
Research \textbf{10} (2009), pp. 997--1070.}
\end{noncompile}

\end{document}